\documentclass[draft,12pt]{article}%
\usepackage{physics}
\usepackage{amscd}
\usepackage{amsmath}
\usepackage{amsfonts}
\usepackage{amssymb}
\usepackage{graphicx}
\usepackage[T1]{fontenc}
\usepackage[margin=0.5in]{geometry}
\usepackage{color}%
\providecommand{\U}[1]{\protect\rule{.1in}{.1in}}
\providecommand{\U}[1]{\protect\rule{.1in}{.1in}}
\providecommand{\U}[1]{\protect\rule{.1in}{.1in}}
\providecommand{\U}[1]{\protect\rule{.1in}{.1in}}
\providecommand{\U}[1]{\protect\rule{.1in}{.1in}}

\numberwithin{equation}{section}
\newtheorem{theorem}{Theorem}[section]
\newtheorem{corollary}[theorem]{Corollary}
\newtheorem{definition}[theorem]{Definition}
\newtheorem{lemma}[theorem]{Lemma}

\newtheorem{assumption}[theorem]{Assumption}

\newtheorem{remark}[theorem]{Remark}

\newenvironment{proof}[1][Proof]{\noindent\textbf{#1.} }{\ \rule{0.5em}{0.5em}}

\newcommand{\ere}{ {\mathbb R}}

\newcommand{\beq}{\begin{equation}}
\newcommand{\ene}{\end{equation}}
\newcommand{\er}{\mathbb R^+}
\def\p2{\mathcal A_{\Phi,2\pi}(B)}
\def\0p2{\mathcal A_{\Phi,2\pi}(0)}
\def\sp2{\mathcal A_{\Phi,2\pi,\hbox{\rm SR}}(B)}
\def\beq{\begin{equation}}
\def\ene{\end{equation}}

\def\qed{\ifhmode\unskip\nobreak\fi\ifmmode\ifinner
\else\hskip5pt\fi\fi\hbox{\hskip5pt\vrule width4pt height6pt
depth1.5pt\hskip1pt}}

\def\+out{x^{\rm out}}

\begin{document}

\title{The Matrix Nonlinear Schr\"{o}dinger Equation with a Potential.
\thanks{Research partially supported by projects PAPIIT-DGAPA UNAM IN103918,
IN100321, IA100422, and CONACYT-FORDECYT-PRONACES 429825/2020. } }
\author{Ivan Naumkin\thanks{Fellow, Sistema Nacional de Investigadores. ivan.naumkin@iimas.unam.mx} and Ricardo Weder\thanks{Emeritus Fellow, 
Sistema Nacional de Investigadores. weder@unam.mx. www.iimas.unam.mx/rweder/rweder.html }\\Departamento de F\'{\i}sica Matem\'{a}tica,\\Instituto de Investigaciones en Matem\'{a}ticas Aplicadas y en Sistemas.\\Universidad Nacional Aut\'{o}noma de M\'{e}xico,\\Apartado Postal 20-126, Ciudad de M\'{e}xico, 01000, M\'{e}xico.}
\date{}
\maketitle

\begin{abstract}
\noindent This paper is devoted to the study of the large-time asymptotics of
the small solutions to the matrix nonlinear Schr\"{o}dinger equation with a
potential on the half-line and with general selfadjoint boundary condition,
and on the line with a potential and a general point interaction, in the whole
supercritical regime. We prove that the small solutions are \textit{scattering
solutions} that asymptotically in time, $t \to\pm\infty, $ behave as solutions
to the associated linear matrix Schr\"odinger equation with the potential
identically zero. The potential can be either \textit{generic or exceptional}.
Our approach is based on detailed results on the spectral and scattering
theory for the associated linear matrix Schr\"{o}dinger equation with a
potential, and in a factorization technique that allows us to control the
large-time behaviour of the solutions in appropriate norms.

\end{abstract}

\section{\label{S1}Introduction}

In this article we study the large-time asymptotics of the small solutions to
the nonlinear Schr\"{o}dinger equation on the half-line with general
selfadjoint boundary condition,
\begin{align}
&  i\partial_{t}u=-\partial_{x}^{2}u+Vu+\mathcal{N}\left(  \left\vert
u_{1}\right\vert ,...,\left\vert u_{n}\right\vert \right)  u,\ t\in
\mathbb{R},\,x\in\mathbb{R}^{+},\label{1.1}\\
&  u\left(  0,x\right)  =u_{0}\left(  x\right)  ,\text{ }x\in\mathbb{R}%
^{+},\label{1.1.ini}\\
&  -B^{\dagger}u\left(  t,0\right)  +A^{\dagger}\left(  \partial_{x}u\right)
\left(  t,0\right)  =0, \label{1.1.bound}%
\end{align}
where $\mathbb{R}^{+}:=(0,+\infty),$ $u(t,x)=\left(  u_{1}\left(  t,x\right)
,...,u_{n}\left(  t,x\right)  \right)  $ is a function from $\mathbb{R}%
\times\mathbb{R}^{+}$ into $\mathbb{C}^{n},$ $A,B$ are constant $n\times n$
matrices, the potential $V(x)$ is an $n\times n$ selfadjoint matrix-valued
function of $x$, i.e.%
\begin{equation}
V\left(  x\right)  =V^{\dagger}\left(  x\right)  ,\text{ }x\in\mathbb{R}^{+},
\label{PotentialHermitian}%
\end{equation}
where the dagger denotes the matrix adjoint and $\mathcal{N}$ is a $n\times n
$ matrix-valued function. The more general selfadjoint boundary condition at
$x=0$ can be written as in \eqref{1.1.bound} where the constant matrices $A;B
$ satisfy, (\cite{WederBook})
\begin{equation}
B^{\dagger}A=A^{\dagger}B, \label{wcon1}%
\end{equation}
and%
\begin{equation}
A^{\dagger}A+B^{\dagger}B>0. \label{wcon2}%
\end{equation}
The theory of the matrix Schr\"{o}dinger equation has its origin at the very
beginning of quantum mechanics. It is important to consider particles with
internal structure such as spin and isospin, and systems of particles such as
collections of atoms molecules and nuclei. A well known example is the Pauli
equation for a half-spin particle. The matrix Schr\"{o}dinger equation is also
relevant in the theory of quantum graphs, that consist of edges that meet at
vertices. The dynamics at each edge is governed by the Schr\"{o}dinger
equation. The matrix Schr\"{o}dinger equation corresponds to a star graph,
that has one vertex and a finite number of semi-infinite edges. Quantum graphs
are important in many problems. For example, in quantum wires, in the design
of elementary gates in quantum computing, and in nanotubes for microscopic
electronic devices. The consideration of general boundary conditions at the
vertices is relevant. For quantum graphs it is crucial that the boundary
condition links the values of the wave-functions and its derivative arriving
from different edges. For our purposes it is also crucial that we study the
general boundary condition because it allows us to consider the case of the
line by means of the unitary equivalence between a $2n\times2n$ system on the
half-line with a potential, and a general boundary condition, and a $n\times n$
system on the line with a potential and a point interaction. The case when the
system on the line has no point interaction appears as a particular case of
the boundary condition of the system on the half-line. General references for
quantum graphs are, for example, the monographs \cite{Berkolaio} and
\cite{hora}. For more information on matrix Scr\"{o}dinger equations and
related problems, as well as for a detailed discussion of the literature see
\cite{WederBook}.

After the seminal work of I. E. Sigal \cite{Segal1}, \cite{Segal2},
\cite{Segal3}, there is a very extensive mathematical literature in nonlinear
evolution equations, and in particular in nonlinear Schr\"odinger,
Klein-Gordon and wave equations. For general references see, for example, the
monographs, \cite{bourgain}, \cite{Cazenave}, \cite{Ginibre}, \cite{racke},
\cite{straussl}, and \cite{Sulem}. General references in the case of
integrable equations, solitons, and the inverse scattering transform are, for
example, \cite{ablow}, \cite{segur}, \cite{fokas}, and \cite{nov}.
A paradigmatic model of nonlinear evolution equation is the following
nonlinear Schr\"odinger equation,%

\begin{equation}
i\partial_{t}u= -\Delta u+ \lambda\left\vert u\right\vert ^{\alpha}u,\text{
}t\in\mathbb{R},\text{ }x\in\mathbb{R}^{d}, \lambda\in\mathbb{R}.
\label{FreeNLS}%
\end{equation}
The qualitative description of the solutions to (\ref{FreeNLS}) is far from
being complete. Up to now, various classes of solutions are known. On the one
hand, there exist the so-called low-energy scattering solutions, which
asymptotically in time $t\rightarrow\pm\infty$ behave as solutions to the
linear Schr\"{o}dinger equation $i\partial_{t}u+\Delta u=0$ (\cite{CazenaveW1,
CN, CN1, GinibreV1, GinibreV2,GinibreOV,
haka1998,HayashiNau1,HayashiNau2,HNST,H-O,NakanishiO,Ozawa1,Strauss,
Strauss2}). On the other hand, there are multisoliton like solutions, which
asymptotically behave like decoupled non trivial trains of solitons. The
soliton like solutions were first explicitly computed in the completely
integrable case $d=1$, $\alpha=2$, and then systematically constructed in more
general settings in \cite{Martel1,MartelMerle,Merle}. A general approach to
the scattering problem for nonlinear equations is presented in \cite{Segal1}%
-\cite{Segal3}, \cite{Strauss,Strauss2} (see also
\cite{Cazenave,Ginibre,Sulem}).

Concerning the low-energy scattering of solutions, the cubic nonlinearity
$\alpha=2$ in dimension $d=1$ is a limiting case. The following is known (see,
for example, \cite{CN,CN1} for an overview, and the references there in). In
the case when $\alpha>2$, there is low-energy scattering, that is, a solution
to~(\ref{FreeNLS}) with sufficiently small initial value (in some appropriate
sense) behaves similar to a solution to the linear Schr\"{o}dinger equation
when $t \rightarrow\pm\infty$. On the other hand, if $\alpha\leq2$, the
low-energy scattering for (\ref{FreeNLS}) cannot be expected \cite{Barab}. In
the limiting case $\alpha=2$ there is modified low-energy scattering instead,
i.e. small (in an appropriate sense) solutions behave like solutions to the
linear Schr\"{o}dinger equation modulated by a phase when $t \rightarrow
\pm\infty$.

A natural generalization of (\ref{FreeNLS}) for $d=1,$ is the nonlinear
Schr\"{o}dinger equation with a potential
\begin{equation}
\left\{
\begin{array}
[c]{c}%
i\partial_{t}u=-\partial_{x}^{2}u+Vu+\lambda\left\vert u\right\vert ^{\alpha
}u,\text{ }t\in\mathbb{R},\text{ }x\in\mathbb{R}, \lambda\in\mathbb{R},\\
u\left(  0,x\right)  =u_{0}\left(  x\right)  ,\text{ }x\in\mathbb{R},
\end{array}
\right.  \label{NLS}%
\end{equation}
where the potential $V\left(  x\right)  $ is a real-valued function. This
equation is related to the Ginzburg-Landau equation of superconductivity
\cite{deGennes}, to one-dimensional self-modulation of a monochromatic wave
\cite{Yajima}, \cite{Taniuti}, propagation of a heat pulse in a solid Langmuir
waves in plasmas \cite{Shimizu}, the self-trapping phenomena of nonlinear
optics \cite{Karpman}, and stationary two-dimensional self-focusing of a plane
wave \cite{Bespalov}. Equation (\ref{NLS}) is an interesting mathematical
question that requires the development of new methods. Besides, this type of
problems arises in the asymptotic stability for special solutions of nonlinear
dispersive equations, such as solitons, traveling waves, kinks. The nonlinear
equations with external potentials appear in the analysis of the full problem
around these special solutions. Equation (\ref{NLS}) is a prototypical model
of a nonlinear system in a presence of external potentials. We refer to
\cite{Cuccagna1} for a nice survey on asymptotic stability of ground states of
nonlinear Schr\"{o}dinger equations. The first results in low-energy
scattering for (\ref{NLS}) in dimension $d=1$ were obtained in
\cite{Weder2000}, and \cite{Weder2001}, where general non-homogenous
nonlinearities were considered, and also the inverse scattering problem of
uniquely reconstruction the nonlinearity and the potential was solved. In
\cite{Weder2000, Weder2001} it is proved that the nonlinear scattering
operator is a homeomorphism from some neighborhood of $0$ in $L^{2}%
({\mathbb{R}})$ onto itself for a class of weighted $L^{1}({\mathbb{R}})$
potentials. In the case of \eqref{NLS} the results of \cite{Weder2000,
Weder2001} correspond to $4\leq\alpha<\infty.$ See also \cite{WederPAMS} for
the case of two or more dimensions. For $2<\alpha<4$, in the case of
\textit{generic potentials,} the existence of free scattering is proved in
\cite{Cuccagna}. The existence of modified scattering in the critical case
$\alpha=2$ for \textit{generic potentials} was studied in \cite{Delort,
Germain,Ivan}. Whereas in \cite{Pusateri,Delort, Ivan1}, the same problem is
studied for the case of \textit{exceptional potentials} under further symmetry
assumptions on the initial data and the potential. The case of a potential
that gives rise to a partial quadratic confinement is considered in
\cite{Carles3}. Up to now, as far as we know, the case of \textit{general
exceptional potentials} remains open and it is not known whether the symmetry
assumptions needed in \cite{Pusateri, Delort, Ivan1} are technical or not. The
scattering problem for nonlinear dynamics in presence of trapping potentials,
i.e. potentials that produce negative eigenvalues, is an interesting and
important problem. In \cite{Wedercenter}, for the nonlinear Schr\"{o}dinger
equation on the line, with a potential that produces one eigenvalue, and a
general nonhomogeneous nonlinearity, a center manifold was constructed, and
the asymptotic stability of nonlinear bound states was proven. For the
nonlinear Schr\"{o}dinger equation on the line, in the case when $\alpha>4,$
and when the potential has exactly one negative eigenvalue, the asymptotic
stability of small solitary waves was studied in \cite{Mizumachi}. In the
critical case $\alpha=2,$ the modified scattering in presence of one bound
state for \textit{generic potentials} was recently studied in \cite{Chen}, and
in presence of any number of bound states in \cite{Cuccagna2}. It is also
worth to mention that the long-time behavior of solutions to a focusing
one-dimensional nonlinear Schr\"{o}dinger equation with a cubic nonlinearity
and a Dirac potential was studied in \cite{deift}. In the related problem of the nonlinear Klein-Gordon equation with a
potential, the first results in direct and inverse scattering for small
solutions were obtained in \cite{Wederkg1} and \cite{Wederkg}. For recent
results in this direction see, \cite{Soffer} and \cite{Soffer1}.

In this paper we study the large-time asymptotics of small solutions in the
whole supercritical case ($\alpha>2$) for the nonlinear matrix Schr\"{o}dinger
equation with a potential on the half-line (\ref{1.1})-\eqref{1.1.bound} for
both \textit{generic and exceptional potentials}. Besides being important on
its own, as mentioned above, we also have the following motivation to consider
this problem. For the nonlinear scalar Schr\"{o}dinger equation on the line,
all of the results mentioned above, \cite{Pusateri, Delort, Ivan1}, for the
case of \textit{exceptional potentials} use in one way or another the symmetry
of the equation, and in particular of the potential, and the symmetry or antisymmetry of the solution, with respect to $x=0$. A useful way to understand
the nonlinear Schr\"{o}dinger equation on the line without any symmetry
assumptions, is to break the symmetry restricting the problem to the half-line
and to consider the large-time asymptotics of the solutions to the matrix
nonlinear Schr\"{o}dinger equation on the half-line with \textit{generic and
exceptional potentials}, and with general boundary condition at $x=0.$ Then, as
mentioned above, we obtain our results for the matrix Schr\"{o}dinger equation
on the line using the unitary equivalence between $2n\times2n$ matrix
Scr\"{o}dinger equations on the half-line with general self-adjoint boundary
condition, and $n\times n$ matrix Schr\"{o}dinger equations on the line with
point interactions. As we already remarked, this is another important reason
to study (\ref{1.1})-\eqref{1.1.bound} with general selfadjoint boundary condition. Up to our
knowledge, the only previous results for (\ref{1.1})-\eqref{1.1.bound} were
obtained in \cite{Wederhalf}, and \cite{Wederhalfscat} in the scalar forced
case, and in particular for the Dirichlet boundary condition. In these papers a
general nonhomogenous nonlinearity was considered. In \cite{Wederhalf} the
Cauchy problem was studied, and in \cite{Wederhalfscat} it is proved that the
nonlinear scattering operator is a homeomorphism from some neighborhood of $0$
in $L^{2}(\mathbb{R}^{+})$ onto itself for a class of weighted $L^{1}%
(\mathbb{R}^{+})$ potentials, and also the inverse scattering problem of the
uniquely reconstruction the nonlinearity and the potential was solved.
Recently, \cite{RicardoIvan} obtained dispersive estimates for the linear
matrix Schr\"{o}dinger equation on the half-line with general selfadjoint
boundary condition, and on the line with point interactions, and with
potentials that are integrable and have a finite first moment. Using these
estimates it is possible to study the scattering operator for small solutions
of (\ref{1.1})-\eqref{1.1.bound} for nonlinear interactions of order
$\alpha\geq4,$ using the methods of \cite{Wederhalfscat}. However, here we are
interested in nonlinear interactions in the whole supercritical regime
$\alpha>2.$

\subsection{Main results.}

Before we state our main results we introduce some notations that we use.

\subsubsection{Notation.}

By $\mathbb{C}$ we denote the complex numbers. For a vector $u= (u_1,\dots, u_n)$ in
$\mathbb{C}^{n},$ we denote by $u_{j},j=1,\dots,n,$ its components. For any
matrix $ L $ we denote by $L^{T}$ its transpose. For a linear operator $L$ in
a Hilbert space we denote by $D[L]$ its domain, by $L^{\dagger}$ the adjoint
of $L$ and by $|L|$ the norm of $L.$ We also use this notation in the particular
case where $L$ is a $n\times n$ matrix. For an open set $U$ of real numbers we
denote by $L^{p}(U;\mathbb{C}^{n})$, for $1\leq p\leq\infty,$ the standard
Lebesgue spaces of $\mathbb{C}^{n}$ valued functions. Further, we designate by
$C^{m}(U;\mathbb{C}^{n})$ the functions with continuous derivatives up to
order $m,m=0,1\dots.$  In the case $m=0$ we use the notation  $C(U; \mathbb C^n)= C^0(U; \mathbb C^n).$  For an integer $m\geq0$ and a real number $1\leq
p\leq\infty,$ $W_{m,p}\left(  U;\mathbb{C}^{n}\right)  $ denotes the standard
Sobolev space. See e.g.~\cite{adams} for the definitions and properties of
these spaces. The space $W_{m,p}^{\left(  0\right)  }\left(  U;\mathbb{C}%
^{n}\right)  $ is the closure of $C_{0}^{\infty}\left(  U;\mathbb{C}%
^{n}\right)  $ in the space $W_{m,p}\left(  U;\mathbb{C}^{n}\right)  $. Here,
$C_{0}^{\infty}\left(  U;\mathbb{C}^{n}\right)  $ designates the infinitely
differentiable functions with compact support in $U $. We write $H^{m}\left(
U;\mathbb{C}^{n}\right)  :=W_{m,2}\left(  U;\mathbb{C}^{n}\right) .$The Japanese
brackets are defined as $\left\langle x\right\rangle =\sqrt{1+\left\vert
x\right\vert ^{2}}.$  The weighted Sobolev spaces are defined by
\[
H_{q}^{m}\left(  U;\mathbb{C}^{n}\right)  =\{u\in L^{2}(U;\mathbb{C}%
^{n})| \langle x\rangle^{q}u \in H^{m}\left(
U;\mathbb{C}^{n}\right) \},q>0.
\]
In the case where $m=0$ we use the notation $L_{q}^{2}\left(  U;\mathbb{C}%
^{n}\right)  :=H^{0,q}\left(  U;\mathbb{C}^{n}\right)  .$ 
If there is no place for misunderstanding, we shall omit $\mathbb{C}^{n}$ in writing
the above spaces. 
 We denote by $\mathcal L_n$ the space of all $ n \times n$ matrices, and by  $L^1( U; \mathcal L_n)$ the  standard Lebesgue space of $\mathcal L_n$ valued functions.  We also use the
weighted Lebesgue spaces,
\[
L_{q}^{1}(U;\mathcal{L}^{n}):=\left\{  u|(1+x)^{q}u(x)\in
L^{1}( U;\mathcal{L}_{n}),q>0\right\}  .
\]
When it is clear from the context, we ommity $\mathcal L_n$  in the notation for these spaces.  We often
take $U={\mathbb{R}}^{+}:=(0,\infty),$ or $U={\mathbb{R}}$. 
Further, for a Banach space $\mathcal B,$ a set  $G \subset \mathbb R^l, l=1, \dots  $, and  $m=0,\dots$ we designate by
 $C^m(G; \mathcal B)$ the continuous function from $G$ into $\mathcal B$ that have $m$ continuous derivatives. In the case $m=0$ we use the notation  $C(G; \mathcal B)= C^0(G; \mathcal B).$

We denote the Fourier transform by,
\[
\mathcal{F}f:=\int_{\mathbb{R}}\,e^{ikx}\,f(x)\,dx,
\]
and the inverse Fourier transform by,
\[
\mathcal{F}^{-1}f:=\frac{1}{2\pi}\,\int_{\mathbb{R}}\,e^{-ikx}\,f(k)\,dk.
\]
These formulae make sense for $f\in L^{1}(\mathbb{R}).$ For $f \in
L^{2}(\mathbb{R})$ they are defined by a limiting process. By $C$ we denote a
positive constant that does not have to take the same value when it appears in
different places. For any set $G$ of real numbers we denote by $\chi_{G}(x)$
the characteristic function of $G,$ $\chi_{G}(x)=1,$ $x\in G,$ $\chi
_{G}(x)=0,$ $x\notin G.$ We define
\begin{equation}
\label{final1}M=\displaystyle e^{ \frac{ix^{2}}{4t}},
\end{equation}
and%
\begin{equation}
\label{final2}\mathcal{D}_{t}\phi=\left(  it\right)  ^{-\frac{1}{2}}%
\phi\left(  xt^{-1}\right)  .
\end{equation}
For later use we introduce the following notation,
\begin{equation}
\label{final3}\mathcal{T}_{+}(a,T):= [a,T), \quad\mathcal{T}_{-}(a,T):= (-T,
-a], \qquad0 \leq a < T \leq\infty.
\end{equation}

\subsubsection{Statement of our main results.}

We now state our main results. First, we define the class of potentials that we
 consider in this paper. Let $V(x),x\in\lbrack0,\infty)$ be a $n\times n$
matrix-valued function. For some $N\geq1,$ let us partition $[0,\infty)$ into
a finite union of intervals, $[0,\infty)=\cup_{j=0}^{N}\,I_{j},$ where
$I_{j}:=[x_{j},x_{j+1}),$ $j=0,\dots,N-1,$ with $x_{0}=0<x_{1}<\dots
x_{N-1}<x_{N} < \infty,$ and $I_{N}:=[x_{N},\infty).$ We obtain a fragmentation of the
potential $V$ by setting
\begin{equation}
V\left(  x\right)  =\sum_{j=0}^{N}V_{j}\left(  x\right)  , \label{3.23}%
\end{equation}
where
\[
V_{j}\left(  x\right)  =\left\{
\begin{array}
[c]{c}%
V\left(  x\right)  ,\text{ \ }x\in I_{j},\\
0,\text{ \ elsewhere.}%
\end{array}
\right.
\]

\begin{definition}{\rm
\textrm{\label{Def1}The potential $V$ admits a \textit{regular decomposition}
if for some $N\geq1,$ there is a partition (\ref{3.23}) such that the restriction of $V_j, $ to $I_j$  extends to an 
absolutely continuous function in $\overline{I_{j}},$ $j=0,1,\dots,N-1,$ 
 $V_N$ is
absolutely continuous in each interval $[x_{N},b],$ $x_{N}<b<\infty,$ and
$<\cdot>^{2+\delta}\,\partial_{x}V\in L^{1}(x_{N},\infty),$ for some
$\delta\geq0.$}}
\end{definition}

We now introduce the class of nonlinear interactions that we consider in the
present work.

\begin{assumption}{\rm
\textrm{Assume that the function $\mathcal{N}$ is defined on $\mathbb{R}^{n}$,
 that it is $n\times n$ matrix-valued, and that $  \mathcal N \in C^{2}\left(  \mathbb{R}^n; \mathcal L_n\right). $ Furthermore,
\begin{equation}
\sum_{l=1}^{n}\left\vert \partial_{l}^{j}\mathcal{N}\left(  \mu_{1}%
,...,\mu_{n}\right)  \right\vert \leq C\left\vert \mu\right\vert ^{\alpha
-j},\text{ }j=0,1,2\text{,} \label{ConNonlinearity}%
\end{equation}
for all $\mu= \left(  \mu_{1},...,\mu_{n}\right)  \in\mathbb{R}^{n}$ and some
$\alpha>2.$}}
\end{assumption}

Let $H_{A,B,V}$ be the linear Schr\"{o}dinger operator associated to problem
(\ref{1.1})-\eqref{1.1.bound}, i.e., to \eqref{1.1}-\eqref{1.1.bound} with $\mathcal{N}$ equal to zero (see
Appendix~\ref{App1}). When there is no possibility of misunderstanding we will
write $H$ instead of $H_{A,B,V}$. We denote by $S(k),k\in\mathbb{R},$ the
scattering matrix (see \eqref{Scatteringmatrix}) associated with $H_{A,B,V}.$
Let $P_{\pm}$ be, respectively, the projectors onto the eigenspaces
corresponding to the eigenvalues $\pm1$ of the scattering matrix at zero
energy, $S(0).$ In fact, $\pm1$ are the only eigenvalues that $S(0)$ can have.
See Remark~\ref{rema.1} in Appendix~\ref{App1}. We define, $m\left(
k,x\right):=f\left(  k,x\right)  e^{-ikx}$, where $f(k,x)$ is the Jost
solution (see (\ref{Jostsolution}) in Appendix~\ref{App1}). Further, we extend the function $m\left(  k,x\right)$  to
an even function defined for $x\in{\mathbb{R}},$ that is, $m(k,-x)=m(k,x),$
$x\in{\mathbb{R}}.$ Observe that $m(k,x)$ is defined in
terms of the potential $V,$ and that it is independent of the boundary
condition in \eqref{1.1.bound}.

\begin{theorem}
\label{Theorem 1.1} Suppose that the potential $V$ is selfadjoint, that $V\in
L^{1}_{4}\left(  \mathbb{R}\right)  ,$ and that $V$ admits a regular
decomposition with $\delta=2$ (see Definition \ref{Def1}). Further, assume
that the boundary matrices $A,B$, satisfy \eqref{wcon1}, \eqref{wcon2}, and
that $H_{A,B,V}$ does not have negative eigenvalues. Moreover, suppose that
the nonlinearity $\mathcal{N}$ fulfills \eqref{ConNonlinearity} and that it
conmutes with $P_{-},$ i.e. $\mathcal{N}\left(  |\mu_{1}| , \dots,|\mu_{n}|
\right)  P_{-}= P_{-} \mathcal{N}\left(  |\mu_{1}| , \dots,|\mu_{n}| \right)
, \mu= (\mu_{1},\dots\mu_{n}) \in\mathbb{R}^{n}.$ Then, there is an
$\varepsilon>0,$ such that for all initial data $u_{0}\in H^{2}(\mathbb{R}%
^{+})\cap H^{1}_{1}(\mathbb{R}^{+})\cap L^{2}_{2}(\mathbb{R}^{+})$ that
satisfy the boundary condition $-B^{\dagger}u_{0}\left(  0\right)
+A^{\dagger} u_{0}^{\prime}\left(  0\right)  =0,$ and with
\begin{equation}
\label{18.2}\left\Vert u_{0}\right\Vert _{H^{2}(\mathbb{R}^{+})}+\left\Vert
u_{0}\right\Vert _{H^{1}_{1}(\mathbb{R}^{+})}+\left\Vert u_{0}\right\Vert
_{L^{2}_{2}(\mathbb{R}^{+})} \leq\varepsilon,
\end{equation}
the initial boundary-value problem \eqref{1.1}-\eqref{1.1.bound} has a unique solution,
\[
u_{+}\in C\left(  [ 0,\infty);H^{2}(\mathbb{R}^{+})\cap L^{2}_{2}%
(\mathbb{R}^{+}) \right)  \cap C^{1}\left(  [0,\infty); L^{2}(\mathbb{R}%
^{+})\right) ,
\]
and a unique solution,
\[
u_{-}\in C\left(  (-\infty, 0];H^{2}(\mathbb{R}^{+})\cap L^{2}_{2}%
(\mathbb{R}^{+}) \right)  \cap C^{1}\left(  (-\infty, 0]; L^{2}(\mathbb{R}%
^{+})\right) .
\]

Further, there exists free final states $w_{\pm\infty}\in H^{1}(\mathbb{R}%
^{+}),$ with $\Vert w_{\pm\infty}\Vert_{H^{1}(\mathbb{R}^{+})}\leq
C\tilde{\varepsilon},$ where $\tilde{\varepsilon}:=\varepsilon(1+\varepsilon
^{\alpha}),$ such that the following asymptotics are valid
\begin{equation}
\left\Vert u_{\pm}\left(  t\right)  -e^{-itH_{0}}\mathbf{F}_{0}^{\dagger
}w_{\pm\infty}\right\Vert _{L^{2}(\mathbb{R}^{+})}\leq C\tilde{\varepsilon
}\left(  \langle t\rangle^{-\left(  \frac{\alpha}{2}-1\right)  }+\langle
t\rangle^{-1/4}\right)  , \label{1.3-1}%
\end{equation}
where $\mathbf{F}_{0}$ is the generalized Fourier map given in \eqref{ap.58},
\eqref{ap.59}, and,%
\begin{equation}
\left\Vert u_{\pm}\left(t\right)-\frac{1}{\sqrt{2}}M\mathcal{D}_{t}\left(  m\left(  \frac
{x}{2},tx\right)  w_{\pm\infty}\left(  \frac{x}{2}\right)  \right)
\right\Vert _{L^{\infty}(\mathbb{R}^{+})}\leq\frac{C\tilde{\varepsilon}%
}{\langle t\rangle^{1/2}}\left(  \tilde{\varepsilon}^{\alpha}\langle
t\rangle^{-\left(  \frac{\alpha}{2}-1\right)  }+\langle t\rangle
^{-1/4}\right)  . \label{1.3-2}%
\end{equation}
Furthermore, the following large-time estimate holds,
\begin{equation}
\left\Vert u_{\pm}\left(  t\right)  \right\Vert _{L^{\infty}(\mathbb{R}^{+}%
)}\leq C\tilde{\varepsilon}\frac{1}{\langle t\rangle^{1/2}}\left(  1+\left(
\tilde{\varepsilon}^{\alpha}\langle t\rangle^{-\left(  \frac{\alpha}%
{2}-1\right)  }+\langle t\rangle^{-1/4}\right)  \right)  . \label{1.3}%
\end{equation}

\end{theorem}

\begin{remark}{\rm
\textrm{\label{genex} \textrm{By Remark 3.8.10 in pages 129-130 of
\cite{WederBook} the potential $V$ is \textit{generic on the half-line} if and
only if there are no bounded solutions to the Schr\"{o}dinger equation
\eqref{MSStationary} with zero energy, $k^{2}=0,$ that satisfy the boundary
condition \eqref{ap.3}, and it is \textit{exceptional on the half-line} if and
only if there is at least one bounded solution to the Schr\"{o}dinger equation
\eqref{MSStationary} with zero energy, $k^{2}=0,$ that satisfies the boundary
condition \eqref{ap.3}. Further, it is \textit{purely exceptional on the
half-line} if and only if there are $n$ linearly independent bounded solutions
to the Schr\"{o}dinger equation \eqref{MSStationary} with zero energy,
$k^{2}=0,$ that satisfy the boundary condition \eqref{ap.3}. The bounded
solutions to the Schr\"{o}dinger equation \eqref{MSStationary} with zero
energy, $k^{2}=0,$ that satisfy the boundary condition \eqref{ap.3} are called
half-bound states or zero-energy resonances. Note that the condition
$\mathcal{N}\,P_{-}=P_{-}\,\mathcal{N}$ is satisfied, in particular, if
$P_{-}=I,$ or $P_{-}=0.$ Further, by   Remark 3.8.10
in pages 129-130, Theorem 3.8.13 in page 137, and Theorem 3.8.14 in pages 138-139, of \cite{WederBook},
$S(0)=-I$ if and only if the potential $V$ is \textit{generic on the half line}. In this case
$P_{-}=I$. Moreover, $S(0)=I,$ and then, $P_{-}=0,$ if and only if the potential $V$
{\it purely exceptional on the half line}. It follows that for \textit{generic and purely
exceptional potentials on the half line}  the condition $\mathcal{N}%
\,P_{-}=P_{-}\,\mathcal{N}$ is always satisfied. Moreover, for a single
equation $\mathcal{N}$ is a scalar, and hence, the condition, $\mathcal{N}%
P_{-}=P_{-}\mathcal{N}, $ is always satisfied. More generally, in the matrix
case if
\[
\mathcal{N}\left(  \mu_{1},...,\mu_{n}\right)  =\mathcal{N}_{\text{scalar}%
}\left(  \mu_{1},...,\mu_{n}\right)  \,I,\qquad(\mu_{1},\dots\mu_{n}%
)\in\mathbb{R}^{n},
\]
where $\mathcal{N}_{\text{scalar}}$ is a scalar function, the condition
$\mathcal{N}P_{-}=P_{-}\mathcal{N}$ is always satisfied.} }
}
\end{remark}

\begin{remark}
\label{wavehalf} {\rm Theorem~\ref{Theorem 1.1} allows us to
construct the inverse wave operators for \eqref{1.1}-\eqref{1.1.bound} as
follows. Let $u_{0}$ satisfy the assumptions of Theorem~\ref{Theorem 1.1}. Let
us denote
\[
u_{\pm\infty}:=\mathbf{F}_{0}^{\dagger}w_{\pm\infty}.
\]
Then, by \eqref{1.3-1}
\[
\lim_{t\rightarrow\pm\infty}\left\Vert u_{\pm}(t)-e^{-itH_{0}}u_{\pm\infty
}\right\Vert _{L^{2}(\mathbb{R}^{+})}=0,
\]
and we can define the inverse wave operators $W_{\pm}^{-1},$ as follows,
\[
W_{\pm}^{-1}u_{0}:=u_{\pm\infty}.
\]
With the methods of this paper we can also construct the direct wave operators
and the scattering matrix for small solutions. We will consider these problems
in a different publication. }

\end{remark}

\subsubsection{The matrix nonlinear Schr\"{o}dinger equation on the
line.}

As it follows from Section 2.4 of \cite{WederBook}, a $2n\times2n$ matrix
Schr\"{o}dinger equation on the half-line is unitarily equivalent to a
$n\times n$ matrix Schr\"{o}dinger equation on the  line with a point
interaction at $x=0$ by the unitary operator $\mathbf{U}$ from $L^{2}\left(
\mathbb{R}^{+};\mathbb{C}^{2n}\right)  $ onto $L^{2}\left(  \mathbb{R}%
;\mathbb{C}^{n}\right)  $%
\begin{equation}
\phi\left(  x\right)  =\mathbf{U}\psi\left(  x\right)  =\left\{
\begin{array}
[c]{c}%
\psi_{1}\left(  x\right)  ,\text{ \ }x\geq0,\\
\psi_{2}\left(  -x\right)  ,\text{ \ }x<0,
\end{array}
\right.  \label{unitarytransform}%
\end{equation}
for a vector-valued function $\psi=\left(  \psi_{1},\psi_{2}\right)  ^{T},$
where $\psi_{j}\in L^{2}\left(  \mathbb{R}^{+};\mathbb{C}^{n}\right)  ,$
$j=1,2.$ Let the potential in (\ref{1.1}) be the block-diagonal matrix
\begin{equation}
\label{blockdiag}V\left(  x\right)  =\operatorname*{diag}\{V_{1}\left(
x\right)  ,V_{2}\left(  x\right)  \},
\end{equation}
where $V_{j},j=1,2$ are selfadjoint $n\times n$ matrix-valued functions that
satisfy $V_{j}\in L_{1}^{1}(\mathbb{R}^{+}).$ Under $\mathbf{U}$ the
Hamiltonian $H_{A,B,V}$ is transformed into the following Hamiltonian on the
 line (see \cite{WederBook})%

\begin{equation}
\label{hamiltonian}H_{\mathbb{R}}=\mathbf{U}\,H_{A,B,V}\mathbf{U}^{\dagger
},\quad D[H_{\mathbb{R}}]=\{\phi\in L^{2}\left(  \mathbb{R};\mathbb{C}%
^{n}\right)  :\mathbf{U}^{\dagger}\phi\in D[H_{A,B,V}]\}.
\end{equation}

Let us write the $2n\times2n$ matrices $A,B$ as follows,
\begin{equation}
A=\left[
\begin{array}
[c]{l}%
A_{1}\\
A_{2}%
\end{array}
\right]  ,\quad\left[
\begin{array}
[c]{l}%
B_{1}\\
B_{2}%
\end{array}
\right]  , \label{matrices}%
\end{equation}
with $A_{j},B_{j},j=1,2,$ being $n\times2n$ matrices. We have that the
functions in the domain of $H_{\mathbb{R}}$ satisfy the following transmission
condition at $x=0,$%

\begin{equation}
-B_{1}^{\dagger}\phi(0+)-B_{2}^{\dagger}\phi(0-)+A_{1}^{\dagger}(\partial
_{x}\phi)(0+)-A_{2}^{\dagger}(\partial_{x}\phi)(0-)=0. \label{bdcond}%
\end{equation}
Let the nonlinearity in (\ref{1.1}) be the  block-diagonal matrix%
\begin{equation}
\mathcal{N}\left(  \mu_{1} ,...,\mu_{2n} \right)  =\operatorname*{diag}%
\{\mathcal{N}_{\mathbb{R},+}\left(  \mu_{1} ,...,\mu_{2n}\right)
,\mathcal{N}_{\mathbb{R},-}\left(  \mu_{1} ,...,\mu_{2n} \right)  \},
\qquad\mu\in\mathbb{R}^{2n} , \label{nonere}%
\end{equation}
where $\mathcal{N}_{\mathbb{R},\pm}$ are $n\times n$ matrix-valued function.
Then, $u(t,x)$ is a solution of the problem (\ref{1.1})-
\eqref{1.1.bound} if and only if $v(t,x):=\mathbf{U}u(t,x)$ is a solution of
the following $n\times n$ system on the  line, where we denote
$v_{0}:=\mathbf{U}u_{0},$
\begin{align}
&  i\partial_{t}v\left(  t,x\right)  =\left(  -\partial_{x}^{2}+Q\left(
x\right)  \right)  v\left(  t,x\right)  +\mathcal{N}_{\mathbb{R},\pm}\left(
\left\vert v_{1}(x)\right\vert ,...,\left\vert v_{n}(x)\right\vert ,\left\vert
v_{1}(-x)\right\vert ,...,\left\vert v_{n}(-x)\right\vert \right)
v(t,x),\label{MSR.1}\\
&  t\in\mathbb{R},\pm x>0,\nonumber\\
&  v\left(  0,x\right)  =v_{0}\left(  x\right)  ,\text{ \ }x\in\mathbb{R}%
,\label{MSR.2}\\
&  -B_{1}^{\dagger}v(t,0+)-B_{2}^{\dagger}v(t,0-)+A_{1}^{\dagger}(\partial
_{x}v)(t,0+)-A_{2}^{\dagger}(\partial_{x}v)(t,0-)=0, \label{MSR.3}%
\end{align}
where,
\begin{equation}
Q\left(  x\right)  =\left\{
\begin{array}
[c]{c}%
V_{1}\left(  x\right)  ,\text{ \ }x\geq0,\\
V_{2}\left(  -x\right)  ,\text{ \ }x<0.
\end{array}
\right.  \label{potere}%
\end{equation}
For example, let us take \cite{WederBook},
\begin{equation}
A=\left[
\begin{array}
[c]{lc}%
0_{n} & I_{n}\\
0_{n} & I_{n}%
\end{array}
\right]  ,\quad B=\left[
\begin{array}
[c]{lc}%
-I_{n} & \Lambda\\
\ I_{n} & 0_{n}%
\end{array}
\right]  , \label{nopoint}%
\end{equation}
where $\Lambda$ is a selfadjoint $n\times n$ matrix. These matrices satisfy
(\ref{wcon1}, \ref{wcon2}). Moreover, the transmission condition in
(\ref{MSR.3}) is given by,
\begin{equation}
v(t,0+)=v(t,0-)=v(t,0),\quad(\partial_{x}v)(t,0+)-(\partial_{x}%
v)(t,0-)=\Lambda v(t,0). \label{bcond2}%
\end{equation}
This transmission condition corresponds to a Dirac delta point interaction at
$x=0$ with coupling matrix $\Lambda$. If $\Lambda=0,$ $v(t,x)$ and
$(\partial_{x}v)(t,x)$ are continuous at $x=0$ and the transmission condition
corresponds to the matrix Schr\"{o}dinger equation on the  line without a
point interaction at $x=0.$

We denote,
\begin{equation}
\mathbf{F}_{0,{\mathbb{R}}}:=\mathbf{U}\mathbf{F}_{0}\mathbf{U}^{\dagger},
\label{f0ere}%
\end{equation}%
\begin{equation}
w_{\pm\infty,{\mathbb{R}}}:=\mathbf{U}w_{\pm\infty}. \label{were}%
\end{equation}

From Theorem \ref{Theorem 1.1} we deduce the following long-time result for
the matrix nonlinear Schr\"{o}dinger equation with a potential, and a point
interaction, on the line.

\begin{theorem}
\label{theoline} Suppose that the potential $Q$ is selfadjoint, $Q^{\dagger
}(x)=Q(x),$ $x\in{\mathbb{R}},$ that $Q\in L_{4}^{1}\left(  \mathbb{R}\right)
$ and that $V_{1}(x):=Q(x)$ and $V_{2}(x):=Q(-x),$ $x\geq0,$ admit a regular
decomposition with $\delta=2$ (see Definition \ref{Def1}). Further, assume
that the boundary matrices $A,B$, satisfy \eqref{wcon1}, \eqref{wcon2}, and
that $H_{{\mathbb{R}}}$ does not have negative eigenvalues. Moreover, suppose
that the nonlinearity $\mathcal{N},$ defined in \eqref{nonere} fulfills
\eqref{ConNonlinearity} and that it conmutes with $P_{-},$ i.e. $\mathcal{N}%
\left(  |\mu_{1}|,\dots,|\mu_{2n}| \right)  P_{-}=P_{-}\mathcal{N}\left(
|\mu_{1}| ,\dots,|\mu_{2n}| \right)  ,$ $\mu\in\mathbb{R }^{2n}.$ Then, there
is an $\varepsilon>0,$ such that for all initial data $v_{0}\in H^{2}%
(\mathbb{R}^{+}\cup\mathbb{R}^{-})\cap H^{1,1}(\mathbb{R}^{+}\cup
\mathbb{R}^{-})\cap L_{2}^{2}({\mathbb{R}})$ that satisfy the transmission
condition
\begin{equation}
-B_{1}^{\dagger}v_{0}(0+)-B_{2}^{\dagger}v_{0}(0-)+A_{1}^{\dagger}(\frac
{d}{dx} v_{0}(0+)-A_{2}^{\dagger}(\frac{d}{dx}v_{0}(0-)=0, \label{transcond}%
\end{equation}
with $A_{1},A_{2},B_{1},$ and $B_{2}$ as in \eqref{matrices}, and where,
\begin{equation}
\left\Vert v_{0}\right\Vert _{H^{2}(\mathbb{R}^{+}\cup\mathbb{R}^{-}%
)}+\left\Vert v_{0}\right\Vert _{H^{1,1}(\mathbb{R}^{+}\cup\mathbb{R}^{-}%
)}+\left\Vert v_{0}\right\Vert _{L_{2}^{2}({\mathbb{R}})}\leq\varepsilon,
\label{18.2.ere}%
\end{equation}
the initial-transmission value problem \eqref{MSR.1}-
\eqref{MSR.3} has a unique solution,
\[
v_{+}\in C\left(  [0,\infty);H^{2}(\mathbb{R}^{+}\cup\mathbb{R}^{-})\cap
L_{2}^{2}({\mathbb{R}})\right)  \cap C^{1}\left(  [0,\infty);L^{2}%
({\mathbb{R}})\right)  ,
\]
and a unique solution,
\[
v_{-}\in C\left(  (-\infty,0];H^{2}(\mathbb{R}^{+}\cup\mathbb{R}^{-})\cap
L_{2}^{2}({\mathbb{R}})\right)  \cap C^{1}\left(  (-\infty,0];L^{2}%
({\mathbb{R}})\right)  .
\]
%
Further, there exists free final states (see \eqref{were}) $w_{\pm
\infty,{\mathbb{R}}}\in H^{1}(\mathbb{R}^{+}\cup\mathbb{R}^{-}),$ with $\Vert
w_{\pm,\infty}\Vert_{H^{1}(\mathbb{R}^{+}\cup\mathbb{R}^{-})}\leq
C\tilde{\varepsilon},$ where $\tilde{\varepsilon}:=\varepsilon(1+\varepsilon
^{\alpha}),$ such that the following asymptotics are valid
\begin{equation}
\left\Vert v_{\pm}\left(  t\right)  -e^{-itH_{0,{\mathbb{R}}}}\mathbf{F}%
_{0,{\mathbb{R}}}^{\dagger}w_{\pm\infty,{\mathbb{R}}}\right\Vert
_{L^{2}({\mathbb{R}})}\leq C\tilde{\varepsilon}\left(  \langle t\rangle
^{-\left(  \frac{\alpha}{2}-1\right)  }+\langle t\rangle^{-1/4}\right)  ,
\label{1.3-1.xx}%
\end{equation}
where $\mathbf{F}_{0,{\mathbb{R}}}$ is the generalized Fourier map given in
\eqref{f0ere}, and,%
\begin{equation}
\left\Vert v_{\pm}(t)-\frac{1}{\sqrt{2}}M \mathbf U\left [ 
 \mathcal{D}_{t}\left(  m\left(  \frac{x}{2},tx\right) \left(\mathbf U^\dagger
  w_{\pm \infty,\mathbb R}\left(  \frac{\vdot}{2}\right) \right)(x) \right) \right] 
  \right\Vert _{L^{\infty}(\mathbb R)}\leq
 \frac{C\tilde{\varepsilon}}{\langle t\rangle^{1/2}}\left(  \tilde{\varepsilon
}^{\alpha}\langle t\rangle^{-\left(  \frac{\alpha}{2}-1\right)  }+\langle
t\rangle^{-1/4}\right).\label{1.3-2.xx}%
\end{equation}
 Furthermore, the
following large-time estimate holds,
\begin{equation}
\left\Vert v_{\pm}\left(  t\right)  \right\Vert _{L^{\infty}({\mathbb{R}}%
)}\leq C\tilde{\varepsilon}\frac{1}{\langle t\rangle^{1/2}}\left(  1+\left(
\tilde{\varepsilon}^{\alpha}\langle t\rangle^{-\left(  \frac{\alpha}%
{2}-1\right)  }+\langle t\rangle^{-1/4}\right)  \right)  . \label{1.3.zz}%
\end{equation}

\end{theorem}

\begin{remark}{\rm
\textrm{\label{relation}\textrm{The scattering theory for the matrix
Schr\"{o}dinger equation on the line without point interaction, i.e., when the
matrices $A,B$ are given by \eqref{nopoint} with $\Lambda=0,$ has been studied
in \cite{akv}. In particular, they consider the low-energy limit of the
scattering matrix for potentials $Q\in L_{1}^{1}({\mathbb{R}}).$ In this
situation we can compare the notions of \textit{generic and exceptional
potentials on the half-line} and \textit{generic and exceptional potentials on
the line}. We consider the case when the potential, $V,$ on the half-line is
block diagonal as in \eqref{blockdiag} and where the potential on the line $Q$
is given by \eqref{potere}. Recall that (see Remark~ \ref{genex}) the
potential $V$ is \textit{generic on the half-line} if and only if there are no
bounded solutions to the Schr\"{o}dinger equation \eqref{MSStationary} with
zero energy, $k^{2}=0,$ that satisfy the boundary condition \eqref{ap.3}, that
$V$ is \textit{exceptional on the half-line} if and only if there is at least
one bounded solution to the Schr\"{o}dinger equation \eqref{MSStationary} with
zero energy, $k^{2}=0,$ that satisfy the boundary condition \eqref{ap.3}, and
that $V$ is \textit{purely exceptional on the half-line} if and only if there
are $2n$ linearly independent bounded solutions to the Schr\"{o}dinger
equation \eqref{MSStationary} with zero energy, $k^{2}=0,$ that satisfy the
boundary condition \eqref{ap.3}. Note that for the potential \eqref{blockdiag}
with the boundary condition given by the matrices \eqref{nopoint} with
$\Lambda=0$ the \textit{purely exceptional case in the half-line} can not
arise, since if we compute the first $n$ components of the solution to
\eqref{MSStationary} with $k^{2}=0,$ using the boundary condition at $x=0$ we
obtain the values of the second n components at $x=0$ of the solution and of
its derivative at $x=0.$ Further, using these initial values we can compute
the second $n$ components of the solution. This implies that there can be at
most $n$ linearly independent bounded solutions to the Schr\"{o}dinger
equation \eqref{MSStationary} with zero energy, $k^{2}=0,$ that satisfy the
boundary condition \eqref{ap.3}. } \textrm{Similarly (see \cite{akv}), the
potential $Q$ is \textit{generic on the line} if and only if there are no
bounded solutions to the Schr\"odiger equation on the line with energy,
$k^{2}=0,$
\begin{equation}
\label{linebounded}- v''+ Q(x) v=0, \qquad x
\in{\mathbb{R}}.
\end{equation}
Moreover, $Q$ is \textit{exceptional on the line} if and only if there is at
least one bounded solution to \eqref{linebounded}, and it is \textit{purely
exceptional on the line} if and only if there are $n$ linearly independent
bounded solutions to \eqref{linebounded}. Then, using the transformation given
by \eqref{unitarytransform} we see that $Q$ is \textit{generic on the line} if
and only if $V$ is \textit{generic on the half-line}, that $Q$ is
\textit{exceptional on the line} if and only if $V$ is \textit{exceptional on
the half-line}, and that $Q$ is \textit{purely exceptional on the line} if and
only if there are $n$ linearly independent bounded solutions to the
Schr\"odinger equation \eqref{MSStationary} with zero energy, $k^{2}=0,$ that
satisfy the boundary condition \eqref{ap.3}}. }}
\end{remark}

\begin{remark}
\textrm{\rm follows from Remarks~\ref{genex} and \ref{relation} that in the
{\it generic case} the condition $P_{-} \mathcal{N }= \mathcal{N }P_{-}$ in Theorem~
\ref{theoline} is always satisfied.}
\end{remark}


\begin{corollary}
\label{corline1} Suppose that in \eqref{nonere} we take $\mathcal{N}%
_{{\mathbb{R}},+}=\mathcal{N}_{{\mathbb{R}},-}=\mathcal{N}_{{\mathbb{R}}}.$
Then, for all even and odd solutions the initial-transmission problem
\eqref{MSR.1}-\eqref{MSR.3} is equivalent to the following
initial-transmission problem
\begin{align}
&  i\partial_{t}v\left(  t,x\right)  =\left(  -\partial_{x}^{2}+Q\left(
x\right)  \right)  v\left(  t,x\right)  +\widehat{\mathcal{N}}_{{\mathbb{R}}%
}(\left\vert v_{1}(x)\right\vert ,\dots,\left\vert v_{n}(x)\right\vert
)v(t,x),t,x\in\mathbb{R},\label{MSR.2.x}\\
&  v\left(  0,x\right)  =v_{0}\left(  x\right)  ,\text{ \ }x\in\mathbb{R},\\
&  -B_{1}^{\dagger}v(t,0+)-B_{2}^{\dagger}v(t,0-)+A_{1}^{\dagger}(\partial
_{x}v)(t,0+)-A_{2}^{\dagger}(\partial_{x}v)(t,0-)=0, \label{MSR.3.x}%
\end{align}
where
\[
\widehat{\mathcal{N}}_{{\mathbb{R}}}(|\mu_{1}|,...,|\mu_{n}|):=\mathcal{N}%
_{\mathbb{R}}\left(  |\mu_{1}|,...,|\mu_{n}|,|\mu_{1}|,...,|\mu_{n}|\right)
,\qquad(\mu_{1},\dots,\mu_{n})\in\mathbb{R}^{n}.
\]
Then, the results of Theorem~\ref{theoline} apply to all even and odd
solutions to the initial-transmission problem \eqref{MSR.2.x}-
\eqref{MSR.3.x}. Moreover, if $\mathcal{N}_{{\mathbb{R}}}=\mathcal{N}%
_{\text{scalar}}\,I,$ with $\mathcal{N}_{\text{scalar}}$ a scalar function
from $\mathbb{R}^{2n}$ into $\mathbb{C}$ and $I$ the $n\times n$ identity
matrix, the nonlinearity $\mathcal{N}$ given in \eqref{nonere} is equal to
$\mathcal{N}=\mathcal{N}_{\text{scalar}}\,I,$ where now $I$ is the
$2n\times2n$ identity matrix. Hence, in this case the condition $P_{-}%
\,\mathcal{N}=P_{-}\,\mathcal{N},$ is always satisfied. This corollary shows
how the condition that the solutions are even or odd appears naturally for the
local nonlinear Schr\"{o}dinger equations on the line.
\end{corollary}

\begin{corollary}
\label{corline2} Let us take a $2n\times2n$ system on the half-line with the
nonlinearity as \eqref{nonere}, with $\mathcal{N}_{{\mathbb{R}},+}%
=\mathcal{N}_{{\mathbb{R}},+}(\mu_{1},\dots,\mu_{n}),$ i.e. independent of
$\mu_{n+1},\dots,\mu_{2n},$ and $\mathcal{N}_{{\mathbb{R}},-}=\mathcal{N}%
_{{\mathbb{R}},-}(\mu_{2n+1},\dots,\mu_{2n}),$ i.e. independent of $\mu
_{1},\dots,\mu_{n}.$ In this case the initial-transmission problem
\eqref{MSR.1}-\eqref{MSR.3} is equivalent to the following
initial-transmission problem,
\begin{align}
&  i\partial_{t}v\left(  t,x\right)  =\left(  -\partial_{x}^{2}+Q\left(
x\right)  \right)  v\left(  t,x\right)  +\mathcal{N}_{{\mathbb{R}},\pm}\left(
\left\vert v_{1}(x)\right\vert , \dots, \left\vert v_{n}(x)\right\vert
\right)  v,\ t\in\mathbb{R},\,\pm x>0,\label{MSR.1.y}\\
&  v\left(  0,x\right)  =v_{0}\left(  x\right)  ,\text{ \ }x\in\mathbb{R}%
,\label{MSR.2.y}\\
&  -B_{1}^{\dagger}v(t,0+)-B_{2}^{\dagger}v(t,0-)+A_{1}^{\dagger}(\partial
_{x}v)(t,0+)-A_{2}^{\dagger}(\partial_{x}v)(t,0-)=0. \label{MSR3.y}%
\end{align}
If furthermore, $\mathcal{N}_{{\mathbb{R}},+}=\mathcal{N}_{{\mathbb{R}}%
,-}=\mathcal{N}_{{\mathbb{R}}},$ the initial-transmission problem
\eqref{MSR.1}-\eqref{MSR.3} is equivalent to the following
initial-transmission problem,%

\begin{align}
&  i\partial_{t}v\left(  t,x\right)  =\left(  -\partial_{x}^{2}+Q\left(
x\right)  \right)  v\left(  t,x\right)  +\mathcal{N}_{ {\mathbb{R}}}\left(
\left\vert v_{1}(x)\right\vert ,\dots, \left\vert v_{n}(x)\right\vert \right)
v ,\ t\in\mathbb{R}, \, x \in{\mathbb{R}},\label{MSR.1.z}\\
&  v\left(  0,x\right)  =v_{0}\left(  x\right)  ,\text{ \ }x\in\mathbb{R}%
,\label{MSR.2.z}\\
&  -B_{1}^{\dagger}v(t,0+)-B_{2}^{\dagger}v(t,0-)+A_{1}^{\dagger}(\partial
_{x}v)(t,0+)-A_{2}^{\dagger}(\partial_{x}v)(t,0-)=0. \label{MSR.3.z}%
\end{align}
Then, the results of Theorem~\ref{theoline} apply to all solutions of the
initial-transmission problems \eqref{MSR.1.y}-\eqref{MSR3.y}, and
\eqref{MSR.1.z}-\eqref{MSR.3.z}. Note, however, that with the exception of the
\textit{generic} case, in these two situations it is still necessary to verify
that the condition $P_{-} \mathcal{N }= \mathcal{N }P_{-}$ holds.
\end{corollary}

\begin{remark}
\textrm{\label{wave} {\rm \textrm{Theorem~\ref{theoline} allows us to construct the
inverse wave operators for \eqref{MSR.1}-\eqref{MSR.3} as
follows. Let $v_{0}$ satisfy the assumptions of Theorem~\ref{theoline}. Let us
denote
\[
u_{\pm\infty,{\mathbb{R}}}:=\mathbf{F}_{0,{\mathbb{R}}}^{\dagger}w_{\pm
\infty,{\mathbb{R}}}.
\]
Then, by \eqref{1.3-1.xx}
\[
\lim_{t\rightarrow\pm\infty}\left\Vert v_{\pm}(t)-e^{-itH_{0,{\mathbb{R}}}%
}u_{\pm\infty,{\mathbb{R}}}\right\Vert _{L^{2}({\mathbb{R}})}=0,
\]
and we can define the inverse wave operators $W_{\pm,{\mathbb{R}}}^{-1}$ as
follows,
\[
W_{\pm,{\mathbb{R}}}^{-1}v_{0}:=u_{\pm\infty,{\mathbb{R}}}.
\]
With the methods of this paper we can also construct the direct wave operators
and the scattering matrix for small solutions. We will consider these problems
in a different publication.}}}
\end{remark}

\subsubsection{Comments on the proof.}

Our approach is based on the factorization technique. For this purpose, we
find it convenient to go to an interaction representation in momentum space
(Fourier space). We define,
\[
w_{\pm}(k):= \left\{
\begin{array}
[c]{l}%
w_{\pm}(k):= (\mathbf{F }e^{it H} u_{\pm})(k), \qquad k \geq0,\\
\\
w_{\pm}(k)= S(k) w_{\pm}(-k), \qquad k \leq0,
\end{array}
\right.
\]
where $\mathbf{F}:= \mathbf{F}^{+},$ and $\mathbf{F}^{+}$ is the generalized
Fourier map for $H_{A,B,V}$ defined in \eqref{gefoma}. We find it convenient
to extend $w_{\pm}$ to negative $k$ with the symmetry $w_{\pm}(k)= S(k)
w_{\pm}(-k), k \in{\mathbb{R}}.$ We then factorize the solutions $u_{\pm}$ to
(\ref{1.1})- \eqref{1.1.bound} as follows,%

\begin{equation}
u_{\pm}\left(  t\right)  =M{\mathcal{D}}_{t}\mathcal{W}\left(  t\right)
w_{\pm}\left(  t\right)  , \label{1.5}%
\end{equation}
where ${\mathcal{W}}\left(  t\right)  $ is defined in \eqref{sp5}. Similar
factorizations were previously used in \cite{Ivan, Ivan1}. Then, if $H_{A,B,V}
$ does not have negative eigenvalues, we obtain the following equation for the
new variable $w_{\pm},$%
\begin{equation}
i\partial_{t}w_{\pm}=\widehat{\mathcal{W}}\left(  t\right)  \left(
\mathcal{N}\left(  \left\vert \left(  \mathcal{W}\left(  t\right)  w\right)
_{1}\right\vert ,...,\left\vert \left(  \mathcal{W}\left(  t\right)  w\right)
_{n}\right\vert \right)  \mathcal{W}\left(  t\right)  w\right)  , \label{1.6}%
\end{equation}
with $\widehat{\mathcal{W}}(t)$ defined in \eqref{sp9}. We then transfer the
problem to the study of $w_{\pm}.$ The advantage of this new equation is the
isolation of the nonlinear interaction. However, the price to pay for this
change of variables is dealing with $\mathcal{W}\left(  t\right)  ,$ and
$\widehat{\mathcal{W}}\left(  t\right)  .$ In \cite{Ivan, Ivan1}, the
$L^{\infty}({\mathbb{R}})$ norm of $w_{\pm}$ is controlled in terms of the
$H^{1}({\mathbb{R}})$ norm of $w_{\pm}.$ For generic potentials or exceptional
potentials with additional symmetry assumptions, it is possible to control
this $H^{1}({\mathbb{R}})$ norm without gaining any time growth. It is seems
problematic to apply the same strategy in the case of general exceptional
potentials since one can only prove $\left\Vert \partial_{k}w\right\Vert
_{L^{2}({\mathbb{R}})}\lesssim\sqrt{|t|}$ by the approach of \cite{Ivan,
Ivan1}. In this paper we use the time-oscillations of $\mathcal{W}\left(
t\right)  $ and of $\widehat{\mathcal{W}}\left(  t\right)  $ in equation
(\ref{1.6}) in order to overcome this difficulty. In this way we manage to
bound uniformly in time the $H^{1}({\mathbb{R}})$ norm of $w_{\pm}.$ We also
control the second derivative $\partial_{k}^{2}w_{\pm},$ and prove that the
$H^{2}({\mathbb{R}})$ norm of $w_{\pm}$ can grow in time at most as
$\sqrt{|t|}.$ Moreover, we prove that also the $L_{2}^{2}({\mathbb{R}})$ norm
of $w_{\pm}$ can as most grow as $\sqrt{|t|}.$ See Remark~\ref{bounds}

\subsection{Outline of the paper.}

The remaining sections of the paper are organized as follows. Section \ref{S7}
is devoted to the proof of our main result, Theorem~\ref{Theorem 1.1}. We
first prove in Theorem~\ref{local}, in Subsection~ \ref{Localsolutions}, the
existence of solutions that are local in time, by means of a contraction
mapping argument. The difficulty here consists in enforcing the boundary
condition in our contraction mapping argument. Then, in
Subsection~\ref{global} we show that the proof of Theorem~\ref{Theorem 1.1}
can equivalently be formulated in our interaction representation in Fourier
(momentum) space, and prove Theorem~\ref{Theorem 1.1} by means of the
\textit{a priori} estimates that we give in Lemma~\ref{LemmaPrincipal}. These
estimates allow us to prove that the local solutions are actually global in
time, and to study their behaviour for large times. The \textit{a priori}
estimates in Lemma~\ref{LemmaPrincipal} are our core estimates, and we prove
them in Section~\ref{Lprincipal}. However, before that, in
Section~\ref{AuxEst}, we prove several auxiliary estimates that we need. The
results of this section relay heavily on Appendix~\ref{App1}. The proof of
Lemma~\ref{LemmaPrincipal} in Section~\ref{Lprincipal} is divided in several
subsections. In subsection~\ref{primeras} we prove \eqref{18.2.3} and we show
that \eqref{18.0} and \eqref{18.0.1.1} follow, respectively, from \eqref{18.1}
and \eqref{18.1.1}. In Subsection~\ref{remaining} we reduce the proof of
\eqref{18.1} and \eqref{18.1.1} to prove Lemmas~\ref{Lth},\thinspace
\ref{Lema3}, and \ref{Lema4}. Actually, in order to control the derivatives of
the nonlinear interaction in the right-hand side of (\ref{1.6}), we decompose
it into the "perturbed" part $\theta_{\pm}$ and the "free" part
$\widehat{\mathcal{V}}\left(  \tau\right)  f_{1,\pm}^{\operatorname*{fr}}$.
Sections \ref{Ltheta} and \ref{ProofLema34} are devoted, respectively, to the
control of the derivatives of $\theta_{\pm},$ and of $\widehat{\mathcal{V}%
}\left(  \tau\right)  f_{1,\pm}^{\operatorname*{fr}}.$ Again, for these
purposes we relay heavily in the results in Appendix~\ref{App1}. Finally, in
Appendix~\ref{App1}, we present results on the spectral and the scattering
theory for the matrix Schr\"{o}dinger equation on the half-line, with the
general selfadjoint boundary condition, that we need, and that are also of
independent interest.

\section{Proof of Theorem \ref{Theorem 1.1}}

\label{S7}

\subsection{Local solutions}

\label{Localsolutions} We first prove that the problem
\eqref{1.1}-\eqref{1.1.bound} has local solutions.

\begin{theorem}
\label{local} \label{Local} Suppose that assumption~(\ref{ConNonlinearity})
holds, that the potential $V$ is selfadjoint, that it belongs to $L_{4}%
^{1}({\mathbb{R}})$ and that it admits a regular decomposition (see Definition
\ref{Def1}). Further, assume that the boundary matrices $A,B$ fulfill
\eqref{wcon1} and \eqref{wcon2}, and that the Hamiltonian $H_{A,B,V}$ has no
negative eigenvalues. Let the initial data $u_{0}\in H^{2}(\mathbb{R}^{+})\cap
H^{1,1}(\mathbb{R}^{+})\cap L_{2}^{2}(\mathbb{R}^{+})$ satisfy the boundary
condition $-B^{\dagger}u_{0}\left(  0\right)  +A^{\dagger}u_{0}^{\prime
}\left(  0\right)  =0$. Then, for some $T>0,$ (\ref{1.1})-\eqref{1.1.bound}
has a unique solution,
\[
u_{+}\in C\left(  \left[  0,T\right]  ;H^{2}(\mathbb{R}^{+})\cap L_{2}%
^{2}(\mathbb{R}^{+})\right)  \cap C^{1}\left(  [0,T];L^{2}(\mathbb{R}%
^{+})\right)  ,
\]
and a unique solution
\[
u_{-}\in C\left(  \left[  -T,0\right]  ;H^{2}(\mathbb{R}^{+})\cap L_{2}%
^{2}(\mathbb{R}^{+})\right)  \cap C^{1}\left(  [-T,0];L^{2}(\mathbb{R}%
^{+})\right)  .
\]
Furthermore, for any $\varepsilon>0$ there is a positive number
$T_{\varepsilon}$ such that if the initial data, $u(0),$ satisfies,
\begin{equation}
\left\Vert u_{0}\right\Vert _{H^{2}(\mathbb{R}^{+})}+\left\Vert u_{0}%
\right\Vert _{H^{1,1}(\mathbb{R}^{+})}+\left\Vert u_{0}\right\Vert _{L_{2}%
^{2}(\mathbb{R}^{+})}+\Vert u_{0}\Vert_{H_{1}(\mathbb{R}^{+})}^{\alpha}\Vert
u_{0}\Vert_{L^{2}(\mathbb{R}^{+})}\leq\varepsilon, \label{7.41.1}%
\end{equation}
the existence time of the solutions, $u_{\pm}(t),$ satisfies $T\geq
T_{\varepsilon},$ respectively, $-T\leq-T_{\varepsilon}.$ Moreover, $u_{\pm
}(t)$ can be extended on a maximal existence interval $[0,T_{+,\max}),$
respectively, $(T_{-,\max},0],$
\[
u_{+}\in{C}\left(  \left[  0,T_{+,\max}\right)  ;H^{2}(\mathbb{R}^{+})\cap
L_{2}^{2}(\mathbb{R}^{+})\right)  \cap{C}^{1}\left(  [0,T_{+,\max}%
];L^{2}(\mathbb{R}^{+})\right)  ,
\]
and
\[
u_{-}\in{C}\left(  \left(  T_{-,\max},0\right]  ;H^{2}(\mathbb{R}^{+})\cap
L_{2}^{2}(\mathbb{R}^{+})\right)  \cap{C}^{1}\left(  \left(  T_{-,\max
},0\right]  ;L^{2}(\mathbb{R}^{+})\right)  ,
\]
Moreover if for some constant $C$
\begin{equation}
\left\Vert u_{+}(t)\right\Vert _{H^{2}(\mathbb{R}^{+})}+\Vert u_{+}%
(t)\Vert_{L_{2}^{2}(\mathbb{R}^{+})}\leq C,\qquad t\in\lbrack0,T_{+,\max}),
\label{7.41}%
\end{equation}
then, $T_{+,\max}=\infty.$ Further, if for some constant $C$
\begin{equation}
\left\Vert u_{-}(t)\right\Vert _{H^{2}(\mathbb{R}^{+})}+\Vert u_{-}%
(t)\Vert_{L_{2}^{2}(\mathbb{R}^{+})}\leq C,\qquad t\in(T_{-,\max},0],
\label{7.41.1.1}%
\end{equation}
then, $T_{-,\max}=-\infty.$
\end{theorem}

\begin{proof}
We depart from the integral equation corresponding to (\ref{1.1}), that is%
\begin{equation}
u\left(  t\right)  =\mathcal{U}\left(  t\right)  u_{0}-i\int_{0}%
^{t}\mathcal{U}\left(  t-\tau\right)  \left(  \mathcal{N}\left(  \left\vert
u_{1}\right\vert , \dots,\left\vert u_{n}\right\vert \right)  u\right)
\left(  \tau\right)  d\tau, \label{7.22}%
\end{equation}
with $\mathcal{U}(t):= e^{-it H_{A,B,V}}.$ For $T >0,$ we denote
\[
I_{+, T}:= [0,T], \qquad I_{-,T}:= [-T,0].
\]

We denote by $G_{\pm,T}$ the Banach space, $C(I_{\pm,T};H^{2}({\mathbb{R}}%
^{+})\cap L_{2}^{2}(\mathbb{R}^{+}))\cap C^{1}(I_{\pm,T};L^{2}(\mathbb{R}%
^{+})),$ with norm,
\begin{equation}
\Vert u\Vert_{G_{\pm,T}}:=\sup_{t\in I_{\pm,T}}\left(  \Vert u(t)\Vert
_{H^{2}(\mathbb{R}^{+})}+\Vert u(t)\Vert_{L_{2}^{2}(\mathbb{R}^{+})}%
+\Vert\partial_{t}u(t)\Vert_{L^{2}(\mathbb{R}^{+})}\right)  . \label{7.26.15}%
\end{equation}
For $u_{0}\in H^{2}(\mathbb{R}^{+})\cap H^{1,1}(\mathbb{R}^{+})\cap L_{2}%
^{2}(\mathbb{R}^{+}),$ that satisfies the boundary condition $-B^{\dagger
}u_{0}(0)+A^{\dagger}\partial_{x}u_{0}(0)=0,$ we define,
\begin{equation}
\mathcal{Q}_{\pm}(u):=\mathcal{U}(t)u_{0}-i\int_{0}^{t}\mathcal{U}\left(
t-\tau\right)  \left(  \mathcal{N}\left(  \left\vert u_{1}\right\vert
,\dots,\left\vert u_{n}\right\vert \right)  u\right)  \left(  \tau\right)
d\tau,\qquad u\in G_{\pm,T}. \label{7.26.16}%
\end{equation}
We will prove that $\mathcal{Q}_{\pm}$ has a unique fixed point in $G_{\pm,T}$
by means of the contraction mapping theorem. This unique fixed point is the
solution to the integral equation \eqref{7.22}. Assume that $u\in G_{\pm,T}.$
Let $\theta\in C^{\infty}\left(  \mathbb{R}\right)  $ be such that
$\theta\left(  k\right)  =1,$ for $\left\vert k\right\vert \leq1,$ and
$\theta\left(  k\right)  =0,$ $\left\vert k\right\vert \geq2.$ Since $H$ has
no eigenvalues, $P_{c}(H)= I.$ Then, using (\ref{Spectralrepresentation1})
with $P_{c}(H)=I,$ we write%
\begin{equation}
\int_{0}^{t}\mathcal{U}\left(  -\tau\right)  \left(  \mathcal{N}\left(
\left\vert u_{1}\right\vert ,\dots,\left\vert u_{n}\right\vert \right)
u\right)  \left(  \tau\right)  d\tau=l_{1}+l_{2}, \label{7.22.0}%
\end{equation}
where%
\begin{equation}%
\begin{array}
[c]{l}%
l_{1}=\int_{0}^{t}\,d\tau\,\frac{1}{\sqrt{2\pi}}\int_{0}^{\infty}dk\left[
f\left(  k,x\right)  +f\left(  -k,x\right)  S(-k)\right]  (1-\theta\left(
k\right)  )e^{i\tau k^{2}}\mathbf{F}\left(  \mathcal{N}\left(  \left\vert
u_{1}\right\vert ,\dots,\left\vert u_{n}\right\vert \right)  u\right)  \left(
\tau\right)  ,
\end{array}
\label{7.22.1}%
\end{equation}
and
\begin{equation}%
\begin{array}
[c]{l}%
l_{2}=\int_{0}^{t}\,d\tau\,\frac{1}{\sqrt{2\pi}}\int_{0}^{\infty}dk\left[
f\left(  k,x\right)  +f\left(  -k,x\right)  S(-k)\right]  \theta\left(
k\right)  e^{i\tau k^{2}}\mathbf{F}\left(  \mathcal{N}\left(  \left\vert
u_{1}\right\vert ,\dots,\left\vert u_{n}\right\vert \right)  u\right)  \left(
\tau\right)  =\mathbf{F}^{\dagger}\frac{1}{<k>^{2}}\tilde{l}_{2},
\end{array}
\label{7.22.2}%
\end{equation}
where,
\begin{equation}
\tilde{l}_{2}:=\int_{0}^{t}\,d\tau\theta\left(  k\right)  <k>^{2}e^{i\tau
k^{2}}\mathbf{F}\left(  \mathcal{N}\left(  \left\vert u_{1}\right\vert
,\dots,\left\vert u_{n}\right\vert \right)  u\right)  \left(  \tau\right)  .
\label{7.22.2.1}%
\end{equation}
Using that $e^{i\tau k^{2}}=\left(  ik^{2}\right)  ^{-1}\partial_{\tau
}e^{i\tau k^{2}}$ and integrating by parts with respect to $\tau$ in $l_{1}$
we have,%
\begin{equation}
l_{1}=l_{11}+l_{12}+l_{13}, \label{7.22.3}%
\end{equation}
where%
\begin{equation}%
\begin{array}
[c]{l}%
l_{11}=-i\frac{1}{\sqrt{2\pi}}\int_{0}^{\infty}dk\left[  f\left(  k,x\right)
+f\left(  -k,x\right)  S(-k)\right]  (1-\theta\left(  k\right)  )(k^{2}%
)^{-1}e^{itk^{2}}\\
\\
\left(  \mathbf{F}\left(  \mathcal{N}\left(  \left\vert u_{1}\right\vert
,\dots,\left\vert u_{n}\right\vert \right)  u\right)  \left(  t\right)
-\mathbf{F}\left(  \mathcal{N}\left(  \left\vert u_{1}\right\vert
,\dots,\left\vert u_{n}\right\vert \right)  u\right)  \left(  0\right)
\right)  =\mathbf{F}^{\dagger}\frac{1}{<k>^{2}}\tilde{l}_{11}(k),
\end{array}
\label{7.22.4}%
\end{equation}
with
\begin{equation}
\tilde{l}_{11}:=-i(1-\theta\left(  k\right)  )<k>^{2}\,k^{-2}e^{itk^{2}%
}\left(  \mathbf{F}\left(  \mathcal{N}\left(  \left\vert u_{1}\right\vert
,\dots,\left\vert u_{n}\right\vert \right)  u\right)  \left(  t\right)
-\mathbf{F}\left(  \mathcal{N}\left(  \left\vert u_{1}\right\vert
,\dots,\left\vert u_{n}\right\vert \right)  u\right)  \left(  0\right)
\right)  . \label{7.22.4.1}%
\end{equation}
Moreover,
\begin{equation}%
\begin{array}
[c]{l}%
l_{12}=-i\frac{1}{\sqrt{2\pi}}\int_{0}^{\infty}dk\left[  f\left(  k,x\right)
+f\left(  -k,x\right)  S(-k)\right]  (1-\theta\left(  k\right)  )(k^{2}%
)^{-1}(e^{itk^{2}}-1)\mathbf{F}\left(  \mathcal{N}\left(  \left\vert
u_{1}\right\vert ,\dots,\left\vert u_{n}\right\vert \right)  u\right)  \left(
0\right) \\
\\
=\mathbf{F}^{\dagger}\frac{1}{<k>^{2}}\tilde{l}_{12}(k),
\end{array}
\label{7.22.5}%
\end{equation}
with
\begin{equation}
\tilde{l}_{12}:=-i(1-\theta\left(  k\right)  )k^{-2}\,<k>^{2}(e^{itk^{2}%
}-1)\mathbf{F}\left(  \mathcal{N}\left(  \left\vert u_{1}\right\vert
,\dots,\left\vert u_{n}\right\vert \right)  u\right)  \left(  0\right)  .
\label{7.22.5.1}%
\end{equation}
Further,
\begin{equation}%
\begin{array}
[c]{l}%
l_{13}=i\int_{0}^{t}\,d\tau\,\frac{1}{\sqrt{2\pi}}\int_{0}^{\infty}dk\left[
f\left(  k,x\right)  +f\left(  -k,x\right)  S(-k)\right]  (1-\theta\left(
k\right)  )(k^{2})^{-1}e^{i\tau k^{2}}\mathbf{F}\partial_{\tau}\left(
\mathcal{N}\left(  \left\vert u_{1}\right\vert ,\dots,\left\vert
u_{n}\right\vert \right)  u\right)  \left(  \tau\right) \\
=\mathbf{F}^{\dagger}\frac{1}{<k>^{2}}\tilde{l}_{13},
\end{array}
\label{7.22.6}%
\end{equation}
where,
\begin{equation}
\tilde{l}_{13}:=i\int_{0}^{t}\,d\tau\,(1-\theta\left(  k\right)
)k^{-2}<k>^{2}\,e^{i\tau k^{2}}\mathbf{F}\partial_{\tau}\left(  \mathcal{N}%
\left(  \left\vert u_{1}\right\vert ,\dots,\left\vert u_{n}\right\vert
\right)  u\right)  \left(  \tau\right)  . \label{7.22.6.1}%
\end{equation}
Using (\ref{ConNonlinearity}) and Sobolev's inequality we estimate%
\begin{equation}%
\begin{array}
[c]{l}%
\left\Vert \left(  \mathcal{N}\left(  \left\vert u_{1}\right\vert
,\dots,\left\vert u_{n}\right\vert \right)  u\right)  \left(  t\right)
-\left(  \mathcal{N}\left(  \left\vert u_{1}\right\vert ,\dots,\left\vert
u_{n}\right\vert \right)  u\right)  \left(  0\right)  \right\Vert
_{L^{2}(\mathbb{R}^{+})}\leq\\
\\
C|t|\sup_{\tau\in\lbrack0,t]}\left[  \left\Vert u(\tau)\right\Vert
_{H^{1}(\mathbb{R}^{+})}^{\alpha}\left\Vert \partial_{\tau}u(\tau)\right\Vert
_{L^{2}(\mathbb{R}^{+})}\right]  .
\end{array}
\label{7.22.7}%
\end{equation}
Then, by \eqref{7.22.4.1}, \eqref{7.22.7}, \eqref{UnitaritySM}, and
\eqref{isom}, recalling that as $H$ has no eigenvalues, $E(\mathbb{R}%
^{+};H)=I,$
\begin{equation}
\Vert\tilde{l}_{11}\Vert_{L^{2}(\mathbb{R}^{+})}\leq C|t|\sup_{\tau\in
\lbrack0,t]}\left\Vert u(\tau)\right\Vert _{H^{1}(\mathbb{R}^{+})}^{\alpha
}\left\Vert \partial_{\tau}u(\tau)\right\Vert _{L^{2}(\mathbb{R}^{+})}.
\label{7.22.7.1}%
\end{equation}
By \eqref{gefoma}, since $k\,e^{\pm ikx}=\mp i\partial_{x}e^{\pm ikx},$
$f(k,x)=e^{ikx}m(k,x),$ and integrating by parts,
\begin{equation}
k\left(  \mathbf{F}\psi\right)  \left(  k\right)  =a_{1}\left(  \psi\right)
(k)+a_{2}\left(  \psi\right)  (k), \label{7.22.8}%
\end{equation}
where
\begin{equation}
a_{1}(\psi)(k):=i\left[  -m(k,0)^{\dagger}+S(k)^{\dagger}m(-k,0)^{\dagger
}\right]  \psi(0), \label{7.22.8.1}%
\end{equation}
and
\begin{equation}
a_{2}(\psi_{2})=i\frac{1}{\sqrt{2\pi}}\,\int_{0}^{\infty}\left(
-e^{ikx}\partial_{x}[m(k,x)^{\dagger}\psi(x)]+e^{-ikx}S(-k)^{\dagger}%
\partial_{x}\left[  m(-k,x)^{\dagger}\psi(x)\right]  \right)  dx.
\label{7.22.8.2}%
\end{equation}
Further, by (\ref{UnitaritySM}), (\ref{3.2}), \eqref{3.6}, Parseval's
identity, and Sobolev's inequality we show that%

\begin{equation}
\label{7.22.9}\left\Vert a_{1}(\psi)\right\Vert _{L^{\infty}(\mathbb{R}^{+}%
)}+\left\Vert a_{2}(\psi)\right\Vert _{L^{2}(\mathbb{R}^{+})}\leq C\left\Vert
\psi\right\Vert _{H^{1}(\mathbb{R}^{+})}.
\end{equation}
Using this estimate, together with (\ref{ConNonlinearity}), \eqref{7.22.8},
and Sobolev's inequality, we show that%
\begin{equation}
\label{7.22.10}k\mathbf{F}\left(  \mathcal{N}\left(  \left\vert u_{1}%
\right\vert , \dots,\left\vert u_{n}\right\vert \right)  u\right)  \left(
0\right)  =a_{1}\left(  (\mathcal{N}\left(  \left\vert u_{1}\right\vert ,
\dots,\left\vert u_{n}\right\vert \right)  u)(0)\right)  +a_{2}\left(  (
\mathcal{N}\left(  \left\vert u_{1}\right\vert , \dots,\left\vert
u_{n}\right\vert \right)  u)(0)\right)  ,
\end{equation}
with%
\begin{equation}
\label{7.22.11}\left\Vert a_{1}\left(  ( \mathcal{N}\left(  \left\vert
u_{1}\right\vert , \dots,\left\vert u_{n}\right\vert \right)  u)(0)\right)
\right\Vert _{L^{\infty}(\mathbb{R}^{+})}+\left\Vert a_{2}\left(  (
\mathcal{N}\left(  \left\vert u_{1}\right\vert , \dots,\left\vert
u_{n}\right\vert \right)  u)(0)\right)  \right\Vert _{L^{2}(\mathbb{R}^{+}%
)}\leq C\left\Vert u(0)\right\Vert _{H^{1}(\mathbb{R}^{+})}^{\alpha+1}.
\end{equation}
Then, since $\left\vert e^{-i t k^{2}}-1\right\vert \leq C\min\{1,|t|
k^{2}\},$ we get
\begin{align}
&  \|\tilde{l}_{12}\|_{L^{2}(\mathbb{R}^{+})}= \left\Vert \left\langle
k\right\rangle ^{2}\left(  1-\theta\left(  k\right)  \right)  k^{-2} \left(
e^{it k^{2}}-1\right)  \mathbf{F}\left(  \mathcal{N}\left(  \left\vert
u_{1}\right\vert , \dots,\left\vert u_{n}\right\vert \right)  u\right)
\left(  0\right)  \right\Vert _{L^{2}(\mathbb{R}^{+})}\nonumber\\
&  \leq C\left\Vert \left\langle k\right\rangle ^{-1}\left\vert e^{-i t k^{2}%
}-1\right\vert ^{1/8}a_{1}\left(  (\mathcal{N}\left(  \left\vert
u_{1}\right\vert , \dots,\left\vert u_{n}\right\vert \right)  u)(0)\right)
\right\Vert _{L^{2}(\mathbb{R}^{+})} +\nonumber\\
&  C\left\Vert \left\langle k\right\rangle ^{-1}\left\vert e^{-it k^{2}%
}-1\right\vert ^{1/2}a_{2}\left(  (\mathcal{N}\left(  \left\vert
u_{1}\right\vert , \dots,\left\vert u_{n}\right\vert \right)  u)(0)\right)
\right\Vert _{L^{2}(\mathbb{R}^{+})}\nonumber\\
&  \leq C |t|^{1/8}\left\Vert \left\langle \cdot\right\rangle ^{-3/4}%
\right\Vert _{L^{2}(\mathbb{R}^{+})} \left\Vert a_{1}\left(  (\mathcal{N}%
\left(  \left\vert u_{1}\right\vert , \dots,\left\vert u_{n}\right\vert
\right)  u) (0)\right)  \right\Vert _{L^{\infty}(\mathbb{R}^{+})} +\nonumber\\
&  C\sqrt{|t|}\left\Vert a_{2}\left(  (\mathcal{N}\left(  \left\vert
u_{1}\right\vert , \dots,\left\vert u_{n}\right\vert \right)  u)(0)\right)
\right\Vert _{L^{2}(\mathbb{R}^{+})} \leq C \left(  |t|^{1/8}+\sqrt
{|t|}\right)  \left\Vert u(0)\right\Vert _{H^{1}(\mathbb{R}^{+})}^{\alpha+1}.
\label{7.22.12}%
\end{align}

Next, we calculate%
\begin{equation}
\partial_{t}\mathcal{N}\left(  \left\vert \left(  g\right)  _{1}\right\vert
,...,\left\vert \left(  g\right)  _{n}\right\vert \right)  =\sum_{j=1}%
^{n}\left(  E_{1}^{\left(  j\right)  }\left(  g\right)  \partial_{t}\left(
g\right)  _{j}+E_{2}^{\left(  j\right)  }\left(  g\right)  \partial
_{t}\overline{\left(  g\right)  _{j}}\right)  , \label{100}%
\end{equation}
where we denote%
\[
E_{1}^{\left(  j\right)  }\left(  g\right)  =\frac{1}{2}\left(  \partial
_{j}\mathcal{N}\right)  \left(  \left\vert \left(  g\right)  _{1}\right\vert
,...,\left\vert \left(  g\right)  _{n}\right\vert \right)  \frac
{\overline{\left(  g\right)  _{j}}}{\left\vert \left(  g\right)
_{j}\right\vert },
\]
and%
\[
E_{2}^{\left(  j\right)  }\left(  g\right)  =\frac{1}{2}\left(  \partial
_{j}\mathcal{N}\right)  \left(  \left\vert \left(  g\right)  _{1}\right\vert
,...,\left\vert \left(  g\right)  _{n}\right\vert \right)  \frac{\left(
g\right)  _{j}}{{\left\vert \left(  g\right)  _{j}\right\vert }}.
\]
Then, via (\ref{ConNonlinearity}) and Sobolev's inequality, we estimate%
\begin{align}
&  \left\Vert \mathbf{\partial}_{\tau}\left(  \mathcal{N}\left(  \left\vert
g_{1}\right\vert ,...,\left\vert g_{n}\right\vert \right)  u\right)  \left(
\tau\right)  \right\Vert _{L^{2}(\mathbb{R}^{+})}\nonumber\\
&  =\left\Vert \sum_{j=1}^{n}\left(  E_{1}^{\left(  j\right)  }\left(
g\right)  \partial_{\tau}g_{j}+E_{2}^{\left(  j\right)  }\left(  g\right)
\partial_{\tau}\overline{g_{j}}\right)  u+\mathcal{N}\left(  \left\vert
g_{1}\right\vert ,...,\left\vert g_{n}\right\vert \right)  \mathbf{\partial
}_{\tau}g\right\Vert _{L^{2}(\mathbb{R}^{+})}\nonumber\\
&  \leq C\left\Vert g\right\Vert _{H^{1}(\mathbb{R}^{+})}^{\alpha}\left\Vert
\partial_{\tau}g\right\Vert _{L^{2}(\mathbb{R}^{+})}. \label{30.0}%
\end{align}
and then,
by \eqref{7.22.6.1}, \eqref{30.0} \eqref{UnitaritySM}, and \eqref{isom}, we
get,
\begin{equation}
\Vert\tilde{l}_{13}\Vert_{L^{2}(\mathbb{R}^{+})}\leq C|t|\sup_{\tau\in
\lbrack0,t]}\left\Vert u(\tau)\right\Vert _{H^{1}(\mathbb{R}^{+})}^{\alpha
}\left\Vert \partial_{\tau}u(\tau)\right\Vert _{L^{2}(\mathbb{R}^{+})}.
\label{30.0.1}%
\end{equation}
Moreover, by (\ref{ConNonlinearity}), \eqref{7.22.2.1}, \eqref{isom}, and
Sobolev's inequality, we obtain,
\begin{equation}
\Vert\tilde{l}_{2}\Vert_{L^{2}(\mathbb{R}^{+})}\leq C|t|\sup_{\tau\in
\lbrack0,t]}\left\Vert u(\tau)\right\Vert _{H^{1}(\mathbb{R}^{+})}^{\alpha
}\left\Vert u(\tau)\right\Vert _{L^{2}(\mathbb{R}^{+})}. \label{30.0.2}%
\end{equation}
We define,
\begin{equation}
\tilde{l}=\tilde{l}_{11}+\tilde{l}_{12}+\tilde{l}_{13}+\tilde{l}_{2}.
\label{30.0.3}%
\end{equation}
Hence, by \eqref{7.22.0}, \eqref{7.22.2}, \eqref{7.22.3}, \eqref{7.22.4},
\eqref{7.22.5}, \eqref{7.22.6}, \eqref{7.22.7.1}, \eqref{7.22.12},
\eqref{30.0.1} and \eqref{30.0.2},
\begin{equation}
\int_{0}^{t}\mathcal{U}\left(  -\tau\right)  \left(  \mathcal{N}\left(
\left\vert u_{1}\right\vert ,\dots,\left\vert u_{n}\right\vert \right)
u\right)  \left(  \tau\right)  d\tau=\mathbf{F}^{\dagger}\left\langle
k\right\rangle ^{-2}\tilde{l}, \label{31.0}%
\end{equation}
with%
\begin{equation}%
\begin{array}
[c]{l}%
\left\Vert \tilde{l}\right\Vert _{L^{^{2}}(\mathbb{R}^{+})}\leq C\left(
|t|^{1/8}+|t|\right)  \sup_{\tau\in\lbrack0,T]}\left(  \left\Vert
u(\tau)\right\Vert _{H^{1}(\mathbb{R}^{+})}^{\alpha+1}+\left\Vert
u(\tau)\right\Vert _{H^{1}(\mathbb{R}^{+})}^{\alpha}\left\Vert \partial_{\tau
}u(\tau)\right\Vert _{L^{2}(\mathbb{R}^{+})}\right)  .
\end{array}
\label{30.0.4}%
\end{equation}
Since the domain of $H$ is contained in $H^{2}(\mathbb{R}^{+}),$ it follows
from \eqref{31.0}, \eqref{30.0.4}, and \eqref{spectralrepr},%

\begin{equation}%
\begin{array}
[c]{l}%
\Vert\int_{0}^{t}\mathcal{U}\left(  t-\tau\right)  \left(  \mathcal{N}\left(
\left\vert u_{1}\right\vert ,\dots,\left\vert u_{n}\right\vert \right)
u\right)  \left(  \tau\right)  d\tau\Vert_{H^{2}(\mathbb{R}^{+})}\leq\\
\\
C\left[  \Vert H\int_{0}^{t}\mathcal{U}\left(  t-\tau\right)  \left(
\mathcal{N}\left(  \left\vert u_{1}\right\vert ,\dots,\left\vert
u_{n}\right\vert \right)  u\right)  \left(  \tau\right)  d\tau\Vert
_{L^{2}(\mathbb{R}^{+})}\right.  +\\
\\
\left.  \Vert\int_{0}^{t}\mathcal{U}\left(  t-\tau\right)  \left(
\mathcal{N}\left(  \left\vert u_{1}\right\vert ,\dots,\left\vert
u_{n}\right\vert \right)  u\right)  \left(  \tau\right)  d\tau\Vert
_{L^{2}(\mathbb{R}^{+})}\right] \\
\leq C\left(  |t|^{1/8}+|t|\right)  \sup_{\tau\in\lbrack0,t]}\left(
\left\Vert u(\tau)\right\Vert _{H^{1}(\mathbb{R}^{+})}^{\alpha}\left\Vert
u(\tau)\right\Vert _{H^{2}(\mathbb{R}^{+})}\right. \\
\\
\left.  +\left\Vert u(\tau)\right\Vert _{H^{1}(\mathbb{R}^{+})}^{\alpha
}\left\Vert \partial_{\tau}u(\tau)\right\Vert _{L^{2}(\mathbb{R}^{+})}\right)
.
\end{array}
\label{30.0.5}%
\end{equation}
Next, integration by parts yields%
\begin{equation}%
\begin{array}
[c]{l}%
\partial_{t}\int_{0}^{t}\mathcal{U}\left(  t-\tau\right)  \left(
\mathcal{N}\left(  \left\vert u_{1}\right\vert ,\dots,\left\vert
u_{n}\right\vert \right)  u\right)  \left(  \tau\right)  d\tau=\mathcal{U}%
(t)\left(  \mathcal{N}\left(  \left\vert u_{1}\right\vert ,\dots,\left\vert
u_{n}\right\vert \right)  u\right)  \left(  0\right)  +\\
\int_{0}^{t}\mathcal{U}\left(  t-\tau\right)  \partial_{\tau}\left(  \left(
\mathcal{N}\left(  \left\vert u_{1}\right\vert ,\dots,\left\vert
u_{n}\right\vert \right)  u\right)  \left(  \tau\right)  \right)  d\tau.
\end{array}
\end{equation}
Then, using (\ref{30.0}) we get%
\begin{equation}%
\begin{array}
[c]{l}%
\left\Vert \partial_{t}\int_{0}^{t}\mathcal{U}\left(  t-\tau\right)  \left(
\mathcal{N}\left(  \left\vert u_{1}\right\vert ,\dots,\left\vert
u_{n}\right\vert \right)  u\right)  \left(  \tau\right)  d\tau\ \right\Vert
_{L^{2}(\mathbb{R}^{+})}\leq C\left\Vert u(0)\right\Vert _{H^{1}%
(\mathbb{R}^{+})}^{\alpha}\Vert u(0)\Vert_{L^{2}(\mathbb{R}^{+})}\\
\\
+C|t|\sup_{\tau\in\lbrack0,t]}\left\Vert u(\tau)\right\Vert _{H^{1}%
(\mathbb{R}^{+})}^{\alpha}\left\Vert \partial_{\tau}u(\tau)\right\Vert
_{L^{2}(\mathbb{R}^{+})}.
\end{array}
\label{7.26}%
\end{equation}
By \eqref{31.0} and \eqref{30.0.4}
\begin{equation}
\int_{0}^{t}\mathcal{U}\left(  -\tau\right)  \left(  \mathcal{N}\left(
\left\vert u_{1}\right\vert ,\dots,\left\vert u_{n}\right\vert \right)
u\right)  \left(  \tau\right)  d\tau\in D[H], \label{7.26.1}%
\end{equation}
where $D[H]$ denotes the domain of $H.$ Then, by \eqref{31.0}, \eqref{30.0.4}
and \eqref{zzz.1}
\begin{equation}%
\begin{array}
[c]{l}%
\left\Vert x^{2}\int_{0}^{t}\mathcal{U}\left(  t-\tau\right)  \left(
\mathcal{N}\left(  \left\vert u_{1}\right\vert ,\dots,\left\vert
u_{n}\right\vert \right)  u\right)  \left(  \tau\right)  d\tau\right\Vert
_{L^{2}(\mathbb{R}^{+})}\leq C|t|(1+|t|^{3})\sup_{\tau\in\lbrack0,t]}\left[
\left\Vert \mathbf{F}^{\dagger}\left\langle k\right\rangle ^{-2}\tilde{l}%
(\tau)\right\Vert _{L_{2}^{2}(\mathbb{R}^{+})}\right. \\
\\
\left.  +\left\Vert \mathbf{F}^{\dagger}\left\langle k\right\rangle
^{-2}\tilde{l}(\tau)\right\Vert _{H^{1,1}(\mathbb{R}^{+})}+\left\Vert
\mathbf{F}^{\dagger}\left\langle k\right\rangle ^{-2}\tilde{l}(\tau
)\right\Vert _{H^{2}(\mathbb{R}^{+})}\right]  .
\end{array}
\label{7.26.2}%
\end{equation}
By \eqref{7.22.4.1} and \eqref{fad}
\begin{equation}%
\begin{array}
[c]{l}%
x^{2}\mathbf{F}^{\dagger}\frac{1}{<k>^{2}}\tilde{l}_{11}=i\frac{1}{\sqrt{2\pi
}}\int_{0}^{\infty}\left[  (\partial_{k}^{2}e^{ikx})m(k,x)+(\partial_{k}%
^{2}e^{-ikx})m(-k,x)S(-k)\right]  (1-\theta\left(  k\right)  )\,k^{-2}%
e^{itk^{2}}\\
\\
\left(  \mathbf{F}\left(  \mathcal{N}\left(  \left\vert u_{1}\right\vert
,\dots,\left\vert u_{n}\right\vert \right)  u\right)  \left(  t\right)
-\mathbf{F}\left(  \mathcal{N}\left(  \left\vert u_{1}\right\vert
,\dots,\left\vert u_{n}\right\vert \right)  u\right)  \left(  0\right)
\right)  \,dk.
\end{array}
\label{7.26.3}%
\end{equation}
Integrating by parts in \eqref{7.26.3}, and using \eqref{UnitaritySM},
\eqref{3.2}, \eqref{3.5}, \eqref{3.5bis},
\eqref{scatmatrixderiv},\eqref{scatmatrixsecondderiv}, and Parseval's identity
we obtain,
\begin{equation}%
\begin{array}
[c]{l}%
\Vert x^{2}\mathbf{F}^{\dagger}\frac{1}{<k>^{2}}\tilde{l}_{11}\Vert
_{L^{2}(\mathbb{R}^{+})}\leq C(1+t^{2})\left[  \Vert\left(  \mathbf{F}\left(
\mathcal{N}\left(  \left\vert u_{1}\right\vert ,\dots,\left\vert
u_{n}\right\vert \right)  u\right)  \left(  t\right)  -\mathbf{F}\left(
\mathcal{N}\left(  \left\vert u_{1}\right\vert ,\dots,\left\vert
u_{n}\right\vert \right)  u\right)  \left(  0\right)  \right)  \Vert
_{L^{2}(\mathbb{R}^{+})}\right. \\
\\
\left.  +\Vert\partial_{k}\left(  \mathbf{F}\left(  \mathcal{N}\left(
\left\vert u_{1}\right\vert ,\dots,\left\vert u_{n}\right\vert \right)
u\right)  \left(  t\right)  -\mathbf{F}\left(  \mathcal{N}\left(  \left\vert
u_{1}\right\vert ,\dots,\left\vert u_{n}\right\vert \right)  u\right)  \left(
0\right)  \right)  \Vert_{L^{2}(\mathbb{R}^{+})}+\right. \\
\\
\left.  \Vert\partial_{k}^{2}\left(  \mathbf{F}\left(  \mathcal{N}\left(
\left\vert u_{1}\right\vert ,\dots,\left\vert u_{n}\right\vert \right)
u\right)  \left(  t\right)  -\mathbf{F}\left(  \mathcal{N}\left(  \left\vert
u_{1}\right\vert ,\dots,\left\vert u_{n}\right\vert \right)  u\right)  \left(
0\right)  \right)  \Vert_{L^{2}(\mathbb{R}^{+})}\right]  .
\end{array}
\label{7.26.4}%
\end{equation}
Moreover, by \eqref{ConNonlinearity}, \eqref{isom} and Sobolev's inequality,
\begin{equation}%
\begin{array}
[c]{l}%
\Vert\left(  \mathbf{F}\left(  \mathcal{N}\left(  \left\vert u_{1}\right\vert
,\dots,\left\vert u_{n}\right\vert \right)  u\right)  \left(  t\right)
-\right. \\
\\
\left.  \mathbf{F}\left(  \mathcal{N}\left(  \left\vert u_{1}\right\vert
,\dots,\left\vert u_{n}\right\vert \right)  u\right)  \left(  0\right)
\right)  \Vert_{L^{2}(\mathbb{R}^{+})}\leq C|t|\sup_{\tau\in\ [0,t]}\left\Vert
u(\tau)\right\Vert _{H^{1}(\mathbb{R}^{+})}^{\alpha}\left\Vert \partial_{\tau
}u(\tau)\right\Vert _{L^{2}(\mathbb{R}^{+})}.
\end{array}
\label{7.26.5}%
\end{equation}
Further, by \eqref{ConNonlinearity}, \eqref{61BIS.2}, and Sobolev's
inequality,
\begin{equation}%
\begin{array}
[c]{l}%
\Vert\partial_{k}\left(  \mathbf{F}\left(  \mathcal{N}\left(  \left\vert
u_{1}\right\vert ,\dots,\left\vert u_{n}\right\vert \right)  u\right)  \left(
t\right)  -\mathbf{F}\left(  \mathcal{N}\left(  \left\vert u_{1}\right\vert
,\dots,\left\vert u_{n}\right\vert \right)  u\right)  \left(  0\right)
\right)  \Vert_{L^{2}(\mathbb{R}^{+})}\\
\\
\leq C|t|\sup_{\tau\in\lbrack0,T]}\left\Vert u(\tau)\right\Vert _{H^{1}%
(\mathbb{R}^{+})}^{\alpha-1}\Vert u(\tau)\Vert_{L_{1}^{2}(\mathbb{R}^{+}%
)}\left\Vert \partial_{\tau}u(\tau)\right\Vert _{L^{2}(\mathbb{R}^{+})}.
\end{array}
\label{7.26.6}%
\end{equation}
Moreover, by \eqref{ConNonlinearity}, \eqref{61BIS.3}, and Sobolev's
inequality,
\begin{equation}%
\begin{array}
[c]{l}%
\Vert\partial_{k}^{2}\left(  \mathbf{F}\left(  \mathcal{N}\left(  \left\vert
u_{1}\right\vert ,\dots,\left\vert u_{n}\right\vert \right)  u\right)  \left(
t\right)  -\mathbf{F}\left(  \mathcal{N}\left(  \left\vert u_{1}\right\vert
,\dots,\left\vert u_{n}\right\vert \right)  u\right)  \left(  0\right)
\right)  \Vert_{L^{2}(\mathbb{R}^{+})}\leq\\
\\
C|t|\sup_{\tau\in\lbrack0,t]}\left\Vert u(\tau)\right\Vert _{H^{1}%
(\mathbb{R}^{+})}^{\alpha-1}\Vert u(\tau)\Vert_{L_{2}^{2}(\mathbb{R}^{+}%
)}\left\Vert \partial_{\tau}u(\tau)\right\Vert _{L^{2}(\mathbb{R}^{+})}.
\end{array}
\label{7.26.7}%
\end{equation}
By \eqref{7.26.4}, \eqref{7.26.5}, \eqref{7.26.6} and \eqref{7.26.7},
\begin{equation}%
\begin{array}
[c]{l}%
\Vert x^{2}\mathbf{F}^{\dagger}\frac{1}{<k>^{2}}\tilde{l}_{11}\Vert
_{L^{2}(\mathbb{R}^{+})}\leq C(1+t^{2})\,|t|\sup_{\tau\in\lbrack0,T]}\left[
\left\Vert u(\tau)\right\Vert _{H^{1}(\mathbb{R}^{+})}^{\alpha}\left\Vert
\partial_{\tau}u(\tau)\right\Vert _{L^{2}(\mathbb{R}^{+})}\right. \\
\\
\left.  +\left\Vert u(\tau)\right\Vert _{H^{1}(\mathbb{R}^{+})}^{\alpha
-1}\Vert u(\tau)\Vert_{L_{2}^{2}(\mathbb{R}^{+})}\left\Vert \partial_{\tau
}u(\tau)\right\Vert _{L^{2}(\mathbb{R}^{+})}\right]  .
\end{array}
\label{7.26.8}%
\end{equation}
As in the proof of \eqref{7.26.8} we show,
\begin{equation}
\Vert x^{2}\mathbf{F}^{\dagger}\frac{1}{<k>^{2}}\tilde{l}_{12}\Vert
_{L^{2}(\mathbb{R}^{+})}\leq C|t|(1+|t|)\left\Vert u(0)\right\Vert
_{H^{1}(\mathbb{R}^{+})}^{\alpha}\left\Vert u(0)\right\Vert _{L_{2}%
^{2}(\mathbb{R}^{+})}, \label{7.26.9}%
\end{equation}
and,
\begin{equation}%
\begin{array}
[c]{l}%
\Vert x^{2}\mathbf{F}^{\dagger}\frac{1}{<k>^{2}}\tilde{l}_{13}\Vert
_{L^{2}(\mathbb{R}^{+})}\leq C(1+t^{2})\,|t|\sup_{\tau\in\lbrack0,T]}\left[
\left\Vert u(\tau)\right\Vert _{H^{1}(\mathbb{R}^{+})}^{\alpha}\left\Vert
\partial_{\tau}u(\tau)\right\Vert _{L^{2}(\mathbb{R}^{+})}\right. \\
\\
+\left.  \left\Vert u(\tau)\right\Vert _{H^{1}(\mathbb{R}^{+})}^{\alpha
-1}\Vert u(\tau)\Vert_{L_{2}^{2}(\mathbb{R}^{+})}\left\Vert \partial_{\tau
}u(\tau)\right\Vert _{L^{2}(\mathbb{R}^{+})}\right]  .
\end{array}
\label{7.26.10}%
\end{equation}
Since $H$ has no eigenvalues, $P_{c}(H)= I,$ and then, by \eqref{projector},
$\mathbf{F}^{\dagger}\mathbf{F}=I.$ Hence, it follows from \eqref{7.22.2.1},
\begin{equation}%
\begin{array}
[c]{l}%
x^{2}\mathbf{F}^{\dagger}\frac{1}{<k>^{2}}\tilde{l}_{2}=x^{2}\mathbf{F}%
^{\dagger}\int_{0}^{t}\,d\tau\left[  \theta\left(  k\right)  e^{i\tau k^{2}%
}-1\right]  \mathbf{F}\left(  \mathcal{N}\left(  \left\vert u_{1}\right\vert
,\dots,\left\vert u_{n}\right\vert \right)  u\right)  \left(  \tau\right)  +\\
\\
x^{2}\int_{0}^{t}\,d\tau\left(  \mathcal{N}\left(  \left\vert u_{1}\right\vert
,\dots,\left\vert u_{n}\right\vert \right)  u\right)  \left(  \tau\right)  .
\end{array}
\label{7.26.11}%
\end{equation}
Then, as in the proof of \eqref{7.26.8} we obtain,
\begin{equation}%
\begin{array}
[c]{l}%
\Vert x^{2}\mathbf{F}^{\dagger}\frac{1}{<k>^{2}}\tilde{l}_{2}\Vert
_{L^{2}(\mathbb{R}^{+})}\leq C(1+t^{2})\,|t|\sup_{\tau\in\lbrack0,T]}\Vert
u(\tau)\Vert_{H_{1}(\mathbb{R}^{+})}^{\alpha}\Vert u(\tau)\Vert_{L_{2}%
^{2}(\mathbb{R}^{+})}.
\end{array}
\label{7.26.12}%
\end{equation}
By \eqref{30.0.3}, \eqref{7.26.8}, \eqref{7.26.9}, \eqref{7.26.10} and
\eqref{7.26.12}
\begin{equation}%
\begin{array}
[c]{l}%
\Vert x^{2}\mathbf{F}^{\dagger}\frac{1}{<k>^{2}}\tilde{l}\Vert_{L^{2}%
(\mathbb{R}^{+})}\leq C|t|(1+t^{2})\sup_{\tau\in\lbrack0,T]}\left[  \left(
\left\Vert u(\tau)\right\Vert _{H^{1}(\mathbb{R}^{+})}^{\alpha}\right.
\right. \\
\\
\left.  \left.  +\left\Vert u(\tau)\right\Vert _{H^{1}(\mathbb{R}^{+}%
)}^{\alpha-1}\Vert u(\tau)\Vert_{L_{2}^{2}(\mathbb{R}^{+})}\right)  \left\Vert
\partial_{\tau}u(\tau)\right\Vert _{L^{2}(\mathbb{R}^{+})}\right. \\
\\
\left.  +\Vert u(\tau)\Vert_{H_{1}(\mathbb{R}^{+})}^{\alpha}\Vert u(\tau
)\Vert_{L_{2}^{2}(\mathbb{R}^{+})}\right]  .
\end{array}
\label{7.26.13}%
\end{equation}
As in the proof of \eqref{7.26.13}, we show,
\begin{equation}
\Vert\mathbf{F}^{\dagger}\frac{1}{<k>^{2}}\tilde{l}\Vert_{L^{2}(\mathbb{R}%
^{+})}\leq C|t|\sup_{\tau\in\lbrack0,T]}\left\Vert u(\tau)\right\Vert
_{H^{1}(\mathbb{R}^{+})}^{\alpha}\left[  \left\Vert \partial_{\tau}%
u(\tau)\right\Vert _{L^{2}(\mathbb{R}^{+})}+\Vert u(\tau)\Vert_{L^{2}%
(\mathbb{R}^{+})}\right]  . \label{7.26.13.1}%
\end{equation}
By \eqref{7.26.13} and \eqref{7.26.13.1}
\begin{equation}%
\begin{array}
[c]{l}%
\Vert x^{2}\mathbf{F}^{\dagger}\frac{1}{<k>^{2}}\tilde{l}\Vert_{L_{2}%
^{2}(\mathbb{R}^{+})}\leq C|t|(1+t^{2})\sup_{\tau\in\lbrack0,T]}\left[
\left(  \left\Vert u(\tau)\right\Vert _{H^{1}(\mathbb{R}^{+})}^{\alpha
}\right.  \right. \\
\\
\left.  \left.  +\left\Vert u(\tau)\right\Vert _{H^{1}(\mathbb{R}^{+}%
)}^{\alpha-1}\Vert u(\tau)\Vert_{L_{2}^{2}(\mathbb{R}^{+})}\right)  \left\Vert
\partial_{\tau}u(\tau)\right\Vert _{L^{2}(\mathbb{R}^{+})}\right. \\
\\
\left.  +\Vert u(\tau)\Vert_{H_{1}(\mathbb{R}^{+})}^{\alpha}\Vert u(\tau
)\Vert_{L_{2}^{2}(\mathbb{R}^{+})}\right]  .
\end{array}
\label{7.26.13.2}%
\end{equation}
Moreover, as in the proof of \eqref{7.26.13} we get,
\begin{equation}%
\begin{array}
[c]{l}%
\Vert x\mathbf{F}^{\dagger}\frac{1}{<k>^{2}}\tilde{l}\Vert_{L^{2}%
(\mathbb{R}^{+})}\leq C|t|(1+|t|)\sup_{\tau\in\lbrack0,T]}\left[  \left(
\left\Vert u(\tau)\right\Vert _{H^{1}(\mathbb{R}^{+})}^{\alpha}\right.
\right. \\
\\
\left.  \left.  +\left\Vert u(\tau)\right\Vert _{H^{1}(\mathbb{R}^{+}%
)}^{\alpha-1}\Vert u(\tau)\Vert_{L_{1}^{2}(\mathbb{R}^{+})}\right)  \left\Vert
\partial_{\tau}u(\tau)\right\Vert _{L^{2}(\mathbb{R}^{+})}\right.  \left.
+\Vert u(\tau)\Vert_{H_{1}(\mathbb{R}^{+})}^{\alpha}\Vert u(\tau)\Vert
_{L_{1}^{2}(\mathbb{R}^{+})}\right]  .
\end{array}
\label{7.26.13.3}%
\end{equation}
Further, as in the proof of \eqref{7.26.13}, and using also \eqref{3.6} we
prove,
\begin{equation}
\Vert\partial_{x}\mathbf{F}^{\dagger}\frac{1}{<k>^{2}}\tilde{l}\Vert
_{L^{2}(\mathbb{R}^{+})}\leq C(\sqrt{|t|}+|t|)\sup_{t\in\lbrack0,T]}%
\,\left\Vert u(\tau)\right\Vert _{H^{1}(\mathbb{R}^{+})}^{\alpha}\left[  \Vert
u(\tau)\Vert_{L^{2}(\mathbb{R}^{+})}+\Vert\partial_{\tau}u(\tau)\Vert
_{L^{2}(\mathbb{R}^{+})}\right]  , \label{7.26.13.4}%
\end{equation}
and, furthermore, using also \eqref{3.7} we obtain,%

\begin{equation}%
\begin{array}
[c]{l}%
\Vert x\partial_{x}\mathbf{F}^{\dagger}\frac{1}{<k>^{2}}\tilde{l}\Vert
_{L^{2}(\mathbb{R}^{+})}\leq C(\sqrt{|t|}+|t|)\sup_{t\in\lbrack0,T]}\,\left[
\left\Vert u(\tau)\right\Vert _{H^{1}(\mathbb{R}^{+})}^{\alpha}\left(  \Vert
u(\tau)\Vert_{L^{2}(\mathbb{R}^{+})}+\Vert\partial_{\tau}u(\tau)\Vert
_{L^{2}(\mathbb{R}^{+})}\right)  \right. \\
\\
\left.  +\left\Vert u(\tau)\right\Vert _{H^{1}(\mathbb{R}^{+})}^{\alpha
-1}\Vert u(\tau)\Vert_{L_{1}^{2}(\mathbb{R}^{+})}\left(  \Vert u(\tau
)\Vert_{L^{2}(\mathbb{R}^{+})}+\Vert\partial_{\tau}u(\tau)\Vert_{L^{2}%
(\mathbb{R}^{+})}\right)  \right]  .
\end{array}
\label{7.26.13.5}%
\end{equation}
By \eqref{7.26.13.1}, \eqref{7.26.13.3}, \eqref{7.26.13.4} and
\eqref{7.26.13.5},
\begin{equation}%
\begin{array}
[c]{l}%
\Vert\mathbf{F}^{\dagger}\frac{1}{<k>^{2}}\tilde{l}\Vert_{H^{1,1}%
(\mathbb{R}^{+})}\leq C(\sqrt{|t|}+|t|)\sup_{t\in\lbrack0,T]}\,\left[
\left\Vert u(\tau)\right\Vert _{H^{1}(\mathbb{R}^{+})}^{\alpha}\left(  \Vert
u(\tau)\Vert_{L^{2}(\mathbb{R}^{+})}+\Vert\partial_{\tau}u(\tau)\Vert
_{L^{2}(\mathbb{R}^{+})}\right)  \right.  \left.  +\left\Vert u(\tau
)\right\Vert _{H^{1}(\mathbb{R}^{+})}^{\alpha-1}\Vert u(\tau)\Vert_{L_{1}%
^{2}(\mathbb{R}^{+})}\left(  \Vert u(\tau)\Vert_{L^{2}(\mathbb{R}^{+})}%
+\Vert\partial_{\tau}u(\tau)\Vert_{L^{2}(\mathbb{R}^{+})}\right)  \right]  .
\end{array}
\label{7.26.13.6}%
\end{equation}
Moreover, as $D[H]\subset H^{2}(\mathbb{R}^{+}),$ it follows from
\eqref{30.0.4} and \eqref{spectralrepr},
\begin{equation}%
\begin{array}
[c]{l}%
\label{7.26.13.7}\Vert\mathbf{F}^{\dagger}\frac{1}{<k>^{2}}\tilde{l}%
\Vert_{H^{2}(\mathbb{R}^{+})}\leq C\left[  \Vert H\mathbf{F}^{\dagger}\frac
{1}{<k>^{2}}\tilde{l}\Vert_{L^{2}(\mathbb{R}^{+})}+\Vert\mathbf{F}^{\dagger
}\frac{1}{<k>^{2}}\tilde{l}\Vert_{L^{2}(\mathbb{R}^{+})}\right]  \leq\\
\\
C\left(  t|^{1/8}+|t|\right)  \sup_{\tau\in\lbrack0,T]}\left(  \left\Vert
u(\tau)\right\Vert _{H^{1}(\mathbb{R}^{+})}^{\alpha+1}\right.  \left.
+\left\Vert u(\tau)\right\Vert _{H^{1}(\mathbb{R}^{+})}^{\alpha}\left\Vert
\partial_{\tau}u(\tau)\right\Vert _{L^{2}(\mathbb{R}^{+})}\right)  .
\end{array}
\end{equation}
By \eqref{7.26.2}, \eqref{7.26.13.2}, \eqref{7.26.13.6}, and
\eqref{7.26.13.7}
\begin{equation}%
\begin{array}
[c]{l}%
\left\Vert \int_{0}^{t}\mathcal{U}\left(  t-\tau\right)  \left(
\mathcal{N}\left(  \left\vert u_{1}\right\vert ,\dots,\left\vert
u_{n}\right\vert \right)  u\right)  \left(  \tau\right)  d\tau\right\Vert
_{L_{2}^{2}(\mathbb{R}^{+})}\leq C|t|(|t|^{1/8}+|t|^{5})\sup_{\tau\in
\lbrack0,T]}\left[  \left(  \left\Vert u(\tau)\right\Vert _{H^{1}%
(\mathbb{R}^{+})}^{\alpha}\right.  \right. \\
\\
\left.  \left.  +\left\Vert u(\tau)\right\Vert _{H^{1}(\mathbb{R}^{+}%
)}^{\alpha-1}\Vert u(\tau)\Vert_{L_{2}^{2}(\mathbb{R}^{+})}\right)  \left\Vert
\partial_{\tau}u(\tau)\right\Vert _{L^{2}(\mathbb{R}^{+})}\right. \\
\left.  +\Vert u(\tau)\Vert_{H_{1}(\mathbb{R}^{+})}^{\alpha}\Vert u(\tau
\Vert_{L_{2}^{2}(\mathbb{R}^{+})}+\Vert u(\tau)\Vert_{H_{1}(\mathbb{R}^{+}%
)}^{\alpha+1}\right]  .
\end{array}
\label{7.26.14}%
\end{equation}
Since by \eqref{domainh2} $u_{0}$ belongs to the domain of $H,$ and as the
domain of $H$ is contained in $H^{2}(\mathbb{R}^{+}),$
\begin{equation}
\Vert\mathcal{U}(t)u_{0}\Vert_{H^{2}(\mathbb{R}^{+})}\leq C\left[  \Vert
H\mathcal{U}(t)u_{0}\Vert_{L^{2}(\mathbb{R}^{+})}+\Vert\mathcal{U}%
(t)u_{0}\Vert_{L^{2}(\mathbb{R}^{+})}\right]  \leq C\Vert u_{0}\Vert
_{H^{2}(\mathbb{R}^{+})}, \label{7.26.17}%
\end{equation}
and, further, by \eqref{zzz.1}
\begin{equation}
\Vert\mathcal{U}(t)u_{0}\Vert_{L_{2}^{2}(\mathbb{R}^{+})}\leq C(1+|t|^{3}%
)\left(  \Vert u_{0}\Vert_{L_{2}^{2}(\mathbb{R}^{+})}+\Vert u_{0}%
\Vert_{H^{1,1}(\mathbb{R}^{+})}+\Vert u_{0}\Vert_{H^{2}(\mathbb{R}^{+}%
)}\right)  . \label{7.26.18}%
\end{equation}
Moreover,
\begin{equation}
\Vert\partial_{t}\mathcal{U}(t)u_{0}\Vert_{L^{2}(\mathbb{R}^{+})}%
=\Vert\mathcal{U}(t)Hu_{0}\Vert_{L^{2}(\mathbb{R}^{+})}\leq C\Vert u_{0}%
\Vert_{H^{2}(\mathbb{R}^{+})}. \label{7.26.19}%
\end{equation}
Then by \eqref{7.26.16}, \eqref{30.0.5}, \eqref{7.26.14}, \eqref{7.26.17},
\eqref{7.26.18}, and \eqref{7.26.19}
\begin{equation}%
\begin{array}
[c]{l}%
\Vert\mathcal{Q}_{\pm}(u)\Vert_{G_{\pm,T}}\leq C\left[  (1+T^{3})\left(  \Vert
u_{0}\Vert_{L_{2}^{2}(\mathbb{R}^{+})}+\Vert u_{0}\Vert_{H^{1,1}%
(\mathbb{R}^{+})}+\Vert u_{0}\Vert_{H^{2}(\mathbb{R}^{+})}\right)  \right. \\
\\
\left.  +\Vert u_{0}\Vert_{H_{1}(\mathbb{R}^{+})}^{\alpha}\Vert u_{0}%
\Vert_{L^{2}(\mathbb{R}^{+})}+(1+T)(T^{1/8}+T^{5})\Vert u\Vert_{G_{\pm,T}%
}^{\alpha+1}\right]  .
\end{array}
\label{7.26.20}%
\end{equation}
Denote,
\begin{equation}
\mathcal{M}:=2C\left[  \Vert u_{0}\Vert_{L_{2}^{2}(\mathbb{R}^{+})}+\Vert
u_{0}\Vert_{H^{1,1}(\mathbb{R}^{+})}+\Vert u_{0}\Vert_{H^{2}(\mathbb{R}^{+}%
)}\right. \label{7.26.21}\\
\\
\left.  +\Vert u_{0}\Vert_{H_{1}(\mathbb{R}^{+})}^{\alpha}\Vert u_{0}%
\Vert_{L^{2}(\mathbb{R}^{+})}\right]  .
\end{equation}
Let $G_{\pm,T}(\mathcal{M})$ be the ball of center zero and radius
$\mathcal{M}$ in $G_{\pm,T}.$ Then, by \eqref{7.26.20} we can take $T$ so
small that,
\begin{equation}
\Vert\mathcal{Q}_{\pm}(u)\Vert_{G_{\pm,T}}\leq\mathcal{M},\qquad u\in
G_{\pm,T}(\mathcal{M}). \label{7.26.22}%
\end{equation}
By a slight modification of the arguments that we used to prove
\eqref{7.26.22} we obtain that we can take $T$ so that also,
\begin{equation}
\Vert\mathcal{Q}_{\pm}(u)-\mathcal{Q}_{\pm}(v)\Vert_{G_{\pm,T}}<\Vert
u-v\Vert_{G_{\pm,T}},\qquad u,v\in G_{\pm,T}(\mathcal{M}). \label{7.26.23}%
\end{equation}
Then, by the contraction mapping theorem (see Theorem V.18 in page 151 of
\cite{rs.2}), the map $\mathcal{Q}_{\pm}$ has a unique fixed point, $u_{\pm
}(t)$ in $G_{I_{\pm,T}}(\mathcal{M})$ that is a solution to \eqref{7.22}.
Moreover, as $u_{0}\in D[H],$ it follows from \eqref{7.26.1} that $u_{\pm
}(t)\in D[H],$ and then, $u_{\pm}(t)$ satisfies the boundary condition
$-B^{\dagger}u(t,0)+A^{\dagger}(\partial_{x}u)(t,0)=0.$
Further, for any initial state $u_{0}$ that satisfies \eqref{7.41.1} we have
that,
\begin{equation}
\mathcal{M}\leq C(\varepsilon+\varepsilon^{\alpha+1}). \label{76.26.23.aaxxaa}%
\end{equation}
Then, by \eqref{7.26.20} we can take a fixed $T_{\varepsilon}$ such that,
\begin{equation}
\Vert\mathcal{Q}_{\pm}(u)\Vert_{G_{\pm,T_{\varepsilon}}}\leq\mathcal{M},\qquad
u\in G_{\pm,T_{\varepsilon}}(\mathcal{M}) \label{7.26.23.1}%
\end{equation}
for all initial data $u_{0}$ that satisfy \eqref{7.41.1}. In a similar way, we
prove that,
\begin{equation}
\Vert\mathcal{Q}_{\pm}(u)-\mathcal{Q}_{\pm}(v)\Vert_{G_{\pm,T_{\varepsilon}}%
}<\Vert u-v\Vert_{G_{\pm,T_{\varepsilon}}},\qquad u,v\in G_{\pm,T_{\varepsilon
}}(\mathcal{M}), \label{7.26.23.2}%
\end{equation}
for all initial data $u_{0}$ that satisfy \eqref{7.41.1}. Hence, the existence
time, $T,$ of all the solutions $u_{\pm}(t)$ with initial data, $u_{0},$ that
satisfy \eqref{7.41.1} is bounded below by $T_{\varepsilon},$ i.e., $T\geq
T_{\varepsilon}.$ Moreover, suppose that \eqref{7.41}, respectively
\eqref{7.41.1.1} holds and that $T_{+,\mathrm{max}}<\infty,$ respectively,
$T_{-,\mathrm{max}}<\infty.$ Hence,
\begin{equation}
\Vert u_{\pm}(t)\Vert_{H^{2}(\mathbb{R}^{+})}+\Vert u_{\pm}(t)\Vert_{L_{2}%
^{2}(\mathbb{R}^{+})}\leq C,\qquad t\in\lbrack0,T_{+,\mathrm{max}%
}),\,\text{respectively, }\,t\in(T_{-,\text{max}},0]. \label{7.26.24}%
\end{equation}
Further, taking the derivative with respect to $t$ of both sides of
\eqref{7.22} we obtain that \eqref{1.1} holds, and then,
\begin{equation}
\Vert\partial_{t}u_{\pm}(t)\Vert_{L^{2}(\mathbb{R}^{+})}\leq C\Vert u_{\pm
}(t)\Vert_{H^{2}(\mathbb{R}^{+})}\left(  1+\Vert u_{\pm}(t)\Vert
_{H^{2}(\mathbb{R}^{+})}^{\alpha}\right)  \leq C, \label{7.26.25}%
\end{equation}
where we used \eqref{ConNonlinearity} and Sobolev's inequality. Then, by
\eqref{7.26.20} we can take the limit as $t\rightarrow T_{+,\mathrm{max}}$ in
both sides of \eqref{7.22} and define,
\begin{equation}
u_{\pm}(T_{\pm,\mathrm{max}})=\lim_{T\rightarrow T_{\pm,\mathrm{max}}}u_{\pm
}(t)=\mathcal{U}\left(  T_{\pm,\mathrm{max}}\right)  u_{0}-i\int_{0}%
^{T_{\pm,\mathrm{max}}}\mathcal{U}\left(  T_{\pm,\mathrm{max}}-\tau\right)
\left(  \mathcal{N}\left(  \left\vert (u_{\pm})_{1}\right\vert ,\dots,
\left\vert (u_{\pm})_{n}\right\vert \right)  u_{\pm}\right)  \left(
\tau\right)  d\tau. \label{7.26.26}%
\end{equation}
Moreover, by \eqref{31.0}, \eqref{30.0.5}, \eqref{7.26.13.2},
\eqref{7.26.13.3}, \eqref{7.26.13.4}, \eqref{7.26.13.5}, \eqref{7.26.17},
\eqref{7.26.26}, \eqref{Uweightderiv}, \eqref{zzz}, and \eqref{zzz.2},
\begin{equation}%
\begin{array}
[c]{l}%
\Vert u_{\pm}(T_{\pm,\mathrm{max}})\Vert_{H^{1,1}(\mathbb{R}^{+})}\leq
C(1+T_{\pm,\mathrm{max}}^{2})\left[  \Vert u(0)\Vert_{H^{2}(\mathbb{R}^{+}%
)}+\Vert u(0)\Vert_{H^{1,1}(\mathbb{R}^{+})}+\Vert u_{0}\Vert_{L_{1}%
^{2}(\mathbb{R}^{+})}\right. \\
\\
\left.  +C(T_{\pm,\mathrm{max}}^{1/8}+T_{\pm,\mathrm{max}}^{5})\Vert u_{\pm
}\Vert_{G_{T_{\pm,\mathrm{max}}}}^{\alpha+1}\right]  .
\end{array}
\label{7.26.27}%
\end{equation}
Then, we can solve the integral equation
\[
u\left(  t\right)  =\mathcal{U}\left(  t-T_{\pm,\mathrm{max}}\right)  u_{\pm
}(T_{\pm,\max})-i\int_{T_{\pm,\mathrm{max}}}^{t}\mathcal{U}\left(
t-\tau\right)  \left(  \mathcal{N}\left(  \left\vert u_{1}\right\vert
,\dots,\left\vert u_{n}\right\vert \right)  u\right)  \left(  \tau\right)
d\tau,
\]
in an interval $t\in\lbrack T_{+,\mathrm{max}},T_{1}],$ respectively,
$[T_{2},T_{-,\mathrm{max}}],$ and extend the solution $u_{\pm}(t),$ to the
interval $[0,T_{1}],T_{1}>T_{\mathrm{max}}$, respectively, $[T_{2}%
,T_{-,\mathrm{max}}],$ in contradiction with the definition of
$T_{+,\mathrm{max}}, $ respectively, $T_{-,\mathrm{max}}.$ This completes the
proof of the theorem.
\end{proof}

\subsection{Global solutions}

\label{global} In this subsection we prove that for small initial data we can
extend the local solutions given by Theorem ~\ref{local} into global
solutions, i.e. solutions with $T_{\pm,\mathrm{max}}= \infty.$ For this
purpose, we find it convenient to go to a interaction representation in
momentum space (Fourier space). We suppose that $H=H_{A,B,V}$ does not have
negative eigenvalues. Then, by Theorem~\ref{theospec} $H$ is absolutely
continuous, $P_{c}(H)=I,$ and the generalized Fourier maps $\mathbf{F}^{\pm}$
are unitary, and in particular $\left(  \mathbf{F}^{\pm}\right)  ^{\dagger
}\mathbf{F}^{\pm}= \mathbf{F}^{\pm}\left(  \mathbf{F}^{\pm}\right)  ^{\dagger
}= I.$ Below we use $\mathbf{F}^{+},$ but we could use $\mathbf{F}^{-}$ as
well. For simplicity we denote $\mathbf{F}:= \mathbf{F}^{+}.$ Let $u_{\pm}$ be
the local solution to \eqref{1.1} given by Theorem~\ref{local}. We define,
\begin{equation}
\label{g.1}w_{\pm}(k):= \left\{
\begin{array}
[c]{l}%
w_{\pm}(k):= (\mathbf{F }e^{it H} u_{\pm})(k), \qquad k \geq0\\
\\
w_{\pm}(k)= S(k) w_{\pm}(-k), \qquad k \leq0.
\end{array}
\right.
\end{equation}
Recall that $u_{\pm}\in D[H] \subset H^{2}(\mathbb{R}^{+}).$ Then,
\begin{equation}
\label{g.2}\| e^{itH }u_{\pm}(t)\|_{H^{2}(\mathbb{R}^{+})} \leq C \left(
\|e^{itH} H u_{\pm}(t)\|_{L^{2}(\mathbb{R}^{+})}+ \|e^{itH}u_{\pm}%
(t)\|_{L^{2}(\mathbb{R}^{+})}\right)  \leq C \|u_{\pm}\|_{H^{2}(\mathbb{R}%
^{+})}.
\end{equation}
Further, by \eqref{7.22}, \eqref{31.0}, and \eqref{7.26.13.2}
\begin{equation}
\label{g.3}\| e^{itH}u_{\pm}(t)\|_{L^{2}_{2}(\mathbb{R}^{+})}\leq C\left[
\|u_{0}\|_{L^{2}_{2}(\mathbb{R}^{+})}+ T (1+T^{2}) \|u_{\pm}\|_{ G_{\pm, T}%
}^{\alpha+1}\right]  .
\end{equation}
Hence, by \eqref{g.2}, \eqref{g.3} and Lemma~\ref{ext},
\begin{equation}
\label{g.4}%
\begin{array}
[c]{l}%
w_{+} \in C([0,T]; H^{2}( {\mathbb{R}})\cap L^{2}_{2}( {\mathbb{R}})),\\
\\
w_{-} \in C([-T,0]; H^{2}( {\mathbb{R}})\cap L^{2}_{2}( {\mathbb{R}})).
\end{array}
\end{equation}
Further, by \eqref{g.1},
\begin{equation}
\label{g.4.1}\partial_{t} w_{\pm}(t,k)= \mathbf{F }\left(  i e^{it H} H
u_{\pm}(t) + e^{it H} \partial_{t} u_{\pm}(t)\right)  (k), \qquad k \geq0.
\end{equation}
Then, by \eqref{g.1} and \eqref{g.4.1}
\begin{equation}
\label{g.4.2}w_{+} \in C^{1}([0,T]; L^{2}( {\mathbb{R}})), \qquad w_{-} \in
C^{1}([-T,0]; L^{2}( {\mathbb{R}})).
\end{equation}

In order to derive the equation for $w_{\pm}\left(  t\right)  ,$ we apply the
operator $\mathbf{F}e^{itH}$ to equation (\ref{1.1}). Since $e^{-itH}$ is the
linear evolution group, we find%
\begin{equation}
\label{g.5}i\partial_{t}\left(  \mathbf{F}e^{itH}u_{\pm}\right)
=\mathbf{F}e^{itH}\left(  \mathcal{N}\left(  \left\vert (u_{\pm}%
)_{1}\right\vert , \dots,\left\vert (u_{\pm})_{n}\right\vert \right)  u_{\pm
}\right)  .
\end{equation}

Let us denote by $\tilde{w}_{\pm}(t,k)$ the restriction of $w_{\pm}(t,k)$ to
$k\in\mathbb{R}^{+}.$ Then, since by \eqref{g.1}, we have,
\begin{equation}
u_{\pm}\left(  t\right)  =e^{-itH}\mathbf{F}^{\dagger}\tilde{w}_{\pm},
\label{g.6}%
\end{equation}
substituting \eqref{g.6} into \eqref{g.5} we obtain the equation for
$\tilde{w}_{\pm}$
\begin{equation}
i\partial_{t}\tilde{w}_{\pm}=\mathbf{F}e^{itH}\left(  \mathcal{N}\left(
\left\vert \left(  e^{-itH}\mathbf{F}^{\dagger}\tilde{w}_{\pm}\right)
_{1}\right\vert ,\dots,\left\vert \left(  e^{-itH}\mathbf{F}^{\dagger}%
\tilde{w}_{\pm}\right)  _{n}\right\vert \right)  e^{-itH}\mathbf{F}^{\dagger
}\tilde{w}_{\pm}\right)  . \label{32}%
\end{equation}
By \eqref{ap.65} with $P_{c}(H)=I,$ and \eqref{sp7.ll} we get,%

\begin{align*}
&  \mathbf{F}e^{itH}\left(  \mathcal{N}\left(  \left\vert \left(
e^{-itH}\mathbf{F}^{\dagger}\tilde{w}_{\pm}\right)  _{1}\right\vert ,\dots,
\left\vert \left(  e^{-itH}\mathbf{F}^{\dagger}\tilde{w}_{\pm}\right)
_{n}\right\vert \right)  e^{-itH}\mathbf{F}^{\dagger}\tilde{w}_{\pm}\right) \\
&  =\mathcal{Q}^{-1}( t) \mathcal{D}_{t}^{-1}\overline{M}\left(
\mathcal{N}\left(  \left\vert ( M \mathcal{D}_{t}\mathcal{Q} \tilde{w}_{\pm
})_{1}\right\vert ,\dots, \left\vert ( M \mathcal{D}_{t}\mathcal{Q} \tilde
{w}_{\pm})_{n} \right\vert \right)  M\mathcal{D}_{t} \mathcal{Q}(t) \tilde
{w}_{\pm}\right) \\
&  =\mathcal{Q}^{-1}\left(  t\right)  \left(  \mathcal{N}\left(  \left\vert
\left(  t^{-1/2} \mathcal{Q}\left(  t\right)  \tilde{w}_{\pm}\right)
_{1}\right\vert , \dots, \left\vert \left(  t^{-1/2} \mathcal{Q}\left(
t\right)  \tilde{w}_{\pm}\right)  _{n} \right\vert \right)  \mathcal{Q}\left(
t\right)  \tilde{w}_{\pm}\right)  .
\end{align*}
Hence, we obtain the following equation for $\tilde{w}_{\pm}$
\begin{equation}
i\partial_{t}\tilde{w}_{\pm}=\mathcal{Q}^{-1}\left(  t\right)  \left(
\mathcal{N}\left(  \left\vert \left(  t^{-1/2} \mathcal{Q}\left(  t\right)
\tilde{w}_{\pm}\right)  _{1}\right\vert ,\dots, \left\vert \left(  t^{-1/2}
\mathcal{Q}\left(  t\right)  \tilde{w}_{\pm}\right)  _{n} \right\vert \right)
\mathcal{Q}\left(  t\right)  \tilde{w}_{\pm}\right)  . \label{eqw.1}%
\end{equation}
The quantities $M$ and $\mathcal{D}_{t}$ are defined in \eqref{final1} and
\eqref{final2}, and $\mathcal{Q}$ is defined in \eqref{sp7} (see also
\eqref{sp3} and \eqref{sp5}). By \eqref{sp7} and \eqref{sp10}, equation
\eqref{eqw.1} can equivalently be written as,
\begin{equation}
\label{eqw.2.1}i\partial_{t}w_{\pm}=\widehat{\mathcal{W}}\left(  t\right)
\left(  \mathcal{N}\left(  \left\vert \left(  t^{-1/2} \mathcal{W}\left(
t\right)  w_{\pm}\right)  _{1}\right\vert ,\dots, \left\vert \left(  t^{-1/2}
\mathcal{W}\left(  t\right)  w_{\pm}\right)  _{n}\right\vert \right)
\mathcal{W}\left(  t\right)  w_{\pm}\right)  , \qquad k \in\mathbb{R}^{+}.
\end{equation}
For the definition of $\widehat{\mathcal{W}}\left(  t\right)  $ see
\eqref{sp9}. Moreover, by \eqref{g.1}, \eqref{UnitaritySM} and \eqref{sp10},
equation \eqref{eqw.2.1} is also satisfied for $k \in\mathbb{R}^{-},$ and
then, we have,%

\begin{equation}
i\partial_{t}w_{\pm}=\widehat{\mathcal{W}}\left(  t\right)  \left(
\mathcal{N}\left(  \left\vert (t^{-1/2}\mathcal{W}\left(  t\right)  w_{\pm
})_{1}\right\vert ,\dots,\left\vert (t^{-1/2}\mathcal{W}\left(  t\right)
w_{\pm})_{n}\right\vert \right)  \mathcal{W}\left(  t\right)  w_{\pm}\right)
,\qquad k\in{\mathbb{R}}. \label{eqw}%
\end{equation}
Furthermore the solution $w_{\pm}$ to \eqref{eqw} satisfies the symmetry,
\begin{equation}
S(k)w_{\pm}(t,-k)=w_{\pm}(t,k),\qquad|t|>0,k\in{\mathbb{R}}. \label{eqw.2}%
\end{equation}
Moreover, integrating \eqref{eqw} from $a>0,$ to $t>a,$ and from $t<-a<0,$ to
$-a,$ we obtain,
\begin{equation}%
\begin{array}
[c]{l}%
w_{\pm}(t)=w(\pm a)\pm\frac{1}{i}\,\int_{\mathcal{T}_{\pm}(a,t)}%
\widehat{\mathcal{W}}\left(  \tau\right)  \left(  \mathcal{N}\left(
\left\vert (\tau^{-1/2}\mathcal{W}\left(  \tau\right)  w_{\pm})_{1}\right\vert
,\dots,\left\vert (\tau^{-1/2}\mathcal{W}\left(  \tau\right)  w_{\pm}%
)_{n}\right\vert \right)  \mathcal{W}\left(  \tau\right)  w_{\pm}\right)
\,d\tau.
\end{array}
\label{eqw.3}%
\end{equation}
From \eqref{eqw.3} we see that $w_{\pm}$ satisfies the symmetry \eqref{eqw.2}
for $t\geq a,$ respectively $t\leq-a,$ if and only if it satisfies it for
$t=\pm a.$ Summing up (see Lemma ~\ref{ext}), solving problem \eqref{1.1} with
a solution $u_{\pm}(t)\in H^{2}(\mathbb{R}^{+})\cap L_{2}^{2}(\mathbb{R}^{+})$
for $t\in\lbrack a,\infty),$ respectively, for $t\in(-\infty,-a],$ $a>0,$ is
equivalent to solving equation \eqref{eqw} with a solution $w_{\pm}(t)\in
H^{2}({\mathbb{R}})\cap L_{2}^{2}({\mathbb{R}})$ for $t\in\lbrack a,\infty),$
respectively, for $t\in(-\infty,-a],a>0,$ that satisfies the symmetry,
$S(k)w_{\pm}(-k)=w(k),k\in{\mathbb{R}},$ where $S(k),k\in{\mathbb{R}},$ is the
scattering matrix for the linear matrix Sch\"{o}dinger equation defined in
\eqref{Scatteringmatrix}. We find it convenient to consider \eqref{eqw.3} in
the full line, for the purpose of constructing global solutions to
\eqref{1.1}, and in order to study their asymptotic behaviour for large times.

We recall that the quantities $\mathcal{T}_{\pm}(a,T), 0 \leq a <T \leq
\infty,$ were defined in \eqref{final3}. For $a >0,$ and a measurable function
$w(t)$ of $t\in\mathcal{T}_{\pm}(a,T),$ with values in $H^{2}({\mathbb{R}})$
we define,
\begin{equation}
\Lambda_{\pm}\left(  t;w\right)  =\left\Vert w(t)\right\Vert _{H^{1}%
({\mathbb{R}})}^{\alpha}\left(  \left\Vert w(t)\right\Vert _{H^{1}%
({\mathbb{R}})}+\left\langle t\right\rangle ^{-1/2}\left\Vert w(t)\right\Vert
_{H^{2}({\mathbb{R}})}\right)  ,\qquad\pm t\geq a. \label{II.1}%
\end{equation}

The main estimates that we need are presented in the following lemma, which we
announce here but defer its proof to Section \ref{Lprincipal} below. We recall
that $P_{\pm}$ are, respectively, the orthogonal projectors onto the
eigenvalues plus and minus one of the scattering matrix at zero energy, $S(0).
$ We recall that $\mathcal{T}_{\pm}(a,t)$ was defined in \eqref{eqw.3}.

\begin{lemma}
\label{LemmaPrincipal}Suppose that the potential $V$ is selfadjoint. Further,
assume that the boundary matrices $A,B$, satisfy \eqref{wcon1}, \eqref{wcon2},
and that $H_{A,B,V}$ does not have negative eigenvalues. Moreover, suppose
that the nonlinearity $\mathcal{N}$ fulfills \eqref{ConNonlinearity}. Let
$a,T$ be positive numbers, $\infty\geq T>a>0.$ Then, for some constant $C>0,$
that is independent of $T,$ and is uniform for $a\geq a_{0}>0,$ the following
estimates are true. For any measurable function, $w_{\pm}(t),$ of
$t\in\mathcal{T}_{\pm}\left(  a,T\right)  ,$ with values in $H^{2}%
({\mathbb{R}})$ we have:

\begin{enumerate}
\item Suppose that $V \in L^{1}_{2+\delta}(\mathbb{R}^{+}), \delta>0.$ Then,
\begin{equation}%
\begin{array}
[c]{l}%
\label{18.2.3} \left\Vert \int_{\mathcal{T}_{\pm}(a,t)} \widehat{\mathcal{W}%
}\left(  \tau\right)  \left(  \mathcal{N}\left(  \left\vert ( t^{-1/2}
\mathcal{W}\left(  \tau\right)  w_{\pm})_{1}\right\vert , \dots, \left\vert (
t^{-1/2} \mathcal{W}\left(  \tau\right)  w_{\pm})_{1}\right\vert \right)
\mathcal{W}\left(  \tau\right)  w_{\pm}\right)  d\tau\right\Vert _{L^{2}(
{\mathbb{R}})}\\
\\
+ \left\Vert \int_{\mathcal{T}_{\pm}(a,t)} \widehat{\mathcal{W}}\left(
\tau\right)  \left(  \mathcal{N}\left(  \left\vert ( t^{-1/2} \mathcal{W}%
\left(  \tau\right)  w_{\pm})_{1}\right\vert , \dots, \left\vert ( t^{-1/2}
\mathcal{W}\left(  \tau\right)  w_{\pm})_{n}\right\vert \right)
\mathcal{W}\left(  \tau\right)  w_{\pm}\right)  d\tau\right\Vert _{L^{\infty}(
{\mathbb{R}})}\\
\\
\leq C \frac{|t\mp a|}{\langle t \rangle} \sup_{\tau\in\mathcal{T}_{\pm}(a,t)}
\left\Vert w_{\pm}(\tau) \right\Vert _{H^{1}( {\mathbb{R}})}^{\alpha+1},
\qquad\pm t \geq a.
\end{array}
\end{equation}

\item If, moreover, $V\in L_{7/2+\tilde{\delta}}^{1}(\mathbb{R}^{+}),$ for
some $\tilde{\delta}>0,$ $V$ admits a regular decomposition with $\delta>3/2,$
the nonlinearity $\mathcal{N}$ commutes with the projector $P_{-}$ onto the
eigenspace of the scattering matrix $S(0)$ corresponding to the eigenvalue
$-1,$ i.e.
\[
\mathcal{N}\left(  \left\vert \mu_{1}\right\vert ,\dots,\left\vert \mu
_{n}\right\vert \right)  P_{-}=P_{-}\mathcal{N}\left(  \left\vert \mu
_{1}\right\vert ,\dots,\left\vert \mu_{n}\right\vert \right)  ,\text{ }\mu
\in\mathbb{R}^{n},
\]
$w_{\pm}$ satisfies (\ref{eqw}), and $w_{\pm}(k)=S(k)w_{\pm}(-k),k\in
{\mathbb{R}},$ the following estimates hold,%

\begin{equation}
\label{18.0}%
\begin{array}
[c]{l}%
\left\Vert \langle k\rangle\int_{\mathcal{T}_{\pm}(a,t)}\widehat{\mathcal{W}%
}\left(  \tau\right)  \left(  \mathcal{N}\left(  \left\vert (\tau^{-1/ 2}
\mathcal{W}\left(  \tau\right)  w_{\pm}(\tau))_{1}\right\vert , \dots,
\left\vert (\tau^{-1/ 2} \mathcal{W}\left(  \tau\right)  w_{\pm}(\tau
))_{n}\right\vert \right)  \mathcal{W}\left(  \tau\right)  w_{\pm}\right)
d\tau\right\Vert _{L^{2}( {\mathbb{R}})}\\
\leq C \sup_{\tau\in\mathcal{T}_{\pm}(a,t)} \left(  \left\Vert w_{\pm
}\right\Vert _{H^{1}( {\mathbb{R}})}^{\alpha}+1\right)  \Lambda_{\pm}\left(
t, w_{\pm}\right)  , \qquad\pm t \geq a,
\end{array}
\end{equation}
\begin{equation}
\label{18.1}%
\begin{array}
[c]{l}%
\left\Vert \partial_{k} \int_{\mathcal{T}_{\pm}(a,t)}\widehat{\mathcal{W}%
}\left(  \tau\right)  \left(  \mathcal{N}\left(  \left\vert (\tau^{-1/ 2}
\mathcal{W}\left(  \tau\right)  w_{\pm}(\tau))_{1}\right\vert , \dots,
\left\vert (\tau^{-1/ 2} \mathcal{W}\left(  \tau\right)  w_{\pm}(\tau
))_{n}\right\vert \right)  \mathcal{W}\left(  \tau\right)  w_{\pm}\right)
d\tau\right\Vert _{L^{2}( {\mathbb{R}})}\\
\leq C \sup_{\tau\in\mathcal{T}_{\pm}(a,t)} \left(  \left\Vert w_{\pm
}\right\Vert _{H^{1}( {\mathbb{R}})}^{\alpha}+1\right)  \Lambda_{\pm}\left(
t, w_{\pm}\right)  , \qquad\pm t \geq a,
\end{array}
\end{equation}
and, if in addition $V\in L^{1}_{4}(\mathbb{R}^{+}),$ and $V$ admits a regular
decomposition with $\delta=2,$ we have,
\begin{equation}
\label{18.0.1.1}%
\begin{array}
[c]{l}%
\left\Vert \langle{k}\rangle^{2} \int_{\mathcal{T}_{\pm}(a,t)}%
\widehat{\mathcal{W}}\left(  \tau\right)  \left(  \mathcal{N}\left(
\left\vert (\tau^{-1/ 2} \mathcal{W}\left(  \tau\right)  w_{\pm}(\tau
))_{1}\right\vert , \dots, \left\vert (\tau^{-1/ 2} \mathcal{W}\left(
\tau\right)  w_{\pm}(\tau))_{n}\right\vert \right)  \mathcal{W}\left(
\tau\right)  w_{\pm}\right)  d\tau\right\Vert _{L^{2}( {\mathbb{R}})}\\
\leq C\sqrt{|t|} \sup_{\tau\in\mathcal{T}_{\pm}(a,t)} \left(  1+\left\Vert
w_{\pm}\right\Vert _{H^{1}( {\mathbb{R}})}^{\alpha}\right)  \left(
\Lambda_{\pm}(\tau; w_{\pm})+\left\Vert w_{\pm}\right\Vert _{H^{1}(
{\mathbb{R}})}^{2\alpha+1}\right)  , \qquad\pm t \geq a,
\end{array}
\end{equation}
and
\begin{equation}
\label{18.1.1}%
\begin{array}
[c]{l}%
\left\Vert \partial_{k}^{2} \int_{\mathcal{T}_{\pm}(a,t)}\widehat{\mathcal{W}%
}\left(  \tau\right)  \left(  \mathcal{N}\left(  \left\vert (\tau^{-1/ 2}
\mathcal{W}\left(  \tau\right)  w_{\pm}(\tau))_{1}\right\vert , \dots,
\left\vert (\tau^{-1/ 2} \mathcal{W}\left(  \tau\right)  w_{\pm}(\tau
))_{n}\right\vert \right)  \mathcal{W}\left(  \tau\right)  w_{\pm}\right)
d\tau\right\Vert _{L^{2}( {\mathbb{R}})}\\
\leq C\sqrt{|t|} \sup_{\tau\in\mathcal{T}_{\pm}(a,t)} \left(  1+\left\Vert
w_{\pm}\right\Vert _{H^{1}( {\mathbb{R}})}^{\alpha}\right)  \left(
\Lambda_{\pm}(\tau; w_{\pm})+\left\Vert w_{\pm}\right\Vert _{H^{1}(
{\mathbb{R}})}^{2\alpha+1}\right)  , \qquad\pm t \geq a.
\end{array}
\end{equation}

\end{enumerate}
\end{lemma}

We now give the proof of Theorem~\ref{Theorem 1.1}

\begin{proof}
[Proof of Theorem~ \ref{Theorem 1.1}]By Theorem~\ref{local} we know that
problem \eqref{1.1} has unique solutions
\[
u_{\pm}(t)\in C\left(  \mathcal{T}_{\pm}\left(  0,T_{\pm,\text{max}}\right)
;H^{2}(\mathbb{R}^{+})\cap L_{2}^{2}(\mathbb{R}^{+})\right)  \cap C^{1}\left(
\mathcal{T}_{\pm}\left(  0,T_{\pm,\text{max}}\right)  ;L^{2}(\mathbb{R}%
^{+})\right)  ,
\]
and that if $\varepsilon$ in \eqref{18.2} is small enough $T_{\pm,\text{max}%
}\geq T_{\varepsilon}>0.$ Further, we know that $w_{\pm}(t)$ defined in
\eqref{g.1} satisfies
\[
w_{\pm}(t)\in C\left(  \mathcal{T}_{\pm}\left(  a, T_{\pm,\text{max}}\right)
;H^{2}(\mathbb{R}^{+})\cap L_{2}^{2}(\mathbb{R}^{+})\right)  \cap C^{1}\left(
\mathcal{T}_{\pm}\left(  a, T_{\pm,\text{max}}\right) ; L^{2}(\mathbb{R}%
^{+})\right) , \qquad T_{\pm, \text{max}} \geq T_{\varepsilon}>a>0.
\]
Moreover, $w_{\pm}$ is a solution to \eqref{eqw.2} and to \eqref{eqw.3}, for
$t\in\mathcal{T}_{\pm}\left(  a, T_{\pm,\text{max}}\right)  $, with
$T_{\pm,\text{max}}\geq T_{\varepsilon}>a>0.$

Using that,
\[
\partial_{k} (k^{2} \partial_{k} w_{\pm}\overline{w} )= k^{2} \partial_{k}
w_{\pm}\partial_{k}\overline{ w_{\pm}}+ 2k \partial_{k} w_{\pm}\overline{
w_{\pm}} + k^{2} \partial_{k}^{2} w_{\pm}\overline{w_{\pm}},
\]
and approximating $w_{\pm}$ by functions in $C^{\infty}_{0}( {\mathbb{R}})$ in
the norm of $H^{2}( {\mathbb{R}})\cap L^{2}_{2}( {\mathbb{R}}),$ we obtain,
\[
\int_{ {\mathbb{R}}} k^{2} \partial_{k} w_{\pm}\partial_{k} \overline{w_{\pm}%
}\, dk= -2 \int_{ {\mathbb{R}}} k \partial_{k} w_{\pm}\overline{w_{\pm}} dk+
\int_{ {\mathbb{R}}} k^{2} \partial^{2}_{k} w_{\pm}\overline{w_{\pm}}\, dk.
\]
Then, by the Cauchy-Schwartz inequality,
\[
\left\Vert k \partial_{k} w_{\pm}\right\Vert ^{2}_{L^{2}( {\mathbb{R}})} \leq2
\left\Vert w_{\pm}\right\Vert _{H^{1}( {\mathbb{R}})} \left\Vert kw_{\pm
}\right\Vert _{L^{2}_{1}( {\mathbb{R}})} + \left\Vert w_{\pm}\right\Vert
_{H^{2}( {\mathbb{R}})} \left\Vert w_{\pm}\right\Vert _{L^{2}_{2}(
{\mathbb{R}})},
\]
and hence,
\begin{equation}
\label{kpartialk}\left\Vert k \partial_{k} w_{\pm}\right\Vert _{L^{2}(
{\mathbb{R}})} \leq C \left[  \left\Vert w_{\pm}\right\Vert _{H^{2}(
{\mathbb{R}})} + \left\Vert w_{\pm}\right\Vert _{L^{2}_{2}( {\mathbb{R}})}
\right]  .
\end{equation}
Note that by \eqref{g.1} and \eqref{spectralrepr},
\begin{equation}
\label{pppccc}e^{-it k^{2}} w_{\pm}= \mathbf{F }u_{\pm}.
\end{equation}
Then, by \eqref{76.26.23.aaxxaa}, \eqref{7.26.23.1}, \eqref{kpartialk},
\eqref{ab.2.1.1}, and \eqref{ab.2.1.2}

\begin{equation}
\label{18.2.1}\|w_{\pm}(a)\|_{H^{2}( {\mathbb{R}})}+ \| w_{\pm}(a)\|_{L^{2}%
_{2}( {\mathbb{R}})} \leq C \tilde{\varepsilon}, \, \text{\textrm{with}}\,
\tilde{\varepsilon}:= \varepsilon(1+ \varepsilon^{\alpha}).
\end{equation}
We will obtain an a priori bound for $w_{\pm}(t),$ $t\in\mathcal{T}_{\pm
}\left(  a, T_{\pm, \text{max}}\right) .$ For this purpose, we introduce the
following notation,
\begin{equation}
\Gamma_{\pm}(t):=\sup_{\tau\in\mathcal{T}_{\pm}(a,t)}\left[  \Vert w_{\pm
}(\tau)\Vert_{H^{1}({\mathbb{R}})}+\frac{1}{\sqrt{\langle\tau\rangle}}\Vert
w_{\pm}(\tau)\Vert_{H^{2}({\mathbb{R}})}\right]  ,\qquad\pm t \geq a .
\label{18.2.2}%
\end{equation}

We first observe that using \eqref{eqw.3}, \eqref{18.2.3}, and \eqref{18.2.1}
we have the estimate
\begin{equation}
\left\Vert w_{\pm}\left(  t\right)  \right\Vert _{L^{2}({\mathbb{R}})}\leq
C\,\tilde{\varepsilon}+C\Gamma_{\pm}(t)^{\alpha+1},\qquad\pm t\geq a.
\label{18.2.4}%
\end{equation}
Moreover, by \eqref{eqw.3}, \eqref{18.1} and \eqref{18.2.1}
\begin{equation}
\left\Vert \partial_{k}w_{\pm}\left(  t\right)  \right\Vert _{L^{2}%
({\mathbb{R}})}\leq C\tilde{\varepsilon}+C(1+\Gamma_{\pm}(t)^{\alpha}%
)\Gamma_{\pm}(t)^{\alpha+1},\qquad\pm t\geq a, \label{18.2.5}%
\end{equation}
and, using \eqref{18.1.1} instead of \eqref{18.1},
\begin{equation}
\left\Vert \partial_{k}^{2}w_{\pm}\left(  t\right)  \right\Vert _{L^{2}%
({\mathbb{R}})}\leq C\tilde{\varepsilon}+\sqrt{|t|}\,C(1+\Gamma_{\pm
}(t)^{\alpha})\Gamma_{\pm}(t)^{\alpha+1},\qquad\pm t\geq a. \label{18.2.6}%
\end{equation}
It follows from \eqref{eqw}, \eqref{eqw.3}, \eqref{18.0.1.1}, and \eqref{18.2.1},%

\begin{equation}
\label{18.2.7}\left\Vert w_{\pm}\left(  t\right)  \right\Vert _{L^{2}_{2}(
{\mathbb{R}})} \leq C \tilde{\varepsilon}+ C \sqrt{|t|} (1+\Gamma_{\pm
}(t)^{\alpha}) \Gamma_{\pm}(t)^{\alpha+1}, \qquad\pm t \geq a, \qquad\pm t
\geq a.
\end{equation}

Let us designate,
\begin{equation}
\beta_{\pm}(t):=\Gamma_{\pm}(t)+\sup_{\tau\in\mathcal{T}_{\pm}(a,t)}\frac
{1}{|\tau|^{1/2}}\left\Vert w_{\pm}\left(  \tau\right)  \right\Vert
_{L_{2}^{2}({\mathbb{R}})}, \qquad\pm t \geq a. \label{18.2.8}%
\end{equation}
Then, by \eqref{18.2.4}, \eqref{18.2.5}, \eqref{18.2.6} and \eqref{18.2.7}
\begin{equation}
\beta_{\pm}(t)\leq6\,C\tilde{\varepsilon}+C\,\left(  \beta^{\alpha}+1\right)
\beta^{\alpha+1},\qquad\pm t\geq a. \label{18.2.9}%
\end{equation}
Take $\tilde{\varepsilon}$ so small that,
\begin{equation}
C\left(  (7C\tilde{\varepsilon})^{\alpha}+1\right)  (7C\tilde{\varepsilon
})^{\alpha+1}<\frac{1}{2}C\tilde{\varepsilon},\qquad\pm t\geq a.
\label{18.2.10}%
\end{equation}
Suppose that for some $\theta_{\pm}$, with $a\leq\theta_{+}<T_{+,\mathrm{max}%
},$ respectively, $T_{-,\mathrm{max}}<\theta_{-}\leq-a,$
\begin{equation}
\beta_{\pm}(\theta_{\pm})=7C\tilde{\varepsilon}. \label{18.2.11.0}%
\end{equation}
Then, by \eqref{18.2.9}, \eqref{18.2.10}, \eqref{18.2.11.0} we reach the
contradiction,
\[
7C\tilde{\varepsilon}\leq\left(  6+\frac{1}{2}\right)  C\tilde{\varepsilon}.
\]
Then for all $\theta_{\pm}$ either $\beta_{\pm}(\theta_{\pm})<7C\tilde
{\varepsilon}$ or $\beta_{\pm}(\theta)>7C\tilde{\varepsilon}.$ However, since
$\beta_{\pm}(\pm a)\leq6\,C\,\tilde{\varepsilon},$ if for some $\theta_{\pm
},\beta_{\pm}(\theta_{\pm})>7C\tilde{\varepsilon},$ by continuity, for some
other $\theta_{1,\pm},$ $\beta_{\pm}(\theta_{1,\pm})=7C\tilde{\varepsilon},$
and this is not possible. Hence, we have proved the following a priori bound,
\begin{equation}
\beta_{\pm}(t)<7C\tilde{\varepsilon},\qquad t\in\mathcal{T}_{\pm}\left( a,
T_{\pm,\text{max}}\right)  . \label{18.2.11.xx}%
\end{equation}
Then, by \eqref{pppccc}, \eqref{18.2.11.xx}, and \eqref{ab.2.1.1},
\begin{equation}
\left\Vert u_{\pm}\right\Vert _{H^{2}(\mathbb{R}^{+})}\leq C\sqrt{|t|}%
\tilde{\varepsilon},\qquad t\in\mathcal{T}_{\pm}\left(  a, T_{\pm,\text{max}%
}\right)  , \label{nnmm}%
\end{equation}
and using \eqref{kpartialk}, and \eqref{ab.2.1.2} instead of \eqref{ab.2.1.1},
we get,
\begin{equation}%
\begin{array}
[c]{l}%
\left\Vert u_{\pm}\right\Vert _{L_{2}^{2}({\mathbb{R}})}\leq C\left\Vert
e^{-itk^{2}}w_{\pm}\right\Vert _{H^{2}({\mathbb{R}})}\leq C|t|^{2}\left[
\left\Vert w_{\pm}\right\Vert _{L_{2}^{2}({\mathbb{R}})}+\left\Vert w_{\pm
}\right\Vert _{H^{2}({\mathbb{R}})}\right] \\
\leq C|t|^{5/2}\,\tilde{\varepsilon},\qquad t\in\mathcal{T}_{\pm}\left( a,
T_{\pm,\text{max}}\right)  .
\end{array}
\label{nn.mm.1}%
\end{equation}
%
Hence, if $T_{\pm,\mathrm{max}}<\infty,$
\[
\sup_{t\in\mathcal{T}_{\pm}\left( a, T_{\pm,\text{max}}\right) } \left(
\left\Vert u_{\pm}(t)\right\Vert _{H^{2}(\mathbb{R}^{+})}+\Vert u_{\pm
}(t)\Vert_{L_{2}^{2}(\mathbb{R}^{+})}\right)  \leq C<\infty.
\]
However, this is in contradiction with Theorem~\ref{local}. In consequence,
$T_{\pm,\mathrm{max}}=\pm\infty.$

It only remains to prove \eqref{1.3-1}, \eqref{1.3-2}, and \eqref{1.3}. Using
(\ref{eqw.3}), (\ref{18.2.11.xx}), and \eqref{EST20}, we show that
\begin{equation}
\left\Vert w_{\pm}\left(  t\right)  -w_{\pm}\left(  s\right)  \right\Vert
_{L^{2}( {\mathbb{R}})}+\left\Vert w_{\pm}\left(  t\right)  -w_{\pm}\left(
s\right)  \right\Vert _{L^{\infty}( {\mathbb{R}})}\leq C\tilde{\varepsilon
}^{\alpha+1}|s|^{-\left(  \alpha/2-1\right)  }, \qquad|t| >|s|. \label{43}%
\end{equation}
Thus, $w_{\pm}\left(  t\right)  $ is a Cauchy sequence, and hence, there
exists a limit $v_{\pm}\in L^{2}( {\mathbb{R}})\cap L^{\infty}( {\mathbb{R}%
}),$
\begin{equation}
\label{43.1}v_{\pm}:= \lim_{t \to\pm\infty} w_{\pm}(t),
\end{equation}
and
\begin{equation}
\label{43.1.2}\|w_{\pm}(t)- v_{\pm}\|_{L^{2}( {\mathbb{R}})}+ \|w_{\pm}(t)-
v_{\pm}\|_{L^{\infty}( {\mathbb{R}})} \leq C\tilde{\varepsilon}^{\alpha
+1}|t|^{-\left(  \alpha/2-1\right)  }.
\end{equation}
Moreover, by \eqref{18.2.11.xx}%

\begin{equation}
\Vert v_{\pm}\Vert_{H^{1}(\mathbb{R}^{+})}\leq7\,C\,\tilde{\varepsilon}.
\label{43.1.2.1}%
\end{equation}
We denote by $w_{\pm\infty}\in L^{2}(\mathbb{R}^{+})\cap L^{\infty}%
(\mathbb{R}^{+}),$ the restriction to $\mathbb{R}^{+}$ of $v_{\pm},$ i.e.,
\begin{equation}
w_{\pm\infty}(k):=v_{\pm}(k),\qquad k\in\mathbb{R}^{+}. \label{43.1.4}%
\end{equation}
Since $H$ has no eigenvalues $P_{c}(H)=I.$ Hence, by \eqref{g.1} and
\eqref{ap.65} we get%
\begin{equation}
u_{\pm}\left(  t\right)  =M\mathcal{D}_{t}\mathcal{W}\left(  t\right)  w_{\pm
}. \label{43.2}%
\end{equation}
Therefore, using \eqref{18.2.11.xx}, \eqref{43.1.2}, \eqref{2.11}, and
\eqref{L2} we get,%
\begin{equation}%
\begin{array}
[c]{l}%
\left\Vert u_{\pm}\left(  t\right)  -M\mathcal{D}_{t}\mathcal{V}\left(
t\right)  v_{\pm}\right\Vert _{L^{2}(\mathbb{R}^{+})}\leq\left\Vert \left(
\mathcal{W}\left(  t\right)  -\mathcal{V}\left(  t\right)  \right)  w_{\pm
}\right\Vert _{L^{2}(\mathbb{R}^{+})}+\left\Vert M\mathcal{D}_{t}%
\mathcal{V}\left(  t\right)  \left(  w_{\pm}-v_{\pm}\right)  \right\Vert
_{L^{2}(\mathbb{R}^{+})}\label{43.2.1}\\
\\
\ \leq C\left(  \tilde{\varepsilon}^{\alpha+1}\,|t|^{-\left(  \frac{\alpha}%
{2}-1\right)  }+\tilde{\varepsilon}|t|^{-1/2}\right)  ,\qquad|t|\geq a.
\end{array}
\end{equation}
Moreover, by \eqref{18.2.11.xx}, \eqref{43.1.2}, \eqref{43.2}, \eqref{3.2} and
\eqref{Linf},
\begin{equation}%
\begin{array}
[c]{l}%
\left\Vert u_{\pm}\left(  t\right)  -\frac{1}{\sqrt{2}}M\mathcal{D}_{t}\left(
m\left(  \frac{x}{2},tx\right)  w_{\pm\infty}\left(  \frac{x}{2}\right)
\right)  \right\Vert _{L^{\infty}(\mathbb{R}^{+})}\leq\left\Vert u_{\pm
}\left(  t\right)  -\frac{1}{\sqrt{2}}M\mathcal{D}_{t}\left(  m\left(
\frac{x}{2},tx\right)  w_{\pm}\left(  t,\frac{x}{2}\right)  \right)
\right\Vert _{L^{\infty}(\mathbb{R}^{+})}\\
\\
+\left\Vert \ \frac{1}{\sqrt{2}}M\mathcal{D}_{t}\left(  m\left(  \frac{x}%
{2},tx\right)  (w_{\pm}\left(  t,\frac{x}{2}\right)  -w_{\pm\infty}\left(
\frac{x}{2}\right)  \right)  \right\Vert _{L^{\infty}(\mathbb{R}^{+})}\leq
C|t|^{-1/2}\left(  \tilde{\varepsilon}^{\alpha+1}|t|^{-\left(  \frac{\alpha
}{2}-1\right)  }+\tilde{\varepsilon}\,|t|^{-1/4}\right)  ,|t|>a.
\end{array}
\label{43.1.3}%
\end{equation}
Equation \eqref{43.1.3} proves \eqref{1.3-2}. Observe that \eqref{43.1.2.1},
\eqref{43.1.4}, \eqref{43.1.3}, and \eqref{3.2} imply the $L^{\infty
}(\mathbb{R}^{+})$ decay estimate (\ref{1.3}). It follows from \eqref{g.1} and
\eqref{43.1}, that
\begin{equation}
v_{\pm}(k)=S(k)v_{\pm}(-k). \label{43.1.4.1}%
\end{equation}
Further, by \eqref{sp4.0} and\eqref{2.5.1}%

\begin{equation}
\label{43.1.5}%
\begin{array}
[c]{l}%
M\mathcal{D}_{t} \mathcal{V}(t)v_{\pm}-e^{-it H_{0}} \mathbf{F}_{0}^{\dagger
}w_{\pm\infty}= M \mathcal{D}_{t} \sqrt{\frac{it}{2\pi}} \int_{-\infty}^{0}
e^{-it(k- \frac{x}{2})^{2}} (S(k)-S_{0}(k)) v_{\pm}(-k) dk\\
\\
= M \mathcal{D}_{t} \sqrt{\frac{it}{2\pi}} \int_{0}^{\infty} e^{-it(k+\frac
{x}{2})^{2}} (S(-k)-S_{0}(-k)) v_{\pm}(k) dk.
\end{array}
\end{equation}
We have,
\begin{equation}
\label{43.1.6}e^{-it\left(  k+\frac{x}{2}\right)  ^{2}}=\frac{\partial
_{k}\left(  \left(  k+\frac{x}{2}\right)  e^{-it\left(  k+\frac{x}{2}\right)
^{2}}\right)  }{1-2it\left(  k+\frac{x}{2}\right)  ^{2}}.
\end{equation}
Using \eqref{43.1.6} and integrating by parts we obtain,
\begin{equation}
\label{43.1.7}%
\begin{array}
[c]{l}%
\sqrt{\frac{it}{2\pi}}\int_{0}^{\infty} e^{-it(k+\frac{x}{2})^{2}}
(S(-k)-S_{0}(-k)) v_{\pm}(k) dk=\displaystyle - \sqrt{\frac{it}{2\pi}}
\frac{x}{2} e^{-it\left(  \frac{x}{2}\right)  ^{2}} \frac{1}{1-2it\left(
\frac{x}{2}\right)  ^{2}} (S(0)-S_{0}(0)) v_{\pm}(0)\\
\\
- \sqrt{\frac{i|t|}{2\pi}} \int_{0}^{\infty} e^{-it(k+\frac{x}{2})^{2}} 4 i t
\frac{\left(  k+\frac{x}{2}\right)  ^{2}}{\left(  1-2it\left(  k+\frac{x}%
{2}\right)  ^{2}\right)  ^{2}} (S(-k)-S_{0}(-k)) v_{\pm}(k) dk\\
\\
- \sqrt{\frac{i|t|}{2\pi}} \int_{0}^{\infty} e^{-it(k+\frac{x}{2})^{2}} \frac{
k+\frac{x}{2}}{1-2it\left(  k+\frac{x}{2}\right)  ^{2}} \partial_{k}\left[
(S(-k)-S_{0}(-k)) v_{\pm}(k)\right]  dk.
\end{array}
\end{equation}
Note that,
\[%
\begin{array}
[c]{l}%
\left|  \sqrt{\frac{it}{2\pi}} \int_{0}^{\infty} e^{-it(k+\frac{x}{2})^{2}} 4
i t \frac{\left(  k+\frac{x}{2}\right)  ^{2}}{\left(  1-2it\left(  k+\frac
{x}{2}\right)  ^{2}\right)  ^{2}} (S(-k)-S_{0}(-k)) v_{\pm}(k) dk \right| \\
\\
\leq C \int_{\sqrt{|t|}x/2}^{\infty} \frac{z^{2}}{\left(  1+z^{2}\right)
^{2}} v_{\pm}(k) dk \|v_{\pm}\|_{L^{\infty}(\mathbb{R}^{+})} \leq C \frac
{1}{1+ \sqrt{|t|}x} \|v_{\pm}\|_{L^{\infty}(\mathbb{R}^{+})}.
\end{array}
\]
Moreover,
\[%
\begin{array}
[c]{l}%
\left|  \sqrt{\frac{it}{2\pi}} \int_{0}^{\infty} e^{-it(k+\frac{x}{2})^{2}}
\frac{ k+\frac{x}{2}}{1-2it\left(  k+\frac{x}{2}\right)  ^{2}} \partial
_{k}\left[  (S(-k)-S_{0}(-k)) v_{\pm}(k)\right]  dk\right| \\
\\
\leq\frac{1}{\sqrt{2\pi}} \frac{1}{1+\sqrt{t}x} \int_{\mathbb{R}^{+}} \left|
\partial_{k}\left[  (S(-k)-S_{0}(-k)) v_{\pm}(k)\right]  \right|  dk
\leq\displaystyle C \frac{1}{1+\sqrt{|t|}x} \|v_{\pm}\|_{L^{2}(\mathbb{R}%
^{+})}.
\end{array}
\]
Hence, by \eqref{43.1.2.1}, \eqref{UnitaritySM}, \eqref{scatmatrixderiv}, and
Sobolev's inequality,
\begin{equation}
\label{43.1.8}%
\begin{array}
[c]{l}%
\left\|  \sqrt{\frac{it}{2\pi}} \int_{0}^{\infty} e^{-it(k+\frac{x}{2})^{2}}
(S(-k)-S_{0}(-k)) v_{\pm}(k) dk \right\|  _{L^{2}(\mathbb{R}^{+})} \leq C
\tilde{\varepsilon} \left[  \sqrt{|t|}\left\|  \frac{x}{1+t x ^{2}} \right\|
_{L^{2}(\mathbb{R}^{+})} \right. \\
\\
\left.  + \left\|  \frac{1}{1+ \sqrt{|t|x}}\right\|  _{L^{2}(\mathbb{R}^{+})}
\right]  .
\end{array}
\end{equation}
Further, by \eqref{43.2.1}, \eqref{43.1.5}, and \eqref{43.1.8}
\begin{equation}
\label{43.1.9}\left\Vert u_{\pm}\left(  t\right)  - e^{-it H_{0}}
\mathbf{F}_{0}^{\dagger}w_{\pm\infty}\right\Vert _{L^{2}(\mathbb{R}^{+})}\leq
C\left(  \tilde{\varepsilon}^{\alpha+1}\, |t|^{-\left(  \frac{\alpha}%
{2}-1\right)  }+ \tilde{\varepsilon} |t|^{-1/4}\right)  , \qquad|t| \geq a.
\end{equation}
This proves \eqref{1.3-1}, and completes the proof of Theorem
\ref{Theorem 1.1}.
\end{proof}

\begin{remark}{\rm
\label{bounds}\textrm{By \eqref{18.2.8} and \eqref{18.2.11.xx} with $T_{\pm,
\text{max}}= \infty,$ }

\textrm{%
\begin{equation}
\label{bound.1}\left\Vert w_{\pm}\right\Vert _{H^{1}( {\mathbb{R}})}+ \frac
{1}{\sqrt{|t|}} \left\Vert w_{\pm}\right\Vert _{H^{2}( {\mathbb{R}})}+
\frac{1}{\sqrt{|t|}} \left\Vert w_{\pm}\right\Vert _{L^{2}_{2}( {\mathbb{R}})}
\leq C \tilde{\varepsilon} , \qquad\pm t \geq a,
\end{equation}
and by \eqref{18.0}, and \eqref{bound.1},
\begin{equation}
\label{bound.2}\left\Vert w_{\pm}\right\Vert _{L^{2}_{1}( {\mathbb{R}})} \leq
C \tilde{\varepsilon} , \qquad\pm t \geq a.
\end{equation}
}

\textrm{Further, by \eqref{76.26.23.aaxxaa}, \eqref{7.26.23.1},\eqref{nnmm}
with $T_{\pm, \text{max}}= \infty,$ we have,
\begin{equation}
\label{bound.3}\left\Vert u_{\pm}\right\Vert _{H^{2}( {\mathbb{R}})}\leq C
\tilde{\varepsilon} \langle\sqrt{|t|}\rangle, \qquad\pm t \geq0.
\end{equation}
Moreover, by \eqref{76.26.23.aaxxaa}, \eqref{7.26.23.1}, \eqref{pppccc},
\eqref{bound.2}, and \eqref{ab.2.1.1.0},
\begin{equation}
\label{bound.4}\left\Vert u_{\pm}\right\Vert _{H^{1}( {\mathbb{R}})}\leq C
\tilde{\varepsilon}, \qquad\pm t \geq0.
\end{equation}
}}
\end{remark}

\section{\bigskip Auxiliary estimates.\label{AuxEst}}

In this section we prove the estimates that we need. We relay on the results
in Appendix ~\ref{App1}. Let $w_{\pm}$ be functions from $t\in\mathcal{T}%
_{\pm}\left(  a,\infty\right)  ,$ $a>0,$ with values in $H^{1}({\mathbb{R}}).
$ We denote
\begin{equation}
f_{1,\pm}:=\mathcal{N}\left(  \left\vert (t^{-1/2}\mathcal{W}\left(  t\right)
w_{\pm})_{1}\right\vert ,\dots,\left\vert (t^{-1/2}\mathcal{W}\left(
t\right)  w_{\pm})_{n}\right\vert \right)  \mathcal{W}\left(  t\right)
w_{\pm}, \label{nnn.1}%
\end{equation}
and%
\begin{equation}
f_{1,\pm}^{\operatorname*{fr}}:=\mathcal{N}\left(  \left\vert (t^{-1/2}%
\mathcal{V}\left(  t\right)  w_{\pm})_{1}\right\vert ,\dots\left\vert
(t^{-1/2}\mathcal{V}\left(  t\right)  w_{\pm})_{n}\right\vert \right)
\mathcal{V}\left(  t\right)  w_{\pm}. \label{nnn.2}%
\end{equation}

\begin{lemma}
Assume that the nonlinear function $\mathcal{N}$ satisfy
(\ref{ConNonlinearity}). Let $w_{\pm}$ be functions from $t\in\mathcal{T}%
_{\pm}\left(  a,\infty\right)  ,$ $a>0,$ with values in $H^{1}({\mathbb{R}}).$
Then, there is a constant $C$ with the following properties.

\begin{enumerate}
\item Suppose that the $V\in L_{2+\delta}^{1}(\mathbb{R}^{+}),$ for some
$\delta>0.$ Then
\begin{equation}
\left\Vert f_{1,\pm}\right\Vert _{L^{2}({\mathbb{R}})}\leq C|t|^{-\alpha
/2}\left\Vert w_{\pm}\right\Vert _{H^{1}({\mathbb{R}})}^{\alpha+1},\qquad\pm
t\geq a, \label{12}%
\end{equation}%
\begin{equation}
\left\Vert f_{1,\pm}\right\Vert _{L^{\infty}({\mathbb{R}})}\leq C|t|^{-\alpha
/2}\left\Vert w_{\pm}\right\Vert _{H^{1}({\mathbb{R}})}^{\alpha+1},\qquad\pm
t\geq a. \label{12.0.0}%
\end{equation}

\item
\begin{equation}
\left\Vert f_{1,\pm}^{\operatorname*{fr}}\right\Vert _{L^{2}( {\mathbb{R}}%
)}\leq|t|^{-\alpha/2} C\left\Vert w_{\pm}\right\Vert _{H^{1}( {\mathbb{R}}%
)}^{\alpha+1}, \qquad\pm t\geq a, \label{12.1}%
\end{equation}
\begin{equation}
\label{12.1.00}\left\Vert f_{1,\pm}^{\operatorname*{fr}}\right\Vert
_{L^{\infty}( {\mathbb{R}})} \leq C |t|^{-\alpha/2} \left\Vert w_{\pm
}\right\Vert _{H^{1}( {\mathbb{R}})}^{\alpha+1}, \qquad\pm t \geq a,
\end{equation}

\begin{equation}
\left\Vert \partial_{x}^{j}f_{1,\pm}^{\operatorname*{fr}}\right\Vert _{L^{2}(
{\mathbb{R}})}\leq C |t|^{- \alpha/2} \left\Vert w_{\pm}\right\Vert _{H^{1}(
{\mathbb{R}})}^{\alpha}\left\Vert w_{\pm}\right\Vert _{H^{j}( {\mathbb{R}})},
\qquad j= 1,2, \pm t \geq a. \label{derivativeffr}%
\end{equation}
Here, for $j=2,$ we assume that $w_{\pm}(t, \cdot) \in H^{2}( {\mathbb{R}}).$
\end{enumerate}
\end{lemma}

\begin{proof}
Equation \eqref{12} follows from \eqref{l2.1}, \eqref{12.0.0} follows from
\eqref{EST4}. Equation \eqref{12.1} follows from \eqref{2.11}, and
\eqref{12.1.00} follows from \eqref{89.xxxx}. Further, \eqref{derivativeffr}
is a consequence of \eqref{ConNonlinearity}, \eqref{EST5fr}, and
\eqref{89.xxxx}. Moreover, to prove each one of these equations we use
Sobolev's inequality.
\end{proof}

We also need the following:

\begin{lemma}
Suppose that $V\in L_{2+\delta}^{1}(\mathbb{R}^{+}),$ for some $\delta>0.$ Let
the nonlinear function $\mathcal{N}$ satisfy (\ref{ConNonlinearity}). Let
$a,T$ be positive numbers, $\infty\geq T>a>0.$ Then, for some constant $C>0,$
that is independent of $T,$ and is uniform for $a\geq a_{0}>0,$ the following
estimates are true. For any measurable function, $w_{\pm}(t),$ of
$t\in\mathcal{T}_{\pm}\left(  a,T\right)  ,$ with values in $H^{1}%
({\mathbb{R}})$ we have
\begin{equation}
\left\Vert \widehat{\mathcal{W}}\left(  t\right)  f_{1,\pm}\right\Vert
_{L^{2}({\mathbb{R}})}+\left\Vert \widehat{\mathcal{W}}\left(  t\right)
f_{1,\pm}\right\Vert _{L^{\infty}({\mathbb{R}})}\leq C|t|^{-\alpha
/2}\left\Vert w_{\pm}\right\Vert _{H^{1}({\mathbb{R}})}^{\alpha+1},\qquad\pm
t\geq a. \label{EST20}%
\end{equation}

\end{lemma}

\begin{proof}
Equation \eqref{EST20} follows from \eqref{l2.1}, \eqref{EST4} and Sobolev's inequality.
\end{proof}

\begin{lemma}
Suppose that $V\in L_{2+\delta}^{1}(\mathbb{R}^{+}),$ for some $\delta>0,$
that the boundary matrices $A,B$, satisfy \eqref{wcon1}, \eqref{wcon2}, and
that $H_{A,B,V}$ does not have negative eigenvalues. Let the nonlinear
function $\mathcal{N}$ satisfy (\ref{ConNonlinearity}). Let $\zeta\in
C^{2}\left(  \mathbb{R}\right)  $ such that%
\begin{align}
\left\vert \partial^{j}\left(  \zeta-1\right)  \right\vert  &  \leq
C\left\langle k\right\rangle ^{-1},\text{ }j=0,1,\label{zeta.0}\\
\text{\ }\left\vert \partial^{2}\left(  \zeta-1\right)  \right\vert  &  \leq
C. \label{zeta}%
\end{align}
Then, for some constant $C>0,$ that is independent of $T,$ and is uniform for
$a\geq a_{0}>0,$ the following estimates are true. Assume that $w_{\pm}$
satisfies (\ref{eqw}) for $t\in\mathcal{T}_{\pm}\left(  a,\infty\right)  ,$
$a>0.$ Then, the following is true:
\begin{equation}
\partial_{t}\left(  \mathcal{V}\left(  t\right)  \left(  \zeta w_{\pm}\right)
\right)  =\mathcal{A}_{1,\pm}(t)+\mathcal{A}_{2,\pm}(t), \label{98}%
\end{equation}
where%
\begin{equation}
\left\Vert \mathcal{A}_{1,\pm}(t)\right\Vert _{L^{2}({\mathbb{R}})}\leq
C|t|^{-3/2}\left(  |t|^{-1/2}\left\Vert w_{\pm}\right\Vert _{H^{2}%
({\mathbb{R}})}+\Lambda_{\pm}\left(  t;w_{\pm}\right)  \right)  ,\qquad\pm
t\geq a, \label{98.1}%
\end{equation}
and%
\begin{equation}
\left\Vert \mathcal{A}_{2,\pm}(t)\right\Vert _{L^{\infty}({\mathbb{R}}%
)}+|t|^{-1/2}\left\Vert \mathcal{A}_{2,\pm}(t)\right\Vert _{H^{1}({\mathbb{R}%
})}\leq C|t|^{-1}\Lambda_{\pm}\left(  t;w_{\pm}\right)  ,\qquad\pm t\geq a.
\label{98.2}%
\end{equation}

\end{lemma}

\begin{proof}
To simplify the notation, in the calculations below we denote by $f_{1,\pm}$,
and $f^{\mathrm{fr}}_{1,\pm},$ respectively, the restriction of $f_{1,\pm}, $
and of $f^{\mathrm{fr}}_{1,\pm},$ to $k \geq0.$ Using that $w_{\pm}$ satisfies
(\ref{eqw}) and \eqref{sp11.2.1} we calculate%
\begin{align}
\partial_{t}\left(  \mathcal{V}\left(  t\right)  \left(  \zeta w_{\pm}\right)
\right)   &  =\frac{i}{4t^{2}}\mathcal{V}\left(  t\right)  \partial_{k}%
^{2}\left(  \zeta w_{\pm}\right)  +\mathcal{V}\left(  t\right)  \zeta
\partial_{t}w_{\pm}\nonumber\\
&  =\frac{i}{4t^{2}}\mathcal{V}\left(  t\right)  \partial_{k}^{2}\left(  \zeta
w_{\pm}\right)  + \mathcal{V}\left(  t\right)  \zeta\left(
\widehat{\mathcal{W}}\left(  t\right)  f_{1,\pm} \right)  . \label{99}%
\end{align}
First, by (\ref{2.11})%
\begin{equation}
\label{99.1}\left\Vert \mathcal{V}\left(  t\right)  \partial_{k}^{2}\left(
\zeta w_{\pm}\right)  \right\Vert _{L^{2}(\mathbb{R}^{+})}\leq C \|w_{\pm
}\|_{H^{2}( {\mathbb{R}})}.
\end{equation}
Observe that by \eqref{sp11.1}
\begin{equation}
\label{99.2}\widehat{\mathcal{W}}(t)f_{1,\pm}(k) = \mathcal{E}\left(
\chi_{(0,\infty)}(k) \widehat{\mathcal{W}}(t) f_{1}(k)\right)  , \qquad k
\in{\mathbb{R}}.
\end{equation}

Next, since $H_{A,B,V}$ does not have negative eigenvalues, it follows from
\eqref{99.2}, \eqref{sp7} and \eqref{sp10},
\begin{equation}
\mathcal{W}(t)\widehat{\mathcal{W}}(t)f_{1,\pm}=\mathcal{W}(t)\mathcal{E}%
\left(  \chi_{(0,\infty)}(k)\widehat{\mathcal{W}}(t)f_{1,\pm}(k)\right)
=\mathcal{Q}(t)\mathcal{Q}^{-1}(t)f_{1,\pm}=f_{1,\pm}. \label{305}%
\end{equation}
Then, we decompose%
\begin{align}
\mathcal{V}\left(  t\right)  \zeta\left(  \widehat{\mathcal{W}}\left(
t\right)  f_{1,\pm}\right)   &  =\left(  \mathcal{V}\left(  t\right)
\zeta-\mathcal{W}\left(  t\right)  \right)  \left(  \left(
\widehat{\mathcal{W}}\left(  t\right)  -\widehat{\mathcal{V}}\left(  t\right)
\right)  f_{1,\pm}\right) \nonumber\\
&  +\left(  \mathcal{V}\left(  t\right)  \zeta-\mathcal{W}\left(  t\right)
\right)  \left(  \widehat{\mathcal{V}}\left(  t\right)  \left(  f_{1,\pm
}-f_{1,\pm}^{\operatorname*{fr}}\right)  \right) \nonumber\\
&  +\left(  \mathcal{V}\left(  t\right)  \left(  \zeta-1\right)  \right)
\left(  \widehat{\mathcal{V}}\left(  t\right)  f_{1,\pm}^{\operatorname*{fr}%
}\right) \nonumber\\
&  +\left(  \mathcal{V}\left(  t\right)  -\mathcal{W}\left(  t\right)
\right)  \left(  \widehat{\mathcal{V}}\left(  t\right)  f_{1\pm,}%
^{\operatorname*{fr}}\right)  +f_{1,\pm}. \label{304}%
\end{align}
Using (\ref{12}), (\ref{l2.1}), (\ref{2.11}), and \eqref{EST12} with $n=0,$ we
show%
\begin{equation}
\left\Vert \left(  \mathcal{V}\left(  t\right)  \zeta-\mathcal{W}\left(
t\right)  \right)  \left(  \widehat{\mathcal{W}}\left(  t\right)
-\widehat{\mathcal{V}}\left(  t\right)  \right)  f_{1,\pm,}\right\Vert
_{L^{2}({\mathbb{R}})}\leq C_{a}|t|^{-1/2}\left\Vert f_{1,\pm}\right\Vert
_{L^{\infty}(\mathbb{R}^{+})}\leq C_{a}\frac{1}{|t|^{\frac{1+\alpha}{2}}%
}\left\Vert w_{\pm}\right\Vert _{H^{1}({\mathbb{R}})}^{\alpha+1},\qquad\pm
t\geq a. \label{95}%
\end{equation}
By (\ref{ConNonlinearity}), (\ref{EST4}) and (\ref{EST11}) we get
\begin{equation}
\left\Vert f_{1,\pm}-f_{1,\pm}^{\operatorname*{fr}}\right\Vert _{L^{2}%
({\mathbb{R}})}\leq C_{a}\frac{1}{|t|^{\frac{\alpha}{2}}}\left\Vert
(\mathcal{W}\left(  t\right)  -\mathcal{V}\left(  t\right)  )w_{\pm
}\right\Vert _{L^{2}({\mathbb{R}})}\left\Vert w_{\pm}\right\Vert
_{H^{1}({\mathbb{R}})}^{\alpha}\leq C_{a}\frac{1}{|t|^{\frac{1+\alpha}{2}}%
}\left\Vert w_{\pm}\right\Vert _{H^{1}({\mathbb{R}})}^{\alpha+1},\qquad\pm
t\geq a. \label{101}%
\end{equation}
Further, it follows from (\ref{l2.1}), (\ref{2.11}), (\ref{isomafr}), and
\eqref{101},
\begin{align}
\left\Vert \left(  \mathcal{V}\left(  t\right)  \zeta-\mathcal{W}\left(
t\right)  \right)  \widehat{\mathcal{V}}\left(  t\right)  \left(  f_{1,\pm
}-f_{1,\pm}^{\operatorname*{fr}}\right)  \right\Vert _{L^{2}({\mathbb{R}})}
&  \leq C\left\Vert f_{1,\pm}-f_{1,\pm}^{\operatorname*{fr}}\right\Vert
_{L^{2}(\mathbb{R}^{+})}\nonumber\\
&  \leq C_{a}\frac{1}{|t|^{\frac{1+\alpha}{2}}}\left\Vert w_{\pm}\right\Vert
_{H^{1}({\mathbb{R}})}^{\alpha+1},\qquad\pm t\geq a. \label{96}%
\end{align}
From (\ref{sp8.1}), integrating by parts, for $\phi\in H^{1}(\mathbb{R}^{+}),$
and using,
\begin{equation}
\displaystyle\partial_{k}e^{it(k\pm x/2)^{2}}=\pm2\partial_{x}e^{it(k\pm
x/2)^{2}}, \label{96.1}%
\end{equation}
we have,%
\begin{equation}
\partial_{k}\mathcal{V}_{\pm}\left(  t\right)  \phi=\mp2\sqrt{\frac{t}{2\pi
i}}e^{itk^{2}}\phi\left(  0\right)  \mp2\mathcal{V}_{\pm}\left(  t\right)
\left(  \partial_{x}\phi\right)  . \label{130}%
\end{equation}
Further, by \eqref{12.1}, \eqref{12.1.00}, \eqref{derivativeffr},
\eqref{zeta.0}, \eqref{130}, (\ref{UnitaritySM}), (\ref{scatmatrixderiv})
\eqref{sp9.1}, \eqref{isomafr}, \eqref{89.xx} and \eqref{89.xxxx},
\begin{equation}%
\begin{array}
[c]{l}%
\left\Vert \mathcal{V}\left(  t\right)  \left(  \zeta-1\right)
\widehat{\mathcal{V}}\left(  t\right)  f_{1,\pm}^{\operatorname*{fr}%
}\right\Vert _{L^{\infty}({\mathbb{R}})}\leq C\sum_{\sigma=\pm}\left(
\left\Vert \mathcal{V}_{\sigma}\left(  t\right)  f_{1,\pm}^{\operatorname*{fr}%
}\right\Vert _{L^{\infty}({\mathbb{R}})}+|t|^{\frac{1}{2}-\frac{1}{2p_{1}}%
}\left\Vert \left\langle k\right\rangle ^{-1}\right\Vert _{L^{q_{1}}%
}\left\Vert f_{1,\pm}^{\operatorname*{fr}}\right\Vert _{L^{\infty}%
(\mathbb{R}^{+})}\right. \\
\\
\left.  +|t|^{-\frac{1}{4}}\left\Vert \mathcal{V}_{\sigma}\left(  t\right)
\partial_{k}f_{1,\pm}^{\operatorname*{fr}}\right\Vert _{L^{2}({\mathbb{R}}%
)}+|t|^{-\frac{1}{4}}\left\Vert \mathcal{V}_{\sigma}\left(  t\right)
f_{1,\pm}^{\operatorname*{fr}}\right\Vert _{L^{2}({\mathbb{R}})}\right)  \leq
C_{a}|t|^{-\alpha/2}\,|t|^{\frac{1}{2}-\frac{1}{2p_{1}}}\,\Vert w_{\pm}%
\Vert_{H^{1}({\mathbb{R}})}^{\alpha+1},\qquad\pm t\geq a.
\end{array}
\label{130.1}%
\end{equation}
Moreover, \eqref{12.1}, \eqref{12.1.00}, \eqref{derivativeffr}, \eqref{130},
(\ref{UnitaritySM}), (\ref{scatmatrixderiv}) \eqref{sp9.1}, \eqref{EST8.0},
\eqref{89.xx}, \eqref{89.xxxx}, and \eqref{EST11}, we have,
\begin{equation}%
\begin{array}
[c]{l}%
\left\langle t\right\rangle ^{1/2}\left\Vert \left(  \mathcal{V}\left(
t\right)  -\mathcal{W}\left(  t\right)  \right)  \widehat{\mathcal{V}}\left(
t\right)  f_{1,\pm}^{\operatorname*{fr}}\right\Vert _{L^{2}({\mathbb{R}})}\leq
C\sum_{\sigma=\pm}\left(  \left\Vert \mathcal{V}_{\sigma}\left(  t\right)
f_{1,\pm}^{\operatorname*{fr}}\right\Vert _{L^{\infty}({\mathbb{R}}%
)}+|t|^{\frac{1}{2}-\frac{1}{2p_{1}}}\left\Vert \left\langle k\right\rangle
^{-1}\right\Vert _{L^{q_{1}}}\left\Vert f_{1,\pm}^{\operatorname*{fr}%
}\right\Vert _{L^{\infty}({\mathbb{R}})}\right. \\
\\
\left.  +|t|^{-\frac{1}{4}}\left\Vert \mathcal{V}_{\sigma}\left(  t\right)
\partial_{k}f_{1,\pm}^{\operatorname*{fr}}\right\Vert _{L^{2}({\mathbb{R}}%
)}+|t|^{-\frac{1}{4}}\left\Vert \mathcal{V}_{\sigma}\left(  t\right)
f_{1,\pm}^{\operatorname*{fr}}\right\Vert _{L^{2}({\mathbb{R}})}\right) \\
\\
\leq C_{a}|t|^{-\alpha/2}\,|t|^{\frac{1}{2}-\frac{1}{2p_{1}}}\,\Vert w_{\pm
}\Vert_{H^{1}({\mathbb{R}})}^{\alpha+1},\qquad\pm t\geq a.\label{300}%
\end{array}
\end{equation}
By (\ref{12.0.0}), (\ref{derivativeffr}), (\ref{130}), and (\ref{89.xx}), we
get
\begin{equation}
\left\Vert \partial_{k}\mathcal{V}_{\pm}\left(  t\right)  f_{1,\pm
}^{\operatorname*{fr}}\right\Vert _{L^{\infty}({\mathbb{R}})}\leq C_{a}%
\sqrt{|t|}\,|t|^{-\alpha/2}\,\Lambda_{\pm}\left(  t;w_{\pm}\right)  ,\qquad\pm
t\geq a. \label{103}%
\end{equation}
Then, using \eqref{12}, \eqref{zeta.0}, (\ref{103}), \eqref{2.11}, and
\eqref{EST8.0} we have%
\begin{align}
\left\Vert \mathcal{V}\left(  t\right)  \left(  \zeta-1\right)  \mathcal{V}%
_{\pm}\left(  t\right)  f_{1,\pm}^{\operatorname*{fr}}\right\Vert
_{H^{1}({\mathbb{R}})}  &  \leq C\left\Vert \mathcal{V}_{\pm}\left(  t\right)
f_{1,\pm}^{\operatorname*{fr}}\right\Vert _{L^{2}({\mathbb{R}})}+C\left\Vert
\left(  \zeta-1\right)  \right\Vert _{L^{2}({\mathbb{R}})}\left\Vert
\partial_{k}\mathcal{V}_{\pm}\left(  t\right)  f_{1,\pm}^{\operatorname*{fr}%
}\right\Vert _{L^{\infty}({\mathbb{R}})}\nonumber\\
&  \leq C_{a}\sqrt{|t|}\,|t|^{-\alpha/2}\,\Lambda_{\pm}(t;w_{\pm}),\qquad\pm
t\geq a. \label{303}%
\end{align}
We denote (see \eqref{100}),
\begin{equation}%
\begin{array}
[c]{l}%
\mathcal{A}_{1,\pm}:=\frac{i}{4t^{2}}\mathcal{V}\left(  t\right)  \partial
_{k}^{2}\left(  \zeta w_{\pm}\right)  +\left(  \mathcal{V}\left(  t\right)
\zeta-\mathcal{W}\left(  t\right)  \right)  \left(  \left(
\widehat{\mathcal{W}}\left(  t\right)  -\widehat{\mathcal{V}}\left(  t\right)
\right)  f_{1,\pm}\right)  \nonumber\\
+\left(  \mathcal{V}\left(  t\right)  \zeta-\mathcal{W}\left(  t\right)
\right)  \left(  \widehat{\mathcal{V}}\left(  t\right)  \left(  f_{1,\pm
}-f_{1,\pm}^{\operatorname*{fr}}\right)  \right)  \nonumber+\left(
\mathcal{V}\left(  t\right)  -\mathcal{W}\left(  t\right)  \right)  \left(
\widehat{\mathcal{V}}\left(  t\right)  f_{1\pm,}^{\operatorname*{fr}}\right)
,\nonumber\\
\end{array}
\label{303.1}%
\end{equation}%
\begin{equation}
\mathcal{A}_{2,\pm}:=\left(  \mathcal{V}\left(  t\right)  \left(
\zeta-1\right)  \right)  \left(  \widehat{\mathcal{V}}\left(  t\right)
f_{1,\pm}^{\operatorname*{fr}}\right)  +f_{1,\pm}. \label{303.2}%
\end{equation}
Therefore, since $\alpha>2,$ and using \eqref{99}, \eqref{304}, \eqref{99.1}
(\ref{95}), (\ref{96}), \eqref{130.1}, \eqref{300}, (\ref{303}), and taking
$p_{1}>1$ sufficiently close to $1,$ we prove (\ref{98}), \eqref{98.1}, and \eqref{98.2}.
\end{proof}

\begin{lemma}
Let the nonlinear function $\mathcal{N}$ satisfy (\ref{ConNonlinearity}). Let
$w_{\pm}$ be functions from $t\in\mathcal{T}_{\pm}\left(  a,\infty\right)  ,$
$a>0,$ with values in $H^{2}({\mathbb{R}}).$ Then, the following is true:

\begin{enumerate}
\item
\begin{equation}
\left\Vert \partial_{t}f_{1,\pm}^{\operatorname*{fr}}\right\Vert
_{L^{2}({\mathbb{R}})}\leq C|t|^{-\alpha/2}\left[  \left\Vert w_{\pm
}\right\Vert _{H^{1}({\mathbb{R}})}^{\alpha}\left(  \left\Vert w_{\pm
}\right\Vert _{H^{2}({\mathbb{R}})}+\left\Vert \partial_{t}w_{\pm}\right\Vert
_{L^{2}({\mathbb{R}})}\right)  \right]  ,\qquad\pm t\geq a.
\label{EST14frL2rough}%
\end{equation}

\item Suppose that $V\in L_{2+\delta}^{1}(\mathbb{R}^{+}),$ for some
$\delta>0,$ and that $H_{A,B,V}$ does not have negative eigenvalues. Assume
that $w_{\pm}$ satisfies (\ref{eqw}) for $t\in\mathcal{T}_{\pm}\left(
a,\infty\right)  ,$ $a>0.$ Then, we have
\begin{equation}
\partial_{t}f_{1}^{\operatorname*{fr}}=\mathcal{A}_{1,\pm}^{\left(  1\right)
}+\mathcal{A}_{2,\pm}^{\left(  1\right)  },\qquad\pm t\geq a, \label{303.3}%
\end{equation}

where%
\begin{equation}
\left\Vert \mathcal{A}_{1,\pm}^{\left(  1\right)  }(t)\right\Vert
_{L^{2}({\mathbb{R}})}\leq C\Vert w_{\pm}\Vert_{H^{1}({\mathbb{R}})}^{\alpha
}|t|^{-\alpha/2}\,|t|^{-3/2}\left(  |t|^{-1/2}\left\Vert w_{\pm}\right\Vert
_{H^{2}({\mathbb{R}})}+\Lambda_{\pm}\left(  t;w_{\pm}\right)  \right)
,\qquad\pm t\geq a, \label{303.4}%
\end{equation}
and,%

\begin{equation}%
\begin{array}
[c]{l}%
\left\Vert \mathcal{A}_{2,\pm}^{\left(  1\right)  }(t)\right\Vert _{L^{\infty
}({\mathbb{R}})}\leq C\Vert w_{\pm}\Vert_{H^{1}({\mathbb{R}})}^{\alpha
}|t|^{-\alpha/2}|t|^{-1}\left(  |t|^{-1/2}\Vert w_{\pm}\Vert_{H^{2}%
(\mathbb{R}^{+})}\right.  +\\
\left.  \Lambda_{\pm}\left(  t;w_{\pm}\right)  \right)  +C|t|^{-1-\alpha
/2}\Vert w_{\pm}\Vert_{H^{(1)}({\mathbb{R}})}^{\alpha+1},\qquad\pm t\geq a.\\
\\
\left\Vert \mathcal{A}_{2,\pm}^{(1)}(t)\right\Vert _{H^{1}({\mathbb{R}})}\leq
C\Vert w_{\pm}\Vert_{H^{1}({\mathbb{R}})}^{\alpha}|t|^{-\alpha/2}%
|t|^{-1/2}\left(  |t|^{-1/2}\Vert w_{\pm}\Vert_{H^{2}(\mathbb{R}^{+})}\right.
+\\
\left.  \Lambda_{\pm}\left(  t;w_{\pm}\right)  \right)  +C|t|^{-1-\alpha
/2}\Vert w_{\pm}\Vert_{H^{(1)}({\mathbb{R}})}^{\alpha+1},\qquad\pm t\ \geq a.
\end{array}
\label{303.5}%
\end{equation}

\item Suppose that $V\in L_{3+\tilde{\delta}}^{1}(\mathbb{R}^{+}),$ for some
$\tilde{\delta}>1/2,$ and that it admits a regular decomposition (see
Definition~\ref{Def1}) with $\delta>0,$ that $w_{\pm}$ satisfies (\ref{eqw})
for $t\in\mathcal{T}_{\pm}\left(  a,\infty\right)  ,$ $a>0.$ Then, we have,%

\begin{equation}
\partial_{t}f_{1,\pm}=\mathcal{A}_{1,\pm}^{\left(  2\right)  }+\mathcal{A}%
_{2,\pm}^{\left(  2\right)  }, \label{EST14}%
\end{equation}
where,%
\begin{equation}
\left\Vert \mathcal{A}_{1,\pm,}^{\left(  2\right)  }\right\Vert _{L^{2}%
({\mathbb{R}})}\leq C|t|^{-\alpha/2}\,|t|^{-3/2}\left\Vert w_{\pm}\right\Vert
_{H^{1}({\mathbb{R}})}^{\alpha}\left[  |t|^{-1/2}\left\Vert w_{\pm}\right\Vert
_{H^{2}({\mathbb{R}})}+\Lambda_{\pm}(t;w_{\pm})\right]  ,\qquad\pm t\geq a.
\label{303.6}%
\end{equation}
and
\begin{equation}
\left\Vert \mathcal{A}_{2,\pm,}^{\left(  2\right)  }\right\Vert _{L^{\infty
}({\mathbb{R}})}\leq C|t|^{-\alpha/2}|t|^{-1}\,\Lambda_{\pm}(t;w_{\pm}%
),\qquad\pm t\geq a. \label{303.7}%
\end{equation}

\end{enumerate}
\end{lemma}

\begin{proof}
From the middle equality in (\ref{99}) with $\zeta=1,$ and (\ref{2.11}) we
obtain%
\begin{equation}
\left\Vert \partial_{t}\left(  \mathcal{V}\left(  t\right)  w\right)
\right\Vert _{L^{2}({\mathbb{R}})}\leq C\left(  \left\Vert w_{\pm}\right\Vert
_{H^{2}({\mathbb{R}})}+\left\Vert \partial_{t}w_{\pm}\right\Vert
_{L^{2}({\mathbb{R}})}\right)  ,\qquad\pm t\geq a. \label{5.0}%
\end{equation}
Then, using \eqref{ConNonlinearity}, (\ref{30.0}) with $g=t^{-1/2}%
\mathcal{V}(t)w_{\pm},$ (\ref{89.xxxx}), and Sobolev's inequality we prove
(\ref{EST14frL2rough}). We denote,
\begin{equation}%
\begin{array}
[c]{l}%
\mathcal{A}_{l,\pm}^{(1)}:=\sum_{j=1}^{n}\left[  E_{1}^{(j)}\left(
t^{-1/2}\mathcal{V}(t)w_{\pm}\right)  t^{-1/2}(\mathcal{A}_{l,\pm})_{j}%
+E_{2}^{(j)}\left(  t^{-1/2}\mathcal{V}(t)w_{\pm}\right)  t^{-1/2}%
\overline{(\mathcal{A}_{l,\pm})_{j}}\right]  w_{\pm}\\
\\
+\mathcal{N}\left(  t^{-1/2}\left\vert \left(  \mathcal{V}(t)w_{\pm}\right)
_{1}\right\vert ,\dots,\left\vert t^{-1/2}\left(  \mathcal{V}(t)w_{\pm
}\right)  _{n}\right\vert \right)  \,\mathcal{A}_{l,\pm}-\\
\displaystyle\frac{1}{2}|t|^{-3/2}\sum_{j=1}^{n}\left[  E_{1}^{(j)}\left(
t^{-1/2}\mathcal{V}(t)w_{\pm}\right)  \mathcal{V}(t)w_{\pm}+E_{2}^{(j)}\left(
t^{-1/2}\mathcal{V}(t)w_{\pm}\right)  \overline{\mathcal{V}(t)w_{\pm}%
(t)}\right]  ,\qquad l=1,2,
\end{array}
\label{5.0.1}%
\end{equation}
where we denote by $(\mathcal{A}_{l,\pm})_{j},j=1,...,n,$ the components of
the vector $A_{l,\pm}.$ Further, using (\ref{ConNonlinearity}), (\ref{100})
with $g=t^{-1/2}\mathcal{V}\left(  t\right)  w$, (\ref{98}) with $\zeta=1,$
and \eqref{89.xxxx}, we prove \eqref{303.3}, \eqref{303.4}, \eqref{303.5}.
Similarly to (\ref{99}) with $\zeta=1,$ using \eqref{305} and
\eqref{sp11.2.0}, we calculate%
\begin{equation}
\partial_{t}\left(  \mathcal{W}\left(  t\right)  w_{\pm}\right)  =\frac
{i}{4t^{2}}\mathcal{V}\left(  t\right)  \partial_{k}^{2}\left(  m\left(
k,tx\right)  w_{\pm}\right)  +\mathcal{V}\left(  t\right)  x\left(
\partial_{x}m\right)  \left(  k,tx\right)  w_{\pm}-if_{1,\pm}. \label{94bis}%
\end{equation}
Using (\ref{3.5}), \eqref{3.6}, (\ref{3.5bis}), and \eqref{l2.1}, we estimate%
\begin{equation}
\left\Vert \frac{i}{4t^{2}}\mathcal{V}\left(  t\right)  \partial_{k}%
^{2}\left(  m\left(  k,tx\right)  w_{\pm}\right)  )\right\Vert _{L^{2}%
({\mathbb{R}})}+\left\Vert \mathcal{V}\left(  t\right)  x\left(  \partial
_{x}m\right)  \left(  k,tx\right)  w_{\pm}\right\Vert _{L^{2}({\mathbb{R}}%
)}\leq C\frac{1}{|t|^{2}}\Vert w_{\pm}\Vert_{H^{2}({\mathbb{R}})},\qquad\pm
t>a. \label{94bis.1}%
\end{equation}
We denote,
\begin{equation}%
\begin{array}
[c]{l}%
\mathcal{A}_{1,\pm}^{(2)}:=\frac{1}{|t|^{1/2}}\sum_{j=1}^{n}\left[
E_{1}^{(j)}\left(  t^{-1/2}\mathcal{W}(t)w_{\pm}\right)  (\frac{i}{4t^{2}%
}\mathcal{V}\left(  t\right)  (\partial_{k}^{2}m\left(  k,tx\right)  w_{\pm
})+\mathcal{V}\left(  t\right)  x\left(  \partial_{x}m\right)  \left(
k,tx\right)  w_{\pm})\right. \\
\\
\left.  +E_{2}^{(j)}\left(  t^{-1/2}\mathcal{W}(t)w_{\pm}\right)
\overline{(\frac{i}{4t^{2}}\mathcal{V}\left(  t\right)  (\partial_{k}%
^{2}m\left(  k,tx\right)  w_{\pm})+\mathcal{V}\left(  t\right)  x\left(
\partial_{x}m\right)  \left(  k,tx\right)  )w_{\pm})}\right]  w_{\pm}\\
\\
+\mathcal{N}\left(  t^{-1/2}\left\vert \left(  \mathcal{V}(t)w_{\pm}\right)
_{1}\right\vert ,\dots,\left\vert t^{-1/2}\left(  \mathcal{V}(t)w_{\pm
}\right)  _{n}\right\vert \right)  \,(\frac{i}{4t^{2}}\mathcal{V}\left(
t\right)  (\partial_{k}^{2}m\left(  k,tx\right)  +(\mathcal{V}\left(
t\right)  x\left(  \partial_{x}m\right)  \left(  k,tx\right)  )w_{\pm}),
\end{array}
\label{94bis.3}%
\end{equation}
and
\begin{equation}%
\begin{array}
[c]{l}%
\mathcal{A}_{2,\pm}^{(2)}:=-i\sum_{j=1}^{n}\left[  E_{1}^{(j)}\left(
t^{-1/2}\mathcal{W}(t)w_{\pm}\right)  f_{1,\pm})\right.  \left.  -E_{2}%
^{(j)}\left(  t^{-1/2}\mathcal{W}(t)w_{\pm}\right)  \overline{f_{1,\pm}%
)}\right]  w_{\pm}\\
\\
-i\mathcal{N}\left(  t^{-1/2}\left\vert \left(  \mathcal{W}(t)w_{\pm}\right)
_{1}\right\vert ,\dots,\left\vert t^{-1/2}\left(  \mathcal{W}(t)w_{\pm
}\right)  _{n}\right\vert \right)  \,f_{1,\pm})\pm\\
\\
-\displaystyle\frac{1}{2}|t|^{-3/2}\sum_{j=1}^{n}\left[  E_{1}^{(j)}\left(
t^{-1/2}\mathcal{W}(t)w_{\pm}\right)  \mathcal{W}(t)w_{\pm}+E_{2}^{(j)}\left(
t^{-1/2}\mathcal{W}(t)w_{\pm}\right)  \overline{\mathcal{W}(t)w_{\pm}%
(t)}\right]  .
\end{array}
\label{94bis.4}%
\end{equation}
Further, using (\ref{ConNonlinearity}), (\ref{100}) with $g=t^{-1/2}%
\mathcal{W}\left(  t\right)  w_{\pm}$, \eqref{94bis}, \eqref{94bis.1},
\eqref{94bis.3}, \eqref{94bis.4}, and \eqref{EST4}, we obtain \eqref{EST14},
\eqref{303.6} and \eqref{303.7}.
\end{proof}

We denote%
\begin{equation}
\label{94bis.4.1}f_{2,\pm}=f_{1,\pm}-f_{1,\pm}^{\operatorname*{fr}}.
\end{equation}
Also, we introduce%
\begin{equation}
\label{94.bis.4.2}N\left(  f\right)  =\left\Vert f\right\Vert _{L^{1}(
{\mathbb{R}})}+|t|^{-1/2}\left\Vert f\right\Vert _{L^{2}( {\mathbb{R}})}.
\end{equation}
Observe that for $\delta>n+1$%
\begin{equation}
N\left(  x^{n}\left\langle tx\right\rangle ^{-\delta}\right)  \leq
C|t|^{-n-1}. \label{n}%
\end{equation}

\begin{lemma}
Suppose that $V\in L_{3+\tilde{\delta}}^{1}(\mathbb{R}^{+}),$ for some
$\tilde{\delta}>n+1/2.$ In addition, assume that $V$ admits a regular
decomposition (see Definition \ref{Def1}) with $\delta=\min[n-1,0].$ Further,
assume that the boundary matrices $A,B$, satisfy \eqref{wcon1},
\eqref{wcon2},and that $H_{A,B,V}$ does not have negative eigenvalues. Let the
nonlinear function $\mathcal{N}$ satisfy (\ref{ConNonlinearity}). Let $w_{\pm
}$ be functions from $t\in\mathcal{T}_{\pm}\left(  a,\infty\right)  , $ $a>0,$
with values in $H^{2}({\mathbb{R}}).$ Then, the following is true.

\begin{enumerate}
\item
\begin{equation}
\left\Vert x^{n}\partial_{t}f_{2,\pm}\right\Vert _{L^{2}( {\mathbb{R}})}\leq C
|t|^{1/2-n-\alpha/2}\left\Vert w_{\pm}\right\Vert _{H^{1}( {\mathbb{R}}%
)}^{\alpha}\left(  \left\Vert w_{\pm}\right\Vert _{H^{2}( {\mathbb{R}}%
)}+\left\Vert \partial_{t}w\right\Vert _{L^{2}( {\mathbb{R}})}\right)  ,
\qquad\pm t \geq a. \label{EST16rough}%
\end{equation}

\item Assume, moreover, that $V\in L^{1}_{4+\tilde{\delta}}( {\mathbb{R}}),$
for some $\tilde{\delta}>n,$ and that $V$ admits a regular decomposition with
$\delta>n.$ Then, if $w_{\pm}$ satisfies (\ref{eqw}) for $\pm t \geq a,$ the
following is true%
\begin{equation}
\left\Vert x^{n}\partial_{t}f_{2,\pm}\right\Vert _{L^{1}( {\mathbb{R}}%
)}+|t|^{-1/2-n}\left\Vert x^{n}\partial_{t}f_{2,\pm}\right\Vert _{L^{2}(
{\mathbb{R}})}\leq C|t|^{-n-2-\alpha/2}\left(  \left\Vert w_{\pm}\right\Vert
_{H^{1}( {\mathbb{R}})}^{\alpha}+1\right)  \Lambda_{\pm}\left(  t;w_{\pm
}\right)  ,\qquad\pm t \geq a . \label{EST16}%
\end{equation}

\end{enumerate}
\end{lemma}

\begin{proof}
Using (\ref{100}) with $g= t^{-1/2} \mathcal{W}\left(  t\right)  w_{\pm}$ and
$g= t^{-1/2}\mathcal{V}\left(  t\right)  w_{\pm}$ we have%
\begin{align}
\partial_{t}f_{2,\pm}  &  =\frac{1}{|t|^{1/2}} \sum_{j=1}^{n}\left[  \left(
E_{1}^{\left(  j\right)  }\left(  t^{-1/2} \mathcal{W}\left(  t\right)
w_{\pm}\right)  \mathcal{W}\left(  t\right)  w_{\pm}\right)  \partial
_{t}\left(  \left(  \mathcal{W}\left(  t\right)  w_{\pm}\right)  _{j}-\left(
\mathcal{V}\left(  t\right)  w_{\pm}\right)  _{j}\right)  \right. \nonumber\\
&  +\left(  E_{1}^{\left(  j\right)  }\left(  t^{-1/2} \mathcal{W}\left(
t\right)  w_{\pm}\right)  \mathcal{W}\left(  t\right)  w_{\pm}-E_{1}^{\left(
j\right)  }\left(  t^{-1/2} \mathcal{V}\left(  t\right)  w_{\pm}\right)
\mathcal{V}\left(  t\right)  w_{\pm}\right)  \partial_{t}\left(
\mathcal{V}\left(  t\right)  w_{\pm}\right)  _{j}\nonumber\\
&  +\left(  E_{2}^{\left(  j\right)  }\left(  t^{-1/2} \mathcal{W}\left(
t\right)  w_{\pm}\right)  \mathcal{W}\left(  t\right)  w_{\pm}\right)
\partial_{t}\left(  \overline{\left(  \mathcal{W}\left(  t\right)  w_{\pm
}\right)  _{j}}-\overline{\left(  \mathcal{V}\left(  t\right)  w_{\pm}\right)
_{j}}\right) \label{110}\\
&  \left.  +\left(  E_{2}^{\left(  j\right)  }\left(  t^{-1/2} \mathcal{W}%
\left(  t\right)  w_{\pm}\right)  \mathcal{W}\left(  t\right)  w_{\pm}%
-E_{2}^{\left(  j\right)  }\left(  t^{-1/2} \mathcal{V}\left(  t\right)
w_{\pm}\right)  \mathcal{V}\left(  t\right)  w_{\pm}\right)  \partial
_{t}\overline{\left(  \mathcal{V}\left(  t\right)  w_{\pm}\right)  _{j}%
}\right] \nonumber\\
&  +\mathcal{N}\left(  \left\vert \left(  t^{-1/2} \mathcal{W}\left(
t\right)  w_{\pm}\right)  _{1}\right\vert ,...,\left\vert \left(  t^{-1/2}
\mathcal{W}\left(  t\right)  w_{\pm}\right)  _{n}\right\vert \right)  \left(
\partial_{t}\mathcal{W}\left(  t\right)  w_{\pm}-\partial_{t}\mathcal{V}%
\left(  t\right)  w_{\pm}\right) \nonumber\\
&  +\left(  \mathcal{N}\left(  \left\vert t^{-1/2} \left(  \mathcal{W}\left(
t\right)  w_{\pm}\right)  _{1}\right\vert ...,\left\vert \left(  t^{-1/2}
\mathcal{W}\left(  t\right)  \right)  w_{\pm}\right)  _{n}\right\vert \right.
\nonumber\\
&  -\mathcal{N}\left.  \left(  \left\vert \left(  t^{-1/2} \mathcal{V}\left(
t\right)  w_{\pm}\right)  _{1}\right\vert ,...,\left\vert t^{-1/2} \left(
\mathcal{V}\left(  t\right)  w_{\pm}\right)  _{n}\right\vert \right)
\partial_{t}\mathcal{V}\left(  t\right)  \right)  w_{\pm}. \pm\nonumber\\
&  \frac{1}{|t|^{3/2}} \sum_{j=1}^{n}\left[  \left(  E_{1}^{\left(  j\right)
}\left(  t^{-1/2} \mathcal{W}\left(  t\right)  w_{\pm}\right)  \mathcal{W}%
\left(  t\right)  w_{\pm}\right)  \left(  \left(  \mathcal{W}\left(  t\right)
w_{\pm}\right)  _{j}-\left(  \mathcal{V}\left(  t\right)  w_{\pm}\right)
_{j}\right)  \right. \nonumber\\
&  +\left(  E_{1}^{\left(  j\right)  }\left(  t^{-1/2} \mathcal{W}\left(
t\right)  w_{\pm}\right)  \mathcal{W}\left(  t\right)  w_{\pm}-E_{1}^{\left(
j\right)  }\left(  t^{-1/2} \mathcal{V}\left(  t\right)  w_{\pm}\right)
\mathcal{V}\left(  t\right)  w_{\pm}\right)  \left(  \mathcal{V}\left(
t\right)  w_{\pm}\right)  _{j}\nonumber\\
&  +\left(  E_{2}^{\left(  j\right)  }\left(  t^{-1/2} \mathcal{W}\left(
t\right)  w_{\pm}\right)  \mathcal{W}\left(  t\right)  w_{\pm}\right)  \left(
\overline{\left(  \mathcal{W}\left(  t\right)  w_{\pm}\right)  _{j}}%
-\overline{\left(  \mathcal{V}\left(  t\right)  w_{\pm}\right)  _{j}}\right)
\nonumber\\
&  \left.  +\left(  E_{2}^{\left(  j\right)  }\left(  t^{-1/2} \mathcal{W}%
\left(  t\right)  w_{\pm}\right)  \mathcal{W}\left(  t\right)  w_{\pm}%
-E_{2}^{\left(  j\right)  }\left(  t^{-1/2} \mathcal{V}\left(  t\right)
w_{\pm}\right)  \mathcal{V}\left(  t\right)  w_{\pm}\right)  \overline{\left(
\mathcal{V}\left(  t\right)  w_{\pm}\right)  _{j}}\right] .\nonumber
\end{align}
Using \eqref{sp11.2.0}, we calculate%
\begin{align}
\label{110.1}\partial_{t}\left(  \mathcal{W}\left(  t\right)  w_{\pm
}-\mathcal{V}\left(  t\right)  w_{\pm}\right)   &  =\frac{i}{4t^{2}%
}\mathcal{V}\left(  t\right)  \left(  \partial_{k}^{2}\left(  m\left(
k,tx\right)  w_{\pm}\right)  -\partial_{k}^{2}w_{\pm}\right)  +\mathcal{V}%
\left(  t\right)  x\left(  \partial_{x}m\right)  \left(  k,tx\right)  w_{\pm
}\nonumber\\
&  +\left(  \mathcal{W}\left(  t\right)  -\mathcal{V}\left(  t\right)
\right)  \left(  \partial_{t}w_{\pm}\right)  .
\end{align}
Using (\ref{3.2}), (\ref{3.5}) and (\ref{3.5bis}) we show%
\begin{equation}%
\begin{array}
[c]{l}%
\label{110.2} \left\Vert x^{n}\mathcal{V}\left(  t\right)  \left(
\partial_{k}^{2}\left(  m\left(  k,tx\right)  w_{\pm}\right)  -\partial
_{k}^{2}w_{\pm}\right)  \right\Vert _{L^{\infty}( {\mathbb{R}})}\leq
C\sqrt{|t|}\left\Vert x^{n}\left\langle tx\right\rangle ^{-\tilde{\delta}%
}\right\Vert _{L^{\infty}( {\mathbb{R}})}\left\Vert w_{\pm}\right\Vert
_{H^{2}( {\mathbb{R}})}\\
\\
\leq C|t|^{1/2-n}\left\Vert w_{\pm}\right\Vert _{H^{2}( {\mathbb{R}})},
\qquad\pm t \geq a.
\end{array}
\end{equation}
Moreover, by (\ref{3.6}) we obtain,
\begin{equation}
\label{110.3}%
\begin{array}
[c]{l}%
\left\Vert x^{n}\mathcal{V}\left(  t\right)  x\left(  \partial_{x}m\right)
\left(  k,tx\right)  w_{\pm}\right\Vert _{L^{\infty}( {\mathbb{R}})}\leq
C\sqrt{|t|}\left\Vert x^{n+1}\left\langle tx\right\rangle ^{-2-\delta
}\right\Vert _{L^{\infty}( {\mathbb{R}})}\left\Vert w_{\pm}\right\Vert
_{H^{1}( {\mathbb{R}})}\\
\\
\leq C|t|^{-n-1/2}\left\Vert w_{\pm}\right\Vert _{H^{1}( {\mathbb{R}})}, \pm t
\geq a.
\end{array}
\end{equation}
From (\ref{EST11b}) it follows that%
\begin{equation}
\label{110.4}\left\Vert x^{n}\left(  \mathcal{W}\left(  t\right)
-\mathcal{V}\left(  t\right)  \right)  \left(  \partial_{t}w_{\pm}\right)
\right\Vert _{L^{\infty}( {\mathbb{R}})}\leq C\sqrt{|t|}\left\Vert
x^{n}\left\langle tx\right\rangle ^{-2-\tilde{\delta} }\right\Vert
_{L^{\infty}( {\mathbb{R}})}\left\Vert \partial_{t}w_{\pm}\right\Vert _{L^{2}(
{\mathbb{R}})}\leq C|t|^{-n+1/2}\left\Vert \partial_{t}w_{\pm}\right\Vert
_{L^{2}( {\mathbb{R}})},
\end{equation}
for $\pm t \geq a.$ Thus, by \eqref{110.1}, \eqref{110.2}, \eqref{110.3} and
\eqref{110.4}, we estimate%
\begin{equation}
\left\Vert \partial_{t}\left(  \mathcal{W}\left(  t\right)  w_{\pm
}-\mathcal{V}\left(  t\right)  w_{\pm}\right)  \right\Vert _{L^{\infty}(
{\mathbb{R}})}\leq C|t|^{1/2-n}\left(  \left\Vert w_{\pm}\right\Vert _{H^{2}(
{\mathbb{R}})}+\left\Vert \partial_{t}w_{\pm}\right\Vert _{L^{2}( {\mathbb{R}%
})}\right)  , \qquad\pm t \geq a. \label{4.0}%
\end{equation}
Therefore, using (\ref{ConNonlinearity}), (\ref{5.0}), \eqref{110.4} with
$w_{\pm}$ instead of $\partial_{t} w_{\pm},$ (\ref{4.0}), (\ref{EST4}),
\eqref{89.xxxx}, (\ref{EST11}), in (\ref{110}), we deduce (\ref{EST16rough}).

Next, we prove (\ref{EST16}). By (\ref{99}) with $\zeta=1,$ and (\ref{94bis})
we get%
\begin{align}
\partial_{t}\left(  \mathcal{W}\left(  t\right)  w_{\pm}-\mathcal{V}\left(
t\right)  w_{\pm}\right)   &  =\frac{i}{4t^{2}}\mathcal{V}\left(  t\right)
\left(  \partial_{k}^{2}\left(  m\left(  k,tx\right)  w_{\pm}\right)
-\partial_{k}^{2}w_{\pm}\right)  +\mathcal{V}\left(  t\right)  x\left(
\partial_{x}m\right)  \left(  k,tx\right)  w_{\pm}\nonumber\\
&  -i\left(  \mathcal{W}\left(  t\right)  -\mathcal{V}\left(  t\right)
\right)  \widehat{ \mathcal{W}}\left(  t\right)  f_{1,\pm} . \label{107}%
\end{align}
Using (\ref{n}), (\ref{3.2}), (\ref{3.5}), and (\ref{3.5bis}) we show%
\begin{equation}
N\left(  x^{n}\mathcal{V}\left(  t\right)  \left(  \partial_{k}^{2}\left(
m\left(  k,tx\right)  w_{\pm}\right)  -\partial_{k}^{2}w_{\pm}\right)
\right)  \leq C|t|^{-n-1/2}\left\Vert w_{\pm}\right\Vert _{H^{2}( {\mathbb{R}%
})}, \qquad\pm t \geq a. \label{108}%
\end{equation}
Moreover, using (\ref{3.6}) we show
\begin{equation}
N\left(  x^{n}\mathcal{V}\left(  t\right)  x\left(  \partial_{x}m\right)
\left(  k,tx\right)  w_{\pm}\right)  \leq C|t|^{-n-2}\left\Vert w\right\Vert
_{L^{2}( {\mathbb{R}})},\qquad\pm t \geq a. \label{108b}%
\end{equation}
Next, we decompose%
\begin{align}
\widehat{\mathcal{W}}\left(  t\right)  f_{1,\pm}  &  =\left(
\widehat{\mathcal{W}}\left(  t\right)  -\widehat{\mathcal{V}}\left(  t\right)
\right)  f_{1,\pm}\nonumber\\
&  +\widehat{\mathcal{V}}\left(  t\right)  \left(  f_{1,\pm} -f_{1,\pm
}^{\operatorname*{fr}} \right)  +\widehat{\mathcal{V}}\left(  t\right)
f_{1,\pm}^{\operatorname*{fr}} . \label{106}%
\end{align}
Using \eqref{12.0.0}, (\ref{101}), (\ref{isomafr}), and (\ref{EST12}) we show%
\[
\left\Vert \left(  \widehat{ \mathcal{W}}\left(  t\right)
-\widehat{\mathcal{V}}\left(  t\right)  \right)  f_{1,\pm} \right\Vert
_{L^{2}( {\mathbb{R}})}+\left\Vert \widehat{\mathcal{V}}\left(  t\right)
\left(  f_{1,\pm} -f_{1,\pm}^{\operatorname*{fr}} \right)  \right\Vert
_{L^{2}( {\mathbb{R}})}\leq C |t|^{-\alpha/2} |t|^{-1/2}\left\Vert w_{\pm
}\right\Vert _{H^{1}( {\mathbb{R}})}^{\alpha+1}, \qquad\pm t \geq a.
\]
Further, from \eqref{101}, \eqref{2.11}, (\ref{EST11b}), and \eqref{EST12} we
deduce that
\begin{align}
&  N\left(  x^{n}\left(  \left(  \mathcal{W}\left(  t\right)  -\mathcal{V}%
\left(  t\right)  \right)  \left(  \widehat{\mathcal{W}}\left(  t\right)
-\widehat{\mathcal{V}}\left(  t\right)  \right)  \right)  f_{1,\pm}\right)
\nonumber\\
&  +N\left(  x^{n}\left(  \mathcal{W}\left(  t\right)  -\mathcal{V}\left(
t\right)  \right)  \widehat{\mathcal{V}}\left(  t\right)  \left(  f_{1,\pm}
-f_{1,\pm}^{\operatorname*{fr}} \right)  \right)  \leq C|t|^{-n-1}\left\Vert
w_{\pm}\right\Vert _{H^{1}( {\mathbb{R}})}^{\alpha+1}, \qquad\pm t \geq a.
\label{104}%
\end{align}
By (\ref{EST11})%
\begin{align*}
&  N\left(  x^{n}\left(  \mathcal{W}\left(  t\right)  -\mathcal{V}\left(
t\right)  \right)  \widehat{\mathcal{V}}\left(  t\right)  f_{1,\pm
}^{\operatorname*{fr}} \right) \\
&  \leq C \frac{1}{|t|^{n+1}}\,\left(  \left\Vert \widehat{\mathcal{V}}\left(
t\right)  f_{1,\pm}^{\operatorname*{fr}} \right\Vert _{L^{\infty}(
{\mathbb{R}})}+|t|^{-\frac{1}{2p}}\left\Vert \left\langle k\right\rangle
^{-1}\partial_{k}\widehat{\mathcal{V}} f_{1,}^{\operatorname*{fr}} \right\Vert
_{L^{q}( {\mathbb{R}})}\right)  , \qquad\pm t \geq a \text{.}%
\end{align*}
Then, it follows from (\ref{12}), (\ref{103}), (\ref{UnitaritySM}),
(\ref{scatmatrixderiv}), and (\ref{89.xxx}),that%
\begin{equation}
N\left(  x^{n}\left(  \mathcal{W}\left(  t\right)  -\mathcal{V}\left(
t\right)  \right)  \widehat{\mathcal{W}}\left(  t\right)  f_{1,\pm
}^{\operatorname*{fr}} \right)  \leq C|t|^{-n-1}|t|^{-\alpha/2} |t|^{\frac
{1}{2}\left(  1-\frac{1}{p}\right)  }\Lambda_{\pm}\left(  t;w_{\pm}\right)  ,
\qquad\pm t \geq a . \label{105}%
\end{equation}
Therefore, using (\ref{104}) and (\ref{105}) with $p$ close enough to one, in
(\ref{106}) we obtain,%
\begin{equation}
N\left(  x^{n}\left(  \mathcal{W}\left(  t\right)  -\mathcal{V}\left(
t\right)  \right)  \widehat{\mathcal{W}}\left(  t\right)  f_{1,\pm} \right)
\leq C|t|^{-n-2}\Lambda_{\pm}\left(  t;w\right)  ,\qquad\pm t \geq a .
\label{109}%
\end{equation}
Hence, by (\ref{107}), (\ref{108}), (\ref{108b}) and (\ref{109}) we obtain%
\begin{equation}
N\left(  x^{n}\partial_{t}\left(  \mathcal{W}\left(  t\right)  w_{\pm
}-\mathcal{V}\left(  t\right)  w_{\pm}\right)  \right)  \leq C|t|^{-n-2}%
\left(  \left\Vert w_{\pm}\right\Vert _{L^{2}( {\mathbb{R}})}+|t|^{-1/2}%
\left\Vert w_{\pm}\right\Vert _{H^{2}( {\mathbb{R}})}+\Lambda_{\pm}\left(
t;w_{\pm}\right)  \right)  ,\qquad\pm t \geq a . \label{111}%
\end{equation}
Therefore, using (\ref{ConNonlinearity}),(\ref{98}),\eqref{98.1},
\eqref{98.2}, (\ref{111}), (\ref{EST4}), (\ref{EST11}), and \eqref{89.xxxx} in
(\ref{110}), we deduce (\ref{EST16}).
\end{proof}

\begin{lemma}
Suppose that $V\in L_{2+\tilde{\delta}}^{1}(\mathbb{R}^{+}),$ for some
$\tilde{\delta}\geq n,$ if $n>0,$ and $\tilde{\delta}>0,$ if $n=0.$ Further,
assume that $H_{A,B,V}$ does not have negative eigenvalues. In addition,
suppose that $V$ admits a regular decomposition (see Definition \ref{Def1})
with $\delta=\min[n-1,0].$ Let the nonlinear function $\mathcal{N}$ satisfy
(\ref{ConNonlinearity}). Let $w_{\pm}$ be functions from $t\in\mathcal{T}%
_{\pm}\left(  a,\infty\right)  ,$ $a>0,$ with values in $H^{1}({\mathbb{R}}).$
Then, the inequality%
\begin{equation}
\left\Vert x^{n}\partial_{x}f_{2,\pm}\right\Vert _{L^{2}({\mathbb{R}})}\leq
C|t|^{1-n-\alpha/2}\left\Vert w\right\Vert _{H^{1}({\mathbb{R}})}^{\alpha
+1},\qquad\pm t\geq a, \label{EST15rough}%
\end{equation}
is true. In addition, if $V$ admits a regular decomposition for some
$\delta\geq n-1+1/2,$ if $n>0,$ and $\delta=0,$ for $n=0,$ the estimate%
\begin{equation}
\left\Vert x^{n}\partial_{x}f_{2,\pm}\right\Vert _{L^{2}({\mathbb{R}})}\leq
C|t|^{1/2-n-\alpha/2}\left\Vert w\right\Vert _{H^{1}({\mathbb{R}})}^{\alpha
+1},\qquad\pm t\geq a, \label{EST15}%
\end{equation}
holds.
\end{lemma}

\begin{proof}
We first prepare an estimate that we need for the proof of this lemma, and
also for later purposes. Since,
\begin{equation}
e^{-it\left(  k\pm\frac{x}{2}\right)  ^{2}}=\frac{\partial_{k}\left(  \left(
k\pm\frac{x}{2}\right)  e^{-it\left(  k\pm\frac{x}{2}\right)  ^{2}}\right)
}{1-2it\left(  k\pm\frac{x}{2}\right)  ^{2}}, \label{nn.1}%
\end{equation}
integrating by parts we have%
\[
\int_{{\mathbb{R}}}e^{-it\left(  k\pm\frac{x}{2}\right)  ^{2}}g\left(
k\right)  dk=-\int_{{\mathbb{R}}}\frac{2it\left(  k\pm\frac{x}{2}\right)
^{2}e^{it\left(  k\pm\frac{x}{2}\right)  ^{2}}}{\left(  1-2it\left(  k\pm
\frac{x}{2}\right)  ^{2}\right)  ^{2}}g\left(  k\right)  dk-\int_{{\mathbb{R}%
}}\frac{\left(  \left(  k\pm\frac{x}{2}\right)  e^{it\left(  k\pm\frac{x}%
{2}\right)  ^{2}}\right)  }{1-2it\left(  k\pm\frac{x}{2}\right)  ^{2}}%
\partial_{x}g\left(  k\right)  dk.
\]
Then, by H\"{o}lder's and Sobolev's inequalities we get%
\begin{equation}%
\begin{array}
[c]{l}%
\left\vert \int_{{\mathbb{R}}}e^{-it\left(  k\pm\frac{x}{2}\right)  ^{2}%
}g\left(  k\right)  dk\right\vert \leq C\left\Vert \frac{1}{1+|t|k^{2}%
}\right\Vert _{L^{1}({\mathbb{R}})}\left\Vert g\right\Vert _{L^{\infty
}({\mathbb{R}})}+C\left\Vert \frac{x}{1+|t|k^{2}}\right\Vert _{L^{p}%
({\mathbb{R}})}\left\Vert g_{1}\right\Vert _{L^{q}({\mathbb{R}})}\\
+C\left\Vert \frac{k}{1+|t|k^{2}}\right\Vert _{L^{p_{1}}({\mathbb{R}}%
)}\left\Vert g_{2}\right\Vert _{L^{q_{2}}({\mathbb{R}})}\leq\frac{C}%
{\sqrt{|t|}}\left\Vert g\right\Vert _{L^{\infty}({\mathbb{R}})}+\frac
{C}{|t|^{\frac{1}{2}\left(  1+\frac{1}{p}\right)  }}\left\Vert g_{1}%
\right\Vert _{L^{q}({\mathbb{R}})}+\frac{C}{|t|^{\frac{1}{2}\left(  1+\frac
{1}{p_{1}}\right)  }}\left\Vert g_{2}\right\Vert _{L^{q_{1}}({\mathbb{R}}%
)},\label{90}%
\end{array}
\end{equation}
for $\frac{1}{p}+\frac{1}{q}=\frac{1}{p_{1}}+\frac{1}{q_{1}}=1$ and
$p,p_{1}>1,$ where $\partial_{x}g=g_{1}+g_{2}.$ In order to prove
(\ref{EST15rough}), we use (\ref{110}) with $\partial_{t}$ replaced by
$\partial_{x}.$ Let us first estimate $x^{n}\partial_{x}\left(  \mathcal{W}%
\left(  t\right)  -\mathcal{V}\left(  t\right)  \right)  w.$ By \eqref{sp5},
\eqref{9}, \eqref{sp5.xxx} and integrating by parts,
\begin{equation}
\partial_{x}\left(  \mathcal{W}\left(  t\right)  -\mathcal{V}\left(  t\right)
\right)  w_{\pm}=\frac{1}{2}\sqrt{\frac{it}{2\pi}}\int_{-\infty}^{\infty
}e^{-it\left(  k-\frac{x}{2}\right)  ^{2}}\left[  \partial_{k}\left(
(m\left(  k,tx\right)  -I)w_{\pm}\left(  k\right)  \right)  +t\partial
_{x}m(k,tx)w_{\pm}(k)\right]  dk. \label{nn.3}%
\end{equation}
By (\ref{3.5}) and (\ref{3.6}),
\begin{equation}
\left\Vert \frac{1}{2}x^{n}\sqrt{\frac{it}{2\pi}}\int_{{\mathbb{R}}%
}e^{-it\left(  k-\frac{x}{2}\right)  ^{2}}\partial_{k}\left(  (m\left(
k,tx\right)  -I)w_{\pm}\left(  k\right)  \right)  dk\right\Vert _{L^{\infty
}({\mathbb{R}})}\leq\left(  t^{1/2-n}\left\Vert w\right\Vert _{H^{1}%
({\mathbb{R}})}\right)  \qquad\pm t\geq a. \label{nn.4}%
\end{equation}
Using (\ref{90}) with $g=x^{n}\partial_{x}m(k,tx)w_{\pm}(k),$ $g_{1}%
=g_{2}=\partial_{k}x^{n}\partial_{x}m(k,tx),$ and $p=q=2,$ (\ref{3.6}), and
(\ref{3.7}), we get,
\begin{align}
&  \left\Vert \frac{1}{2}\sqrt{\frac{it}{2\pi}}\int_{{\mathbb{R}}%
}e^{-it\left(  k-\frac{x}{2}\right)  ^{2}}t\partial_{x}m(k,tx)w_{\pm
}(k)dk\right\Vert _{L^{\infty}({\mathbb{R}})}\nonumber\\
&  \leq C\left\Vert tx^{n}\left\langle tx\right\rangle ^{-1-\delta}\right\Vert
_{L^{\infty}({\mathbb{R}})}\left\Vert w\right\Vert _{H^{1}({\mathbb{R}}%
)}\nonumber\\
&  \leq C\left(  |t|^{1-n}\left\Vert w\right\Vert _{H^{1}({\mathbb{R}}%
)}\right)  ,\qquad\pm t\geq a. \label{113}%
\end{align}
Then, estimate (\ref{EST15rough}) follows from (\ref{ConNonlinearity}),
(\ref{110}) with $\partial_{t},$ replaced by $\partial_{x},$ \eqref{nn.3},
\eqref{nn.4}, \eqref{113} \eqref{l2.1}, (\ref{EST4}), (\ref{2.11}), and
(\ref{89.xxxx}).

Further, if $V$ admits a regular decomposition for some $\delta\geq n-1 +1/2,
$ if $n>0,$ and $\delta=0,$ for $n=0,$ using (\ref{90}) with $g= x^{n}%
\partial_{x} m(k,tx) w_{\pm}(k),$ $g_{1}=g_{2}= \partial_{k} x^{n}
\partial_{x} m(k,tx), $ and $p=q=2,$ (\ref{3.6}), and (\ref{3.7}), we obtain,
\begin{align}
&  \left\Vert \frac{1}{2} \sqrt{\frac{it}{2\pi}}\int_{ {\mathbb{R}}%
}e^{-it\left(  k-\frac{x}{2}\right)  ^{2}} t \partial_{x} m(k,tx) w_{\pm}(k)
dk \right\Vert _{L^{2}( {\mathbb{R}})}\nonumber\\
&  \leq C \left\Vert tx^{n}\left\langle tx\right\rangle ^{-1-\delta
}\right\Vert _{L^{2}( {\mathbb{R}})}\left\Vert w\right\Vert _{H^{1}(
{\mathbb{R}})}\nonumber\\
&  \leq C\left(  |t|^{1/2-n}\left\Vert w\right\Vert _{H^{1}( {\mathbb{R}}%
)}\right)  , \qquad\pm t \geq a . \label{113.1}%
\end{align}
Finally, estimate (\ref{EST15}) follows from (\ref{ConNonlinearity}),
(\ref{110}) with $\partial_{t},$ replaced by $\partial_{x},$ \eqref{nn.3},
\eqref{nn.4}, \eqref{113.1} \eqref{l2.1}, (\ref{EST4}), (\ref{2.11}), and
(\ref{89.xxxx}).
\end{proof}

\section{Proof of Lemma \ref{LemmaPrincipal}.\label{Lprincipal}}

This Section is devoted to the proof of the core estimates presented in Lemma
\ref{LemmaPrincipal}. We divide the proof in several parts.

\subsection{Prof of \eqref{18.2.3}, \eqref{18.0}, and \eqref{18.0.1.1}}

\label{primeras}

\begin{proof}
[Proof of \eqref{18.2.3}]\label{Lprincipal.1} \noindent Equation
\eqref{18.2.3} follows from \eqref{EST20}.
\end{proof}

We now proceed to prove \eqref{18.0}.

\begin{proof}
[Proof of \eqref{18.0}]Note that,
\begin{equation}
\label{14}it (k\pm x/2)\, e^{it(k\pm x/2)^{2}} = \pm\partial_{x} e^{it(k\pm
x/2)^{2}} = \frac{1}{2} \partial_{k} e^{it(k\pm x/2)^{2}} .
\end{equation}

By (\ref{14}) we have%
\begin{equation}
k\widehat{\mathcal{W}}\left(  \tau\right)  f_{1,\pm}=\frac{1}{2}\left[
\partial_{k}\widehat{\mathcal{W}}\left(  \tau\right)  f_{1,\pm}+D_{1,\pm
}\right]  , \label{400}%
\end{equation}
where
\begin{align}
\label{400.1} &  D_{1,\pm}=-\left(  i\tau\right)  ^{-1}\left(  \partial
_{k}S\left(  k\right)  \right)  \mathcal{V}_{ + }\left(  \tau\right)  \left(
m^{\dagger}\left(  -k,\tau x\right)  f_{1,\pm}\right) \nonumber\\
&  +\left(  i\tau\right)  ^{-1}S\left(  k\right)  \mathcal{V}_{+}\left(
\tau\right)  \left(  \left(  \partial_{k}m^{\dagger}\right)  \left(  -k,\tau
x\right)  f_{1,\pm}\right) \nonumber\\
&  -\left(  i\tau\right)  ^{-1}\left(  \mathcal{V}_{- }\left(  \tau\right)
\right)  \left(  \left(  \partial_{k}m^{\dagger}\right)  \left(  k,\tau
x\right)  f_{1,\pm}\right) \\
&  -S\left(  k\right)  \mathcal{V}_{ + }\left(  \tau\right)  \left(  x\left(
m^{\dagger}\left(  -k,\tau x\right)  f_{1,\pm}\right)  \right)  +\mathcal{V}_{
- }\left(  \tau\right)  \left(  xm^{\dagger}\left(  k,\tau x\right)  f_{1,\pm
}\right)  .\nonumber
\end{align}

By \eqref{ConNonlinearity}, \eqref{UnitaritySM}, \eqref{3.2}, \eqref{3.5},
\eqref{scatmatrixderiv}, \eqref{EST8},\eqref{EST4}, \eqref{EST8.0}, and
\eqref{fin.1},
\begin{equation}
\left\Vert D_{1,\pm}\right\Vert _{L^{2}({\mathbb{R}})}\leq\left\Vert \langle
x\rangle f_{1,\pm}\right\Vert _{L^{2}({\mathbb{R}})}\leq C|\tau|^{-\alpha
/2}\,\left\Vert w_{\pm}\right\Vert _{H^{1}({\mathbb{R}})}^{\alpha}\left(
\left\Vert w_{\pm}\right\Vert _{H^{1}({\mathbb{R}})}+\left\Vert w_{\pm
}\right\Vert _{L_{1}^{2}({\mathbb{R}})}\right)  ,\qquad\pm\tau\geq a.
\label{400.3}%
\end{equation}
Then, \eqref{18.0} follows from \eqref{400}, \eqref{400.3}, and \eqref{18.1}.
\end{proof}

\begin{proof}
[Proof of \eqref{18.0.1.1}]By (\ref{400})%
\begin{align}
k^{2}\widehat{\mathcal{W}}\left(  \tau\right)  f_{1,\pm}  &  =\frac{1}{2}
\left[  k\partial_{k}\widehat{\mathcal{W}}\left(  \tau\right)  f_{1,\pm
}+kD_{1,\pm}\right] \nonumber\\
&  =\frac{1}{2}\left[  \partial_{k}\left(  k\widehat{\mathcal{W}}\left(
\tau\right)  f_{1,\pm}\right)  -\widehat{\mathcal{W}}\left(  \tau\right)
f_{1,\pm}+kD_{1,\pm}\right] \label{400.2}\\
&  =\frac{1}{2}\left[  \frac{1}{2}\partial_{k}^{2}\widehat{\mathcal{W}}\left(
\tau\right)  f_{1,\pm}+\frac{1}{2}\partial_{k}D_{1,\pm}-\widehat{\mathcal{W}%
}\left(  \tau\right)  f_{1,\pm}+kD_{1,\pm}\right]  .\nonumber
\end{align}
By \eqref{ConNonlinearity}, \eqref{kpartialk},\eqref{12}, \eqref{130},
\eqref{UnitaritySM}, \eqref{3.2}, \eqref{3.5}, \eqref{3.6}, \eqref{3.7},
\eqref{3.5bis}, \eqref{scatmatrixderiv}, \eqref{scatmatrixsecondderiv},
\eqref{EST8},\linebreak\eqref{EST4}, \eqref{EST8.0}, and \eqref{fin.1},%

\begin{equation}
\label{400.3.1}\left\Vert \partial_{k}D_{1,\pm}\right\Vert _{L^{2}
{\mathbb{R}}} C |\tau|^{-\alpha/2} \, \left\Vert w_{\pm}\right\Vert _{H^{1}(
{\mathbb{R}})}^{\alpha}\left(  \left\Vert w_{\pm}\right\Vert _{H^{2}(
{\mathbb{R}})}+\left\Vert w_{\pm}\right\Vert _{L^{2}_{2}( {\mathbb{R}}%
)}\right)  , \qquad\pm\tau\geq a,
\end{equation}
and using also \eqref{14}, and \eqref{fin.2}, we get,
\begin{equation}
\label{400.4}\left\Vert k D_{1,\pm}\right\Vert _{L^{2} {\mathbb{R}}} C
|\tau|^{-\alpha/2} \, \left\Vert w_{\pm}\right\Vert _{H^{1}( {\mathbb{R}}%
)}^{\alpha}\left(  \left\Vert w_{\pm}\right\Vert _{H^{2}( {\mathbb{R}}%
)}+\left\Vert w_{\pm}\right\Vert _{L^{2}_{2}( {\mathbb{R}})}\right)  ,
\qquad\pm\tau\geq a,
\end{equation}
%
Finally, (\ref{18.0.1.1}) follows from \eqref{18.1.1}, \eqref{400.2},
\eqref{400.3.1}, and \eqref{400.4}.
\end{proof}

\subsection{Proof of \eqref{18.1}, and \eqref{18.1.1}}

\label{remaining} In this subsection we prove the remaining estimates of Lemma
\ref{LemmaPrincipal}, namely, \eqref{18.1}, and \eqref{18.1.1}. Let us define,%
\begin{equation}
\label{dd.1}\theta_{\pm}:=\widehat{\mathcal{W}}\left(  \tau\right)  f_{1,\pm
}-\widehat{\mathcal{V}}\left(  \tau\right)  f_{1,\pm}^{\operatorname*{fr}}.
\end{equation}
We decompose
\begin{equation}
\label{dd.2}\widehat{\mathcal{W}}\left(  \tau\right)  f_{1,\pm}=\theta_{\pm
}+\widehat{\mathcal{V}}\left(  \tau\right)  f_{1,\pm}^{\operatorname*{fr}}.
\end{equation}
Then, Lemma \ref{LemmaPrincipal} is consequence of the following estimates for
the derivatives $\partial_{k}^{j}\theta_{\pm}$ and $\partial_{k}%
^{j}\widehat{\mathcal{W}}\left(  \tau\right)  f_{1,\pm}^{\operatorname*{fr}},$
$j=1,2.$

\begin{lemma}
\label{Lth}Suppose that $V\in L_{5+\tilde{\delta}}^{1}({\mathbb{R}}),$ for
some $\tilde{\delta}>0$ and that $H_{A,B,V}$ does not have negative
eigenvalues. In addition assume that $V$ admits a regular decomposition with
$\delta>1$ (see Definition \ref{Def1}). Let the nonlinear function
$\mathcal{N}$ satisfy (\ref{ConNonlinearity}). Suppose that $w_{\pm}$
satisfies (\ref{eqw}) for $t\in\mathcal{T}_{\pm}\left(  a,\infty\right)  ,$
$a>0.$ Then, for some $C>0,$ uniformly for $a\geq a_{0}>0,$ we have,%

\begin{equation}
\left\Vert \int_{\mathcal{T}_{\pm}(a,t)}\partial_{k}\theta_{\pm}%
d\tau\right\Vert _{L^{2}( {\mathbb{R}})}\leq C \sup_{\tau\in\mathcal{T}_{\pm
}(a,t)} \left(  \left\Vert w_{\pm}\right\Vert _{H^{1}( {\mathbb{R}})}^{\alpha
}+1\right)  \Lambda_{\pm}\left(  \tau;w_{\pm}\right)  , \qquad\pm t \geq a,
\label{Lth1}%
\end{equation}
and%
\begin{equation}
\left\Vert \int_{\mathcal{T}_{\pm}(a,t)}\partial_{k}^{2}\theta_{\pm}%
d\tau\right\Vert _{L^{2}( {\mathbb{R}})}\leq C\sqrt{|t|} \sup_{\tau
\in\mathcal{T}_{\pm}(a,t)} \left(  \left\Vert w_{\pm}\right\Vert _{H^{1}(
{\mathbb{R}})}^{\alpha}+1 \right)  \Lambda_{\pm}\left(  \tau;w_{\pm}\right)  ,
\qquad\pm t \geq a. \label{Lth2}%
\end{equation}

\end{lemma}

\begin{proof}
See Subsection \ref{Ltheta}.
\end{proof}

\begin{lemma}
\label{Lema3}Suppose that $V\in L_{7/2+\tilde{\delta}}^{1}(\mathbb{R}^{+}),$
for some $\tilde{\delta}>0$ and that $H_{A,B,V}$ does not have negative
eigenvalues. In addition, suppose that $V$ admits a regular decomposition with
$\delta>3/2$ (see Definition \ref{Def1}). Let the nonlinear function
$\mathcal{N}$ satisfy (\ref{ConNonlinearity}), and assume that it commutes
with the projectors $P_{-}$ onto the eigenspace of the scattering matrix
$S(0)$ corresponding to the eigenvalue $-1$ (see Appendix~ \ref{App1}) . Then,
for some $C=C\left(  a\right)  >0,$ uniformly for $a\geq a_{0}>0,$ and for any
measurable function, $w_{\pm}(t),$ of $t\in\mathcal{T}_{\pm}\left(
a,\infty\right)  ,$ with values in $H^{2}({\mathbb{R}})$ that satisfies
(\ref{eqw}), and $w_{\pm}(k)=S(k)w_{\pm}(-k),$ $k\in{\mathbb{R}},$ we have
\begin{equation}
\left\Vert \partial_{k}\int_{\mathcal{T}_{\pm}(a,t)}\widehat{\mathcal{V}%
}\left(  \tau\right)  f_{1,\pm}^{\operatorname*{fr}}d\tau\right\Vert
_{L^{2}({\mathbb{R}})}\leq C\sup_{\tau\in\mathcal{T}_{\pm}(a,t)}\left(
\left\Vert w_{\pm}\right\Vert _{H^{1}({\mathbb{R}})}^{\alpha}+1\right)
\Lambda_{\pm}\left(  t,w_{\pm}\right)  ,\qquad\pm t\geq a. \label{Lema3F1}%
\end{equation}

\end{lemma}

\begin{proof}
See Subsection \ref{ProofLema34}.
\end{proof}

\begin{lemma}
\label{Lema4}Suppose that $V\in L_{4}^{1}(\mathbb{R}^{+})$ and that
$H_{A,B,V}$ does not have negative eigenvalues. In addition, suppose that $V$
admits a regular decomposition with $\delta=2$ (see Definition \ref{Def1}).
Assume that the nonlinear function $\mathcal{N}$ satisfy
(\ref{ConNonlinearity}) and that it commutes with the projector $P_{-}$ onto
the eigenspace of the scattering matrix $S(0)$ corresponding to the eigenvalue
$-1$ (see Appendix \ref{App1}). Then, for some $C>0,$ uniformly for $a\geq
a_{0}>0,$ and for any measurable function, $w_{\pm}(t),$of $t\in
\mathcal{T}_{\pm}\left(  a,\infty\right)  ,$ with values in $H^{2}%
({\mathbb{R}})$ that satisfies (\ref{eqw}), and $w_{\pm}(k)=S(k)w_{\pm}(-k),$
$k\in{\mathbb{R}},$ we have
\begin{equation}
\left\Vert \partial_{k}^{2}\int_{\mathcal{T}_{\pm}(a,t)}\widehat{\mathcal{V}%
}\left(  \tau\right)  f_{1,\pm}^{\operatorname*{fr}}d\tau\right\Vert
_{L^{2}({\mathbb{R}})}\leq C\sqrt{|t|}\sup_{\tau\in\mathcal{T}_{\pm}%
(a,t)}\left(  1+\left\Vert w_{\pm}\right\Vert _{H^{1}({\mathbb{R}})}^{\alpha
}\right)  \left(  \Lambda_{\pm}(\tau;w_{\pm})+\left\Vert w_{\pm}\right\Vert
_{H^{1}({\mathbb{R}})}^{2\alpha+1}\right)  , \label{Lema3F2}%
\end{equation}
for $\pm t\geq a.$
\end{lemma}

\begin{proof}
See Subsection \ref{ProofLema34}.
\end{proof}

Therefore, we see that the proof of \eqref{18.1} is reduced to proving Lemmas
\ref{Lth}, \ref{Lema3} and \ref{Lema4}. We prove these lemmas in the next subsections.

\subsection{Estimates for the derivatives of $\theta_{\pm}.$ Proof of Lemma
\ref{Lth} \label{Ltheta}}

This Section is devoted to the proof of Lemma \ref{Lth}. In order to prove
this result, we prepare several estimates. Let us denote%
\begin{equation}
\label{pp.1}\theta_{\pm}^{(1)}:=\mathcal{W}_{+}\left(  \tau\right)  \left(
xf_{1,\pm}\right)  -\mathcal{V}_{ + }\left(  \tau\right)  \left(  xf_{1,\pm
}^{\operatorname*{fr}}\right)  , \theta_{\pm}^{(2)}:=\mathcal{W}_{-}\left(
\tau\right)  \left(  xf_{1,\pm}\right)  -\mathcal{V}_{ - }\left(  \tau\right)
\left(  xf_{1,\pm}^{\operatorname*{fr}}\right) ,
\end{equation}
and%
\begin{equation}
\label{pp.2}\tilde{\theta}_{\pm}:=\widehat{\mathcal{W}}\left(  \tau\right)
\left(  x^{2}f_{1,\pm}\right)  -\widehat{\mathcal{V}}\left(  \tau\right)
\left(  x^{2}f_{1,\pm}^{\operatorname*{fr}}\right)  .
\end{equation}
We define%
\begin{equation}
\label{pp.3}R_{1}^{\pm}\left(  \phi\right)  :=\mp\sqrt{\frac{\tau}{2\pi i}%
}\int_{0}^{\infty}e^{i\tau\left(  k\pm\frac{x}{2}\right)  ^{2}}\left(
\partial_{k}m^{\dagger}\right)  \left(  \mp k,\tau x\right)  \phi\left(
x\right)  dx,
\end{equation}
and%
\begin{equation}
\label{pp.4}R_{2}^{\pm}\left(  \phi\right)  :=R_{21}^{\pm}\left(  \phi\right)
+R_{22}^{\pm}\left(  \phi\right)  ,
\end{equation}
with%
\begin{equation}
\label{pp.5}R_{21}^{\pm}\left(  \phi\right)  :=\mp4i \tau\left(  \sqrt
{\frac{\tau}{2\pi i}}\int_{0}^{\infty}e^{i\tau\left(  k\pm\frac{x}{2}\right)
^{2}}\left(  \partial_{k}m^{\dagger}\right)  \left(  \mp k,\tau x\right)
\left(  k\pm\frac{x}{2}\right)  \phi\left(  x\right)  dx\right)  ,
\end{equation}
and%
\begin{equation}
\label{pp.6}R_{22}^{\pm}\left(  \phi\right)  :=\sqrt{\frac{\tau}{2\pi i}}%
\int_{0}^{\infty}e^{i\tau\left(  k\pm\frac{x}{2}\right)  ^{2}}\left(
\partial_{k}^{2}m^{\dagger}\right)  \left(  \mp k,\tau x\right)  \phi\left(
x\right)  dx.
\end{equation}
Also, we set%
\begin{equation}
\label{pp.7}\Phi_{1,\pm}:=|\tau|^{1+\alpha/2}\left\Vert \theta_{\pm}%
^{(1)}\right\Vert _{L^{2}( {\mathbb{R}})}+|\tau|^{1+\alpha/2}\left\Vert
\theta_{\pm}^{(2)}\right\Vert _{L^{2}( {\mathbb{R}})},
\end{equation}%
\begin{equation}
\label{pp.8}%
\begin{array}
[c]{l}%
\Phi_{2,\pm}:=|\tau|^{\alpha/2}\left[  \left\Vert R_{1}^{+}\left(  f_{1,\pm
}\right)  \right\Vert _{L^{2}( {\mathbb{R}})}+\left\Vert R_{1}^{-}\left(
f_{1,\pm}\right)  \right\Vert _{L^{2}( {\mathbb{R}})}+\right. \\
\left.  \left\Vert \left(  \partial_{k}S\left(  k\right)  \right)
\mathcal{W}_{+}\left(  \tau\right)  f_{1,\pm}\right\Vert _{L^{2}( {\mathbb{R}%
})}+\left\Vert \left(  \partial_{k}S\left(  k\right)  \right)  \mathcal{V}%
{_{+}} \left(  \tau\right)  f_{1,\pm}^{\operatorname*{fr}}\right\Vert _{L^{2}(
{\mathbb{R}})}\right]  ,
\end{array}
\end{equation}%
\begin{equation}
\label{pp.9}\Phi_{3,\pm}:= |\tau|^{\alpha/2}| \left\Vert \tau\theta_{\pm
}\right\Vert _{L^{2}( {\mathbb{R}})}+\left\Vert \tau^{2}\tilde{\theta}_{\pm
}\right\Vert _{L^{2}( {\mathbb{R}})},
\end{equation}
\begin{equation}
\label{pp.10}%
\begin{array}
[c]{l}%
\Phi_{4,\pm}:=|\tau|^{\alpha/2}\left[  \left\Vert R_{2}^{+}\left(  f_{1,\pm
}\right)  \right\Vert _{L^{2} {\mathbb{R}})}+\left\Vert R_{2}^{-}\left(
f_{1,\pm}\right)  \right\Vert _{L^{2}( {\mathbb{R}})}\right]  ,
\end{array}
\end{equation}

\begin{equation}
\Phi_{5,\pm}:=|\tau|^{\alpha/2}\left[  \left\Vert \left(  \partial_{k}S\left(
k\right)  \right)  \left(  \partial_{k}\mathcal{W}_{+}\left(  \tau\right)
\right)  f_{1,\pm}\right\Vert _{L^{2}({\mathbb{R}})}+\left\Vert \left(
\partial_{k}S\left(  k\right)  \right)  \left(  \partial_{k}\mathcal{V}%
_{\left(  +\right)  }\left(  \tau\right)  \right)  f_{1,\pm}%
^{\operatorname*{fr}}\right\Vert _{L^{2}({\mathbb{R}})}\right]  ,
\label{pp.11}%
\end{equation}%
\begin{equation}
\Phi_{6,\pm}:=|\tau|^{\alpha/2}\left[  \left\Vert \left(  \partial_{k}%
^{2}S\left(  k\right)  \right)  \mathcal{W}_{+}\left(  \tau\right)  f_{1,\pm
}\right\Vert _{L^{2}({\mathbb{R}})}+\left\Vert \left(  \partial_{k}%
^{2}S\left(  k\right)  \right)  \mathcal{V}_{+}\left(  \tau\right)  f_{1,\pm
}^{\operatorname*{fr}}\right\Vert _{L^{2}({\mathbb{R}})}\right]  ,
\label{pp.12}%
\end{equation}%
\begin{equation}
\Omega_{1,\pm}:=\left\Vert \int_{\mathcal{T}_{\pm}(a,t)}\tau k\theta_{\pm
}d\tau\right\Vert _{L^{2}({\mathbb{R}})},\Omega_{2,\pm}:=\left\Vert
\int_{\mathcal{T}_{\pm}(a,t)}\tau^{2}k\theta_{\pm}^{(1)}d\tau\right\Vert
_{L^{2}({\mathbb{R}})}+\left\Vert \int_{\mathcal{T}_{\pm}(a,t)}\tau^{2}%
k\theta_{\pm}^{(2)}d\tau\right\Vert _{L^{2}({\mathbb{R}})}, \label{pp.13}%
\end{equation}
and%
\begin{equation}
\Omega_{3,\pm}:=\left\Vert \int_{\mathcal{T}_{\pm}(a,t)}\tau^{2}k^{2}%
\theta_{\pm}d\tau\right\Vert _{L^{2}({\mathbb{R}})}. \label{pp.14}%
\end{equation}
We begin with the following:

\begin{lemma}
Suppose that $V\in L_{4}^{1}(\mathbb{R}^{+}).$ Let the nonlinear function
$\mathcal{N}$ satisfy (\ref{ConNonlinearity}). Then, for some $C>0,$ uniformly
for $a\geq a_{0}>0,$ and for any measurable function, $w_{\pm}(t), $ of
$t\in\mathcal{T}_{\pm}\left(  a,\infty\right)  ,$ with values in
$H^{1}({\mathbb{R}}),$ we have
\begin{equation}
\left\Vert \int_{\mathcal{T}_{\pm}(a,t)}\partial_{k}\theta_{\pm}%
d\tau\right\Vert _{L^{2}({\mathbb{R}})}\leq C\left(  \Omega_{1,\pm}+\sup
_{\tau\in\mathcal{T}_{\pm}(a,t)}\left(  \Phi_{1,\pm}+\Phi_{2,\pm}\right)
\right)  ,\qquad\pm t\geq a, \label{114}%
\end{equation}
and%
\begin{equation}
\left\Vert \int_{\mathcal{T}_{\pm}(a,t)}\partial_{k}^{2}\theta_{\pm}%
d\tau\right\Vert _{L^{2}({\mathbb{R}})}\leq C\left(  \Omega_{2,\pm}%
+\Omega_{3,\pm}+\sup_{\tau\in\mathcal{T}_{\pm}(a,t)}\left(  \sum_{j=3}^{6}%
\Phi_{j\pm,}\right)  \right)  ,\qquad\pm t\geq a. \label{115}%
\end{equation}

\end{lemma}

\begin{proof}
We calculate%
\[
\partial_{k}\mathcal{W}_{\pm}\left(  \tau\right)  \phi=2i\tau k\mathcal{W}%
_{\pm}\left(  \tau\right)  \phi\pm i\tau\mathcal{W}_{\pm}\left(  \tau\right)
\left(  x\phi\left(  x\right)  \right)  +R_{1}^{\pm}\left(  \phi\right)  ,
\]
and
\begin{align*}
\partial_{k}^{2}\mathcal{W}_{\pm}\left(  \tau\right)  \phi &  =2i\tau
\mathcal{W}_{\pm}\left(  \tau\right)  \phi-4\tau^{2}k^{2}\mathcal{W}_{\pm
}\left(  \tau\right)  \phi\\
&  \mp4\tau^{2}k\mathcal{W}_{\pm}\left(  \tau\right)  \left(  x\phi\left(
x\right)  \right)  -\tau^{2}\mathcal{W}_{\pm}\left(  \tau\right)  \left(
x^{2}\phi\left(  x\right)  \right)  +R_{2}^{\pm}\left(  \phi\right)  + 2 i
\tau k R^{\pm}_{1}(\phi)\pm i \tau R^{\pm}_{1}(x \phi) .
\end{align*}
Then,
\begin{align}
\partial_{k}\widehat{\mathcal{W}}\left(  \tau\right)  \phi &  =2i\tau
k\widehat{\mathcal{W}}\left(  \tau\right)  \phi-i\tau\mathcal{W}_{-}\left(
\tau\right)  \left(  x\phi\left(  x\right)  \right)  +i\tau\mathcal{W}%
_{+}\left(  \tau\right)  \left(  xS\left(  k\right)  \phi\left(  x\right)
\right) \label{pp.15}\\
&  +R_{1}^{-}\left(  \phi\right)  +R_{1}^{+}\left(  S\left(  k\right)
\phi\right)  +\left(  \partial_{k}S\left(  k\right)  \right)  \mathcal{W}%
_{+}\left(  \tau\right)  \phi,\nonumber
\end{align}
and%
\begin{align}
\label{pp.16}\partial_{k}^{2}\widehat{\mathcal{W}}\left(  \tau\right)  \phi &
=2i\tau\widehat{\mathcal{W}}\left(  \tau\right)  \phi-4\tau^{2}k^{2}%
\widehat{\mathcal{W}}\left(  \tau\right)  \phi-\tau^{2}\widehat{\mathcal{W}%
}\left(  \tau\right)  \left(  x^{2}\phi\left(  x\right)  \right) \nonumber\\
&  +4\tau^{2}k\mathcal{W}_{-}\left(  \tau\right)  \left(  x\phi\left(
x\right)  \right)  -4\tau^{2}kS\left(  k\right)  \mathcal{W}_{+}\left(
\tau\right)  \left(  x\phi\left(  x\right)  \right) \nonumber\\
&  +R_{2}^{-}\left(  \phi\right)  +R_{2}^{+}\left(  S\left(  k\right)
\phi\right) \\
&  +2\partial_{k}S\left(  k\right)  \left(  \partial_{k}\mathcal{W}_{+}\left(
\tau\right)  \right)  \phi+\partial_{k}^{2}S\left(  k\right)  \mathcal{W}%
_{+}\left(  \tau\right)  \phi+ 2i \tau k \left[  R^{-}_{1}(\phi)+ S(k)
R^{+}_{1}(\phi) \right]  +\nonumber\\
&  i \tau\left[  - R^{-}_{1}(x\phi)+ S(k) R^{+}_{1}(x\phi) \right]  .\nonumber
\end{align}
Using the above relations replacing the function $m$ by $1,$ we get

\begin{align}
\label{pp.17}\partial_{k}\widehat{\mathcal{V}}\left(  \tau\right)  \phi &
=2i\tau k\widehat{\mathcal{V}}\left(  \tau\right)  \phi-i\tau\mathcal{V}_{ -
}\left(  \tau\right)  \left(  x\phi\left(  x\right)  \right) \nonumber\\
&  +i\tau\mathcal{V}_{ + }\left(  \tau\right)  \left(  xS\left(  k\right)
\phi\left(  x\right)  \right)  +\left(  \partial_{k}S\left(  k\right)
\right)  \mathcal{V}_{+}\left(  \tau\right)  \phi,
\end{align}
and%
\begin{align}
\label{pp.18}\partial_{k}^{2}\widehat{\mathcal{V}}\left(  \tau\right)  \phi &
=2i\tau\widehat{\mathcal{V}}\left(  \tau\right)  \phi-4\tau^{2}k^{2}%
\widehat{\mathcal{V}}\left(  \tau\right)  \phi-\tau^{2}\widehat{\mathcal{V}%
}\left(  \tau\right)  \left(  x^{2}\phi\left(  x\right)  \right) \nonumber\\
&  +4\tau^{2}k\mathcal{V}_{ - }\left(  \tau\right)  \left(  x\phi\left(
x\right)  \right)  -4\tau^{2}kS\left(  k\right)  \mathcal{V}_{ + }\left(
\tau\right)  \left(  x\phi\left(  x\right)  \right) \\
&  +2\partial_{k}S\left(  k\right)  \left(  \partial_{k}\mathcal{V}_{ +
}\left(  \tau\right)  \right)  \phi+\partial_{k}^{2}S\left(  k\right)
\mathcal{V}_{ + }\left(  \tau\right)  \phi.\nonumber
\end{align}
Then, \eqref{114} and \eqref{115} follow from \eqref{pp.1}-\eqref{pp.18},
\eqref{scatmatrixderiv} and \eqref{scatmatrixsecondderiv}.
\end{proof}

\begin{lemma}
\label{Lema1}Suppose that $V\in L_{4}^{1}(\mathbb{R}^{+}),$ and that it admits
a regular decomposition with $\delta>0$ (see Definition \ref{Def1}). Let the
nonlinear function $\mathcal{N}$ satisfy (\ref{ConNonlinearity}). Then, for
some $C>0,$ uniformly for $a\geq a_{0}>0,$ and for any measurable function,
$w_{\pm}(t),$ of $t\in\mathcal{T}_{\pm}\left(  a,\infty\right)  , $ with
values in $H^{1}({\mathbb{R}}),$ we have%

\begin{equation}
\sqrt{|\tau|} \Phi_{1,\pm}(\tau)+\Phi_{2,\pm}(\tau)\leq C \left\Vert w_{\pm
}(\tau)\right\Vert _{H^{1}( {\mathbb{R}})}^{\alpha+1}, \qquad\pm\tau\geq a,
\label{Lema1F1}%
\end{equation}
and%
\begin{equation}
\left(  \sum_{j=3}^{5}\Phi_{j,\pm}(\tau)\right)  + \sqrt{|\tau|}\Phi_{6,\pm}
\leq C \sqrt{|\tau|} \left\Vert w_{\pm}(\tau)\right\Vert _{H^{1}( {\mathbb{R}%
})}^{\alpha+1}, \qquad\pm\tau\geq a. \label{Lema1F2}%
\end{equation}

\end{lemma}

\begin{proof}
We decompose%
\begin{equation}
\theta_{\pm}^{(j)}=\theta^{(j)}_{\pm, 1}+\theta_{\pm,2}^{(j)}, \qquad j=1,2,
\label{6}%
\end{equation}
where%
\begin{equation}
\theta_{\pm,1}^{\left(  1\right)  }:=\left(  \mathcal{W}_{+}\left(
\tau\right)  -\mathcal{V}_{+}\left(  \tau\right)  \right)  \left(  xf_{1,\pm
}\right)  , \qquad\theta_{\pm,1}^{\left(  2\right)  }:=\left(  \mathcal{W}%
_{-}\left(  \tau\right)  -\mathcal{V}_{-}\left(  \tau\right)  \right)  \left(
xf_{1,\pm}\right)  , \label{6.1}%
\end{equation}
and%
\begin{equation}
\theta_{\pm,2}^{\left(  1\right)  }=\mathcal{V}_{+}\left(  \tau\right)
\left(  xf_{2,\pm}\right)  ,\qquad\theta_{\pm,2}^{\left(  2\right)
}=\mathcal{V}_{-}\left(  \tau\right)  \left(  xf_{2,\pm}\right)  . \label{6.2}%
\end{equation}
Using (\ref{12}) and (\ref{EST12j}) with $n=1$ we have%
\[
\left\Vert \theta_{\pm,1}^{\left(  j\right)  }\right\Vert _{L^{2}(
{\mathbb{R}})}\leq C|\tau|^{-3/2}\left\Vert f_{1,\pm}\right\Vert _{L^{\infty}(
{\mathbb{R}})}\leq C|\tau|^{-3/2- \alpha/2}\left\Vert w_{\pm}\right\Vert
_{H^{1}( {\mathbb{R}})}^{\alpha+1}, \qquad j=1,2, \pm\tau\geq a.
\]
Since
\begin{equation}
\left\vert f_{2,\pm}\right\vert \leq C\left\vert \mathcal{W}\left(
\tau\right)  w_{\pm}-\mathcal{V}\left(  \tau\right)  w_{\pm}\right\vert
\left(  \left\vert \mathcal{W}\left(  \tau\right)  w_{\pm}\right\vert
^{\alpha}\mathcal{+}\left\vert \mathcal{V}\left(  \tau\right)  w_{\pm
}\right\vert ^{\alpha}\right)  , \label{11}%
\end{equation}
by (\ref{EST4}) and (\ref{EST11}) with $\delta>n+1/2$ we get%
\begin{equation}
\left\Vert x^{n}f_{2,\pm}\right\Vert _{L^{2}( {\mathbb{R}})}\leq C\left\Vert
x^{n}\left\langle \tau x\right\rangle ^{ -\delta}\right\Vert _{L^{2}(
{\mathbb{R}})} |\tau|^{-\alpha/2} \left\Vert w_{\pm}\right\Vert _{H^{1}(
{\mathbb{R}})}^{\alpha+1}\leq C|\tau|^{-1/2-n-\alpha/2}\left\Vert w_{\pm
}\right\Vert _{H^{1}( {\mathbb{R}})}^{\alpha+1}. \label{13}%
\end{equation}
Then, putting $n=1$ and using (\ref{isomafr}) we show that%
\[
\left\Vert \theta_{\pm,2}^{\left(  j\right)  }\right\Vert _{L^{2}(
{\mathbb{R}})}\leq C|\tau|^{-3/2-\alpha/2}\left\Vert w\right\Vert
_{H^{1}(\mathbb{R}^{+})}^{\alpha+1}, \qquad j=1,2, \pm t \geq a.
\]
Therefore%
\begin{equation}
\left\Vert \theta_{\pm}^{(j)}\right\Vert _{L^{2}( {\mathbb{R}})}\leq
C|\tau|^{-3/2-\alpha/2}\left\Vert w_{\pm}\right\Vert _{H^{1}( {\mathbb{R}%
(\mathbb{R}^{+})})}^{\alpha+1},\qquad j=1,2, \pm t \geq a, \label{th1}%
\end{equation}
and hence,%
\begin{equation}
\Phi_{1,\pm}\leq C|\tau|^{-1/2}\left\Vert w_{\pm}\right\Vert _{H^{1}(
{\mathbb{R}})}^{\alpha+1}, \qquad\pm t\geq a. \label{16}%
\end{equation}
Using (\ref{3.5}) with $\delta>1/2$ and (\ref{12}) we get%
\[
\left\Vert R_{1}^{\pm} f_{1,+} \right\Vert _{L^{2}( {\mathbb{R}})}\leq
C\sqrt{|\tau|}\left\Vert \left\langle \cdot\right\rangle ^{-1}\right\Vert
_{L^{2}( {\mathbb{R}})}\left\Vert \left\langle \tau x\right\rangle ^{-\delta
}\right\Vert _{L^{2}( {\mathbb{R}})}\left\Vert f_{1,+}\right\Vert _{L^{2}(
{\mathbb{R}})}\leq C |\tau|^{-\alpha/2} \left\Vert w_{+}\right\Vert _{H^{1}%
}^{\alpha+1},
\]
and
\[
\left\Vert R_{1}^{\pm} f_{1,-} \right\Vert _{L^{2}( {\mathbb{R}})}\leq
C\sqrt{|\tau|}\left\Vert \left\langle \cdot\right\rangle ^{-1}\right\Vert
_{L^{2}( {\mathbb{R}})}\left\Vert \left\langle \tau x\right\rangle ^{-\delta
}\right\Vert _{L^{2}( {\mathbb{R}})}\left\Vert f_{1,-}\right\Vert _{L^{2}(
{\mathbb{R}})}\leq C |\tau|^{-\alpha/2} \left\Vert w_{-}\right\Vert _{H^{1}%
}^{\alpha+1}.
\]

Moreover, by (\ref{12}), (\ref{scatmatrixderiv}), (\ref{EST8}),and
(\ref{EST8.0}) we show%
\[
\left\Vert \left(  \partial_{k}S\left(  k\right)  \right)  \mathcal{W}%
_{+}\left(  \tau\right)  f_{1,\pm}\right\Vert _{L^{2}( {\mathbb{R}}%
)}+\left\Vert \left(  \partial_{k}S\left(  k\right)  \right)  \mathcal{V}_{+}
\left(  \tau\right)  f_{1,\pm}^{\operatorname*{fr}}\right\Vert _{L^{2}(
{\mathbb{R}})}\leq C |\tau|^{-\alpha/2} \left\Vert w_{\pm}\right\Vert _{H^{1}(
{\mathbb{R}})}^{\alpha+1}.
\]
Hence,%
\begin{equation}
\Phi_{2,\pm}\leq C \left\Vert w_{\pm}\right\Vert _{H^{1}( {\mathbb{R}}%
)}^{\alpha+1}, \qquad\pm\tau\geq a. \label{17}%
\end{equation}
Using (\ref{16}) and (\ref{17}) we prove (\ref{Lema1F1}).

Writing
\[
\widehat{\mathcal{W}}\left(  \tau\right)  \left(  x^{n}f_{1,\pm}\right)
-\widehat{\mathcal{V}}\left(  \tau\right)  \left(  x^{n}f_{1,\pm
}^{\operatorname*{fr}}\right)  =\left(  \widehat{\mathcal{W}}\left(
\tau\right)  -\widehat{\mathcal{V}}\left(  \tau\right)  \right)  \left(
x^{n}f_{1,\pm}\right)  +\widehat{\mathcal{V}}\left(  \tau\right)  \left(
x^{n}f_{2,\pm}\right)  ,
\]
and using (\ref{12}), (\ref{13}), (\ref{EST12}), and (\ref{isomafr}), we get%
\begin{equation}
\left\Vert \widehat{\mathcal{W}}\left(  \tau\right)  \left(  x^{n}f_{1,\pm
}\right)  -\widehat{\mathcal{V}}\left(  \tau\right)  \left(  x^{n}f_{1,\pm
}^{\operatorname*{fr}}\right)  \right\Vert _{L^{2}( {\mathbb{R}})}\leq C|\tau|
^{-1/2-n-\alpha/2} \left\Vert w_{\pm}\right\Vert _{H^{1}( {\mathbb{R}}%
)}^{\alpha+1}. \label{21.0}%
\end{equation}
Using this last equation with $n=0,$ we estimate%
\begin{equation}
\left\Vert \theta_{\pm}\right\Vert _{L^{2}( {\mathbb{R}})}\leq C |\tau|
^{-1/2-\alpha/2} \left\Vert w_{\pm}\right\Vert _{H^{1}( {\mathbb{R}})}%
^{\alpha+1}, \label{th}%
\end{equation}
and using it with $n=2$ gives us
\[
\left\Vert \tilde{\theta}_{\pm}\right\Vert _{L^{2}( {\mathbb{R}})}\leq
C|\tau|^{-5/2-\alpha/2}\left\Vert w\right\Vert _{H^{1}}^{\alpha+1}.
\]
Therefore%
\begin{equation}
\Phi_{3,\pm} \leq C|\tau|^{1/2}\left\Vert w_{\pm}\right\Vert _{H^{1}(
{\mathbb{R}})}^{\alpha+1}, \qquad\pm\tau\geq a. \label{18}%
\end{equation}
Using (\ref{14}) and integrating by parts, we get
\begin{align*}
R_{21}^{\pm}\left(  \phi\right)   &  =2 \sqrt{\frac{2\tau}{\pi i}}e^{i\tau
k^{2}}\left(  \partial_{k}m^{\dagger}\right)  \left(  \mp k,0\right)
\phi\left(  0\right) \\
&  +2\sqrt{\frac{2\tau}{2\pi i}}\int_{0}^{\infty}e^{i\tau\left(  k\pm\frac
{x}{2}\right)  ^{2}}\partial_{x}\left(  \left(  \partial_{k}m^{\ast}\right)
\left(  \mp k,\tau x\right)  \phi\left(  x\right)  \right)  dx.
\end{align*}
Using (\ref{EST5}), and (\ref{EST4}) we show,%
\begin{equation}
\left\Vert \partial_{x}f_{1,\pm}\right\Vert _{L^{2}( {\mathbb{R}})}\leq C
|\tau|^{-\alpha/2} \left\Vert \partial_{x}\mathcal{W}\left(  \tau\right)
w_{\pm}\right\Vert _{L^{2}( {\mathbb{R}})}\left\Vert \mathcal{W}\left(
\tau\right)  w_{\pm}\right\Vert _{L^{\infty}( {\mathbb{R}})}^{\alpha}\leq
C\sqrt{|\tau|} |\tau|^{-\alpha/2} \left\Vert w_{\pm}\right\Vert _{H^{1}(
{\mathbb{R}})}^{\alpha+1}, \qquad\pm\tau\geq a. \label{15}%
\end{equation}
Then, using (\ref{12.0.0}), (\ref{3.5}) with $\delta>1/2$, and (\ref{3.7}), we
show
\begin{align*}
\left\Vert R_{21}^{\pm}\left(  f_{1,+}\right)  \right\Vert _{L^{2}(
{\mathbb{R}})}  &  \leq C\sqrt{\tau}\left\Vert \left\langle \cdot\right\rangle
^{-1}\right\Vert _{L^{2}( {\mathbb{R}})}\left(  \left\Vert f_{1,+}\right\Vert
_{L^{\infty}( {\mathbb{R}})}+\tau\left\Vert \left\langle \tau x\right\rangle
^{-1-\hat{\delta}}\right\Vert _{L^{1}( {\mathbb{R}})}\left\Vert f_{1,+}%
\right\Vert _{L^{\infty}( {\mathbb{R}})}+\right. \\
&  \left.  \left\Vert \left\langle \tau x\right\rangle ^{-\delta}\right\Vert
_{L^{2}( {\mathbb{R}})}\left\Vert \partial_{x}f_{1,+}\right\Vert _{L^{2}(
{\mathbb{R}})}\right) \\
&  \leq C |\tau|^{1/2-\alpha/2} \left\Vert w_{+}\right\Vert _{H^{1}(
{\mathbb{R}})}^{\alpha+1}, \qquad\tau\geq a,
\end{align*}
and in the same way,
\[
\left\Vert R_{21}^{\pm}\left(  f_{1,-}\right)  \right\Vert _{L^{2}(
{\mathbb{R}})} \leq C |\tau|^{1/2-\alpha/2} \left\Vert w_{-}\right\Vert
_{H^{1}( {\mathbb{R}})}^{\alpha+1}, \qquad- \tau\geq a.
\]
Moreover, using (\ref{12.0.0}), and (\ref{3.5bis}) we get%
\[
\left\Vert R_{22}^{\pm}\left(  f_{1,+}\right)  \right\Vert _{L^{2}(
{\mathbb{R}})}\leq C\sqrt{|\tau|}\left\Vert \left\langle \cdot\right\rangle
^{-1}\right\Vert _{L^{2}( {\mathbb{R}})}\left\Vert \left\langle \tau
x\right\rangle ^{-\delta}\right\Vert _{L^{2}( {\mathbb{R}})}\left\Vert
f_{1,+}\right\Vert _{L^{\infty}( {\mathbb{R}})}\leq C |\tau|^{-\alpha/2}
\left\Vert w_{+}\right\Vert _{H^{1}( {\mathbb{R}})}^{\alpha+1}, \qquad\tau\geq
a,
\]
and in the same way,
\[
\left\Vert R_{22}^{\pm}\left(  f_{1,-}\right)  \right\Vert _{L^{2}(
{\mathbb{R}})}\leq C |\tau|^{-\alpha/2} \left\Vert w_{-}\right\Vert _{H^{1}(
{\mathbb{R}})}^{\alpha+1}, \qquad-\tau\geq a.
\]
Hence%
\begin{equation}
\Phi_{4,\pm}\leq C |\tau|^{1/2} \left\Vert w\right\Vert _{H^{1}( {\mathbb{R}%
})}^{\alpha+1}, \qquad\pm\tau\geq a. \label{19}%
\end{equation}
Using (\ref{14}) and integrating by parts we have%
\begin{align}
\left(  \partial_{k}\mathcal{W}_{+}\left(  \tau\right)  \right)  f_{1,\pm}  &
=-2\sqrt{\frac{\tau}{2\pi i}}e^{i\tau k^{2}}m^{\dagger}\left(  -k,0\right)
f_{1,\pm}\left(  0\right)  -2\mathcal{W}_{+}\left(  \tau\right)  \left(
\partial_{x}f_{1,\pm}\right) \nonumber\\
&  -2\tau\sqrt{\frac{\tau}{2\pi i}}\int_{0}^{\infty}e^{i\tau\left(  k+\frac
{x}{2}\right)  ^{2}}\left(  \partial_{x}m^{\dagger}\right)  \left(  -k,\tau
x\right)  f_{1,\pm}\left(  x\right)  dx\nonumber\\
&  -\sqrt{\frac{\tau}{2\pi i}}\int_{0}^{\infty}e^{i\tau\left(  k+\frac{x}%
{2}\right)  ^{2}}\left(  \partial_{k}m^{\dagger}\right)  \left(  -k,\tau
x\right)  f_{1,\pm}\left(  x\right)  dx. \label{52}%
\end{align}
Then, by (\ref{12}), \eqref{12.0.0}, (\ref{130}), (\ref{15}), (\ref{3.2}),
(\ref{3.5}) with $\delta>1/2$, (\ref{3.6}),(\ref{scatmatrixderiv}),
(\ref{EST8}), (\ref{isomafr}), we show%
\begin{equation}
\Phi_{5,\pm}\leq C |\tau|^{1/2} \left\Vert w_{\pm}\right\Vert _{H^{1}(
{\mathbb{R}})}^{\alpha+1}, \qquad\pm\tau\geq a. \label{20}%
\end{equation}
Moreover, using (\ref{12}), (\ref{scatmatrixsecondderiv}),(\ref{EST8}), and
(\ref{EST8.0}) we have%
\begin{equation}
\Phi_{6,\pm}\leq C \left\Vert w_{\pm}\right\Vert _{H^{1}( {\mathbb{R}}%
)}^{\alpha+1}, \qquad\pm\tau\geq a. \label{21}%
\end{equation}
Finally, by (\ref{18}), (\ref{19}), (\ref{20}), and (\ref{21}), we attain
(\ref{Lema1F2}).
\end{proof}

\begin{lemma}
\bigskip\label{Lema2}Suppose that $V\in L_{5+\tilde{\delta}}^{1}%
(\mathbb{R}^{+}),$ for some $\tilde{\delta}>0$ and that $H_{A,B,V}$ does not
have negative eigenvalues. In addition, assume that $V$ admits a regular
decomposition with $\delta>1$ (see Definition \ref{Def1}). Let the nonlinear
function $\mathcal{N}$ satisfy (\ref{ConNonlinearity}). Suppose that $w_{\pm}$
satisfies (\ref{eqw}) for $t\in\mathcal{T}_{\pm}\left(  a,\infty\right)  ,$
$a>0.$ Then, for some $C>0,$ uniformly for $a\geq a_{0}>0,$ we have,%

\begin{equation}
\Omega_{1,\pm}+\Omega_{2,\pm}\leq C \sup_{\tau\in\mathcal{T}_{\pm}(a,t)}
\left(  \left\Vert w_{\pm}\right\Vert _{H^{1}( {\mathbb{R}})}^{\alpha
}+1\right)  \Lambda_{\pm}\left(  \tau;w_{\pm}\right)  , \qquad\pm t \geq a,
\label{Lema2F1}%
\end{equation}
and%
\begin{equation}
\Omega_{3,\pm}\leq C \sqrt{|t|} \sup_{\tau\in\mathcal{T}_{\pm}(a,t)}\left(
\left\Vert w_{\pm}\right\Vert _{H^{1}( {\mathbb{R}})}^{\alpha}+1\right)
\Lambda_{\pm}\left(  \tau;w_{\pm}\right) . \qquad\pm t \geq a. \label{Lema2F2}%
\end{equation}

\end{lemma}

\begin{proof}
Define%
\[
\Theta^{\left(  1\right)  }_{\pm}\left(  \tau,k\right)  :=\tau k\theta_{\pm},
\]%
\[
\Theta^{\left(  2\right)  }_{\pm}\left(  \tau,k\right)  :=\tau^{2}k\theta
_{\pm}^{(1)},
\]%
\[
\Theta^{\left(  3\right)  }_{\pm}\left(  \tau,k\right)  :=\tau^{2}k\theta
_{\pm}^{(2)},
\]
and%
\[
\Theta^{\left(  4\right)  }_{\pm}\left(  \tau,k\right)  :=\tau^{2}k^{2}%
\theta_{\pm}.
\]
Then%
\[
\Omega_{1,\pm}= \left\Vert \int_{\mathcal{T}_{\pm}(a,t)}\Theta^{\left(
1\right)  }_{\pm}\left(  \tau,k\right)  d\tau\right\Vert _{L^{2}( {\mathbb{R}%
})},
\]%
\[
\Omega_{2,\pm} = \left\Vert \int_{\mathcal{T}_{\pm}(a,t)} \Theta^{\left(
2\right)  }_{\pm}\left(  \tau,k\right)  d\tau\right\Vert _{L^{2}( {\mathbb{R}%
})}+\left\Vert \int_{\mathcal{T}_{\pm}(a,t)}\Theta^{\left(  3\right)  }_{\pm
}\left(  \tau,k\right)  d\tau\right\Vert _{L^{2}( {\mathbb{R}})}%
\]
and%
\[
\Omega_{3,\pm}=\left\Vert \int_{\mathcal{T}_{\pm}(a,t)}\Theta^{\left(
4\right)  }_{\pm}\left(  \tau,k\right)  d\tau\right\Vert _{L^{2}( {\mathbb{R}%
})}.
\]
We have,
\begin{equation}
1= e^{-i\tau k^{2}}\frac{\partial_{\tau}\left(  \tau e^{i\tau k^{2}}\right)
}{1+i\tau k^{2}}. \label{2.20}%
\end{equation}

Using (\ref{2.20}) and integrating by parts, we have%
\begin{equation}
\int_{\mathcal{T}_{\pm}(a,t)}\Theta^{\left(  j\right)  }_{\pm}\left(
\tau,k\right)  d\tau=\Theta_{1,\pm}^{\left(  j\right)  }+\Theta_{2,\pm
}^{\left(  j\right)  }+\Theta_{3,\pm}^{\left(  j\right)  }, \label{22}%
\end{equation}
where%
\[
\Theta_{1,\pm}^{\left(  j\right)  }= \pm\bigskip\left.  \tau\frac{1}{1+i\tau
k^{2}}\Theta^{\left(  j\right)  }_{\pm}\left(  \tau,k\right)  \right\vert
_{\pm a}^{t}, \qquad\pm t \geq a.
\]%
\[
\Theta_{2,\pm}^{\left(  j\right)  }=\int_{\mathcal{T}_{\pm}(a,t)}\frac{i\tau
k^{2}}{\left(  1+i\tau k^{2}\right)  ^{2}}\Theta^{\left(  j\right)  }_{\pm
}\left(  \tau,k\right)  d\tau, \qquad\pm t \geq a,
\]
and%
\[
\Theta_{3,\pm}^{\left(  j\right)  }=- \int_{1}^{t}\tau\frac{e^{i\tau k^{2}}%
}{1+i\tau k^{2}}\partial_{\tau}\left(  e^{-i\tau k^{2}}\Theta^{\left(
j\right)  }_{\pm}\left(  \tau,k\right)  \right)  d\tau, \qquad\pm t\geq a.
\]
for $j=1,2,3.$ Using (\ref{th1}) and (\ref{th}), we obtain%
\begin{equation}
\left\Vert \frac{1}{1+i\tau k^{2}}\Theta^{\left(  j\right)  }_{\pm}\left(
\tau,k\right)  \right\Vert _{L^{2}( {\mathbb{R}})} \leq C |\tau|^{-\alpha/2}
\left\Vert w_{\pm}\right\Vert _{H^{1}( {\mathbb{R}})}^{\alpha+1}, \qquad
\pm\tau\geq a, \label{F26}%
\end{equation}
for $j=1,2,3$ and
\begin{equation}
\left\Vert \frac{1}{1+i\tau k^{2}}\Theta^{\left(  4\right)  }\left(
\tau,k\right)  \right\Vert _{L^{2}}\leq C |\tau|^{1/2-\alpha/2} \left\Vert
w\right\Vert _{H^{1}}^{\alpha+1}, \qquad\pm\tau\geq a. \label{F27}%
\end{equation}
Next, we estimate $\partial_{\tau}\Theta^{\left(  j\right)  }_{\pm}\left(
\tau,k\right)  .$ Let us control first $\partial_{\tau}\left(  e^{-i\tau
k^{2}}\theta_{\pm}\right)  .$ We decompose $\theta_{\pm}$ as
\begin{equation}
\theta_{\pm}=\theta_{\pm}^{\left(  a\right)  }+\theta_{\pm}^{\left(  b\right)
}, \label{3}%
\end{equation}
with%
\[
\theta_{\pm}^{\left(  a\right)  }=\left(  \widehat{\mathcal{W}}\left(
\tau\right)  -\widehat{\mathcal{V}}\left(  \tau\right)  \right)  f_{1,\pm}%
\]
and%
\[
\theta^{\left(  b\right)  }_{\pm}=\widehat{\mathcal{V}}\left(  \tau\right)
f_{2,\pm}.
\]
Observing that
\[
\partial_{\tau}\left(  e^{-i\tau k^{2}}e^{i\tau\left(  k\pm\frac{x}{2}\right)
^{2}}\right)  =\pm\frac{1}{\tau}e^{-i\tau k^{2}}x\partial_{x}e^{i\tau\left(
k\pm\frac{x}{2}\right)  ^{2}}-ie^{-i\tau k^{2}}\frac{x^{2}}{4}e^{i\tau\left(
k\pm\frac{x}{2}\right)  ^{2}},
\]
we calculate%
\begin{align}
&  e^{i\tau k^{2}}\partial_{\tau}\left(  e^{-i\tau k^{2}}\mathcal{W}_{\pm
}\left(  \tau\right)  f\right) \nonumber\\
&  =\mathcal{W}_{\pm}\left(  \tau\right)  \left(  \left[  \frac{1}{2\tau}
\mp\frac{2}{\tau}\right]  f \mp\frac{1}{\tau} x \partial_{x}f\left(  x\right)
-\frac{i}{4}x^{2}f\left(  x\right)  +\partial_{\tau}f) \right) \nonumber\\
&  \mp\sqrt{\frac{\tau}{2\pi i}}\int_{0}^{\infty}e^{i\tau\left(  k \pm\frac
{x}{2}\right)  ^{2}} x\left(  \partial_{x}m^{\dagger}\right)  \left(  \mp
k,\tau x\right)  f\left(  x\right)  dx, \label{7}%
\end{align}
and%
\begin{equation}
e^{i\tau k^{2}}\partial_{\tau}\left(  e^{-i\tau k^{2}}\mathcal{V}_{\pm}\left(
\tau\right)  f\right)  =\mathcal{V}_{\pm}\left(  \tau\right)  \left(  \left[
\frac{1}{2\tau} \mp\frac{2}{\tau}\right]  f \mp\frac{1}{\tau} x \partial_{x}
f\left(  x\right)  -\frac{i}{4}x^{2}f\left(  x\right)  +\partial_{\tau
}f\right)  . \label{8}%
\end{equation}
Then, using (\ref{7}) and (\ref{8}) with $f=f_{1,\pm},$ and \eqref{sp9},
\eqref{sp9.1} we get%
\begin{align*}
e^{i\tau k^{2}}\partial_{\tau}\left(  e^{-i\tau k^{2}}\theta^{\left(
a\right)  }_{\pm}\right)   &  = S(k) \left(  \mathcal{W}_{+}\left(
\tau\right)  -\mathcal{V}_{+}\left(  \tau\right)  \right)  \left(  -\frac
{3}{2\tau}f_{1,\pm}-\frac{1}{\tau}x\partial_{x}f_{1,\pm}\left(  x\right)
-\frac{i}{4}x^{2}f_{1,\pm}\left(  x\right)  +\partial_{\tau}f_{1,\pm}\right)
\\
&  - S(k) \sqrt{\frac{\tau}{2\pi i}}\int_{0}^{\infty}e^{i\tau\left(
k+\frac{x}{2}\right)  ^{2}}x\left(  \partial_{x}m^{\dagger}\right)  \left(  -
k,\tau x\right)  f_{1,\pm}\left(  x\right)  dx\\
&  + \left(  \mathcal{W}_{-}\left(  \tau\right)  -\mathcal{V}_{-}\left(
\tau\right)  \right)  \left(  \frac{5}{2\tau}f_{1,\pm}+\frac{1}{\tau}%
x\partial_{x}f_{1,\pm}\left(  x\right)  -\frac{i}{4}x^{2}f_{1,\pm}\left(
x\right)  +\partial_{\tau}f_{1,\pm}\right) \\
&  + \sqrt{\frac{\tau}{2\pi i}}\int_{0}^{\infty}e^{i\tau\left(  k-\frac{x}%
{2}\right)  ^{2}}x\left(  \partial_{x}m^{\dagger}\right)  \left(  k,\tau
x\right)  f_{1,\pm}\left(  x\right)  dx.
\end{align*}
Hence, using (\ref{EST14}), \eqref{UnitaritySM}, (\ref{3.6}), \eqref{EST5},
\eqref{EST4}, and (\ref{EST12}) we obtain,%
\begin{equation}
\label{bb.8.1}\left\Vert e^{i\tau k^{2}}\partial_{\tau}\left(  e^{-i\tau
k^{2}}\theta^{\left(  a\right)  }_{\pm}\right)  \right\Vert _{L^{2}(
{\mathbb{R}})}\leq C |\tau|^{-3/2-\alpha/2}\left(  \left\Vert w_{\pm
}\right\Vert _{H^{1}( {\mathbb{R}})}^{\alpha}+1\right)  \Lambda_{\pm}\left(
\tau;w_{\pm}\right)  , \quad\pm\tau\geq a.
\end{equation}
By (\ref{8}) with $f=f_{2,\pm}$ we get%
\[%
\begin{array}
[c]{l}%
e^{i\tau k^{2}}\partial_{\tau}\left(  e^{-i\tau k^{2}}\theta^{\left(
b\right)  }_{\pm}\right)  = S(k) \mathcal{V}_{+}\left(  \tau\right)  \left(
-\frac{3}{2\tau}f_{2,\pm}-\frac{1}{\tau}x\partial_{x}f_{2,\pm}\left(
x\right)  -\frac{i}{4}x^{2}f_{2,\pm}\left(  x\right)  +\partial_{\tau}%
f_{2,\pm}\right) \\
+ \ \mathcal{V}_{-} \left(  \frac{5}{2\tau}f_{2,\pm}+\frac{1}{\tau}%
x\partial_{x}f_{2,\pm}\left(  x\right)  -\frac{i}{4}x^{2}f_{2,\pm}\left(
x\right)  +\partial_{\tau}f_{2,\pm}\right)  .
\end{array}
\]
Then, using (\ref{EST8.0}), (\ref{EST15}) with $n=1$, (\ref{EST16}) with
$n=0,$ and (\ref{13}) we show that%
\begin{equation}
\label{bb.8.2}\left\Vert e^{i\tau k^{2}}\partial_{\tau}\left(  e^{-i\tau
k^{2}}\theta^{\left(  b\right)  }_{\pm}\right)  \right\Vert _{L^{2}(
{\mathbb{R}})}\leq C |\tau| ^{-3/2-\alpha/2}\left(  \left\Vert w_{\pm
}\right\Vert _{H^{1}( {\mathbb{R}})}^{\alpha}+1\right)  \Lambda_{\pm}\left(
\tau;w_{\pm}\right)  ,\qquad\pm\tau\geq a .
\end{equation}
Therefore, it follows from (\ref{3}), \eqref{bb.8.1} and \eqref{bb.8.2},
\begin{equation}
\left\Vert e^{i\tau k^{2}}\partial_{\tau}\left(  e^{-i\tau k^{2}}\theta_{\pm
}\right)  \right\Vert _{L^{2}( {\mathbb{R}})}\leq C|\tau|^{-3/2-\alpha
/2}\left(  \left\Vert w_{\pm}\right\Vert _{H^{1}( {\mathbb{R}})}^{\alpha
}+1\right)  \Lambda_{\pm}\left(  \tau;w_{\pm}\right)  , \qquad\pm\tau\geq a.
\label{F22}%
\end{equation}
Next, we control $\partial_{\tau}\left(  e^{-i\tau k^{2}}\theta_{\pm}%
^{(j)}\right)  , j=1,2 .$ Observing that%
\[
\partial_{\tau}\left(  e^{-i\tau k^{2}}e^{i\tau\left(  k\pm\frac{x}{2}\right)
^{2}}\right)  =i\left(  \pm kx+\frac{x^{2}}{4}\right)  e^{-i\tau k^{2}%
}e^{i\tau\left(  k\pm\frac{x}{2}\right)  ^{2}},
\]
we show%
\begin{align*}
e^{i\tau k^{2}}\partial_{\tau}\left(  e^{-i\tau k^{2}}\mathcal{W}_{\pm}\left(
\tau\right)  f\right)   &  =\mathcal{W}_{\pm}\left(  \tau\right)  \left(
\frac{1}{2\tau}f+i\left(  \pm k x+\frac{x^{2}}{4}\right)  f\right) \\
&  +\sqrt{\frac{\tau}{2\pi i}}\int_{0}^{\infty}e^{i\tau\left(  k\pm\frac{x}%
{2}\right)  ^{2}}x\left(  \partial_{x}m^{\dagger}\right)  \left(  \mp k,\tau
x\right)  f\left(  x\right)  dx,
\end{align*}
and%
\[
e^{i\tau k^{2}}\partial_{\tau}\left(  e^{-i\tau k^{2}}\mathcal{V}_{\pm}\left(
\tau\right)  f\right)  =\mathcal{V}_{\pm}\left(  \tau\right)  \left(  \frac
{1}{2\tau}f+i\left(  \pm k x+\frac{x^{2}}{4}\right)  f\right)  .
\]
Then, with $\theta_{\pm,1}^{\left(  j\right)  },$ and $\theta_{\pm,2}^{\left(
j\right)  }, j=1,2 $ given respectively by (\ref{6.1}) and (\ref{6.2})
\begin{align}
e^{i\tau k^{2}}\partial_{\tau}\left(  e^{-i\tau k^{2}}\theta_{\pm, 1}^{\left(
1\right)  }\right)   &  =\left(  \mathcal{W}_{+}\left(  \tau\right)
-\mathcal{V}_{+}\left(  \tau\right)  \right)  \left(  \frac{1}{2\tau}%
xf_{1,\pm}+i\left(  kx+\frac{x^{2}}{4}\right)  xf_{1,\pm}+x\partial_{\tau
}f_{1,\pm}\right) \nonumber\\
&  +\sqrt{\frac{\tau}{2\pi i}}\int_{0}^{\infty}e^{i\tau\left(  k+\frac{x}%
{2}\right)  ^{2}}x\left(  \partial_{x}m^{\dagger}\right)  \left(  - k,\tau
x\right)  f_{1,\pm}\left(  x\right)  dx, \label{F18}%
\end{align}
\begin{align}
e^{i\tau k^{2}}\partial_{\tau}\left(  e^{-i\tau k^{2}}\theta_{\pm, 1}^{\left(
2\right)  }\right)   &  =\left(  \mathcal{W}_{-}\left(  \tau\right)
-\mathcal{V}_{-}\left(  \tau\right)  \right)  \left(  \frac{1}{2\tau}%
xf_{1,\pm}+i\left(  - kx+\frac{x^{2}}{4}\right)  xf_{1,\pm}+x\partial_{\tau
}f_{1,\pm}\right) \nonumber\\
&  +\sqrt{\frac{\tau}{2\pi i}}\int_{0}^{\infty}e^{i\tau\left(  k-\frac{x}%
{2}\right)  ^{2}}x \left(  \partial_{x}m^{\dagger}\right)  \left(  k,\tau
x\right)  f_{1,\pm}\left(  x\right)  dx, \label{F18.1}%
\end{align}

\begin{equation}
e^{i\tau k^{2}}\partial_{\tau}\left(  e^{-i\tau k^{2}}\theta_{\pm,2}^{\left(
1\right)  }\right)  =\mathcal{V}_{+}\left(  \tau\right)  \left(  \frac
{1}{2\tau}xf_{2,\pm}+i\left(  kx+\frac{x^{2}}{4}\right)  xf_{2,\pm}%
+x\partial_{\tau}f_{2,\pm}\right)  , \label{F19}%
\end{equation}
and
\begin{equation}
e^{i\tau k^{2}}\partial_{\tau}\left(  e^{-i\tau k^{2}}\theta_{\pm,2}^{\left(
2\right)  }\right)  =\mathcal{V}_{-}\left(  \tau\right)  \left(  \frac
{1}{2\tau}x f_{2,\pm}+i\left(  - kx+\frac{x^{2}}{4}\right)  xf_{2,\pm
}+x\partial_{\tau}f_{2,\pm}\right)  . \label{F19.1}%
\end{equation}
Using \eqref{EST14}, (\ref{3.6}), and (\ref{EST12j}) in (\ref{F18}),
\eqref{F18.1} we show
\begin{equation}
\left\Vert \frac{k}{1+i\tau k^{2}}e^{i\tau k^{2}}\partial_{\tau}\left(
e^{-i\tau k^{2}}\theta_{\pm,1}^{\left(  j\right)  }\right)  \right\Vert
_{L^{2}( {\mathbb{R}})}\leq C|\tau|^{-3-\alpha/2}\left(  \left\Vert w_{\pm
}\right\Vert _{H^{1}( {\mathbb{R}})}^{\alpha}+1\right)  \Lambda_{\pm}\left(
\tau;w_{\pm}\right)  , \label{F20}%
\end{equation}
for $j=1,2$, $\pm\tau\geq a.$ Since by \eqref{sp8.1},
\[
\left\Vert \mathcal{V}_{\pm}\left(  \tau\right)  \phi\right\Vert _{L^{\infty}(
{\mathbb{R}})}\leq C\sqrt{|\tau|}\left\Vert \phi\right\Vert _{L^{1}%
(\mathbb{R}^{+})},
\]
using (\ref{isomafr}) in (\ref{F19}), \eqref{F19.1} we have%
\begin{align*}
&  \left\Vert \frac{k}{1+i\tau k^{2}}e^{i\tau k^{2}}\partial_{\tau}\left(
e^{-i\tau k^{2}}\theta_{\pm,2}^{\left(  j\right)  }\right)  \right\Vert
_{L^{2}( {\mathbb{R}})}\\
&  \leq\frac{C}{|\tau|^{3/2}}\left\Vert xf_{2,\pm}\right\Vert _{L^{2}(
{\mathbb{R}})}+\frac{C}{|\tau|}\left\Vert x^{2}f_{2,\pm}\right\Vert _{L^{2}(
{\mathbb{R}})}+\frac{C}{\sqrt{|\tau|}}\left\Vert x^{3}f_{2,\pm}\right\Vert
_{L^{2}( {\mathbb{R}})}+C\left\Vert x\partial_{t}f_{2,\pm}\right\Vert _{L^{1}(
{\mathbb{R}})}.
\end{align*}
Then, by (\ref{EST16}) with $n=1$ and (\ref{13}) we derive%
\begin{equation}
\left\Vert \frac{k}{1+i\tau k^{2}}e^{i\tau k^{2}}\partial_{\tau}\left(
e^{-i\tau k^{2}}\theta_{\pm,2 }^{\left(  j\right)  }\right)  \right\Vert
_{L^{2}( {\mathbb{R}})}\leq C|\tau|^{-3-\alpha/2}\left(  \left\Vert w_{\pm
}\right\Vert _{H^{1}( {\mathbb{R}})}^{\alpha}+1\right)  \Lambda_{\pm}\left(
\tau;w_{\pm}\right)  , j=1,2, \pm\tau\geq a . \label{F21}%
\end{equation}
Hence, using (\ref{F20}) and (\ref{F21}) in (\ref{6}) we obtain%
\begin{equation}
\left\Vert \frac{k}{1+i\tau k^{2}}e^{i\tau k^{2}}\partial_{\tau}\left(
e^{-i\tau k^{2}}\theta_{\pm}^{(j)}\right)  \right\Vert _{L^{2}( {\mathbb{R}}%
)}\leq C |\tau| ^{-3-\alpha/2}\left(  \left\Vert w_{\pm}\right\Vert _{H^{1}(
{\mathbb{R}})}^{\alpha}+1\right)  \Lambda_{\pm}\left(  \tau;w_{\pm}\right)  ,
\label{F23}%
\end{equation}
for $j=1,2, \pm\tau\geq a.$ Using \eqref{th1}, \eqref{th}, (\ref{F22}) and
(\ref{F23}), we prove%
\begin{equation}
\left\Vert \frac{e^{i\tau k^{2}}}{1+i\tau k^{2}}\partial_{\tau}\left(
e^{-i\tau k^{2}}\Theta^{\left(  j\right)  }_{\pm}\left(  \tau,k\right)
\right)  \right\Vert _{L^{2}( {\mathbb{R}})}\leq C |\tau|^{-1-\alpha/2}\left(
\left\Vert w_{\pm}\right\Vert _{H^{1}( {\mathbb{R}})}^{\alpha}+1\right)
\Lambda_{\pm}\left(  \tau;w_{\pm}\right)  , \label{F28}%
\end{equation}
for $\pm\tau\geq a, $ for $j=1,2,3$ and
\begin{equation}
\left\Vert \frac{e^{i\tau k^{2}}}{1+i\tau k^{2}}\partial_{\tau}\left(
e^{-i\tau k^{2}}\Theta^{\left(  4\right)  }\left(  \tau,k\right)  \right)
\right\Vert _{L^{2}( {\mathbb{R}})}\leq C\tau^{-1/2-\alpha/2}\left(
\left\Vert w\right\Vert _{H^{1}( {\mathbb{R}})}^{\alpha}+1\right)
\Lambda\left(  \tau;w\right)  . \label{F29}%
\end{equation}
Therefore, it follows from (\ref{F26}), (\ref{F27}), (\ref{F28}) and
(\ref{F29})
\[
\sum_{i=1}^{3}\left\Vert \Theta_{i}^{\left(  j\right)  }\right\Vert _{L^{2}(
{\mathbb{R}})}\leq Ct^{-\frac{\alpha}{2}+1}\left(  \left\Vert w\right\Vert
_{H^{1}( {\mathbb{R}})}^{\alpha}+1\right)  \Lambda\left(  t;w\right)  ,
\]
for $j=1,23,$ and%
\[
\sum_{i=1}^{3}\left\Vert \Theta_{i}^{\left(  4 \right)  }\right\Vert _{L^{2}(
{\mathbb{R}})}\leq Ct^{-\frac{\alpha}{2}+3/2}\left(  \left\Vert w\right\Vert
_{H^{1}( {\mathbb{R}})}^{\alpha}+1\right)  \Lambda\left(  t;w\right)  .
\]
Hence, from (\ref{22}) we attain (\ref{Lema2F1}) and (\ref{Lema2F2}).
\end{proof}

\begin{proof}
[Proof of Lemma \ref{Lth}]Using (\ref{Lema1F1}) and (\ref{Lema2F1}) in
(\ref{114}) we get (\ref{Lth1}). Moreover, by (\ref{Lema1F2}), (\ref{Lema2F1})
and (\ref{Lema2F2}) in (\ref{115}) we get (\ref{Lth2}).
\end{proof}

\subsection{Estimates for the derivatives of $\protect\widehat{\mathcal{V}%
}\left(  \tau\right)  f_{1,\pm}^{\operatorname*{fr}}.$ Proof of Lemmas
\ref{Lema3} and \ref{Lema4} \label{ProofLema34}}

Let $\chi\in C^{\infty}_{0}\left(  \mathbb{R}\right)  $ be such that $\chi
\geq0,$%
\begin{equation}
\left\{
\begin{array}
[c]{c}%
\chi\left(  x\right)  =1,\text{ for\ \ }\left\vert x\right\vert \leq1,\\
\chi\left(  x\right)  =0,\text{ for\ }\left\vert x\right\vert \geq2.
\end{array}
\right.  \label{chi}%
\end{equation}

In order to estimate the derivatives $\partial_{k}^{j}\widehat{\mathcal{V}%
}\left(  \tau\right)  f_{1,\pm}^{\operatorname*{fr}},$ $j=1,2,$ we prepare the
following lemma.

\begin{lemma}
\label{Lema5}\bigskip Suppose that $V\in L^{1,3+\tilde{\delta}}(\mathbb{R}%
^{+}),$ for some $\tilde{\delta}>1/2$ and that $H_{A,B,V}$ does not have
negative eigenvalues. In addition, suppose that $V$ admits a regular
decomposition (see Definition \ref{Def1}). Let the nonlinear function
$\mathcal{N}$ satisfy (\ref{ConNonlinearity}). Let $\psi\in C^{2}\left(
\mathbb{R}\right)  $ be such that%
\begin{equation}%
\begin{array}
[c]{l}%
\left\vert \psi\left(  k\right)  \right\vert \leq Ck\left\langle
k\right\rangle ^{-1}\text{, \ }\left\vert \psi^{\prime}\left(  k\right)
\right\vert \leq C\left\langle k\right\rangle ^{-1},\\
\left\vert \psi^{\prime\prime}\left(  k\right)  \right\vert \leq C.\label{128}%
\end{array}
\end{equation}
Then, if $w_{\pm}$ satisfies (\ref{eqw}) for $t\in\mathcal{T}_{\pm}\left(
a,\infty\right)  ,$ there is a constant $C>0$, such that uniformly for $a\geq
a_{0}>0,$ we have,
\begin{equation}
\left\Vert \partial_{k}\int_{\mathcal{T}_{\pm}(a,t)}\psi\left(  k\right)
\mathcal{V}_{+}\left(  \tau\right)  \chi f_{1,\pm}^{\operatorname*{fr}}%
d\tau\right\Vert _{L^{2}({\mathbb{R}})}\leq C\sup_{\tau\in\mathcal{T}_{\pm
}(a,t)}\left(  \left\Vert w_{\pm}\right\Vert _{H^{1}({\mathbb{R}})}^{\alpha
}+1\right)  \Lambda_{\pm}\left(  \tau;w_{\pm}\right)  ,\pm t\geq a,
\label{L5-1}%
\end{equation}
and,
\begin{equation}
\left\Vert \partial_{k}^{2}\int_{\mathcal{T}_{\pm}(a,t)}\psi\left(  k\right)
\mathcal{V}_{+}\left(  \tau\right)  \chi f_{1,\pm}^{\operatorname*{fr}}%
d\tau\right\Vert _{L^{2}({\mathbb{R}})}\leq C\sqrt{|t|}\ \sup_{\tau
\in\mathcal{T}_{\pm}(a,t)}\left(  \left\Vert w_{\pm}\right\Vert _{H^{1}%
({\mathbb{R}})}^{\alpha}+1\right)  \Lambda_{\pm}\left(  \tau;w_{\pm}\right)  ,
\label{L5-2}%
\end{equation}
for $\pm t\geq a.$
\end{lemma}

\begin{proof}
Using (\ref{2.20}) and integrating by parts we have%
\begin{equation}
\int_{\mathcal{T}_{\pm}(a,t)}\psi\left(  k\right)  \mathcal{V}_{ + }\left(
\tau\right)  \chi f_{1,\pm}^{\operatorname*{fr}}d\tau=A^{\left(  1\right)
}+A^{\left(  2\right)  }+\tilde{A}, \label{51b}%
\end{equation}
where%
\[
A^{\left(  1\right)  }= \pm\left.  \tau\frac{\psi\left(  k\right)  }{1+i\tau
k^{2}}\mathcal{V}_{ + }\left(  \tau\right)  \chi f_{1,\pm}^{\operatorname*{fr}%
}\right\vert _{\pm a}^{t},
\]
%
\[
A^{\left(  2\right)  }=\int_{\mathcal{T}_{\pm}(a,t)} \frac{i\tau k^{2}%
\psi\left(  k\right)  }{\left(  1+i\tau k^{2}\right)  ^{2}}\mathcal{V}_{ +
}\left(  \tau\right)  \chi f_{1,\pm}^{\operatorname*{fr}}d\tau,
\]
and%
\[
\tilde{A}=-\int_{\mathcal{T}_{\pm}(a,t)}\tau\frac{\psi\left(  k\right)
}{1+i\tau k^{2}}e^{i\tau k^{2}}\partial_{\tau}\left(  e^{-i\tau k^{2}%
}\mathcal{V}_{ + }\left(  \tau\right)  \chi f_{1,\pm}^{\operatorname*{fr}%
}\right)  d\tau.
\]
Using (\ref{14}) and integrating by parts we get%
\begin{equation}
i\tau\mathcal{V}_{ + }\left(  \tau\right)  \left(  k+\frac{x}{2}\right)
\phi=-\sqrt{\frac{\tau}{2\pi i}}\phi\left(  0\right)  -\mathcal{V}_{ +
}\left(  \tau\right)  \left(  \partial_{x}\phi\right)  . \label{80}%
\end{equation}
Then, using that
\[
\partial_{\tau}\left(  e^{-i\tau k^{2}}e^{i\tau\left(  k+\frac{x}{2}\right)
^{2}}\right)  =i\left(  \frac{1}{2}x\left(  k+\frac{x}{2}\right)  +\frac{1}%
{2}kx\right)  e^{-i\tau k^{2}}e^{i\tau\left(  k+\frac{x}{2}\right)  ^{2}}%
\]
we obtain,%
\begin{align*}
e^{i\tau k^{2}}\partial_{\tau}\left(  e^{-i\tau k^{2}}\mathcal{V}^{\left(
+\right)  }\left(  \tau\right)  \chi f_{1}^{\operatorname*{fr}}\right)   &
=\frac{i}{2}\mathcal{V}^{\left(  +\right)  }\left(  \tau\right)  \left(
k+\frac{x}{2}\right)  x\chi f_{1}^{\operatorname*{fr}}+\frac{i}{2}%
k\mathcal{V}^{\left(  +\right)  }\left(  \tau\right)  x\chi f_{1}%
^{\operatorname*{fr}}\\
&  +\mathcal{V}^{\left(  +\right)  }\left(  \tau\right)  \chi\partial_{\tau
}f_{1}^{\operatorname*{fr}}+\frac{1}{2\tau}\mathcal{V}^{\left(  +\right)
}\left(  \tau\right)  \chi f_{1}^{\operatorname*{fr}}\\
&  =-\frac{1}{2\tau}\mathcal{V}^{\left(  +\right)  }\left(  \tau\right)
\partial_{x}\left(  x\chi f_{1}^{\operatorname*{fr}}\right)  +\frac{i}%
{2}k\mathcal{V}^{\left(  +\right)  }\left(  \tau\right)  \left(  x\chi
f_{1}^{\operatorname*{fr}}\right) \\
&  +\mathcal{V}^{\left(  +\right)  }\left(  \tau\right)  \chi\partial_{\tau
}f_{1}^{\operatorname*{fr}}+\frac{1}{2\tau}\mathcal{V}^{\left(  +\right)
}\left(  \tau\right)  \chi f_{1}^{\operatorname*{fr}}.
\end{align*}
Thus, we decompose $\tilde{A}$ as follows%
\[
\tilde{A}=A^{\left(  3\right)  }+A^{\left(  4\right)  }+A^{\left(  5\right)
}+ A^{\left(  6\right)  },
\]
where
\[
A^{\left(  3\right)  }=\frac{1}{2}\int_{\mathcal{T}_{\pm}(a,t)}\frac
{\psi\left(  k\right)  }{1+i\tau k^{2}}\mathcal{V}_{ + }\left(  \tau\right)
\partial_{x}\left(  x\chi f_{1,\pm}^{\operatorname*{fr}}\right)  d\tau,
\]%
\[
A^{\left(  4\right)  }=-\frac{i}{2}\int_{\mathcal{T}_{\pm}(a,t)}\tau
\frac{k\psi\left(  k\right)  }{1+i\tau k^{2}}\mathcal{V}_{ + }\left(
\tau\right)  x\chi f_{1,\pm}^{\operatorname*{fr}}d\tau,
\]%
\[
A^{\left(  5\right)  }= -\int_{\mathcal{T}_{\pm}(a,t)} \frac{ \psi\left(
k\right)  }{1+i\tau k^{2}}\mathcal{V}_{ + } \left(  \tau\right)  \left(
\chi\partial_{\tau}f_{1,\pm}^{\operatorname*{fr}}\right)  d\tau,
\]
and%
\[
A^{\left(  6\right)  }=-\frac{1}{2}\int_{\mathcal{T}_{\pm}(a,t)}\frac
{\psi\left(  k\right)  }{1+i\tau k^{2}}e^{i\tau k^{2}}\mathcal{V}_{ + }\left(
\tau\right)  \chi f_{1,\pm}^{\operatorname*{fr}}d\tau.
\]
Observe that%
\begin{equation}
\sqrt{|\tau|}\left\Vert \left\langle k\right\rangle \frac{k^{m}\psi\left(
k\right)  }{1+i\tau k^{2}}\right\Vert _{L^{\infty}( {\mathbb{R}})}+\left\Vert
\left\langle k\right\rangle \partial_{k}\frac{k^{m}\psi\left(  k\right)
}{1+i\tau k^{2}}\right\Vert _{L^{\infty}( {\mathbb{R}})}+\sqrt{|\tau
|}\left\Vert k\partial_{k}\frac{k^{m}\psi\left(  k\right)  }{1+i\tau k^{2}%
}\right\Vert _{L^{\infty}( {\mathbb{R}})}\leq C|\tau|^{-\frac{m}{2}},
\label{170}%
\end{equation}
for $m=0,1.$ Then, by (\ref{12.1}), (\ref{derivativeffr}), (\ref{130}), and
(\ref{EST8.0}) we show that%
\[
\left\Vert \partial_{k}A^{\left(  j\right)  } \right\Vert _{L^{2}(
{\mathbb{R}})}\leq C \sup_{\tau\in\mathcal{T}_{\pm}(a,t)} \left\Vert w_{\pm
}\right\Vert _{H^{1}( {\mathbb{R}})}^{\alpha+1}, \qquad\pm t \geq a,
\]
for $j=1,2,3,6.$ Next, using (\ref{130}), and (\ref{80}) we decompose%
\[
\partial_{k}A^{\left(  4\right)  }=A^{\left(  41\right)  }+A^{\left(
42\right)  }+A^{\left(  43\right)  }+A^{\left(  44\right)  },
\]
where%
\[
A^{\left(  41\right)  }=\sqrt{\frac{1}{2\pi i}}\int_{\mathcal{T}_{\pm}%
(a,t)}\sqrt{\tau}\left(  \partial_{k}\frac{k\psi\left(  k\right)  }{1+i\tau
k^{2}}\right)  \left(  f_{1,\pm}^{\operatorname*{fr}}\right)  \left(
0\right)  d\tau,
\]%
\[
A^{\left(  42\right)  }=\int_{\mathcal{T}_{\pm}(a,t)}\left(  \partial_{k}%
\frac{k\psi\left(  k\right)  }{1+i\tau k^{2}}\right)  \mathcal{V}_{ + }\left(
\tau\right)  \left(  \partial_{x}\left(  \chi f_{1,\pm}^{\operatorname*{fr}%
}\right)  \right)  d\tau,
\]%
\[
A^{\left(  43\right)  }=i\int_{\mathcal{T}_{\pm}(a,t)}\tau\left(
k\partial_{k}\frac{k\psi\left(  k\right)  }{1+i\tau k^{2}}\right)
\mathcal{V}_{ + }\left(  \tau\right)  \chi f_{1,\pm}^{\operatorname*{fr}}%
d\tau,
\]
and%
\[
A^{\left(  44\right)  }=i\int_{\mathcal{T}_{\pm}(a,t)}\tau\frac{k\psi\left(
k\right)  }{1+i\tau k^{2}}\mathcal{V}_{ + }\left(  \tau\right)  \left(
\partial_{x}\left(  x\chi f_{1,\pm}^{\operatorname*{fr}}\right)  \right)
d\tau.
\]
Thus, by (\ref{12.1}), \eqref{12.1.00}, (\ref{derivativeffr}), (\ref{170}),
and \eqref{EST8.0} we prove%
\[
\left\Vert \partial_{k}A^{\left(  4\right)  }\right\Vert _{L^{2}( {\mathbb{R}%
})}\leq C \sup_{\tau\in\mathcal{T}_{\pm}(a,t)} \left\Vert w_{\pm}\right\Vert
_{H^{1}( {\mathbb{R}})}^{\alpha+1}, \pm t \geq a.
\]
Next, using (\ref{303.3}), as well as (\ref{130}) we decompose%
\begin{align*}
\partial_{k} A^{\left(  5\right)  }  &  =-\int_{\mathcal{T}_{\pm}(a,t)}%
\tau\left(  \partial_{k}\frac{\psi\left(  k\right)  }{1+i\tau k^{2}}\right)
\mathcal{V}_{ + }\left(  \tau\right)  \left(  \chi\partial_{\tau}f_{1,\pm
}^{\operatorname*{fr}}\right)  d\tau\\
&  -2i\int_{\mathcal{T}_{\pm}(a,t)}\tau\frac{\psi\left(  k\right)  }{1+i\tau
k^{2}}\tau\mathcal{V}_{ + }\left(  \tau\right)  \left(  k+\frac{x}{2}\right)
\left(  \chi\mathcal{A}_{1,\pm}^{\left(  1\right)  }\right)  d\tau\\
&  +\sqrt{\frac{2}{\pi i}}\int_{\mathcal{T}_{\pm}(a,t)}\tau^{3/2}\frac
{\psi\left(  k\right)  e^{i\tau k^{2}}}{1+i\tau k^{2}}\mathcal{A}_{2,\pm
}^{\left(  1\right)  }\left(  0\right)  d\tau\\
&  +2\int_{\mathcal{T}_{\pm}(a,t)}\tau\frac{\psi\left(  k\right)  }{1+i\tau
k^{2}}\mathcal{V}_{+ }\left(  \tau\right)  \partial_{x}\left(  \chi
\mathcal{A}_{2,\pm}^{\left(  1\right)  }\right)  d\tau.
\end{align*}
Hence, by \eqref{303.4}, \eqref{303.5}, (\ref{170}) and (\ref{EST8.0}), we get%
\[
\left\Vert \partial_{k}A^{\left(  5\right)  }\right\Vert _{L^{2}( {\mathbb{R}%
})}\leq C \sup_{\tau\in\mathcal{T}_{\pm}(a,t)} \left(  \left(  1+\left\Vert
w_{\pm}\right\Vert _{H^{1}() {\mathbb{R}}}^{\alpha}\right)  \Lambda_{\pm
}\left(  \tau;w_{\pm}\right)  \right)  , \qquad\pm t \geq a .
\]
Therefore, using these estimates for $A^{\left(  j\right)  },$ $j=1,...,6,$ in
(\ref{51b}) we obtain (\ref{L5-1}).

Now, we estimate (\ref{L5-2}). Using (\ref{130}) we have
\begin{align}
\partial_{k}^{2}\left(  \mathcal{V}_{\left(  +\right)  }\left(  \tau\right)
\varphi\right)   &  =2i\tau\partial_{k}\left(  \mathcal{V}_{\left(  +\right)
}\left(  \tau\right)  \left(  k+\frac{x}{2}\right)  \varphi\right) \nonumber\\
&  =-4i\tau\sqrt{\frac{\tau}{2\pi i}}ke^{itk^{2}}\varphi\left(  0\right)
+2i\tau\mathcal{V}_{\left(  +\right)  }\left(  \tau\right)  \varphi\nonumber\\
&  -4i\tau\mathcal{V}_{\left(  +\right)  }\left(  \tau\right)  \left(
\partial_{x}\left(  \left(  k+\frac{x}{2}\right)  \varphi\right)  \right)  .
\label{54}%
\end{align}
Then, it follows from \eqref{12.1}, (\ref{12.1.00}), (\ref{derivativeffr}),
(\ref{130}), \eqref{80}, \eqref{170}, (\ref{54}), and (\ref{EST8.0}) that%
\[
\left\Vert \partial_{k}^{2}\left(  \frac{\psi\left(  k\right)  }{1+i\tau
k^{2}}\left(  \frac{i\tau k^{2}}{1+i\tau k^{2}}\right)  ^{m}\mathcal{V}%
_{\left(  +\right)  }\left(  \tau\right)  \chi f_{1,\pm}^{\operatorname*{fr}%
}\right)  \right\Vert _{L^{2}( {\mathbb{R}})}\leq C\sqrt{|\tau|}\left\Vert
w_{\pm}\right\Vert _{H^{1}( {\mathbb{R}})}^{\alpha+1},\qquad\pm\tau\geq a,
\]
for $m=0,1.$ Hence, we show that%
\begin{equation}
\left\Vert \partial_{k}^{2}A^{\left(  j\right)  }\right\Vert _{L^{2}(
{\mathbb{R}})}\leq C\sqrt{|t|}\sup_{\tau\in\mathcal{T}_{\pm}(a,t)}\left\Vert
w_{\pm}\right\Vert _{H^{1}( {\mathbb{R}})}^{\alpha+1}, \qquad\pm\tau\geq a,
\label{65b}%
\end{equation}
for $j=1,2.$ We decompose%
\begin{equation}
\partial_{k}^{2}\tilde{A}=\tilde{A}^{\left(  1\right)  }+\tilde{A}^{\left(
2\right)  }+\tilde{A}^{\left(  3\right)  }, \label{63b}%
\end{equation}
where%
\[
\tilde{A}^{\left(  1\right)  }=-\int_{\mathcal{T}_{\pm}(a,t)}\tau\left[
\partial_{k}^{2}\left(  \frac{\psi\left(  k\right)  }{1+i\tau k^{2}}\right)
\right]  e^{i\tau k^{2}}\partial_{\tau}\left(  e^{-i\tau k^{2}}\mathcal{V}_{ +
}\left(  \tau\right)  \chi f_{1,\pm}^{\operatorname*{fr}}\right)  d\tau,
\]%
\[
\tilde{A}^{\left(  2\right)  }=-2\int_{\mathcal{T}_{\pm}(a,t)}\tau\left[
\partial_{k}\left(  \frac{\psi\left(  k\right)  }{1+i\tau k^{2}}\right)
\right]  \partial_{k}\left(  e^{i\tau k^{2}}\partial_{\tau}\left(  e^{-i\tau
k^{2}}\mathcal{V}_{ + }\left(  \tau\right)  \chi f_{1,\pm}^{\operatorname*{fr}%
}\right)  \right)  d\tau,
\]
and
\[
\tilde{A}^{\left(  3\right)  }=-\int_{\mathcal{T}_{\pm}(a,t)}\tau\frac
{\psi\left(  k\right)  }{1+i\tau k^{2}}\partial_{k}^{2}\left(  e^{i\tau k^{2}%
}\partial_{\tau}\left(  e^{-i\tau k^{2}}\mathcal{V}_{ + }\left(  \tau\right)
\chi f_{1,\pm}^{\operatorname*{fr}}\right)  \right)  d\tau.
\]
First, we estimate $\tilde{A}^{\left(  1\right)  }.$ Integrating by parts in
$\tau$ we have,
\begin{align*}
\tilde{A}^{\left(  1\right)  }  &  = \mp\left.  \tau\left[  \partial_{k}%
^{2}\left(  \frac{\psi\left(  k\right)  }{1+i\tau k^{2}}\right)  \right]
\mathcal{V}_{ + }\left(  \tau\right)  \chi f_{1,\pm}^{\operatorname*{fr}%
}\right\vert _{\pm a}^{t}\\
&  +\int_{\mathcal{T}_{\pm}(a,t)}\tau ik^{2}\left[  \partial_{k}^{2}\left(
\frac{\psi\left(  k\right)  }{1+i\tau k^{2}}\right)  \right]  \mathcal{V}_{ +
}\left(  \tau\right)  \chi f_{1,\pm}^{\operatorname*{fr}}d\tau\\
&  + \int_{\mathcal{T}_{\pm}(a,t)} \left[  \partial_{k}^{2}\left(  \frac
{\psi\left(  k\right)  }{1+i\tau k^{2}}\right)  \right]  \mathcal{V}_{ +
}\left(  \tau\right)  \chi f_{1,\pm}^{\operatorname*{fr}}d\tau\\
&  -\int_{\mathcal{T}_{\pm}(a,t)} \left[  \partial_{k}^{2}\left(  \psi\left(
k\right)  \frac{i\tau k^{2}}{\left(  1+i\tau k^{2}\right)  ^{2}}\right)
\right]  \mathcal{V}_{ + }\left(  \tau\right)  \chi f_{1,\pm}%
^{\operatorname*{fr}}d\tau.
\end{align*}
Observe that%
\[
\left\Vert \left(  \tau k^{2}\right)  ^{l}\partial_{k}^{2}\left(  \frac
{\psi\left(  k\right)  }{1+i\tau k^{2}}\left(  \frac{\tau k^{2}}{1+i\tau
k^{2}}\right)  ^{m}\right)  \right\Vert _{L^{\infty}( {\mathbb{R}})}\leq
C\sqrt{|\tau|},
\]
for $m,l=0,1.$Then, by \eqref{12.1}, (\ref{12.1.00}), and (\ref{EST8.0}) we
show%
\begin{equation}
\left\Vert \tilde{A}^{\left(  1\right)  }\right\Vert _{L^{2}( {\mathbb{R}}%
)}\leq C\sqrt{|t|}\sup_{\tau\in\mathcal{T}_{\pm}(a,t)} \left\Vert w_{\pm
}\right\Vert _{H^{1}( {\mathbb{R}})}^{\alpha+1}, \pm t \geq a. \label{62b}%
\end{equation}
Next $,$ using that%
\[
e^{i\tau k^{2}}\partial_{\tau}\left(  e^{-i\tau k^{2}}\mathcal{V}_{ + }\left(
\tau\right)  \chi f_{1,\pm}^{\operatorname*{fr}}\right)  =\partial_{\tau
}\mathcal{V}_{ + }\left(  \tau\right)  \chi f_{1,\pm}^{\operatorname*{fr}%
}-ik^{2}\mathcal{V}_{ + }\left(  \tau\right)  \chi f_{1,\pm}%
^{\operatorname*{fr}}%
\]
we write $\tilde{A}^{\left(  2\right)  }$ as%
\[
\tilde{A}^{\left(  2\right)  }=\tilde{A}^{\left(  21\right)  }+\tilde
{A}^{\left(  22\right)  }+\tilde{A}^{\left(  23\right)  },
\]
with%
\[
\tilde{A}^{\left(  21\right)  }=-2\int_{\mathcal{T}_{\pm}(a,t)}\tau\left[
\partial_{k}\left(  \frac{\psi\left(  k\right)  }{1+i\tau k^{2}}\right)
\right]  \partial_{\tau}\left(  \partial_{k}\mathcal{V}^{\left(  +\right)
}\left(  \tau\right)  \chi f_{1,\pm}^{\operatorname*{fr}}\right)  d\tau,
\]%
\[
\tilde{A}^{\left(  22\right)  }=4i\int_{\mathcal{T}_{\pm}(a,t)}\tau k\left[
\partial_{k}\left(  \frac{\psi\left(  k\right)  }{1+i\tau k^{2}}\right)
\right]  \mathcal{V}_{ + }\left(  \tau\right)  \chi f_{1,\pm}%
^{\operatorname*{fr}}d\tau
\]
and%
\[
\tilde{A}^{\left(  23\right)  }=2i\int_{\mathcal{T}_{\pm}(a,t)}\tau k^{2}
\left[  \partial_{k}\left(  \frac{\psi\left(  k\right)  }{1+i\tau k^{2}%
}\right)  \right]  \partial_{k}\mathcal{V}_{ + }\left(  \tau\right)  \chi
f_{1,\pm}^{\operatorname*{fr}}d\tau.
\]
Integrating by parts in $\tau$ we have
\begin{align*}
\tilde{A}^{\left(  21\right)  }  &  =\mp2\left.  \tau\left[  \partial
_{k}\left(  \frac{\psi\left(  k\right)  }{1+i\tau k^{2}}\right)  \right]
\partial_{k}\mathcal{V}_{ + }\left(  \tau\right)  \chi f_{1,\pm}%
^{\operatorname*{fr}}\right\vert _{\pm a}^{t}\\
&  + 2 \int_{\mathcal{T}_{\pm}(a,t)} \left[  \partial_{k}\left(  \frac
{\psi\left(  k\right)  }{1+i\tau k^{2}}\right)  \right]  \left(  \partial
_{k}\mathcal{V}_{ + }\left(  \tau\right)  \chi f_{1,\pm}^{\operatorname*{fr}%
}\right)  d\tau\\
&  -2\int_{\mathcal{T}_{\pm}(a,t)} \left[  \partial_{k}\left(  \psi\left(
k\right)  \frac{i\tau k^{2}}{\left(  1+i\tau k^{2}\right)  ^{2}}\right)
\right]  \left(  \partial_{k}\mathcal{V}_{ + }\left(  \tau\right)  \chi
f_{1,\pm}^{\operatorname*{fr}}\right)  d\tau.
\end{align*}
Observe that%
\[
\left\Vert \left\langle k\right\rangle k^{l}\partial_{k}\left(  \frac
{\psi\left(  k\right)  }{1+i\tau k^{2}}\left(  \frac{\tau k^{2}}{1+i\tau
k^{2}}\right)  ^{m}\right)  \right\Vert _{L^{\infty}( {\mathbb{R}})}\leq
C|\tau|^{-\frac{l}{2}},
\]
for $m=0,1$ and $l=0,1,2.$ Then, using (\ref{12.1}), (\ref{130}),
\eqref{derivativeffr}, and (\ref{EST8.0}) we estimate $\tilde{A}^{\left(
2\right)  }$ as%
\begin{equation}
\left\Vert \tilde{A}^{\left(  2\right)  }\right\Vert _{L^{2}( {\mathbb{R}}%
)}\leq\left\Vert \tilde{A}^{\left(  21\right)  }\right\Vert _{L^{2}(
{\mathbb{R}})}+\left\Vert \tilde{A}^{\left(  22\right)  }\right\Vert _{L^{2}(
{\mathbb{R}})}+\left\Vert \tilde{A}^{\left(  23\right)  }\right\Vert _{L^{2}(
{\mathbb{R}})}\leq C\sqrt{|t|}\sup_{\tau\in\mathcal{T}_{\pm}(a,t)}\left\Vert
w_{\pm}\right\Vert _{H^{1}( {\mathbb{R}})}^{\alpha+1}, \pm t \geq a.
\label{61b}%
\end{equation}
We now estimate $\tilde{A}^{\left(  3\right)  }.$ Let us define,%
\begin{align*}
b_{1}  &  =-4i\mathcal{V}_{\ + }\left(  \tau\right)  \partial_{x}\left(  x\chi
f_{1,\pm}^{\operatorname*{fr}}\right)  +4\sqrt{\frac{\tau}{2\pi i}}ike^{i\tau
k^{2}}f_{1,\pm}^{\operatorname*{fr}}\left(  0\right)  ,\\
&  +8ik\mathcal{V}_{ + }\left(  \tau\right)  \partial_{x}\left(  \chi
f_{1,\pm}^{\operatorname*{fr}}\right)  +4ik\mathcal{V}_{ + }\left(
\tau\right)  \left(  x\partial_{x}^{2}\left(  \chi f_{1,\pm}%
^{\operatorname*{fr}}\right)  \right)  ,
\end{align*}%
\begin{align*}
b_{2}  &  =i\mathcal{V}_{ +\ }\left(  \tau\right)  \partial_{x}\left(  \left(
\partial_{x}\left(  x^{2}\chi\right)  \right)  f_{1,\pm}^{\operatorname*{fr}%
}\right)  -2ik\mathcal{V}_{ + }\left(  \tau\right)  \partial_{x}\left(
x\chi\left(  \partial_{x}f_{1,\pm}^{\operatorname*{fr,\pm}}\right)  \right) \\
&  +i\mathcal{V}_{+ }\left(  \tau\right)  \left(  x\chi\partial_{x}f_{1,\pm
}^{\operatorname*{fr}}\right)  ,
\end{align*}%
\[
b_{3}=-2i\tau\left(  \partial_{\tau}\mathcal{V}_{ + }\left(  \tau\right)
\right)  \left(  x\chi\partial_{x}f_{1,\pm}^{\operatorname*{fr}}\right)  ,
\]
and%
\[
b_{4}=-4\tau^{2}k\mathcal{V}_{ + }\left(  \tau\right)  \left(  k+\frac{x}%
{2}\right)  \left(  \chi\partial_{\tau}f_{1,\pm}^{\operatorname*{fr}}\right)
-2i\tau\mathcal{V}_{ + }\left(  \tau\right)  \left(  x\left(  \partial_{x}%
\chi\right)  \partial_{\tau}f_{1,\pm}^{\operatorname*{fr}}\right)
-2i\tau\mathcal{V}_{ + }\left(  \tau\right)  \left(  x\chi\partial_{\tau
}\partial_{x}f_{1,\pm}^{\operatorname*{fr}}\right)  .
\]
Using (\ref{14}) in the last term in the right-hand side of (\ref{54}) we have%
\begin{align}
\partial_{k}^{2}\left(  \mathcal{V}_{ + }\left(  \tau\right)  \varphi\right)
&  =-4i\tau\sqrt{\frac{\tau}{2\pi i}}ke^{itk^{2}}\varphi\left(  0\right)
+4\sqrt{\frac{\tau}{2\pi i}}e^{itk^{2}}\left(  \partial_{x}\varphi\right)
\left(  0\right) \nonumber\\
&  +4\mathcal{V}_{ + }\left(  \tau\right)  \partial_{x}^{2}\varphi.
\label{54bis}%
\end{align}
Then, by (\ref{130}) and (\ref{54bis}) we calculate%
\[
\partial_{k}^{2}\left(  ik\mathcal{V}_{ + }\left(  \tau\right)  x\chi
f_{1,\pm}^{\operatorname*{fr}}\right)  =b_{1}.
\]

By (\ref{14}) and (\ref{130}) we have%
\begin{align*}
-2i\tau\left(  \partial_{\tau}\mathcal{V}_{ + }\left(  \tau\right)  \right)
\left(  x\chi\partial_{x}f_{1,\pm}^{\operatorname*{fr}}\right)   &
=-i\mathcal{V}_{ + }\left(  \tau\right)  \left(  x\chi\partial_{x}f_{1,\pm
}^{\operatorname*{fr}}\right)  -i\mathcal{V}_{ + }\left(  \tau\right)
2i\tau\left(  k+\frac{x}{2}\right)  ^{2}\left(  x\chi\partial_{x}f_{1,\pm
}^{\operatorname*{fr}}\right) \\
&  =-i\mathcal{V}_{ + }\left(  \tau\right)  \left(  x\chi\partial_{x}f_{1,\pm
}^{\operatorname*{fr}}\right)  -i\left(  \partial_{k}\mathcal{V}_{ + }\left(
\tau\right)  \right)  \left(  k+\frac{x}{2}\right)  \left(  x\chi\partial
_{x}f_{1,\pm}^{\operatorname*{fr}}\right) \\
&  =2i\mathcal{V}_{ + }\left(  \tau\right)  \left(  \left(  k+\frac{x}%
{2}\right)  \partial_{x}\left(  x\chi\partial_{x}f_{1,\pm}^{\operatorname*{fr}%
}\right)  \right) \\
&  =i\mathcal{V}_{ + }\left(  \tau\right)  \left(  x\partial_{x}\left(
x\chi\partial_{x}f_{1,\pm}^{\operatorname*{fr}}\right)  \right)
+2ik\mathcal{V}^{\left(  +\right)  }\left(  \tau\right)  \left(  \partial
_{x}\left(  x\chi\partial_{x}f_{1,\pm}^{\operatorname*{fr}}\right)  \right) \\
&  =-i\mathcal{V}_{ + }\left(  \tau\right)  \left(  x\chi\partial_{x}f_{1,\pm
}^{\operatorname*{fr}}\right)  +i\mathcal{V}_{ + }\left(  \tau\right)  \left(
\partial_{x}\left(  x^{2}\chi\partial_{x}f_{1,\pm}^{\operatorname*{fr}%
}\right)  \right) \\
&  +2ik\mathcal{V}_{ + }\left(  \tau\right)  \left(  \partial_{x}\left(
x\chi\left(  \partial_{x}f_{1,\pm}^{\operatorname*{fr}}\right)  \right)
\right)  .
\end{align*}
Then%
\begin{align*}
&  i\mathcal{V}_{ + }\left(  \tau\right)  \left(  \partial_{x}\left(
x^{2}\chi\partial_{x}f_{1,\pm}^{\operatorname*{fr}}\right)  \right) \\
&  =b_{3}-2ik\mathcal{V}_{ + }\left(  \tau\right)  \left(  \partial_{x}\left(
x\chi\left(  \partial_{x}f_{1,\pm}^{\operatorname*{fr}}\right)  \right)
\right)  +i\mathcal{V}_{ + }\left(  \tau\right)  \left(  x\chi\partial
_{x}f_{1,\pm}^{\operatorname*{fr}}\right)  .
\end{align*}
Thus, it follows from (\ref{54bis}) that%
\[
\partial_{k}^{2}\left(  i\mathcal{V}_{ + }\left(  \tau\right)  \left(
\frac{x^{2}}{4}\right)  \chi f_{1,\pm}^{\operatorname*{fr}}\right)
=b_{2}+b_{3}.
\]
Moreover, since%
\[
\partial_{k}^{2}\left(  \mathcal{V}_{ + }\left(  \tau\right)  \varphi\right)
=-4\tau^{2}k\mathcal{V}_{ + }\left(  \tau\right)  \left(  k+\frac{x}%
{2}\right)  \varphi-2i\tau\mathcal{V}_{ + }\left(  \tau\right)  \left(
x\partial_{x}\varphi\right)  ,
\]
we show,%
\[
\partial_{k}^{2}\left(  \mathcal{V}_{ + }\left(  \tau\right)  \chi
\partial_{\tau}f_{1,\pm}^{\operatorname*{fr}}\right)  =b_{4}.
\]
Hence
\[
\partial_{k}^{2}\left(  e^{i\tau k^{2}}\partial_{\tau}\left(  e^{-i\tau k^{2}%
}\mathcal{V}_{ + }\left(  \tau\right)  \chi f_{1,\pm}^{\operatorname*{fr}%
}\right)  \right)  =\sum_{j=1}^{4}b_{j}.
\]
Therefore, we decompose $\tilde{A}^{\left(  3\right)  }$ as%
\begin{equation}
\tilde{A}^{\left(  3\right)  }=\tilde{A}^{\left(  31\right)  }+\tilde
{A}^{\left(  32\right)  }, \label{59b}%
\end{equation}
where%
\[
\tilde{A}^{\left(  31\right)  }=-\int_{\mathcal{T}_{\pm}(a,t)}\tau\frac
{\psi\left(  k\right)  }{1+i\tau k^{2}}\left(  b_{1}+b_{2}\right)  d\tau,
\]
and%
\[
\tilde{A}^{\left(  32\right)  }=-\int_{1}^{t}\tau\frac{\psi\left(  k\right)
}{1+i\tau k^{2}}\left(  b_{3}+b_{4}\right)  d\tau.
\]
Using (\ref{12.1}), (\ref{12.1.00}), (\ref{derivativeffr}), (\ref{170}), and
\eqref{EST8.0} we show,%
\begin{equation}
\left\Vert \tilde{A}^{\left(  31\right)  }\right\Vert _{L^{2}( {\mathbb{R}}%
)}\leq C\sqrt{|t|} \sup_{\tau\in\mathcal{T}_{\pm}(a,t)} \|w_{\pm}\|_{H^{1}(
{\mathbb{R}})}^{\alpha+1}. \label{57b}%
\end{equation}
In order to estimate $\tilde{A}^{\left(  32\right)  }$, we observe that%
\begin{align*}
b_{3}+b_{4}  &  =-4\tau^{2}k\mathcal{V}_{ + }\left(  \tau\right)  \left(
k+\frac{x}{2}\right)  \left(  \chi\partial_{\tau}f_{1,\pm}^{\operatorname*{fr}%
}\right) \\
&  -2i\tau\mathcal{V}_{ + }\left(  \tau\right)  \left(  x\left(  \partial
_{x}\chi\right)  \partial_{\tau}f_{1,\pm}^{\operatorname*{fr}}\right)
-2i\tau\partial_{\tau}\left(  \mathcal{V}_{ + }\left(  \tau\right)  \left(
x\chi\partial_{x}f_{1,\pm}^{\operatorname*{fr}}\right)  \right)  .
\end{align*}
Then, using (\ref{303.3}) we write $\tilde{A}^{\left(  32\right)  }$ as
follows%
\begin{equation}
\tilde{A}^{\left(  32\right)  }=\tilde{A}_{1}^{\left(  32\right)  }+\tilde
{A}_{2}^{\left(  32\right)  }+\tilde{A}_{3}^{\left(  32\right)  }+\tilde
{A}_{4}^{\left(  32\right)  }, \label{203}%
\end{equation}
where%
\[
\tilde{A}_{1}^{\left(  32\right)  }=4\int_{\mathcal{T}_{\pm}(a,t)}\tau
^{3}\frac{k\psi\left(  k\right)  }{1+i\tau k^{2}}\mathcal{V}_{ + }\left(
\tau\right)  \left(  k+\frac{x}{2}\right)  \left(  \chi\mathcal{A}%
_{1}^{\left(  1\right)  }\right)  d\tau,
\]%
\[
\tilde{A}_{2}^{\left(  32\right)  }=4\int_{\mathcal{T}_{\pm}(a,t)}\tau
^{3}\frac{k\psi\left(  k\right)  }{1+i\tau k^{2}}\mathcal{V}_{ + }\left(
\tau\right)  \left(  k+\frac{x}{2}\right)  \left(  \chi\mathcal{A}%
_{2}^{\left(  1\right)  }\right)  d\tau,
\]%
\[
\tilde{A}_{3}^{\left(  32\right)  }=2i\int_{\mathcal{T}_{\pm}(a,t)}\tau
^{2}\frac{\psi\left(  k\right)  }{1+i\tau k^{2}}\mathcal{V}_{ + }\left(
\tau\right)  \left(  x\left(  \partial_{x}\chi\right)  \left(  \mathcal{A}%
_{1}^{\left(  1\right)  }+\mathcal{A}_{2}^{\left(  1\right)  }\right)
\right)  d\tau,
\]
and%
\[
\tilde{A}_{4}^{\left(  32\right)  }=2i\int_{\mathcal{T}_{\pm}(a,t)}\tau
^{2}\frac{\psi\left(  k\right)  }{1+i\tau k^{2}}\partial_{\tau}\left(
\mathcal{V}_{ +\ }\left(  \tau\right)  \left(  x\chi\partial_{x}f_{1,\pm
}^{\operatorname*{fr}}\right)  \right)  d\tau.
\]
Then, using(\ref{303.4}), \eqref{303.5}, (\ref{170}), and (\ref{EST8.0}), we
get%
\begin{equation}
\left\Vert \tilde{A}_{1}^{\left(  32\right)  }\right\Vert _{L^{2}(
{\mathbb{R}})}+\left\Vert \tilde{A}_{3}^{\left(  32\right)  }\right\Vert
_{L^{2}( {\mathbb{R}})}\leq C\sqrt{|t|} \sup_{\tau\in\mathcal{T}_{\pm}(a,t)}
\left(  \left\Vert w_{\pm}\right\Vert _{H^{1}( {\mathbb{R}})}^{\alpha
}+1\right)  \Lambda_{\pm}\left(  t;w_{\pm}\right)  , \qquad\pm t \geq a .
\label{200}%
\end{equation}
Using (\ref{80}), it follows from (\ref{303.4}), and (\ref{EST8.0}),
\begin{equation}%
\begin{array}
[c]{l}%
\left\Vert \tilde{A}_{2}^{\left(  32\right)  }\right\Vert _{L^{2}(
{\mathbb{R}})}\leq C \sup_{\tau\in\mathcal{T}_{\pm}(a,t)} \left(  |\tau
|^{3/2}\left\Vert \mathcal{A}_{2}^{\left(  1\right)  }\right\Vert _{L^{\infty
}( {\mathbb{R}})}+|\tau|\left\Vert \mathcal{A}_{2}^{\left(  1\right)
}\right\Vert _{H^{1}( {\mathbb{R}})( {\mathbb{R}})}\right)  \leq\\
\\
C\sqrt{|t|} \sup_{\tau\in\mathcal{T}_{\pm}(a,t)} \left(  \left\Vert w_{\pm
}\right\Vert _{H^{1}( {\mathbb{R}})}^{\alpha}+1\right)  \Lambda_{\pm}\left(
\tau;w\right)  , \pm t \geq a . \label{201}%
\end{array}
\end{equation}
Integrating by parts we get%
\begin{align*}
\tilde{A}_{4}^{\left(  32\right)  }  &  =\pm2i\left.  \tau^{2}\frac
{\psi\left(  k\right)  }{1+i\tau k^{2}}\mathcal{V}_{ + }\left(  \tau\right)
x\chi\partial_{x}f_{1,\pm}^{\operatorname*{fr}}\right\vert _{\pm a}^{t}\\
&  -2i\int_{\mathcal{T}_{\pm}(a,t)}\psi\left(  k\right)  \left[
\partial_{\tau}\left(  \frac{\tau^{2}}{1+i\tau k^{2}}\right)  \right]  \left(
\mathcal{V}_{ + }\left(  \tau\right)  x\chi\partial_{x}f_{1,\pm}%
^{\operatorname*{fr}}\right)  d\tau.
\end{align*}
Then, using (\ref{derivativeffr}), (\ref{170}), and (\ref{EST8.0}), we show%
\begin{equation}
\left\Vert \tilde{A}_{4}^{\left(  32\right)  }\right\Vert _{L^{2}(
{\mathbb{R}})}\leq C\sqrt{|t|} \sup_{\tau\in\mathcal{T}_{\pm}(a,t)} \|w_{\pm
}\|_{H^{1}( {\mathbb{R}})}^{\alpha+1}, \qquad\pm t \geq a. \label{202}%
\end{equation}
Hence, by (\ref{203}), (\ref{200}), (\ref{201}) and (\ref{202}) we get%
\begin{equation}
\left\Vert \tilde{A}^{\left(  32\right)  }\right\Vert _{L^{2}( {\mathbb{R}}%
)}\leq C \sqrt{|t|} \sup_{\tau\in\mathcal{T}_{\pm}(a,t)} \left(  \left\Vert
w_{\pm}\right\Vert _{H^{1}( {\mathbb{R}})}^{\alpha}+1\right)  \Lambda_{\pm
}\left(  \tau;w\right)  , \pm t \geq a . \label{58b}%
\end{equation}
From (\ref{59b}), (\ref{57b}) and (\ref{58b}), we deduce%
\begin{equation}
\left\Vert \tilde{A}^{\left(  3\right)  }\right\Vert _{L^{2}( {\mathbb{R}}%
)}\leq C \sqrt{|t|} \sup_{\tau\in\mathcal{T}_{\pm}(a,t)} \left(  \left\Vert
w_{\pm}\right\Vert _{H^{1}( {\mathbb{R}})}^{\alpha}+1\right)  \Lambda_{\pm
}\left(  \tau;w\right)  , \pm t \geq a . \label{60b}%
\end{equation}
Introducing (\ref{62b}), (\ref{61b}) and (\ref{60b}) into (\ref{63b}) we prove%
\begin{equation}
\left\Vert \partial_{k}^{2}\tilde{A}\right\Vert _{L^{2}( {\mathbb{R}})}\leq C
\sqrt{|t|} \sup_{\tau\in\mathcal{T}_{\pm}(a,t)} \left(  \left\Vert w_{\pm
}\right\Vert _{H^{1}( {\mathbb{R}})}^{\alpha}+1\right)  \Lambda_{\pm}\left(
\tau;w\right)  , \pm t \geq a . \label{64b}%
\end{equation}
Therefore, using (\ref{65b}) and (\ref{64b}) in (\ref{51b}) we deduce
(\ref{L5-2}).
\end{proof}

\begin{proof}
[Proof of Lemma \ref{Lema3}]We begin by decomposing
\begin{equation}
\widehat{\mathcal{V}}\left(  \tau\right)  f_{1,\pm}^{\operatorname*{fr}}%
=B_{1}+B_{2}, \label{35.0}%
\end{equation}
with
\[
B_{1}=\widehat{\mathcal{V}}\left(  \tau\right)  \left(  \left(  1-\chi\right)
f_{1,\pm}^{\operatorname*{fr}}\right)  ,
\]
and%
\[
B_{2}=\widehat{\mathcal{V}}\left(  \tau\right)  \left(  \chi f_{1,\pm
}^{\operatorname*{fr}}\right)  ,
\]
where we recall that $\chi$ is defined by (\ref{chi}). Using (\ref{14}) and
integrating by parts we have
\begin{align}
\partial_{k}B_{1}  &  =\left(  \partial_{k}S\left(  k\right)  \right)
\mathcal{V}_{ + }\left(  \tau\right)  \left(  1-\chi\right)  f_{1,\pm
}^{\operatorname*{fr}}\nonumber\\
&  -2S\left(  k\right)  \mathcal{V}_{ + }\left(  \tau\right)  \partial
_{x}\left(  \left(  1-\chi\right)  f_{1,\pm}^{\operatorname*{fr}}\right)
+2\mathcal{V}_{ - }\left(  \tau\right)  \partial_{x}\left(  \left(
1-\chi\right)  f_{1,\pm}^{\operatorname*{fr}}\right)  . \label{39}%
\end{align}
Then, using (\ref{12.1}), (\ref{derivativeffr}), (\ref{UnitaritySM}),
(\ref{scatmatrixderiv}), (\ref{EST8.0}), we get%
\begin{equation}
\left\Vert \partial_{k}B_{1}\right\Vert _{L^{2}( {\mathbb{R}})}\leq C
|\tau|^{-\alpha/2} \left\Vert w_{\pm}\right\Vert _{H^{1}( {\mathbb{R}}%
)}^{\alpha+1}, \qquad\pm\tau\geq a. \label{35}%
\end{equation}

Next, we calculate%
\begin{equation}
\partial_{k}B_{2}=\left(  \partial_{k}S\left(  k\right)  \right)
\mathcal{V}_{ + }\left(  \tau\right)  \left(  \chi f_{1,\pm}%
^{\operatorname*{fr}}\right)  +S\left(  k\right)  \partial_{k}\mathcal{V}_{ +
}\left(  \tau\right)  \left(  \chi f_{1,\pm}^{\operatorname*{fr}}\right)
+\partial_{k}\mathcal{V}_{ - }\left(  \tau\right)  \left(  \chi f_{1,\pm
}^{\operatorname*{fr}}\right)  . \label{120}%
\end{equation}
Recall that $P_{\pm}$ are the projectors on the eigenspaces of $S(0)$
corresponding to the eigenvalues $\pm1,$ respectively (see Appendix
\ref{App1}). Using (\ref{130}) we decompose%
\begin{align}
\partial_{k}B_{2}  &  =\partial_{k}\left(  \left(  S\left(  k\right)
-S\left(  0\right)  \right)  \mathcal{V}_{ + }\left(  \tau\right)  \left(
\chi f_{1,\pm}^{\operatorname*{fr}}\right)  \right) \nonumber\\
&  -2P_{+}\left(  \mathcal{V}_{ + }\left(  \tau\right)  \left(  \partial
_{x}\left(  \chi f_{1,\pm}^{\operatorname*{fr}}\right)  \right)
-\mathcal{V}_{ - }\left(  \tau\right)  \left(  \partial_{x}\left(  \chi
f_{1,\pm}^{\operatorname*{fr}}\right)  \right)  \right) \nonumber\\
&  +P_{-}\left(  \partial_{k}\mathcal{V}_{ - }\left(  \tau\right)  \left(
\chi f_{1,\pm}^{\operatorname*{fr}}\right)  -\partial_{k}\mathcal{V}_{ +
}\left(  \tau\right)  \left(  \chi f_{1,\pm}^{\operatorname*{fr}}\right)
\right)  . \label{126}%
\end{align}
We use Lemma \ref{Lema5} with $\psi\left(  k\right)  =S\left(  k\right)
-S\left(  0\right)  $ to estimate the first term on the right-hand side of the
above equality. By (\ref{S-1}), (\ref{scatmatrixderiv}), and
\eqref{scatmatrixsecondderiv} $\psi$ satisfies (\ref{128})$.$ Then, it follows
from (\ref{L5-1}) that%
\begin{equation}
\label{36}%
\begin{array}
[c]{l}%
\left\Vert \partial_{k}\int_{\mathcal{T}_{\pm}(a,t)} \left(  \left(  S\left(
k\right)  -S\left(  0\right)  \right)  \mathcal{V}_{ + }\left(  \tau\right)
\left(  \chi f_{1,\pm}^{\operatorname*{fr}}\right)  \right)  d\tau\right\Vert
_{L^{2}( {\mathbb{R}})}\leq C \sup_{\tau\in\mathcal{T}_{\pm}(a,t)} \left(
\left\Vert w_{\pm}\right\Vert _{H^{1}( {\mathbb{R}})}^{\alpha}+1\right)
\Lambda_{\pm}\left(  \tau;w_{\pm}\right)  , 
\end{array}
\end{equation}
for $\pm t \geq a.$
Moreover, by (\ref{derivativeffr}), and (\ref{EST8.0}),%
\begin{equation}
\left\Vert \mathcal{V}_{ \pm}\left(  \tau\right)  \partial_{x}\left(  \chi
f_{1,\pm}^{\operatorname*{fr}}\right)  \right\Vert _{L^{2}( {\mathbb{R}})}\leq
C\left\Vert \partial_{x}\left(  \chi f_{1,\pm}^{\operatorname*{fr}}\right)
\right\Vert _{L^{2}( {\mathbb{R}})}\leq C |\tau|^{-\alpha/2} \left\Vert
w\right\Vert _{H^{1}( {\mathbb{R}})}^{\alpha+1}, \qquad\pm\tau\geq a.
\label{127}%
\end{equation}
Then,%
\begin{equation}
\label{127.1}\left\Vert 2P_{+}\left(  \mathcal{V}_{ + }\left(  \tau\right)
\left(  \partial_{x}\left(  \chi f_{1,\pm}^{\operatorname*{fr}}\right)
\right)  -\mathcal{V}_{ - }\left(  \tau\right)  \left(  \partial_{x}\left(
\chi f_{1,\pm}^{\operatorname*{fr}}\right)  \right)  \right)  \right\Vert
_{L^{2}( {\mathbb{R}})}\leq C |\tau|^{-\alpha/2} \left\Vert w_{\pm}\right\Vert
_{H^{1}( {\mathbb{R}})}^{\alpha+1}, \qquad\pm\tau\geq a.
\end{equation}
Let us assume that $w_{\pm}(k)= S(k) w_{\pm}(-k), k \in{\mathbb{R}},$ or
equivalently,
\begin{equation}
w_{\pm}\left(  -k\right)  =S\left(  -k\right)  w_{\pm}\left(  k\right)  ,
\qquad k \in{\mathbb{R}}. \label{68}%
\end{equation}
Using (\ref{68}) we have%
\begin{equation}
\left(  \mathcal{V}\left(  \tau\right)  w_{\pm}\right)  \left(  0\right)
=\sqrt{\frac{i\tau}{2\pi}}\int_{0}^{\infty}e^{-i\tau k^{2}}\left(  S\left(
-k\right)  +I\right)  w_{\pm}\left(  k\right)  dk. \label{129}%
\end{equation}
Let $\psi$ satisfy (\ref{128}). Since $e^{-i\tau k^{2}}=\left(  1-2i\tau
k^{2}\right)  ^{-1}\partial_{k}\left(  ke^{-i\tau k^{2}}\right)  $,
integrating by parts we show
\begin{align}
&  \left\vert \sqrt{\frac{i\tau}{2\pi}}\int_{0}^{\infty}e^{-i\tau k^{2}}%
\psi\left(  k\right)  w_{\pm}\left(  k\right)  dk\right\vert \nonumber\\
&  \leq C\sqrt{|\tau|}\int_{0}^{\infty}\left\vert k\partial_{k}\left(
\frac{\psi\left(  k\right)  }{1-2itk^{2}}w_{\pm}\left(  k\right)  \right)
\right\vert dk\nonumber\\
&  \leq C\sqrt{|\tau|}\int_{0}^{\infty}\left\vert \frac{k\partial_{k}%
\psi\left(  k\right)  }{1-2i\tau k^{2}}w_{\pm}\left(  k\right)  \right\vert
dk+C\sqrt{|\tau|}\int_{0}^{\infty}\left\vert \frac{\psi\left(  k\right)
}{1-2i\tau k^{2}}w_{\pm}\left(  k\right)  \right\vert dk\nonumber\\
&  +C\sqrt{|\tau|}\int_{0}^{\infty}\left\vert k\frac{\psi\left(  k\right)
}{1-2i\tau k^{2}}\partial_{k}w_{\pm}\left(  k\right)  \right\vert dk.
\label{70}%
\end{align}
Hence%
\begin{equation}
\left\vert \sqrt{\frac{i\tau}{2\pi}}\int_{0}^{\infty}e^{-i\tau k^{2}}%
\psi\left(  k\right)  w_{\pm}\left(  k\right)  dk\right\vert \leq
C|\tau|^{-\rho/2}\left\Vert w_{\pm}\right\Vert _{H^{1}( {\mathbb{R}})},
\label{71}%
\end{equation}
for any $0\leq\rho<1.$ By (\ref{S-1}), (\ref{scatmatrixderiv}), and
\eqref{scatmatrixsecondderiv}, $\psi=P_{-}\left(  S\left(  -k\right)
+I\right)  $ satisfies (\ref{128}). Then, using (\ref{71}) we show%
\begin{equation}
\left\vert P_{-}\left(  \mathcal{V}\left(  \tau\right)  w_{\pm}\right)
\left(  0\right)  \right\vert \leq C|\tau|^{-\rho/2}\left\Vert w_{\pm
}\right\Vert _{H^{1}( {\mathbb{R}})}. \label{76}%
\end{equation}
Since $\mathcal{N}$ commutes with $P_{-},$ using this estimate together with
(\ref{89.xxxx}), we get%
\begin{equation}
\left\vert P_{-} f_{1,\pm}^{\operatorname*{fr}}\left(  0\right)  \right\vert
\leq C|\tau|^{-\rho/2-\alpha/2} \left\Vert w_{\pm}\right\Vert _{H^{1}(
{\mathbb{R}})}^{\alpha+1}, \qquad\pm\tau\geq a, \label{140}%
\end{equation}
for any $0\leq\rho<1.$ Further, by (\ref{130})%
\begin{align}
&  \partial_{k}\mathcal{V}_{ - }\left(  \tau\right)  \left(  \chi f_{1,\pm
}^{\operatorname*{fr}}\right)  -\partial_{k}\mathcal{V}_{ + }\left(
\tau\right)  \left(  \chi f_{1,\pm}^{\operatorname*{fr}}\right)  =4\sqrt
{\frac{\tau}{2\pi i}}e^{i\tau k^{2}}\left(  \chi f_{1,\pm}^{\operatorname*{fr}%
}\right)  \left(  0\right) \nonumber\\
&  +2\mathcal{V}_{ + }\left(  \tau\right)  \left(  \partial_{x}\left(  \chi
f_{1,\pm}^{\operatorname*{fr}}\right)  \right)  +2\mathcal{V}_{ - }\left(
\tau\right)  \left(  \partial_{x}\left(  \chi f_{1,\pm}^{\operatorname*{fr}%
}\right)  \right)  . \label{130bis}%
\end{align}
Then, it follows from (\ref{127}) and (\ref{140}) that%
\begin{equation}
\left\Vert P_{-}\left[  \partial_{k}\mathcal{V}_{ - }\left(  \tau\right)
\left(  \chi f_{1,\pm}^{\operatorname*{fr}}\right)  -\partial_{k}\mathcal{V}_{
+ }\left(  \tau\right)  \left(  \chi f_{1,\pm}^{\operatorname*{fr}}\right)
\right]  \right\Vert _{L^{2}( {\mathbb{R}})}\leq C |\tau|^{\frac{1-\rho}%
{2}-\alpha/2}\left\Vert w_{\pm}\right\Vert _{H^{1}( {\mathbb{R}})}^{\alpha+1},
\qquad\pm\tau\geq a, \label{133}%
\end{equation}
$0 \leq\rho<1.$ Hence, using (\ref{36}), (\ref{127.1}) and (\ref{133}) in
(\ref{126}), and taking $\rho$ sufficiently close to $1,$ since $\alpha>2$ we
obtain%
\begin{equation}
\left\Vert \partial_{k}\int_{\mathcal{T}_{\pm}(a,t)} B_{2}\,d\tau\right\Vert
_{L^{2}}\leq C \sup_{\tau\in\mathcal{T}_{\pm}(a,t)} \left(  \left\Vert w_{\pm
}\right\Vert _{H^{1}( {\mathbb{R}})}^{\alpha}+1\right)  \Lambda_{\pm}\left(
\tau;w_{\pm}\right)  , \qquad\pm t \geq a. \label{35.1}%
\end{equation}

Combining estimates (\ref{35}) and (\ref{35.1})\ in (\ref{35.0}) we attain
(\ref{Lema3F1}).
\end{proof}

\begin{proof}
[Proof of Lemma \ref{Lema4}]Derivating (\ref{39}), using (\ref{14}) and
integrating by parts we have%
\begin{align*}
\partial_{k}^{2}B_{1}  &  =\left(  \partial_{k}^{2}S\left(  k\right)  \right)
\mathcal{V}_{+}\left(  \tau\right)  \left(  1-\chi\right)  f_{1,\pm
}^{\operatorname*{fr}}\\
&  -4\left(  \partial_{k}S\left(  k\right)  \right)  \mathcal{V}{_{+}}\left(
\tau\right)  \partial_{x}\left(  \left(  1-\chi\right)  f_{1,\pm
}^{\operatorname*{fr}}\right) \\
&  +4S\left(  k\right)  \mathcal{V}_{+}\left(  \tau\right)  \partial_{x}%
^{2}\left(  \left(  1-\chi\right)  f_{1,\pm}^{\operatorname*{fr}}\right)
+4\mathcal{V}{\ }_{-}\left(  \tau\right)  \partial_{x}^{2}\left(  \left(
1-\chi\right)  f_{1,\pm}^{\operatorname*{fr}}\right)  .
\end{align*}
Using (\ref{derivativeffr}), (\ref{UnitaritySM}), (\ref{scatmatrixderiv}),
(\ref{scatmatrixsecondderiv}), and (\ref{EST8.0}), we get%
\begin{equation}
\left\Vert \partial_{k}^{2}B_{1}\right\Vert _{L^{2}({\mathbb{R}})}\leq
C\sqrt{|\tau|}\,|\tau|^{-\alpha/2}\,\Lambda_{\pm}\left(  \tau;w_{\pm}\right)
. \label{146}%
\end{equation}
To estimate $\partial_{k}^{2}B_{2}$ we decompose $B_{2}$ as%
\begin{align}
B_{2}  &  =\left(  S\left(  k\right)  -S\left(  0\right)  \right)
\mathcal{V}_{+}\left(  \tau\right)  \left(  \chi f_{1,\pm}^{\operatorname*{fr}%
}\right) \nonumber\\
&  +P_{+}\left(  1-\chi\right)  \left(  \mathcal{V}_{+}\left(  \tau\right)
\left(  \chi f_{1,\pm}^{\operatorname*{fr}}\right)  +\mathcal{V}_{-}\left(
\tau\right)  \left(  \chi f_{1,\pm}^{\operatorname*{fr}}\right)  \right)
\nonumber\\
&  +P_{+}\chi\left(  \mathcal{V}_{+}\left(  \tau\right)  \left(  \chi
f_{1,\pm}^{\operatorname*{fr}}\right)  +\mathcal{V}_{-}\left(  \tau\right)
\left(  \chi f_{1,\pm}^{\operatorname*{fr}}\right)  \right) \label{146.1}\\
&  +P_{-}\left(  -\mathcal{V}_{+}\left(  \tau\right)  \left(  \chi f_{1,\pm
}^{\operatorname*{fr}}\right)  +\mathcal{V}_{-}\left(  \tau\right)  \left(
\chi f_{1,\pm}^{\operatorname*{fr}}\right)  \right)  .
\end{align}
From now on we take $\chi$ even. Then, using (\ref{68}) we have%
\begin{align*}
\mathcal{V}_{-}\left(  \tau\right)  \left(  \chi f_{1,\pm}^{\operatorname*{fr}%
}\right)   &  =\sqrt{\frac{\tau}{2\pi i}}\int_{-\infty}^{0}e^{i\tau\left(
k+\frac{x}{2}\right)  ^{2}}\chi f_{1,\pm}^{\operatorname*{fr}}\left(
-x\right)  dx\\
&  =\sqrt{\frac{\tau}{2\pi i}}\int_{-\infty}^{0}e^{i\tau\left(  k+\frac{x}%
{2}\right)  ^{2}}\chi f_{1,\pm}^{\operatorname*{fr}}dx+\mathcal{\widetilde{W}%
}\left(  \tau\right)  \left(  \chi a_{1}\right)  ,
\end{align*}
with%
\[
\mathcal{\widetilde{W}}\left(  \tau\right)  \phi=\mathcal{V}_{-}\left(
t\right)  \left(  \phi\left(  -x\right)  \right)  =\sqrt{\frac{\tau}{2\pi i}%
}\int_{-\infty}^{0}e^{i\tau\left(  k+\frac{x}{2}\right)  ^{2}}\phi dx,
\]
and%
\[
a_{1}\left(  x\right)  =f_{1,\pm}^{\operatorname*{fr}}\left(  S(-k)w_{\pm
}\left(  k\right)  \right)  -f_{1,\pm}^{\operatorname*{fr}}\left(  w_{\pm
}\right)  =f_{1,\pm}^{\operatorname*{fr}}\left(  w_{\pm}\left(  -k\right)
\right)  -f_{1,\pm}^{\operatorname*{fr}}\left(  w_{\pm}\right)  ,
\]
where we used \eqref{68}. Then, by \eqref{146.1}%
\begin{equation}
B_{2}=B_{2}^{\left(  1\right)  }+B_{2}^{\left(  2\right)  }, \label{767}%
\end{equation}
with%
\begin{align*}
B_{2}^{\left(  1\right)  }  &  =\left(  S\left(  k\right)  -S\left(  0\right)
\right)  \mathcal{V}_{+}\left(  \tau\right)  \left(  \chi f_{1,\pm
}^{\operatorname*{fr}}\right) \\
&  +P_{+}\left(  1-\chi\right)  \left(  \mathcal{V}_{+}\left(  \tau\right)
\left(  \chi f_{1,\pm}^{\operatorname*{fr}}\right)  +\mathcal{V}_{-}\left(
\tau\right)  \left(  \chi f_{1,\pm}^{\operatorname*{fr}}\right)  \right) \\
&  +P_{-}\left(  1-\chi\right)  \left(  \mathcal{V}_{-}\left(  \tau\right)
\left(  \chi f_{1,\pm}^{\operatorname*{fr}}\right)  -\mathcal{V}_{+}\left(
\tau\right)  \left(  \chi f_{1,\pm}^{\operatorname*{fr}}\right)  \right) \\
&  +P_{+}\chi\left(  \mathcal{\widetilde{V}}\left(  \tau\right)  \left(  \chi
f_{1,\pm}^{\operatorname*{fr}}\right)  +\mathcal{\widetilde{W}}\left(
\tau\right)  \left(  \chi a_{1}\right)  \right) ,
\end{align*}
and%
\[
B_{2}^{\left(  2\right)  }=P_{-}\chi\left(  \mathcal{V}_{\ -}\left(
\tau\right)  \left(  \chi f_{1,\pm}^{\operatorname*{fr}}\right)
-\mathcal{V}_{+}\left(  \tau\right)  \left(  \chi f_{1,\pm}%
^{\operatorname*{fr}}\right)  \right) ,
\]
where we denote%
\[
\mathcal{\widetilde{V}}\left(  \tau\right)  \phi=\sqrt{\frac{\tau}{2\pi i}%
}\int_{-\infty}^{\infty}e^{i\tau\left(  k+\frac{x}{2}\right)  ^{2}}\phi dx.
\]

We calculate%
\begin{equation}
\label{767.1}\partial_{k}^{2}B_{2}^{\left(  1\right)  }=B_{21}+B_{22}+B_{23},
\end{equation}
where%
\begin{align*}
B_{21}  &  =\partial_{k}^{2}\left(  \left(  S\left(  k\right)  -S\left(
0\right)  \right)  \mathcal{V}_{ + }\left(  \tau\right)  \left(  \chi
f_{1,\pm}^{\operatorname*{fr}}\right)  \right) \\
&  +P_{+}\partial_{k}^{2}\left(  \left(  1-\chi\right)  \left(  \mathcal{V}_{
+ }\left(  \tau\right)  \left(  \chi f_{1,\pm}^{\operatorname*{fr}}\right)
+\mathcal{V}_{ -}\left(  \tau\right)  \left(  \chi f_{1,\pm}%
^{\operatorname*{fr}}\right)  \right)  \right) \\
&  +P_{-}\partial_{k}^{2}\left(  \left(  1-\chi\right)  \left(  \mathcal{V}_{
- }\left(  t\right)  \left(  \chi f_{1,\pm}^{\operatorname*{fr}}\right)
-\mathcal{V}_{ + }\left(  \tau\right)  \left(  \chi f_{1,\pm}%
^{\operatorname*{fr}}\right)  \right)  \right)  ,
\end{align*}%
\[
B_{22}= \left(  \partial^{2}_{k} \chi\right)  P_{+}\mathcal{\widetilde{V}%
}\left(  \tau\right)  \left(  \chi f_{1,\pm}^{\operatorname*{fr}}\right)  - 4
(\partial_{k} \chi) P_{+}\mathcal{\widetilde{V}}\left(  \tau\right)
\partial_{x}\left(  \chi f_{1,\pm}^{\operatorname*{fr}}\right)  +
4P_{+}\mathcal{\tilde{V}}\left(  \tau\right)  \partial_{x}^{2}\left(  \chi
f_{1,\pm}^{\operatorname*{fr}}\right)  ,
\]
and%
\[
B_{23}=P_{+}\partial_{k}^{2}\left(  \chi\mathcal{\widetilde{W}}\left(
\tau\right)  \left(  \chi a_{1}\right)  \right)  .
\]
By (\ref{S-1}), \eqref{scatmatrixderiv}, and (\ref{scatmatrixsecondderiv})
$\psi=S\left(  k\right)  -S\left(  0\right)  $ satisfies (\ref{128}).
Moreover, also $\psi=1-\chi$ satisfies (\ref{128}). Further, as
\[
\left(  \mathcal{V}_{ - }\left(  \tau\right)  \phi\right)  \left(  k\right)
=\left(  \mathcal{V}_{ + }\left(  \tau\right)  \phi\right)  \left(  -k\right)
,
\]
using (\ref{L5-2}) of Lemma \ref{Lema5}, (\ref{UnitaritySM}),
(\ref{scatmatrixderiv}) and (\ref{scatmatrixsecondderiv}) we get
\begin{equation}
\label{B21}\left\Vert \int_{\mathcal{T}_{\pm}(a,t)} B_{21}d\tau\right\Vert
_{L^{2}( {\mathbb{R}})}\leq C \sqrt{|t|}\, \sup_{\tau\in\mathcal{T}_{\pm
}(a,t)} \left(  \left\Vert w_{\pm}\right\Vert _{H^{1}( {\mathbb{R}})}^{\alpha
}+1\right)  \Lambda_{\pm}\left(  \tau;w_{\pm}\right)  , \pm t \geq a.
\end{equation}
Moreover, by (\ref{derivativeffr}), and\eqref{EST8.0}, we have
\begin{equation}
\left\Vert B_{22}\right\Vert _{L^{2}( {\mathbb{R}})}\leq C\sqrt{|\tau|}
|\tau|^{-\alpha/2} \Lambda\left(  \tau;w_{\pm}\right)  . \label{B22}%
\end{equation}
Next, we decompose
\begin{equation}
\label{B22.1}B_{23}= B_{23}(1)+ B_{23}(2)+B_{23}(3),
\end{equation}
where%
\[
B_{23}\left(  1\right)  :=P_{+}\left(  \partial_{k}^{2}\chi\right)
\mathcal{\widetilde{W}}\left(  \tau\right)  \left(  \chi a_{1}\right)  ,
\]%
\[
B_{23}\left(  2\right)  :=2 P_{+} \left(  \partial_{k}\chi\right)
\partial_{k}\mathcal{\widetilde{W}}\left(  \tau\right)  \left(  \chi
a_{1}\right)  ,
\]
and%
\[
B_{23}\left(  3\right)  := P_{+} \chi\partial_{k}^{2}\mathcal{\widetilde{W}%
}\left(  \tau\right)  \left(  \chi a_{1}\right)  .
\]
By (\ref{12.1}), and (\ref{EST8.0}) we estimate,%
\begin{equation}
\left\Vert B_{23}\left(  1\right)  \right\Vert _{L^{2}( {\mathbb{R}})}\leq C
|\tau|^{-\alpha/2} \left\Vert w_{\pm}\right\Vert _{H^{1}( {\mathbb{R}}%
)}^{\alpha+1}, \qquad\pm\tau\geq a. \label{B231}%
\end{equation}
Note that by (\ref{68}),
\begin{equation}
a_{1}\left(  -x\right)  =-a_{1}\left(  x\right)  , \label{69}%
\end{equation}
and as $a_{1} \in H^{2}( {\mathbb{R}}),$ we have, $a_{1}(0)=0.$

Then,
\begin{equation}
\partial_{k}\mathcal{\widetilde{W}}\left(  \tau\right)  \left(  \chi
a_{1}\right)  =-2\mathcal{\widetilde{W}}\left(  \tau\right)  \partial
_{x}\left(  \chi a_{1}\right)  , \label{135}%
\end{equation}
and it follows from (\ref{derivativeffr}), and \eqref{EST8.0},%
\begin{equation}
\left\Vert B_{23}\left(  2\right)  \right\Vert _{L^{2}}\leq C |\tau
|^{-\alpha/2} \left\Vert w\right\Vert _{H^{1}}^{\alpha+1}. \label{B232}%
\end{equation}
Moreover, derivating (\ref{135}) we have
\begin{equation}
\label{B232.1}\partial_{k}^{2}\mathcal{\widetilde{W}}\left(  \tau\right)
\left(  \chi a_{1}\right)  =-4\sqrt{\frac{\tau}{2\pi i}}e^{i\tau k^{2}}\left(
\partial_{x}\left(  \chi a_{1}\right)  \right)  \left(  0\right)
+4\mathcal{\widetilde{W}}\left(  \tau\right)  \partial_{x}^{2}\left(  \chi
a_{1}\right)  .
\end{equation}
Using (\ref{derivativeffr}), (\ref{UnitaritySM}), (\ref{scatmatrixderiv}),
(\ref{scatmatrixsecondderiv}), and (\ref{EST8.0}) we get%
\begin{equation}
\left\Vert \mathcal{\tilde{W}}\left(  \tau\right)  \partial_{x}^{2}\left(
\chi a_{1}\right)  \right\Vert _{L^{2}( {\mathbb{R}})}\leq C |\tau
|^{-\alpha/2} \left\Vert w_{\pm}\right\Vert _{H^{1}( {\mathbb{R}})}^{\alpha
}\left\Vert w_{\pm}\right\Vert _{H^{2}( {\mathbb{R}})}. \label{136}%
\end{equation}
Since $\mathcal{N}$ conmutes with $P_{+},$ using \eqref{UnitaritySM} we can
write%
\begin{equation}
P_{+}a_{1}\left(  x\right)  =P_{+}\mathcal{N}\left(  \left\vert \tau^{-1/2}
\left(  \mathcal{V}\left(  \tau\right)  w_{\pm}\right)  _{1}\right\vert
,...,\left\vert \left(  \tau^{-1/2} \mathcal{V}\left(  \tau\right)  w_{\pm
}\right)  _{n}\right\vert \right)  \left(  \mathcal{V}\left(  \tau\right)
\left(  S\left(  -k\right)  -S\left(  0\right)  \right)  w_{\pm}\right)  .
\label{136bis}%
\end{equation}
Then, it follows from (\ref{100}) with $g= \tau^{-1/2} \mathcal{V}\left(
\tau\right)  w_{\pm}$ and $\partial_{t}$ replaced with $\partial_{x}$ that%
\begin{align*}
P_{+}\partial_{x}a_{1}  &  = P_{+} \sum_{j=1}^{n}\left(  E_{1}^{\left(
j\right)  }\left(  \tau^{-1/2} \mathcal{V}\left(  \tau\right)  w_{\pm}\right)
\partial_{x}\left(  \tau^{-1/2} \mathcal{V}\left(  \tau\right)  w_{\pm
}\right)  _{j}+E_{2}^{\left(  j\right)  }\left(  \tau^{-1 /2} \mathcal{V}%
\left(  \tau\right)  w_{\pm}\right)  \partial_{x}\overline{\left(  \tau^{-1/2}
\mathcal{V}\left(  \tau\right)  w_{\pm}\right)  _{j}}\right) \\
&  \left(  \mathcal{V}\left(  \tau\right)  \left(  S\left(  -k\right)
-S\left(  0\right)  \right)  w_{\pm}\right)  +P_{+}\left(  \mathcal{N}\left(
\left\vert \tau^{-1/2} \left(  \mathcal{V}\left(  \tau\right)  w_{\pm}\right)
_{1}\right\vert ,...,\left\vert \left(  \tau^{-1/2} \mathcal{V}\left(
\tau\right)  w_{\pm}\right)  _{n}\right\vert \right)  \right) \\
&  \partial_{x}\left(  \mathcal{V}\left(  \tau\right)  \left(  S\left(
-k\right)  -S\left(  0\right)  \right)  w_{\pm}\right)  .
\end{align*}
Then, using (\ref{ConNonlinearity}), \eqref{96.1}, and integrating by parts we
estimate,
\begin{align*}
\left\vert P_{+}\partial_{x}a_{1}\right\vert  &  \leq C |\tau|^{-\alpha/2}
\left\vert \mathcal{V}\left(  \tau\right)  w_{\pm}\right\vert ^{\alpha
-1}\left\vert \mathcal{V}\left(  \tau\right)  \left(  S\left(  -k\right)
-S\left(  0\right)  \right)  w_{\pm}\left(  k\right)  \right\vert \left\vert
\mathcal{V}\left(  \tau\right)  \partial_{k}w_{\pm}\right\vert \\
&  +C |\tau|^{-\alpha/2} \left\vert \mathcal{V}\left(  \tau\right)  w_{\pm
}\right\vert ^{\alpha}\left\vert \mathcal{V}\left(  \tau\right)  \partial
_{k}\left(  S\left(  -k\right)  -S\left(  0\right)  \right)  w_{\pm}\left(
k\right)  \right\vert , \qquad\pm\tau\geq a .
\end{align*}
Thus, it follows from (\ref{71}), (\ref{UnitaritySM}), (\ref{S-1}),
(\ref{scatmatrixderiv}), (\ref{scatmatrixsecondderiv}), and (\ref{89.xxxx}),
\[
\left\vert P_{+}\left(  \partial_{x}a_{1}\right)  \left(  0\right)
\right\vert \leq C |\tau|^{-\alpha/2} \left\Vert w_{\pm}\right\Vert _{H^{1}(
{\mathbb{R}})}^{\alpha}\left(  \left\Vert w_{\pm}\right\Vert _{H^{1}(
{\mathbb{R}})}+|\tau|^{-\rho/2}\left\Vert w_{\pm}\right\Vert _{H^{2}(
{\mathbb{R}})}\right)  , \qquad\pm\tau\geq a ,
\]
for any $0\leq\rho<1.$ Hence, using \eqref{B232.1}, and (\ref{136}) we deduce
\begin{equation}
\left\Vert B_{23}\left(  3\right)  \right\Vert _{L^{2}( {\mathbb{R}})}\leq C
|\tau|^{1-\rho/2- \alpha/2}\Lambda_{\pm}\left(  \tau;w_{\pm}\right)  ,
\qquad\pm\tau\geq a . \label{B233}%
\end{equation}
By (\ref{B231}), (\ref{B232}) and (\ref{B233}) we get%
\begin{equation}
\left\Vert B_{23}\right\Vert _{L^{2}( {\mathbb{R}})}\leq C|\tau|^{1-\rho
/2-\alpha/2}\Lambda_{\pm}\left(  \tau;w_{\pm}\right)  , \qquad\pm\tau\geq a .
\label{B23}%
\end{equation}
Therefore$,$ using (\ref{B21}), (\ref{B22}) and (\ref{B23}) and taking
$0\leq\rho<1$ sufficiently close to one we conclude that
\begin{equation}
\left\Vert \partial_{k}^{2}\int_{\mathcal{T}_{\pm}(a,t)} B_{2}^{\left(
1\right)  }d\tau\right\Vert _{L^{2}( {\mathbb{R}})}\leq C \sqrt{|t|}\,
\sup_{\tau\in\mathcal{T}_{\pm}(a,t)} \left(  \left\Vert w_{\pm}\right\Vert
_{H^{1}( {\mathbb{R}})}^{\alpha}+1\right)  \Lambda_{\pm}\left(  \tau;w_{\pm
}\right)  . \label{145}%
\end{equation}

Next, we estimate $\partial_{k}^{2}B_{2}^{\left(  2\right)  }.$ We calculate
\begin{align}
\partial_{k}^{2}B_{2}^{\left(  2\right)  }  &  =2P_{-}\left(  \partial_{k}%
\chi\right)  \left(  \partial_{k}\mathcal{V}_{ - }\left(  \tau\right)  \left(
\chi f_{1,\pm}^{\operatorname*{fr}}\right)  -\partial_{k}\mathcal{V}_{ +
}\left(  \tau\right)  \left(  \chi f_{1,\pm}^{\operatorname*{fr}}\right)
\right) \nonumber\\
&  +\left(  \partial_{k}^{2}\chi\right)  P_{-}\left(  \mathcal{V}_{ - }\left(
\tau\right)  \left(  \chi f_{1,\pm}^{\operatorname*{fr}}\right)
-\mathcal{V}_{ + }\left(  \tau\right)  \left(  \chi f_{1,\pm}%
^{\operatorname*{fr}}\right)  \right) \nonumber\\
&  +\chi P_{-}\partial_{k}^{2}\left(  -\mathcal{V}_{ + }\left(  \tau\right)
\left(  \chi f_{1,\pm}^{\operatorname*{fr}}\right)  +\mathcal{V}_{ - }\left(
\tau\right)  \left(  \chi f_{1,\pm}^{\operatorname*{fr}}\right)  \right)  .
\label{144}%
\end{align}
Let us define,%
\[
f_{2}^{\operatorname*{fr}}\left(  w_{\pm}\right)  :=\mathcal{N}\left(
\left\vert \left(  \mathcal{V}\left(  \tau\right)  w_{\pm}\right)
_{1}\right\vert ,...,\left\vert \left(  \mathcal{V}\left(  \tau\right)
w_{\pm}\right)  _{n}\right\vert \right)  \left(  0\right)  .
\]
Using (\ref{130}) we get%
\begin{align}
\label{qpqpqp}\chi P_{-}  &  \partial_{k}^{2}\left(  -\mathcal{V}_{ + }\left(
\tau\right)  +\mathcal{V}_{ - }\left(  \tau\right)  \right)  \left(  \chi
f_{1,\pm}^{\operatorname*{fr}}\right) \nonumber\\
&  =8 \chi P_{-}i\tau\sqrt{\frac{\tau}{2\pi i}}ke^{i\tau k^{2}}f_{2}%
^{\operatorname*{fr}}\left(  \chi w_{\pm}\right)  \left(  \mathcal{V}\left(
\tau\right)  w_{\pm}\right)  \left(  0\right) \nonumber\\
&  +8\chi P_{-}8 i\tau\sqrt{\frac{\tau}{2\pi i}}ke^{i\tau k^{2}}\left(
f_{2}^{\operatorname*{fr}}\left(  w_{\pm}\right)  -f_{2}^{\operatorname*{fr}%
}\left(  \chi w_{\pm}\right)  \right)  \left(  \mathcal{V}\left(  \tau\right)
w_{\pm}\right)  \left(  0\right) \\
&  -4 \chi P_{-}\mathcal{V}_{ + }\left(  \tau\right)  \left(  \partial_{x}%
^{2}\left(  \chi f_{1,\pm}^{\operatorname*{fr}}\right)  \right)
+4\mathcal{V}_{\left(  -\right)  }\left(  \tau\right)  \left(  \partial
_{x}^{2}\left(  \chi f_{1\pm,}^{\operatorname*{fr}}\right)  \right)
.\nonumber
\end{align}
Let $\zeta\in C^{\infty}_{0}\left(  \mathbb{R}\right)  $ be symmetric and such
that%
\[
\left\{
\begin{array}
[c]{c}%
\zeta\left(  y\right)  =1,\text{ for\ \ }\left\vert y\right\vert \leq2,\\
\zeta\left(  y\right)  =0,\text{ for\ }\left\vert y\right\vert \geq3.
\end{array}
\right.
\]
We decompose%
\begin{align}
&  8 \chi P_{-} i \tau\sqrt{\frac{\tau}{2\pi i}}ke^{i\tau k^{2}}%
f_{2}^{\operatorname*{fr}}\left(  \chi w_{\pm}\right)  \left(  \mathcal{V}%
\left(  \tau\right)  w_{\pm}\right)  \left(  0\right) \nonumber\\
&  =8 \chi i \tau\sqrt{\frac{\tau}{2\pi i}}ke^{i\tau k^{2}}f_{2}%
^{\operatorname*{fr}}\left(  \chi w_{\pm}\right)  P_{-} \left(  \mathcal{V}%
\left(  \tau\right)  \left(  \zeta\left(  \frac{\cdot}{k}\right)  w_{\pm
}\right)  \right)  \left(  0\right) \nonumber\\
&  +8\chi i \tau\sqrt{\frac{\tau}{2\pi i}}ke^{i\tau k^{2}}f_{2}%
^{\operatorname*{fr}}\left(  \chi w_{\pm}\right)  P_{-} \left(  \mathcal{V}%
\left(  \tau\right)  \left(  \left(  1-\zeta\right)  \left(  \frac{\cdot}%
{k}\right)  w_{\pm}\right)  \right)  \left(  0\right)  , \label{29.0}%
\end{align}
where we used that $P_{-}$ conmutes with $\mathcal{N}.$ Using (\ref{2.20}) and
integrating by parts, we get
\begin{align}
&  \chi\int_{\mathcal{T}_{\pm}(a,t)}\tau^{3/2}ke^{i\tau k^{2}}f_{2}%
^{\operatorname*{fr}}\left(  \chi w_{\pm}\right)  P_{-} \left(  \mathcal{V}%
\left(  \tau\right)  \left(  \zeta\left(  \frac{\cdot}{k}\right)  w_{\pm
}\right)  \right)  \left(  0\right)  d\tau\nonumber\\
&  = \pm\chi\left.  \tau^{3/2}\frac{\tau ke^{i\tau k^{2}}}{1+i\tau k^{2}}%
f_{2}^{\operatorname*{fr}}\left(  \chi w_{\pm}\right)  P_{-} \left(
\mathcal{V}\left(  \tau\right)  \left(  \zeta\left(  \frac{\cdot}{k}\right)
w_{\pm}\right)  \right)  \left(  0\right)  \right\vert _{\pm a}^{t}\nonumber\\
&  -\chi\int_{\mathcal{T}_{\pm}(a,t)}\frac{\tau ke^{i\tau k^{2}}}{1+i\tau
k^{2}}\left[  \partial_{\tau}\left(  \frac{\tau^{3/2}}{1+i\tau k^{2}}%
f_{2}^{\operatorname*{fr}}\left(  \chi w_{\pm}\right)  \right)  \right]  P_{-}
\left(  \mathcal{V}\left(  \tau\right)  \left(  \zeta\left(  \frac{\cdot}%
{k}\right)  w_{\pm}\right)  \right)  \left(  0\right)  d\tau\nonumber\\
&  -\chi\int_{\mathcal{T}_{\pm}(a,t)} \tau^{3/2}k\frac{\tau e^{i\tau k^{2}}%
}{1+i\tau k^{2}}f_{2}^{\operatorname*{fr}}\left(  \chi w\right)
\partial_{\tau} P_{-} \left(  \mathcal{V}\left(  \tau\right)  \left(
\zeta\left(  \frac{\cdot}{k}\right)  w_{\pm}\right)  \right)  \left(
0\right)  d\tau. \label{26.0}%
\end{align}
By applying \eqref{129}, (\ref{70}), (\ref{S-1}), and (\ref{scatmatrixderiv}),
we control%
\begin{equation}
\label{26.0.1}\left\vert P_{-}\mathcal{V}\left(  \tau\right)  \left(
\zeta\left(  \frac{\cdot}{k}\right)  w_{\pm}\right)  (0) \right\vert \leq C
|\tau| ^{-\rho/2}\left\Vert w_{\pm}\right\Vert _{H^{1}( {\mathbb{R}})}.
\end{equation}
From (\ref{99}), via (\ref{derivativeffr}), (\ref{130}), (\ref{21.0}),
\eqref{76}, (\ref{UnitaritySM}), (\ref{scatmatrixderiv}), \eqref{89.xxxx},
using that $P_{+}+P_{-}= I, P_{+}^{2}= P_{+},$ and since $\alpha>2$ we
estimate%
\begin{align*}
\left\vert \partial_{\tau}\left(  \mathcal{V}\left(  \tau\right)  \left(  \chi
w_{\pm}\right)  \right)  (0) \right\vert  &  \leq C|\tau|^{-2}\left\Vert
\mathcal{V}\left(  \tau\right)  \partial_{k}^{2}\left(  \chi w_{\pm}\right)
\right\Vert _{L^{\infty}( {\mathbb{R}})}\\
&  +C\left\Vert \mathcal{V}\left(  \tau\right)  \chi\left(
\widehat{\mathcal{W}}\left(  \tau\right)  f_{1,\pm}-\widehat{\mathcal{V}%
}\left(  \tau\right)  f_{1,\pm}^{\operatorname*{fr}}\right)  \right\Vert
_{L^{\infty}( {\mathbb{R}})}\\
&  +C \left\vert \left(  P_{+} \mathcal{V}\left(  \tau\right)  \chi
\widehat{\mathcal{V}}\left(  \tau\right)  f_{1,\pm}^{\operatorname*{fr}%
}\right)  (0)\right\vert + \left\vert \left(  P_{-} \mathcal{V}\left(
\tau\right)  \chi\widehat{\mathcal{V}}\left(  \tau\right)  f_{1,\pm
}^{\operatorname*{fr}}\right)  (0)\right\vert \\
&  \leq C|\tau|^{-3/2}\left\Vert w_{\pm}\right\Vert _{H^{2}( {\mathbb{R}})}+C
|\tau|^{-1}\left\Vert w_{\pm}\right\Vert _{H^{1}( {\mathbb{R}})}^{\alpha+1},
\pm\tau\geq a,
\end{align*}
and then, via (\ref{ConNonlinearity}), (\ref{100}) with $g= \tau^{-1/2}
\mathcal{V}\left(  \tau\right)  w_{\pm},$ and (\ref{89.xxxx}),%
\begin{equation}
\left\vert \partial_{\tau}f_{2}^{\operatorname*{fr}}\left(  \chi w_{\pm
}\right)  \right\vert \leq C|\tau|^{-1-\alpha/2}\left(  \left\Vert w_{\pm
}\right\Vert _{H^{1}( {\mathbb{R}})}^{\alpha+1}+|\tau|^{-1/2}\left\Vert
w_{\pm}\right\Vert _{H^{2}( {\mathbb{R}})}\right)  \left\Vert w_{\pm
}\right\Vert _{H^{1}( {\mathbb{R}})}^{\alpha-1}, \qquad\pm\tau\geq a.
\label{23.0}%
\end{equation}
Also, by (\ref{ConNonlinearity}) and (\ref{EST4}) we have%
\begin{equation}
\left\vert f_{2}^{\operatorname*{fr}}\left(  \chi w_{\pm}\right)  \right\vert
\leq C |\tau|^{-\alpha/2} \left\Vert w_{\pm}\right\Vert _{H^{1}( {\mathbb{R}%
})}^{\alpha}. \label{f2}%
\end{equation}

Moreover, since $\alpha>2,$ by \eqref{26.0.1}, \eqref{f2}, the first two terms
on the right-hand side of (\ref{26.0}) are estimated by%
\[
C\sqrt{|t|}\sup_{\tau\in\mathcal{T}_{\pm}(a,t)}\left\Vert w\right\Vert
_{H^{1}({\mathbb{R}})}^{\alpha}\left(  \left\Vert w_{\pm}\right\Vert
_{H^{1}({\mathbb{R}})}+\left\Vert w_{\pm}\right\Vert _{H^{1}({\mathbb{R}}%
)}^{\alpha+1}+|\tau|^{-1/2}\left\Vert w_{\pm}\right\Vert _{H^{2}({\mathbb{R}%
})}\right)  ,\qquad\pm t\geq a.
\]
Using (\ref{99}), via \eqref{129}, and (\ref{S-1}), we estimate%
\begin{align*}
\left\vert P_{-}\partial_{\tau}\left(  \mathcal{V}\left(  \tau\right)  \left(
\zeta\left(  \frac{\cdot}{k}\right)  w_{\pm}\right)  \right)  (0)\right\vert
&  \leq\frac{C}{|\tau|^{3/2}}\left(  \left\Vert w_{\pm}\right\Vert
_{L^{\infty}({\mathbb{R}})}+|k|^{-1/2}\left\Vert w_{\pm}\right\Vert
_{H^{1}({\mathbb{R}})}+|k|^{1/2}\left\Vert w_{\pm}\right\Vert _{H^{2}%
({\mathbb{R}})}\right) \\
&  +C\left\vert P_{-}\mathcal{V}\left(  \tau\right)  \left(  \zeta\left(
\frac{\cdot}{k}\right)  \left(  \widehat{\mathcal{W}}\left(  \tau\right)
f_{1,\pm}\left(  w_{\pm}\right)  \right)  \right)  (0)\right\vert .
\end{align*}
Then, using (\ref{derivativeffr}), (\ref{130}), (\ref{70}), (\ref{EST4}),
(\ref{UnitaritySM}), (\ref{3.2}), (\ref{3.5}), \eqref{3.6}, (\ref{S-1}),
(\ref{scatmatrixderiv}), to estimate the second term on the right-hand side of
the last inequality we show that%
\begin{align*}
\left\vert P_{-}\partial_{\tau}\left(  \mathcal{V}\left(  \tau\right)  \left(
\zeta\left(  \frac{\cdot}{k}\right)  w_{\pm}\right)  \right)  (0)\right\vert
&  \leq\frac{C}{|\tau|^{3/2}}\left(  \left\Vert w_{\pm}\right\Vert
_{L^{\infty}({\mathbb{R}})}+|k|^{-1/2}\left\Vert w_{\pm}\right\Vert
_{H^{1}({\mathbb{R}})}+|k|^{1/2}\left\Vert w_{\pm}\right\Vert _{H^{2}%
({\mathbb{R}})}\right) \\
&  +C|\tau|^{-\frac{\alpha}{2}}\left(  \frac{1}{\sqrt{|\tau|}}+|k|\right)
\left\Vert w_{\pm}\right\Vert _{H^{1}({\mathbb{R}})}^{\alpha+1},\qquad\pm
\tau\geq a.
\end{align*}
Then, by (\ref{f2}), since $\alpha>2,$ we estimate the third term on the
right-hand side of (\ref{26.0}) as%
\begin{align*}
&  \left\Vert \chi\int_{\mathcal{T}_{\pm}(a,t)}\tau^{3/2}k\frac{\tau e^{i\tau
k^{2}}}{1+i\tau k^{2}}f_{2}^{\operatorname*{fr}}\left(  \chi w_{\pm}\right)
\partial_{\tau}P_{-}\left(  \mathcal{V}\left(  \tau\right)  \left(
\zeta\left(  \frac{\cdot}{k}\right)  w_{\pm}\right)  \right)  \left(
0\right)  d\tau\right\Vert _{L^{2}({\mathbb{R}})}\\
&  \leq C\int_{\mathcal{T}_{\pm}(a,t)}|\tau|^{-\alpha/2}\left\Vert \chi
\frac{\tau}{1+i\tau k^{2}}\left(  k\left\Vert w_{\pm}\right\Vert _{L^{\infty
}({\mathbb{R}})}+|k|^{1/2}\left\Vert w_{\pm}\right\Vert _{H^{1}({\mathbb{R}}%
)}+|k|^{3/2}\left\Vert w_{\pm}\right\Vert _{H^{2}({\mathbb{R}})}\right)
\right\Vert _{L^{2}({\mathbb{R}})}\left\Vert w_{\pm}\right\Vert _{H^{1}%
({\mathbb{R}})}^{\alpha}d\tau\\
&  +C\int_{\mathcal{T}_{\pm}(a,t)}|\tau|^{3/2-\alpha/2}\left\Vert \chi\frac
{k}{1+i\tau k^{2}}\left(  \frac{1}{\sqrt{|\tau|}}+|k|\right)  \right\Vert
_{L^{2}({\mathbb{R}})}\left\Vert w_{\pm}\right\Vert _{H^{1}({\mathbb{R}}%
)}^{\alpha+1}\left\Vert w_{\pm}\right\Vert _{H^{1}({\mathbb{R}})}^{\alpha
}d\tau\\
&  \leq C\sqrt{|t|}\sup_{\tau\in\mathcal{T}_{\pm}(a,t)}\left\Vert w_{\pm
}\right\Vert _{H^{1}({\mathbb{R}})}^{\alpha}\left(  \left\Vert w_{\pm
}\right\Vert _{H^{1}({\mathbb{R}})}+\left\Vert w_{\pm}\right\Vert
_{H^{1}({\mathbb{R}})}^{\alpha+1}+|\tau|^{-1/2}\left\Vert w_{\pm}\right\Vert
_{H^{2}({\mathbb{R}})}\right)  ,\qquad\pm t\geq a.
\end{align*}
Hence, we see from (\ref{26.0}),%
\begin{align}
&  \left\Vert \chi\int_{\mathcal{T}_{\pm}(a,t)}|\tau|^{3/2}ke^{i\tau k^{2}%
}f_{2}^{\operatorname*{fr}}\left(  \chi w_{\pm}\right)  P_{-}\left(
\mathcal{V}\left(  \tau\right)  \left(  \zeta\left(  \frac{\cdot}{k}\right)
w_{\pm}\right)  \right)  \left(  0\right)  d\tau\right\Vert _{L^{2}%
({\mathbb{R}})}\nonumber\\
&  \leq C\sqrt{|t|}\sup_{\mathcal{T}_{\pm}(a,t)}\left(  \left\Vert w_{\pm
}\right\Vert _{H^{1}({\mathbb{R}})}+\left\Vert w_{\pm}\right\Vert
_{H^{1}({\mathbb{R}})}^{\alpha+1}+|\tau|^{-1/2}\left\Vert w_{\pm}\right\Vert
_{H^{2}({\mathbb{R}})}\right)  \left\Vert w_{\pm}\right\Vert _{H^{1}%
({\mathbb{R}})}^{\alpha},\pm t\geq a. \label{27.0}%
\end{align}
By the definition of $\mathcal{V}\left(  \tau\right)  ,$ in \eqref{sp5.xxx} we
have%
\begin{align*}
&  4i\tau\sqrt{\frac{\tau}{2\pi i}}ke^{i\tau k^{2}}f_{2}^{\operatorname*{fr}%
}\left(  \chi w_{\pm}\right)  \left(  \mathcal{V}\left(  \tau\right)  \left(
\left(  1-\zeta\right)  \left(  \frac{\cdot}{k}\right)  w_{\pm}\right)
\right)  \left(  0\right) \\
&  =2i\frac{\tau^{2}}{\pi}ke^{i\tau k^{2}}f_{2}^{\operatorname*{fr}}\left(
\chi w_{\pm}\right)  \int_{-\infty}^{\infty}e^{-itk_{1}^{2}}\left(
1-\zeta\right)  \left(  \frac{k_{1}}{k}\right)  w_{\pm}\left(  k_{1}\right)
dk_{1}.
\end{align*}
Using that
\[
e^{i\tau\left(  k^{2}-k_{1}^{2}\right)  }=\frac{\partial_{\tau}\left(  \tau
e^{i\tau\left(  k^{2}-k_{1}^{2}\right)  }\right)  }{1+i\tau\left(  k^{2}%
-k_{1}^{2}\right)  },
\]
and integrating by parts, we obtain%
\begin{align}
&  \chi\int_{\mathcal{T}_{\pm}(a,t)}4i\tau\sqrt{\frac{\tau}{2\pi i}}ke^{i\tau
k^{2}}f_{2}^{\operatorname*{fr}}(\chi w_{\pm})P_{-}\left(  \mathcal{V}\left(
\tau\right)  \left(  \left(  1-\zeta\right)  \left(  \frac{\cdot}{k}\right)
w\right)  \right)  \left(  0\right)  d\tau\nonumber\\
&  =\pm\chi\frac{2i}{\pi}\left.  kf_{2}^{\operatorname*{fr}}\left(  \chi
w_{\pm}\right)  e^{i\tau k^{2}}\int_{-\infty}^{\infty}\frac{\tau^{3}e^{-i\tau
k_{1}^{2}}}{1+i\tau\left(  k^{2}-k_{1}^{2}\right)  }\left(  1-\zeta\right)
\left(  \frac{k_{1}}{k}\right)  P_{-}w_{\pm}\left(  k_{1}\right)
dk_{1}\right\vert _{\pm a}^{t}\nonumber\\
&  -\chi\frac{2i}{\pi}\int_{\mathcal{T}_{\pm}(a,t)}ke^{i\tau k^{2}}%
\int_{-\infty}^{\infty}\tau e^{-i\tau k_{1}^{2}}\partial_{\tau}\left(
\frac{\tau^{2}}{1+i\tau\left(  k^{2}-k_{1}^{2}\right)  }f_{2}%
^{\operatorname*{fr}}\left(  \chi w_{\pm}\right)  P_{-}w_{\pm}\left(
k_{1}\right)  \right)  \left(  1-\zeta\right)  \left(  \frac{k_{1}}{k}\right)
dk_{1}d\tau. \label{20.0}%
\end{align}
First, using (\ref{68}) we have%
\begin{align}
&  kf_{2}^{\operatorname*{fr}}(\chi w_{\pm})e^{i\tau k^{2}}\int_{-\infty
}^{\infty}\frac{\tau^{3}e^{-i\tau k_{1}^{2}}}{1+i\tau\left(  k^{2}-k_{1}%
^{2}\right)  }\left(  1-\zeta\right)  \left(  \frac{k_{1}}{k}\right)
P_{-}w_{\pm}\left(  k_{1}\right)  dk_{1}\nonumber\\
&  =kf_{2}^{\operatorname*{fr}}(\chi w_{\pm})e^{i\tau k^{2}}\int_{0}^{\infty
}\frac{\tau^{3}e^{-i\tau k_{1}^{2}}}{1+i\tau\left(  k^{2}-k_{1}^{2}\right)
}\left(  1-\zeta\right)  \left(  \frac{k_{1}}{k}\right)  P_{-}\left(  S\left(
-k_{1}\right)  +I\right)  w_{\pm}\left(  k_{1}\right)  dk_{1}. \label{19.0}%
\end{align}
Let $g\in W^{1,\infty}.$ Similarly to (\ref{70}), by using that
\[
\left\vert \frac{\left(  1-\zeta\right)  \left(  \frac{k_{1}}{k}\right)
}{1+i\tau\left(  k^{2}-k_{1}^{2}\right)  }\right\vert \leq C\left\vert \left(
1-\zeta\right)  \left(  \frac{k_{1}}{k}\right)  \right\vert \min\left\{
\frac{1}{1+\tau k^{2}},\frac{1}{1+\tau k_{1}^{2}}\right\}  ,
\]
and (\ref{UnitaritySM}), (\ref{S-1}), (\ref{scatmatrixderiv}), we estimate%
\begin{align}
&  K:=\left\vert \int_{0}^{\infty}\frac{ke^{-i\tau k_{1}^{2}}}{1+i\tau\left(
k^{2}-k_{1}^{2}\right)  }\left(  1-\zeta\right)  \left(  \frac{k_{1}}%
{k}\right)  P_{-}\left(  S\left(  -k_{1}\right)  +I\right)  g\left(
k_{1}\right)  dk_{1}\right\vert \nonumber\label{yy.zz.nn}\\
&  \leq C\left\vert \int_{0}^{\infty}kk_{1}\partial_{k_{1}}\left(  \frac
{1}{1+i\tau\left(  k^{2}-k_{1}^{2}\right)  }\frac{1}{1-i\tau k_{1}^{2}}\left(
1-\zeta\right)  \left(  \frac{k_{1}}{k}\right)  P_{-}\left(  S\left(
-k_{1}\right)  +I\right)  g\left(  k_{1}\right)  \right)  dk_{1}\right\vert
\nonumber\\
&  \leq C\frac{\left\vert k\right\vert }{1+\tau k^{2}}\int_{0}^{\infty}%
\frac{1}{1+\tau k_{1}^{2}}\left\vert P_{-}\left(  S\left(  -k_{1}\right)
+I\right)  g\left(  k_{1}\right)  \right\vert dk_{1}\\
&  +C\frac{\left\vert k\right\vert }{1+\tau k^{2}}\int_{0}^{\infty}\frac
{k_{1}}{1+\tau k_{1}^{2}}\left\vert S^{\prime}\left(  -k_{1}\right)  g\left(
k_{1}\right)  \right\vert dk_{1}\nonumber\\
&  +\frac{C}{\tau}\int_{0}^{\infty}\frac{1}{1+\tau k_{1}^{2}}\left\vert
P_{-}\left(  S\left(  -k_{1}\right)  +I\right)  g^{\prime}\left(
k_{1}\right)  \right\vert dk_{1}.\nonumber
\end{align}
Then, by \eqref{yy.zz.nn},
\begin{equation}
K\leq\frac{C}{|\tau|^{\rho}}\frac{\left\vert k\right\vert }{1+\tau k^{2}%
}\left\Vert g\right\Vert _{L^{\infty}({\mathbb{R}})}+\frac{C}{\tau^{1+\rho}%
}\left\Vert g^{\prime}\right\Vert _{L^{\infty}({\mathbb{R}})},
\label{yy.zz.nn.1}%
\end{equation}
$0\leq\rho<1.$ Using \eqref{yy.zz.nn.1}, from (\ref{19.0}) via (\ref{f2}), we
get%
\begin{align}
&  \left\Vert \chi\frac{2i}{\pi}\left.  kf_{2}^{\operatorname*{fr}}e^{i\tau
k^{2}}\int_{-\infty}^{\infty}\frac{\tau^{3}e^{-i\tau k_{1}^{2}}}%
{1+i\tau\left(  k^{2}-k_{1}^{2}\right)  }\left(  1-\zeta\right)  \left(
\frac{k_{1}}{k}\right)  P_{-}w_{\pm}\left(  k_{1}\right)  dk_{1}\right\vert
_{\pm a}^{t}\right\Vert _{L^{2}({\mathbb{R}})}\nonumber\\
&  \leq C\sqrt{|t|}\sup_{\tau\in\mathcal{T}_{\pm}(a,t)}\left(  \left\Vert
w_{\pm}\right\Vert _{H^{1}({\mathbb{R}})}+|\tau|^{-1/2}\left\Vert w_{\pm
}\right\Vert _{H^{2}({\mathbb{R}})}\right)  \left\Vert w_{\pm}\right\Vert
_{H^{1}({\mathbb{R}})}^{\alpha},\qquad\pm t\geq a. \label{22.0}%
\end{align}
From the equation (\ref{eqw}) for $w_{\pm},$ using (\ref{ConNonlinearity}),
(\ref{130}) , \eqref{UnitaritySM}, (\ref{3.2}), \eqref{3.5}, (\ref{3.6}),
\eqref{scatmatrixderiv}, \eqref{89.1}, and (\ref{EST4}), we have%
\begin{equation}
\left\Vert \partial_{\tau}w_{\pm}\right\Vert _{L^{\infty}({\mathbb{R}})}\leq
C|\tau|^{-\frac{\alpha}{2}}\left\Vert w_{\pm}\right\Vert _{H^{1}({\mathbb{R}%
})}^{\alpha+1}, \label{24.0}%
\end{equation}
and%
\begin{equation}
\left\Vert \partial_{k}\partial_{\tau}w_{\pm}\right\Vert _{L^{\infty
}({\mathbb{R}}9}\leq C|\tau|^{-\frac{\alpha}{2}+\frac{1}{2}}\left\Vert w_{\pm
}\right\Vert _{H^{1}({\mathbb{R}})}^{\alpha}\left(  \left\Vert w_{\pm
}\right\Vert _{H^{1}({\mathbb{R}})}+|\tau|^{-\frac{1}{2}}\left\Vert w_{\pm
}\right\Vert _{H^{2}({\mathbb{R}})}\right)  . \label{25.0}%
\end{equation}
Therefore, arguing as in the proof of (\ref{22.0}), by using (\ref{23.0}),
\eqref{yy.zz.nn}, (\ref{24.0}) and (\ref{25.0}) we control the second term on
the right-hand side of (\ref{20.0}) by%
\[
C\sqrt{|t|}\sup_{\tau\in\mathcal{T}_{\pm}(a,t)}\left(  \left\Vert w_{\pm
}\right\Vert _{H^{1}({\mathbb{R}})}+\left\Vert w_{\pm}\right\Vert
_{H^{1}({\mathbb{R}})}^{\alpha+1}+|\tau|^{-1/2}\left\Vert w_{\pm}\right\Vert
_{H^{2}({\mathbb{R}})}\right)  \left\Vert w_{\pm}\right\Vert _{H^{1}%
({\mathbb{R}})}^{\alpha}\left(  1+\left\Vert w_{\pm}\right\Vert _{H^{1}%
({\mathbb{R}})}^{\alpha}\right)  ,\pm t\geq a.
\]
Together with (\ref{22.0}) this proves that%
\begin{align}
&  \left\Vert \chi P_{-}\int_{\mathcal{T}_{\pm}(a,t)}4i\tau\sqrt{\frac{\tau
}{2\pi i}}ke^{i\tau k^{2}}f_{2}^{\operatorname*{fr}}(\chi w_{\pm})\left(
\mathcal{V}\left(  \tau\right)  \left(  \left(  1-\zeta\right)  \left(
\frac{\cdot}{k}\right)  w_{\pm}\right)  \right)  \left(  0\right)
d\tau\right\Vert _{L^{2}({\mathbb{R}})}\nonumber\\
&  \leq C\sqrt{|t|}\sup_{\tau\in\mathcal{T}_{\pm}(a,t)}\left(  \left\Vert
w_{\pm}\right\Vert _{H^{1}({\mathbb{R}})}+\left\Vert w_{\pm}\right\Vert
_{H^{1}({\mathbb{R}})}^{\alpha+1}+|\tau|^{-1/2}\left\Vert w_{\pm}\right\Vert
_{H^{2}({\mathbb{R}})}\right)  \left\Vert w_{\pm}\right\Vert _{H^{1}%
({\mathbb{R}})}^{\alpha}\left(  1+\left\Vert w_{\pm}\right\Vert _{H^{1}%
({\mathbb{R}})}^{\alpha}\right)  . \label{28.0}%
\end{align}
Gathering (\ref{27.0}) and (\ref{28.0}) in (\ref{29.0}) we arrive to
\begin{align}
&  \left\Vert \chi P_{-}\int_{\mathcal{T}_{\pm}(a,t)}4i\tau\sqrt{\frac{\tau
}{2\pi i}}ke^{i\tau k^{2}}f_{2}^{\operatorname*{fr}}\left(  \chi w_{\pm
}\right)  \left(  \mathcal{V}\left(  \tau\right)  w_{\pm}\right)  \left(
0\right)  d\tau\right\Vert _{L^{2}({\mathbb{R}})}\nonumber\label{w5w5w5}\\
&  \leq C\sqrt{|t|}\sup_{\tau\in\mathcal{T}_{\pm}(a,t)}\left(  \left\Vert
w_{\pm}\right\Vert _{H^{1}({\mathbb{R}})}+\left\Vert w_{\pm}\right\Vert
_{H^{1}({\mathbb{R}})}^{\alpha+1}+|\tau|^{-1/2}\left\Vert w_{\pm}\right\Vert
_{H^{2}({\mathbb{R}})}\right)  \left\Vert w_{\pm}\right\Vert _{H^{1}%
({\mathbb{R}})}^{\alpha}\left(  1+\left\Vert w_{\pm}\right\Vert _{H^{1}%
({\mathbb{R}})}^{\alpha}\right)  =\\
&  C\sqrt{|t|}\sup_{\tau\in\mathcal{T}_{\pm}(a,t)}\left(  \Lambda_{\pm}%
(\tau;w_{\pm})+\left\Vert w_{\pm}\right\Vert _{H^{1}({\mathbb{R}})}%
^{2\alpha+1}\right)  \left(  1+\left\Vert w_{\pm}\right\Vert _{H^{1}%
({\mathbb{R}})}^{\alpha}\right)  ,\qquad\pm t\geq a.\nonumber
\end{align}
By (\ref{ConNonlinearity}) and (\ref{EST4})%
\[
\left\vert f_{2}^{\operatorname*{fr}}\left(  w\right)  -f_{2}%
^{\operatorname*{fr}}\left(  \chi w\right)  \right\vert \leq C|\tau
|^{-\alpha/2}\left\vert \mathcal{V}\left(  \tau\right)  \left(  \left(
1-\chi\right)  w\right)  \right\vert \left\Vert w\right\Vert _{H^{1}}%
^{\alpha-1}.
\]
Then, as $1-\chi$ satisfies \eqref{128}, by \eqref{71}, and, as $P_{-}$
conmutes with $\mathcal{N},$ (\ref{71}), and \eqref{89.xxxx}, we get,%
\begin{equation}
\left\Vert \chi P_{-}4i\tau\sqrt{\frac{\tau}{2\pi i}}ke^{i\tau k^{2}}\left(
f_{2}^{\operatorname*{fr}}\left(  w_{\pm}\right)  -f_{2}^{\operatorname*{fr}%
}\left(  \chi w_{\pm}\right)  \right)  \left(  \mathcal{V}\left(  \tau\right)
w_{\pm}\right)  \left(  0\right)  \right\Vert _{L^{2}({\mathbb{R}})}\leq
C|\tau|^{\frac{3}{2}-\rho-\alpha/2}\left\Vert w_{\pm}\right\Vert
_{H^{1}({\mathbb{R}})}^{\alpha+1},\qquad\pm\tau\geq a, \label{aaaaa}%
\end{equation}
$0\leq\rho<1.$ Finally, by using (\ref{derivativeffr}), and (\ref{EST8.0}) we
show
\begin{equation}
\left\Vert \mathcal{V}^{\left(  \pm\right)  }\left(  \tau\right)  \left(
\partial_{x}^{2}\left(  \chi f_{1,\pm}^{\operatorname*{fr}}\right)  \right)
\right\Vert _{L^{2}({\mathbb{R}})}\leq C|\tau|^{-\alpha/2}\left\Vert w_{\pm
}\right\Vert _{H^{1}({\mathbb{R}})}^{\alpha}\left\Vert w_{\pm}\right\Vert
_{H^{2}({\mathbb{R}})}. \label{xyxy}%
\end{equation}
Hence, by \eqref{qpqpqp}, \eqref{w5w5w5}, \eqref{aaaaa}, \eqref{xyxy}, and
since $\alpha>2,$ we prove that
\begin{align}
&  \left\Vert \chi P_{-}\int_{\mathcal{T}_{\pm}(a,t)}\partial_{k}^{2}\left(
-\mathcal{V}_{+}\left(  \tau\right)  +\mathcal{V}_{-}\left(  \tau\right)
\right)  \left(  \chi f_{1,\pm}^{\operatorname*{fr}}\right)  d\tau\right\Vert
_{L^{2}({\mathbb{R}})}\nonumber\\
&  \leq C\sqrt{|t|}\sup_{\tau\in\mathcal{T}_{\pm}(a,t)}\left(  \Lambda_{\pm
}(\tau;w_{\pm})+\left\Vert w_{\pm}\right\Vert _{H^{1}({\mathbb{R}})}%
^{2\alpha+1}\right)  \left(  1+\left\Vert w_{\pm}\right\Vert _{H^{1}%
({\mathbb{R}})}^{\alpha}\right)  ,\qquad\pm t\geq a. \label{rrrrr}%
\end{align}
Using (\ref{derivativeffr}), and(\ref{2.11}), we show
\begin{equation}
\left\Vert \left(  \partial_{k}^{j}\chi\right)  \mathcal{V}^{\left(
\pm\right)  }\left(  \tau\right)  \left(  \chi f_{1,\pm}^{\operatorname*{fr}%
}\right)  \right\Vert _{L^{2}({\mathbb{R}})}\leq C|\tau|^{-\alpha/2}\left\Vert
w\right\Vert _{H^{1}({\mathbb{R}})}^{\alpha+1},\qquad\pm\tau\geq a,
\label{134}%
\end{equation}
for $j=0,1,2.$ Therefore, using (\ref{133}), (\ref{rrrrr}) and (\ref{134}) in
(\ref{144}) we prove
\begin{equation}
\left\Vert \partial_{k}^{2}\int_{a}^{t}B_{2}^{\left(  2\right)  }%
d\tau\right\Vert _{L^{2}({\mathbb{R}})}\leq C\sqrt{|t|}\sup_{\tau
\in\mathcal{T}_{\pm}(a,t)}\left(  \Lambda_{\pm}(\tau;w_{\pm})+\left\Vert
w_{\pm}\right\Vert _{H^{1}({\mathbb{R}})}^{2\alpha+1}\right)  \left(
1+\left\Vert w_{\pm}\right\Vert _{H^{1}({\mathbb{R}})}^{\alpha}\right)
,\qquad\pm t\geq a. \label{143}%
\end{equation}
Then, using (\ref{767}), (\ref{145}), and \eqref{143}, we get%
\begin{equation}
\left\Vert \partial_{k}^{2}\int_{\mathcal{T}_{\pm}(a,t)}B_{2}d\tau\right\Vert
_{L^{2}({\mathbb{R}})}\leq C\sqrt{|t|}\sup_{\tau\in\mathcal{T}_{\pm}%
(a,t)}\left(  \Lambda_{\pm}(\tau;w_{\pm})+\left\Vert w_{\pm}\right\Vert
_{H^{1}({\mathbb{R}})}^{2\alpha+1}\right)  \left(  1+\left\Vert w_{\pm
}\right\Vert _{H^{1}({\mathbb{R}})}^{\alpha}\right)  , \label{146bis}%
\end{equation}
for $\pm t\geq a.$ Finally, using (\ref{146}) and (\ref{146bis}) in
(\ref{35.0}), we attain (\ref{Lema3F2}).
\end{proof}

\appendix
\renewcommand{\theequation}{\thesection.\arabic{equation}}
\newtheorem{theorem2}{THEOREM}[section]
\renewcommand{\thetheorem}{\arabic{section}.\arabic{theorem}}
\newtheorem{prop2}[theorem2]{PROPOSITION} \newtheorem{lemma2}[theorem2]{LEMMA} \newtheorem{remark2}[theorem2]{Remark}

\section{Appendix. The linear matrix Schr\"{o}dinger equation on the half
line\label{App1}}

\label{appendix} In this appendix we present results on the spectral and the
scattering theory for the matrix Schr\"odinger equation on the half line, with
the general selfadjoint boundary condition, that we need. Some of the results
that we give here are borrowed from \cite{WederBook}, and others are new.

We consider the matrix Schr\"{o}dinger equation on the half line given by,
\begin{equation}
-\psi^{\prime\prime}+V\left(  x\right)  \psi=k^{2}\psi,\text{ }x\in
\mathbb{R}^{+}. \label{MSStationary}%
\end{equation}
The prime denotes the derivative with respect to the spacial coordinate $x,$
the scalar $k^{2}$ is the complex-valued spectral parameter, the potential $V
$ is a $n\times n$ self-adjoint matrix-valued function that belongs at least
to $L^{1}(\mathbb{R}^{+}).$ Later we impose further conditions of $V.$
Moreover, $n$ is any positive integer. The selfadjointness of $V$ means that,
\begin{equation}
V(x)=V^{\dagger}(x),\qquad x\in\mathbb{R}^{+}, \label{ap.2}%
\end{equation}
where the dagger denotes the matrix adjoint. The wavefunction that appears in
\eqref{MSStationary} may be an $n\times n$ matrix-valued function or it may be
a column vector with n components. Our interest is to study
\eqref{MSStationary} under the general self-adjoint boundary condition at
$x=0.$ For a discussion of the general self-adjoint boundary condition at $x=0
$ see Section 2.2 in Chapter 2, and Sections 3.4 and 3.6 in Chapter 3 of
\cite{WederBook}. We find it convenient to use the parametrization of the
general self-adjoint boundary condition at $x=0,$ given in \cite{WederBook} in
terms of two constant $n\times n$ matrices $A$ and $B$ as,
\begin{equation}
-B^{\dagger}\psi(0)+A^{\dagger}\psi^{\prime}(0)=0, \label{ap.3}%
\end{equation}
where $A$ and $B$ satisfy,
\begin{equation}
B^{\dagger}A =A^{\dagger}B, \label{ap.4}%
\end{equation}%
\begin{equation}
A^{\dagger}A+B^{\dagger}B>0. \label{ap.5}%
\end{equation}
We call $A$ and $B$ the boundary matrices. The condition in \eqref{ap.4} means
that the matrix $A^{\dagger}B$ is self-adjoint, and the condition in
\eqref{ap.5} means that the matrix $A^{\dagger}A+B^{\dagger}B$ is positive
definite. One can directly verify that the boundary condition
\eqref{ap.3}-\eqref{ap.5} remains invariant if the boundary matrices $A,$ and
$B$ are replaced, respectively, by $AT$ and $BT$ where $T$ is an arbitrary
$n\times n $ invertible matrix. Actually, it is proved in Proposition 2.2.1 in
page 23 of \cite{WederBook} that this is the only freedom in the choice of the
matrices $A,B$. In other words, the boundary condition
\eqref{ap.3}-\eqref{ap.5} is uniquely determined by the matrix pair $(A,B)$
modulo an invertible matrix $T,$ or, equivalently, the boundary condition
\eqref{ap.3}-\eqref{ap.5} determines the matrix pair $(A,B)$ modulo $T.$ There
is a simpler equivalent form of the boundary condition \eqref{ap.3} where the
matrices $A,B$ are diagonal. In Section 3.4 of Chapter three of
\cite{WederBook}, it is given the explicit steps to go from any pair of
matrices $A$ and $B$ appearing in the selfadjoint boundary condition
\eqref{ap.3}, and that satisfy \eqref{ap.4}, and (\ref{ap.5}) to a pair
$\tilde{A}$ and $\tilde{B},$ given by%
\begin{equation}
\tilde{A}=-\operatorname*{diag}\left\{  \sin\theta_{1},...,\sin\theta
_{n}\right\}  ,\text{ }\tilde{B}=\operatorname*{diag}\left\{  \cos\theta
_{1},...,\cos\theta_{n}\right\}  , \label{A,Btilde}%
\end{equation}
with appropriate real parameters $\theta_{j}\in(0,\pi].$ The matrices
$\tilde{A}, \tilde{B}$ satisfy (\ref{ap.4}), (\ref{ap.5}). In the case of the
matrices $\tilde{A}$, $\tilde{B},$ the boundary condition (\ref{ap.3}) is
given by%
\begin{equation}
\cos\theta_{j}Y_{j}\left(  0\right)  +\sin\theta_{j}\frac{d}{d x}Y_{j} \left(
0\right)  =0,\text{ \ \ }j=1,2,...,n. \label{PSM13}%
\end{equation}

The case $\theta_{j}=\pi$ corresponds to the Dirichlet boundary condition and
the case $\theta_{j}=\pi/2$ corresponds to the Neumann boundary condition. In
the general case, there are $n_{\operatorname*{N}}\leq n$ values with
$\theta_{j}=\pi/2$ and $n_{\operatorname*{D}}\leq n$ values with $\theta
_{j}=\pi$. Further, there are $n_{\operatorname*{M}}$ remaining values, where
$n_{\operatorname*{M}}=n-n_{\operatorname*{N}}-n_{\operatorname*{D}}$ such
that those $\theta_{j}$-values lie in the interval $(0,\pi/2)$ or $(\pi
/2,\pi).$ They correspond to Robin, or mixed, boundary conditions. It is
proven in Section 3.4 of Chapter three of \cite{WederBook}, that for any pair
of matrices $(A,B)$ that satisfy (\ref{ap.4}), and (\ref{ap.5}) there is a
pair of matrices $(\tilde{A},\tilde{B})$ as in (\ref{A,Btilde}), a unitary
matrix $M$ and two invertible matrices $T_{1},T_{2}$ such%

\begin{equation}
A=M\,\tilde{A}T_{1}M^{\dagger}T_{2},\quad B=M\,\tilde{B}T_{1}M^{\dagger}T_{2},
\label{trans}%
\end{equation}
and further, the Hamiltonians with the boundary condition given by matrices
$A,B$ and with the matrices $\tilde{A}, \tilde{B},$ are unitarily equivalent,
as we will see.

We construct a selfadjoint realization of the matrix Schr\"{o}dinger operator
$-\frac{d^{2}}{dx^{2}}+V(x)$ by quadratic forms methods. For the following
discussion see Sections 3.3 and 3.5 of Chapter three in \cite{WederBook}. Let
$\theta_{j}$ be as in equations (\ref{A,Btilde}). We denote
\begin{equation}
{H}_{j}^{(1)}(\mathbb{R}^{+};\mathbb{C}):=W_{1,2}^{\left(  0\right)
}(\mathbb{R}^{+};\mathbb{C}),\text{ if }\theta_{j}=\pi, \text{ and} \,{H}%
_{j}^{(1)}(\mathbb{R}^{+} ;\mathbb{C}):=H^{1}(\mathbb{R}^{+};\mathbb{C}%
),\text{ if }\theta_{j}\neq\pi. \label{spaceW1j}%
\end{equation}
We put%
\[
\widetilde{{H}}^{(1)}(\mathbb{R}^{+};\mathbb{C}^{n}):=\oplus_{j=1}^{n}{H}%
_{j}^{(1)}(\mathbb{R}^{+};\mathbb{C}).
\]
We define%
\[
\Theta:=\operatorname*{diag}[\widehat{\cot}\theta_{1},...,\widehat{\cot}%
\theta_{n}],
\]
where $\widehat{\cot}\theta_{j}=0,$ if $\theta_{j}=\pi/2,$ or $\theta_{j}%
=\pi,$ and $\widehat{\cot}\theta_{j}=\cot\theta_{j},$ if $\theta_{j}\neq
\pi/2,\pi.$ Suppose that the potential $V$ is self-adjoint and that it belongs
to $L^{1}(\mathbb{R}^{+}).$

The following quadratic form is closed, symmetric and bounded below,
\begin{equation}%
\begin{array}
[c]{l}%
h\left(  Y,Z\right)  :=\left(  \frac{d}{dx}Y,\frac{d}{dx}Z\right)
_{L^{2}(\mathbb{R}^{+},\mathbb{C}^{n})}-\left\langle M\Theta M^{\dag}Y\left(
0\right)  ,Z\left(  0\right)  \right\rangle +\left(  VY,Z\right)
_{L^{2}(\mathbb{R}^{+},\mathbb{C}^{n})},\\
Q\left(  h\right)  :={H}_{A,B}^{1}(\mathbb{R}^{+};\mathbb{C}^{n}),
\end{array}
\label{quadratic}%
\end{equation}
where by $Q(h)$ we denote the domain of $h$ and,
\begin{equation}
{H}^{1}_{A,B}(\mathbb{R}^{+};\mathbb{C}^{n}):=M\widetilde{{H}}^{(1)}%
(\mathbb{R}^{+};\mathbb{C}^{n})\subset H^{1}(\mathbb{R}^{+};\mathbb{C}^{n}).
\label{W1AB}%
\end{equation}
For simplicity we will also use the notation,
\[
H^{1}_{A,B}(\mathbb{R}^{+})\equiv H^{1}_{A,B}(\mathbb{R}^{+}; \mathbb{C}%
^{n}).
\]
We denote by $H_{A,B,V}$ the selfadjoint bounded below operator associated to
$h$ \cite{Kato}. The operator $H_{A,B,V}$ is the selfadjoint realization of
$\displaystyle-\frac{d^{2}}{dx^{2}}+V\left(  x\right)  $ with the selfadjoint
boundary condition (\ref{ap.3}). Actually, it is proven in Proposition 3.5.1
in page 107 of \cite{WederBook} that the domain of $H_{A,B,V}$ is given by,
\begin{equation}
\label{domainh}%
\begin{array}
[c]{l}%
D[H_{A,B,V}]= \left\{  Z \in H^{1}(\mathbb{R}^{+}) : Z^{\prime}\,
\text{\textrm{is absolutely continuous in any}} \, [0,a]\, \text{\textrm{for
all}}\, a >0,\right. \\
\left.  -B^{\dagger}Z(0) + A^{\dagger}Z^{\prime}(0)=0, \, \text{\textrm{and}%
}\, -\frac{d^{2}}{d x^{2}} Z(x)+ V(x) Z(x)\in L^{2}(\mathbb{R}^{+})\right\}  .
\end{array}
\end{equation}
If, moreover $V\in L^{1}(\mathbb{R}^{+})\cap L^{\infty}(\mathbb{R}^{+})$ the
domain of $H_{A,B,V}$ can be, equivalently, characterized as follows,
\begin{equation}
\label{domainh2}D[H_{A,B,V}]= \left\{  Z \in H^{2}(\mathbb{R}^{+}) :
-B^{\dagger}Z(0) + A^{\dagger}Z^{\prime}(0)=0\right\}  .
\end{equation}
When there is no possibility of misunderstanding we will use the notation $H,
$ i.e., $H\equiv H_{A,B,V}.$ It is proven in Section 3.6 of Chapter three of
\cite{WederBook} that,
\begin{equation}
H_{A,B,V}=MH_{\tilde{A},\tilde{B},M^{\dagger}VM}M^{\dagger}.
\label{diagonalization}%
\end{equation}

In the next proposition we introduce the Jost solution given in
\cite{AgrMarch}. See also Sections 3.1 and 3.2 of \cite{WederBook}.

\begin{prop2}
\label{PJostsolution}Suppose that the potential $V$ is self-adjoint and that it belongs to $L^1(\mathbb R^+).$
Then, for each fixed
$k\in\overline{\mathbb{C}^{+}}\backslash\{0\}$ there exists a unique $n\times
n$ matrix-valued Jost solution $f\left(  k,x\right)  $ to equation
(\ref{MSStationary}) satisfying the asymptotic condition%
\begin{equation}
f\left(  k,x\right)  =e^{ikx}\left(  I+o\left(1\right)  \right)
,\qquad x\rightarrow+\infty. \label{Jostsolution}%
\end{equation}
Moreover, for any fixed $x\in\lbrack0,\infty),$ $f\left(  k,x\right)  $ is
analytic in $k\in\mathbb{C}^{+}$ and continuous in $k\in\overline
{\mathbb{C}^{+}}\setminus\{0\}$. If, moreover, $V$  belongs to $L^1_1(\mathbb R^+),$ the Jost solution also exists at $k=0,$ and for each  fixed $x\in\lbrack0,\infty),$ $f\left(  k,x\right)  $ is continuous in $k\in\overline
{\mathbb{C}^{+}}.$ Furthermore if $V$ belongs to $L^1_1(\mathbb R^+),$ for $k\in\overline{\mathbb{C}^{+}}\backslash\{0\},$ the $o(1)$ term  in \eqref{Jostsolution} can be replaced by $o\left(\frac{1}{x}\right).$
\end{prop2}
Using the Jost solution and the boundary matrices $A$ and $B$ satisfying
(\ref{ap.4})-(\ref{ap.5}), we construct the Jost matrix $J\left(  k\right)
,$
\begin{equation}
J\left(  k\right)  =f\left(  -k^{*},0\right)  ^{\dagger}B-f^{\prime}\left(
-k^{*} ,0\right)  ^{\dagger}A,\text{ \ }k\in\overline{\mathbb{C}^{+}},
\label{Jostmatrix}%
\end{equation}
where the asterisk denotes complex conjugation. For the following result
Theorem 3.8.1 in page 114 of \cite{WederBook}.

\begin{prop2}
\label{Jostnozero}Suppose that the potential $V$ is self-adjoint and that
it belongs to $L^1(\mathbb R^+).$ Then,
$J\left(  k\right)  $ is analytic for $k\in\mathbb{C}^{+}$, continuous for
$k\in\overline{\mathbb{C}^{+}} \setminus\{0\},$ and invertible for $k\in\mathbb{R}%
\setminus\{0\}.$ Furthermore, if $V$ belongs to $L^1_1(\mathbb R^+),$ the Jost matrix $J(k)$ is continuous at $k=0.$
\end{prop2}

The scattering matrix, $S\left(  k\right)  ,$ is a $n\times n$ matrix-valued
function of $k\in\mathbb{R}$ that is given by
\begin{equation}
S\left(  k\right)  =-J\left(  -k\right)  J\left(  k\right)  ^{-1},\text{ }%
k\in\mathbb{R}\text{.} \label{Scatteringmatrix}%
\end{equation}

In the \textit{exceptional case} where $J(0)$ is not invertible the scattering
matrix is defined by \eqref{Scatteringmatrix} only for $k\neq0.$ However, it
is proved in Theorem 3.8.14 in pages 138-139 of \cite{WederBook}, that for
self-adjoint potentials $V$ that belong to $L_{1}^{1}(\mathbb{R}^{+})$ the
limit $S(0):=\lim_{k\rightarrow0}S(k)$ exists in the exceptional case and,
moreover, a formula for $S(0)$ is given. Furthermore, it is proved in
Proposition 3.8.17 in page 140 of \cite{WederBook},
\begin{equation}
S(-k)=S(k)^{\dagger}=S(k)^{-1}, \vert S(k)\vert=1, \qquad k \in\mathbb{R}.
\label{UnitaritySM}%
\end{equation}
It is proved in Theorem 3.10.6 in pages 196-197 of \cite{WederBook} that the
following limit exist,
\begin{equation}
S_{\infty}:=\lim_{|k|\rightarrow\infty}S\left(  k\right)  . \label{sinft}%
\end{equation}

\begin{remark2}\label{rema.1}{\rm
It follows from equation (3.10.37) in page 197 of \cite{WederBook} that $S_\infty=I_n$ if and only if there are no Dirichlet boundary conditions in the diagonal representation of the boundary matrices given in \eqref{A,Btilde}, \eqref{PSM13}, and \eqref{trans}.
Moreover, if $V \in L^{1}_{1}(\mathbb R^+),$ by Proposition 3.8.18 in page 142 of \cite {WederBook} the only two eigenvalues that  $S(0)$ can have are $+1$ and $-1$. We denote by $P_\pm$ the projectors onto the eigenspaces of $S(0)$ corresponding to the eigenvalues $\pm1,$ respectively.
Further, it follows from  Theorem 3.8.13 in page 137 of \cite{WederBook}  and Theorem 3.8.14 in pages 138-139 of \cite{WederBook}, that $S(0)=-I$ if and only if  we are in the {\it generic case} where $J(0)$ is invertible, and that  $S(0)=I$ if and only if the geometric multiplicity, $\mu,$ of the eigenvalue zero of the zero energy Jost matrix, $J(0),$  (see \eqref{Jostmatrix}) is equal to $n.$ Moreover, by Remark 3.8.10 in pages 129-130 of \cite{WederBook} the geometric multiplicity of the eigenvalue zero of $J(0)$ is equal to $n$ if and only if there are $n$ linearly independent bounded solutions to the Schr\"odinger equation \eqref{MSStationary} with zero energy, $k^2=0,$ that satisfy the boundary condition \eqref{ap.3}. This corresponds to the { \it purely exceptional case} where there are $n$ linearly independent half-bound states, or zero-energy resonances.  }
\end{remark2}

In terms of the Jost solution $f(k,x)$ and the scattering matrix $S(k)$ we
construct the physical solution to \eqref{MSStationary}, see equation (2.2.29)
in page 26 of \cite{WederBook},%
\begin{equation}
\Psi\left(  k,x\right)  =f\left(  -k,x\right)  +f\left(  k,x\right)  S\left(
k\right)  ,\text{ }k\in\mathbb{R}\text{.} \label{Physicalsolution}%
\end{equation}
By Proposition 3.2.12 in page 89 of \cite{WederBook} the physical solution
satisfies the boundary condition \eqref{ap.3}. The physical solution
$\Psi(k,x)$ is the main input to construct the generalized Fourier maps,
$\mathbf{F}^{\pm},$ for the absolutely continuous subspace of $H,$ that are
defined in equation (4.3.44) in page 284 of \cite{WederBook} (see also
Proposition 4.3.4 in page 287 of \cite{WederBook}),
\begin{equation}
\left(  \mathbf{F}^{\pm}Y\right)  \left(  k\right)  =\sqrt{\frac{1}{2\pi}}%
\int_{0}^{\infty}\left(  \Psi\left(  \mp k,x\right)  \right)  ^{\dagger
}Y\left(  x\right)  \,dx, \label{gefoma}%
\end{equation}
for $Y\in L^{1}(\mathbb{R}^{+})\cap L^{2}(\mathbb{R}^{+}).$

We have (see equation (4.3.46) in page 284 of \cite{WederBook}),
\begin{equation}
\left\Vert \mathbf{F}^{\pm}Y\right\Vert _{L^{2}(\mathbb{R}^{+})}=\left\Vert
E(\mathbb{R}^{+};H)Y\right\Vert _{L^{2}(\mathbb{R}^{+})}, \label{isom}%
\end{equation}
where, for any Borel set $U$ we denote by $E(U;H)$ the spectral projector of
$H$ for the Borel set $U.$ Thus, the $\mathbf{F}^{\pm}$ extend to bounded
operators on $L^{2}(\mathbb{R}^{+})$ that we also denote by $\mathbf{F}^{\pm
}.$

The following results on the spectral theory of $H$ are proved in Theorem
3.11.1 in pages 199-200, Theorem 4.3.3 in page 284 and Proposition 4.3.4 in
page 287, of \cite{WederBook}.

\begin{theorem2} \label{theospec}
\label{spect}Suppose that the potential $V$ is self-adjoint, that it belongs to $L^1(\mathbb R^+),$ and that the constant boundary matrices $A, B$ fulfill \eqref{ap.4}, \eqref{ap.5}. Then, the Hamiltonian $H$
has no positive eigenvalues and the negative spectrum of $H$ consists of
isolated eigenvalues of multiplicity smaller or equal than $n$, that can
accumulate only at zero. Furthermore, $H$ has no singular continuous spectrum
and its absolutely continuous spectrum is given by $[0,\infty)$. The
generalized Fourier maps $\mathbf{F}^{\pm}$ are partially isometric with
initial subspace the absolutely continuous subspace of $H$   and
final subspace $L^{2}(\mathbb R^+)$. Moreover, the adjoint operators are given by%
\beq \label{fad}
\left(  \left(  \mathbf{F}^{\pm}\right)  ^{\dagger}Z\right)  \left(
x\right)  =\sqrt{\frac{1}{2\pi}}\int_{0}^{\infty}\Psi\left(  \mp k,x\right)
Z\left(  k\right)  \,dk,
\ene
for $Z\in L^{1}(\mathbb R^+) \cap L^{2}(\mathbb R^+).$ Furthermore,
\begin{equation}
\mathbf{F}^{\pm}H\left(  \mathbf{F}^{\pm}\right)  ^{\dagger}=\mathcal{M},
\label{spectralrepr}%
\end{equation}
where $\mathcal{M}$ is the operator of multiplication by $k^{2}.$ If, in
addition, $V\in L^{1}_1(\mathbb R^+),$ zero is not an eigenvalue and the number of
eigenvalues of $H$ including multiplicities is finite.
\end{theorem2}

Note that by (\ref{isom}) $\left(  \mathbf{F}^{\pm}\right)  ^{\dagger
}\mathbf{F}^{\pm}$ is the orthogonal projector onto the absolutely continuous
subspace of $H,$ and as there is no singular continuous spectrum the
absolutely continuous subspace of $H$ coincides with the continuous subspace
of $H.$ Then,
\begin{equation}
\left(  \mathbf{F}^{\pm}\right)  ^{\dagger}\mathbf{F}^{\pm}%
=P_{\operatorname*{c}}(H) , \label{projector}%
\end{equation}
where $P_{\operatorname*{ c}}(H)$ is the orthogonal projector onto the
continuous subspace of $H.$ Moreover, from (\ref{spectralrepr}) and functional
calculus, it follows,%
\begin{equation}
\mathbf{F}^{\pm}e^{-itH}\left(  \mathbf{F}^{\pm}\right)  ^{\ast}%
=e^{-it\mathcal{M}}. \label{spectralgroup}%
\end{equation}
Assuming faster decay of the potential $V$ at infinity we now obtain further
properties of the Jost solution and of the scattering matrix.

Let us denote,
\begin{equation}
\label{m}m\left(  k,x\right)  :=e^{-ikx}f\left(  k,x\right)  .
\end{equation}

Assuming that $V \in L^{1}(\mathbb{R}),$ it is proved in the proof of
Proposition 3.2.1 in Chapter~3 of \cite{WederBook} that the function $m(k,x),
$ $k \in\overline{\mathbb{C}^{+}}\setminus\{0\}, $ is the unique solution of
the following Volterra integral equation (see equation (3.2.9) in page 53 of
\cite{WederBook}),%

\begin{equation}
m\left(  k,x\right)  =1+\int_{x}^{\infty}D\left(  y-x,k\right)  V\left(
y\right)  m\left(  k,y\right)  dy, \label{3.24}%
\end{equation}
where,
\begin{equation}
D\left(  k,x\right)  :=\frac{1}{2ik}\left[  e^{2ikx}-1\right]  =\int_{0}%
^{x}e^{2iky}\,dy. \label{ap.6}%
\end{equation}
We have the following result. \begin{prop2}\label{Proposition 3.1}
Assume that the potential $V$ is selfadjoint. Then.
\begin{enumerate}
\item Suppose that $V\in{L}^{1}_{2+\delta}(\mathbb R^+),$ for some $\delta\geq 0.$  Then,
for all $k\in\overline{\mathbb{C}^{+}}$ and $x\geq0,$ we have%
\begin{equation}
\left\vert m\left(  k,x\right)  -1\right\vert \leq C\left\langle
k\right\rangle ^{-1}\left\langle x\right\rangle ^{-1-\delta}, \label{3.2}%
\end{equation}%
\begin{equation}
\left\vert \partial_{k}m\left(  k,x\right)  \right\vert \leq C\left\langle
k\right\rangle ^{-1}\left\langle x\right\rangle ^{-\delta}. \label{3.5}%
\end{equation}
\item Assume that $V\in{L}^{1}_{2+\tilde{\delta}}(\mathbb R^+),\ $ for  some $\tilde{\delta} \geq 0,$ and that it admits a regular decomposition. Then, for
all $k\in\overline{\mathbb{C}^{+}},$ and $x\geq0,$ the following estimates
hold,%
\begin{equation}
\left\vert \partial_{x}m\left(  k,x\right)  \right\vert \leq C\left\langle
k\right\rangle ^{-1}\left\langle x\right\rangle ^{-2-\hat{\delta}}, \hat{\delta}:= \min[ \tilde{\delta}, \delta ], \label{3.6}%
\end{equation}
and%
\begin{equation}
\left\vert \partial_{k}\partial_{x}m\left(  k,x\right)  \right\vert \leq
C\left\langle k\right\rangle ^{-1}\left\langle x\right\rangle ^{-1-\hat{\delta}},   \hat{\delta}:= \min[ \tilde{\delta}, \delta].
\label{3.7}%
\end{equation}
\item Suppose that $V\in{L}^{1}_{3+\delta}(\mathbb R^+),$ for some $\delta\geq0.$ Then, for
all $k\in\overline{\mathbb{C}^{+}}$ and $x\geq0,$ we have%
\begin{equation}
\left\vert \partial_{k}^{2}m\left(  k,x\right)  \right\vert \leq C\left\langle
k\right\rangle ^{-1}\left\langle x\right\rangle ^{-\delta}, \label{3.5bis}%
\end{equation}
and%
\begin{equation}
\left\vert \partial_{k}^{2}\partial_{x}m\left(  k,x\right)  \right\vert \leq
C\left\langle x\right\rangle ^{-1-\delta}.\text{ \ } \label{3.6bis}%
\end{equation}
\item If $V\in{L}^{1}_{4+\delta}(\mathbb R^+),$then, for all $k\in\overline{\mathbb{C}^{+}%
}$ and $x\geq0,$ we have%
\begin{equation}
\left\vert \partial_{k}^{3}m\left(  k,x\right)  \right\vert \leq C\left\langle
k\right\rangle ^{-1}\left\langle x\right\rangle ^{-\delta}, \label{3.5bisbis}%
\end{equation}
and%
\begin{equation}
\left\vert \partial_{k}^{3}\partial_{x}m\left(  k,x\right)  \right\vert \leq
C\left\langle x\right\rangle ^{-1-\delta}.\text{ \ } \label{3.6bisbis}%
\end{equation}
\end{enumerate}
\end{prop2}

\begin{proof}
By equations (3.2.13), (3.2.14), (3.2.15) and (3.2.16) in page 54 of
\cite{WederBook}, and since,
\begin{equation}
|e^{z}-1|\leq C\frac{|z|}{1+|z|},\qquad z\in\mathbf{C}, \label{ap.7}%
\end{equation}
we have
\begin{equation}
|m(k,x)|\leq C,\qquad k\in\overline{\mathbb{C}^{+}},x\in\mathbb{R}^{+}.
\label{ap.8}%
\end{equation}
Then, \eqref{3.2} follows from \eqref{3.24}, the first equality in
\eqref{ap.6}, \eqref{ap.7} and \eqref{ap.8}. Differentiating with respect to
$k$ both sides of \eqref{3.24} we get,
\begin{equation}
\partial_{k}m(k,x)=\dot m_{0}(k,x)+\int_{x}^{\infty}%
\,dy\,D(k(y-x))\,V(y)\partial_{k}m(k,y), \label{ap.9}%
\end{equation}
where, we have defined,
\begin{equation}
\dot m_{0}(k,x):=\frac{1}{2ik^{2}}\,\int_{x}^{\infty}\,dy\,\left[
1-e^{2ik(y-x)}+2ik(y-x)e^{2ik(y-x)}\right]  \,V(y)\,m(k,y). \label{ap.10}%
\end{equation}
We have that,
\begin{equation}
\left\vert e^{iz}-iz-1\right\vert \leq\frac{C|z|^{2}}{1+|z|}. \label{ap.11}%
\end{equation}
Moreover, in equation (3.9.15) in page 157 of \cite{WederBook} it is proved,
\begin{equation}
\left\vert \partial_{k}m(k,x)\right\vert \leq C. \label{ap.12}%
\end{equation}
Equation \eqref{3.5} follows from the first equality in \eqref{ap.6}, and
\eqref{ap.7}-\eqref{ap.12}. This completes the proof of 1. Let us prove 2.
Taking the derivative of both sides of \eqref{3.24} with respect to $x$ and
using the first equality in \eqref{ap.6} we obtain,
\begin{equation}
\partial_{x}m(k,x)=-\int_{x}^{\infty}\,dye^{2ik(y-x)}\,V(y)-\int_{x}^{\infty
}\,dye^{2ik(y-x)}\,V(y)\,(m(k,y)-1). \label{ap.13}%
\end{equation}
Note that
\begin{equation}
|V(x)|=\left\vert \int_{x}^{\infty}\,|\partial_{y}V(y)|dy\right\vert \leq
\frac{1}{|x|^{2+\delta}}\,\int_{x}^{\infty}\,|y|^{2+\delta}\,|\partial
_{y}V(y)|,\qquad x\geq x_{N}. \label{ap.14}%
\end{equation}
Then,
\begin{equation}
|V(x)|\leq C\frac{1}{<x>^{2+\delta}}. \label{ap.15}%
\end{equation}
Using \eqref{ap.15},
\[
e^{2ik(y-x)}=\frac{1}{2ik}\,\partial_{y}e^{2ik(y-x)},
\]
that $<\cdot>^{2+\delta}\,\partial_{x}V\in L^{1}(x_{N},\infty),$ and
integrating by parts we prove that,
\begin{equation}
\left\vert \int_{x}^{\infty}\,dye^{2ik(y-x)}\,V(y)\right\vert \leq
C\left\langle k\right\rangle ^{-1}\left\langle x\right\rangle ^{-2-\delta}.
\label{ap.16}%
\end{equation}
Hence, \eqref{3.6} follows from \eqref{3.2}, \eqref{ap.13} and \eqref{ap.16}.
Equation \eqref{3.7} is proved taking the derivative with respect to $k$ of
both sides of \eqref{ap.13}, using \eqref{3.5}, and arguing as in the proof of
\eqref{3.6}. This finishes the proof of 2. We now proceed to prove 3.
Differentiating both sides of (\ref{3.24}) two times with respect to the $k$
variable, we have%
\begin{align}
\partial_{k}^{2}m\left(  k,x\right)   &  =\int_{x}^{\infty}\partial_{k}%
^{2}D\left(  y-x,k\right)  V\left(  y\right)  m\left(  k,y\right)
dy+2\int_{x}^{\infty}\partial_{k}D\left(  y-x,k\right)  V\left(  y\right)
\partial_{k}m\left(  k,y\right)  dy\nonumber\\
&  +\int_{x}^{\infty}D\left(  y-x,k\right)  V\left(  y\right)  \partial
_{k}^{2}m\left(  k,y\right)  dy. \label{116}%
\end{align}
From the second equality in \eqref{ap.6} we obtain,
\begin{equation}
\partial_{k}^{j}D(k,x)=(2i)^{j}\,\int_{0}^{x}\,y^{j}\,e^{2iky}dy,\qquad
j=0,1,\dots. \label{ap.17}%
\end{equation}
Hence, by \eqref{ap.17}
\begin{equation}
\left\vert \partial_{k}^{j}D(k,x)\right\vert \leq\frac{2^{j}}{j+1}%
x^{j+1},\qquad j=0,1,\dots. \label{ap.18}%
\end{equation}
Moreover, using again \eqref{ap.17}, and integrating by parts,
\begin{equation}
\partial_{k}^{j}D(k,x)=\frac{(2i)^{j}}{2ik}\,\int_{0}^{x}\,dy\,y^{j}%
\,\partial_{y}e^{2iky}=\frac{(2i)^{j}}{2ik}\left[  x^{j}e^{2ikx}-j\int_{0}%
^{x}\,dy\,y^{j-1}e^{2iky}\right]  . \label{ap.19}%
\end{equation}
By \eqref{ap.19}
\begin{equation}
\left\vert \partial_{k}^{j}D(k,x)\right\vert \leq2^{j}\,\frac{x^{j}}%
{|k|},\qquad j=1,2,\dots. \label{ap.20}%
\end{equation}

Moreover, by \eqref{ap.18} and \eqref{ap.20}%
\begin{equation}
\left\vert \partial_{k}^{j}D\left(  k,x\right)  \right\vert \leq C\min\left\{
\left\vert x\right\vert ^{j+1},\frac{\left\vert x\right\vert ^{j}}{\left\vert
k\right\vert }\right\}  \leq C\frac{\left\langle x\right\rangle ^{j+1}%
}{\left\langle k\right\rangle }, \label{ap.21}%
\end{equation}
for $j=1,2,\dots,$ and using (\ref{3.5}), and (\ref{ap.8}) we show that the
two first terms on the right-hand side of (\ref{116}) are controlled as
\begin{equation}%
\begin{array}
[c]{l}%
\left\vert \int_{x}^{\infty}\partial_{k}^{2}D\left(  y-x,k\right)  V\left(
y\right)  m\left(  k,y\right)  dy+2\int_{x}^{\infty}\partial_{k}D\left(
y-x,k\right)  V\left(  y\right)  \partial_{k}m\left(  k,y\right)
dy\right\vert \\
\\
\leq C\left\langle k\right\rangle ^{-1}\left\langle x\right\rangle ^{-\delta
}\left\Vert V\right\Vert _{{L}_{3+\delta}^{1}},
\end{array}
\label{ap.22}%
\end{equation}
for $\delta\geq0.$ The Volterra integral equation \eqref{116} is solved by
successive approximations as follows,
\begin{equation}
\partial_{k}^{2}m\left(  k,x\right)  =\partial_{k}^{2}m\left(  k,x\right)
_{0}+\sum_{l=1}^{\infty}\,\partial_{k}^{2}m\left(  k,x\right)  _{l},
\label{ap.23}%
\end{equation}
where,
\begin{equation}
\partial_{k}^{2}m\left(  k,x\right)  _{0}:=\int_{x}^{\infty}\partial_{k}%
^{2}D\left(  y-x,k\right)  V\left(  y\right)  m\left(  k,y\right)
dy+2\int_{x}^{\infty}\partial_{k}D\left(  y-x,k\right)  V\left(  y\right)
\partial_{k}m\left(  k,y\right)  \,dy, \label{ap.24}%
\end{equation}
and
\begin{equation}
\partial_{k}^{2}m\left(  k,x\right)  _{l+1}=\int_{x}^{\infty}D\left(
y-x,k\right)  V\left(  y\right)  \partial_{k}^{2}m\left(  k,y\right)
_{l}dy,\qquad l=1,2,\dots. \label{ap.25}%
\end{equation}
Further, by \eqref{ap.18} with $j=0,$ \eqref{ap.22}, \eqref{ap.24} and
iterating \eqref{ap.25} we obtain,
\begin{equation}
|\partial_{k}^{2}m\left(  k,x\right)  _{l}|\leq C\frac{1}{l!}\,\left(
\int_{x}^{\infty}\,y|V(y)|\,dy\right)  ^{l}\,\left\langle k\right\rangle
^{-1}\left\langle x\right\rangle ^{-\delta}\left\Vert V\right\Vert
_{L^{1}_{3+\delta}(\mathbb{R}^{+})}. \label{ap.26}%
\end{equation}
Hence, by \eqref{ap.22}-\eqref{ap.24} and \eqref{ap.26} we have that,
\begin{equation}
|\partial_{k}^{2}m\left(  k,x\right)  |\leq Ce^{\int_{x}^{\infty
}\,y|V(y)|\,dy}\,\left\langle k\right\rangle ^{-1}\left\langle x\right\rangle
^{-\delta}\left\Vert V\right\Vert _{L^{1}_{3+\delta}(\mathbb{R}^{+})}.
\label{ap.27}%
\end{equation}
This proves \eqref{3.5bis}.

Derivating both sides of (\ref{ap.13}) two times with respect to $k$ we get%
\begin{align*}
\partial_{k}^{2}\partial_{x}m\left(  k,x\right)   &  =-\left(  2i\right)  ^{2}%
{\displaystyle\int\limits_{x}^{\infty}}
\left(  y-x\right)  ^{2}e^{2ik\left(  y-x\right)  }V\left(  y\right)  m\left(
k,y\right)  dy\\
&  -4i%
{\displaystyle\int\limits_{x}^{\infty}}
e^{2ik\left(  y-x\right)  }\left(  y-x\right)  V\left(  y\right)  \partial
_{k}m_{+}\left(  k,y\right)  dy-%
{\displaystyle\int\limits_{x}^{\infty}}
e^{2ik\left(  y-x\right)  }V\left(  y\right)  \partial_{k}^{2}m_{+}\left(
k,y\right)  dy.
\end{align*}
Then, using (\ref{3.5}), (\ref{3.5bis}), and (\ref{ap.8}), we deduce
(\ref{3.6bis}). This completes the proof of 3. We prove \eqref{3.5bisbis}
derivating both sides of \eqref{116} and using an iteration argument as in the
proof of \eqref{3.5bis}. Finally we obtain \eqref{3.6bisbis} arguing in a
similar way as in the proof of \eqref{3.6bis}.
\end{proof}

In the following lemma we collect results on the scattering matrix that we need.

\begin{lemma2}\label{scontinuity}
Assume that the potential $V$ is self-adjoint, and that the constant boundary matrices $A,B$ fulfill \eqref{ap.4}, and \eqref{ap.5}. Then:
\begin{enumerate}
\item
Suppose that $V \in L^{1}_{1}(\mathbb R^+).$  Then, the scattering matrix $S(k)$ is continuous for $ k \in \mathbb R,$ and
\begin{equation}
\left\vert S\left(  k\right)  -S_{0}\left(  k\right)  \right\vert \leq
C\left\langle k\right\rangle ^{-1},\label{S-S0}%
\end{equation}
where $S_0(k)$ is the scattering matrix with the boundary condition \eqref{ap.3} given by the boundary matrices $A,B$ and the potential identically zero.
\item
Moreover, if $V\in L^{1}_{2}(\mathbb R^+)$ the following relation
is  true
\begin{equation}
\left\vert S\left(  k\right)  -S\left(  0\right)  \right\vert \leq
Ck\left\langle k\right\rangle ^{-1}. \label{S-1}%
\end{equation}
\item
Further, if $V\in {L}^{1}_{3}(\mathbb R^+), S\in C^{1}\left(  \mathbb{R}\right)  $
and
\begin{equation}
\left\vert \dot{S}\left(  k\right)  \right\vert \leq C\left\langle
k\right\rangle ^{-1}.\label{scatmatrixderiv}%
\end{equation}
\item
Furthermore, if $V\in {L}^{1}_{4}(\mathbb R^+),$ $S\in C^{2}\left(  \mathbb{R}\right)
$ and
\begin{equation}
\left\vert \ddot{S}\left(  k\right)  \right\vert \leq C\left\langle
k\right\rangle ^{-1}.\label{scatmatrixsecondderiv}%
\end{equation}
\end{enumerate}
\end{lemma2}

\begin{proof}
By the definition of $S(k)$ in \eqref{Scatteringmatrix} and by Theorem 3.8.1
in page 114, Theorem 3.8.13 in page 137, and Theorem 3.8.14 in pages 138 and
139, of \cite{WederBook} $S(k)$ is continuous for $k\in\mathbb{R}.$ Further
\eqref{S-S0} is proved in equation (3.10.41) in page 198 of \cite{WederBook}.
This proves 1. Item 2 is proved in Theorem 3.9.15 in pages 189 and 190 of
\cite{WederBook}. Item 3 is proved in Propositions ~A.1 and A.2 of \cite{Wlp}.
Let us prove 4. By \eqref{Jostmatrix} and \eqref{m}
\begin{equation}
\dot{J}(k)=-\dot{m}(-k,0)^{\dagger}B-im(-k,0)^{\dagger}A+ik\dot{m}%
(-k,0)^{\dagger}A+\dot{m}^{\prime\dagger}A,\qquad k\in\mathbb{R},
\label{ap.28}%
\end{equation}
where by $\dot{m}(k,x)$ we denote the derivative with respect to $k$ of
$m(k,x).$ By Proposition ~\ref{Jostnozero}, ${J}(k)$ is continuous for
$k\in\mathbb{R}.$ Further, by equation (3.9.17) in page 158 of
\cite{WederBook}
\begin{equation}
|\dot{m}^{\prime}(k,0)|\leq C,\qquad k\in\overline{\mathbb{C}^{+}}.
\label{ap.29}%
\end{equation}
Then, by \eqref{ap.8}, \eqref{3.5},\eqref{ap.28} and \eqref{ap.29}
\begin{equation}
\dot{J}(k)=O(1)A+O\left(  \frac{1}{|k|}\right)  ,\qquad|k|\rightarrow
\infty,\text{ }k\in\mathbb{R}. \label{ap.30}%
\end{equation}
Derivating (\ref{ap.28}) and using (\ref{3.5}), (\ref{3.5bis}) and
(\ref{3.6bis}) we show that $\ddot{J}\left(  k\right)  $ is continuous for
$k\in\mathbb{R}$ and
\begin{equation}
\ddot{J}\left(  k\right)  =O(1)A+O\left(  \frac{1}{|k|}\right)  ,\qquad
|k|\rightarrow\infty,\text{ }k\in\mathbb{R}. \label{ap.31}%
\end{equation}
Moreover, derivating (\ref{Scatteringmatrix}) we have%
\begin{equation}%
\begin{array}
[c]{l}%
\ddot{S}\left(  k\right)  =-\ddot{J}\left(  -k\right)  J\left(  k\right)
^{-1}-2\dot{J}\left(  -k\right)  J\left(  k\right)  ^{-1}\dot{J}\left(
k\right)  J\left(  k\right)  ^{-1}-\\
\\
2J\left(  -k\right)  J\left(  k\right)  ^{-1}\dot{J}\left(  k\right)  J\left(
k\right)  ^{-1}\dot{J}\left(  k\right)  J\left(  k\right)  ^{-1}-J\left(
-k\right)  J\left(  k\right)  ^{-1}\ddot{J}(k)J\left(  k\right)  ^{-1}.
\end{array}
\label{153}%
\end{equation}
As by Proposition~\ref{Jostnozero} $J(k)$ is continuous and invertible for
$k\in\mathbb{R}\setminus\{0\}$ it follows from (\ref{153}) that $\ddot{S}(k) $
is continuous for $k\in\mathbb{R}\setminus\{0\}.$ Further, by equation
(3.7.11) in page 113 and equation 3.10.24) in page 195 of \cite{WederBook},
\begin{equation}
J\left(  k\right)  ^{-1}=O(1),\qquad k\rightarrow\infty,\text{ \textrm{in}%
}\,\overline{\mathbb{C}^{+}}. \label{ap.32}%
\end{equation}
Moreover, by equations (3.7.3) and (3.7.4) in page 113 and equations (3.10.17)
and (3.10.18) in page 194 of \cite{WederBook},
\begin{equation}
J\left(  -k\right)  J\left(  k\right)  ^{-1}=O(1),\qquad|k|\rightarrow
\infty,\text{ }k\in\mathbb{R}. \label{ap.33}%
\end{equation}

Furthermore, by equation (3.7.12) in page 113 and taking the inverse in both
sides of equation (3.10.23) in page 195 of \cite{WederBook},
\begin{equation}
AJ\left(  k\right)  ^{-1}=O\left(  \frac{1}{|k|}\right)  ,\qquad
|k|\rightarrow\infty,\text{ \textrm{in}}\,\overline{\mathbb{C}^{+}}.
\label{ap.34}%
\end{equation}

Hence by \eqref{ap.30}-\eqref{ap.34}
\begin{equation}
\label{ap.33.1}\ddot{S}\left(  k\right)  =O\left(  \frac{1}{|k|}\right)  ,
\qquad|k| \to\infty, k \in\mathbb{R}.
\end{equation}
Moreover, if $J\left(  0\right)  $ is invertible, $\ddot{S}\left(  k\right)  $
is continuous also at $k=0.$ In the case where $J\left(  0\right)  $ is not
invertible, it remains for us to control the behaviour of $\ddot{S}\left(
k\right)  $ in a neighborhood of $k=0.$ Derivating twice with respect to $k$
both sides of \eqref{ap.28} and using \eqref{3.5bis}-\eqref{3.6bisbis} we see
that $\dddot{J}(k)$ is continuous for $k\in\mathbb{R}$ and then,
\begin{equation}
\label{ap.34.1}J\left(  k\right)  =J\left(  0\right)  +\dot{J}\left(
0\right)  k+\frac{\ddot{J}\left(  0\right)  }{2}k^{2}+\frac{\dddot{J}\left(
0\right)  }{6}k^{3}+o\left(  k^{3}\right)  ,\qquad k \to0, k \in\mathbb{R},
\end{equation}
\begin{equation}
\label{ap.35}\dot{J}\left(  k\right)  =\dot{J}\left(  0\right)  +\ddot
{J}\left(  0\right)  k+\frac{\dddot{J}\left(  0\right)  }{2}k^{2}+o\left(
k^{2}\right)  , \qquad k \to0, k \in\mathbb{R},
\end{equation}
and
\begin{equation}
\label{ap.36}\ddot{J}(k)= \ddot{J}(0)+ \dddot{J}(0) k+ o(|k|), \qquad k \to0,
k \in\mathbb{R}.
\end{equation}
Moreover, by \eqref{3.5}, \eqref{3.5bis} and \eqref{3.5bisbis}, for any fixed
$a \geq0,$
\begin{equation}
\label{ap.37}f(k,a)= f(0,a)+ \dot{f}(0,a) k + \ddot{f}(0,a) \frac{k^{2}}{2}+
\dddot{f}(0,a)\frac{k^{3}}{6}+o(k^{3}), \qquad k \to0, k \in\overline
{\mathbb{C}^{+}}.
\end{equation}
In equation (3.9.239) in page 190 of \cite{WederBook} the following small $k$
expansion of $S(k)$ is given, assuming that $V\in L^{1}_{2}(\mathbb{R}^{+}),$
\[
S(k)= S(0)+ \dot{S}(0) k+ o(k), \qquad\qquad k \to0.
\]
It is easily verified that if $V\in L^{1}_{4}(\mathbb{R}^{+})$ the error term
$o(k)$ can be improved using \eqref{ap.37} to obtain,
\begin{equation}
\label{ap.38}S(k)= S(0)+ \dot{S}(0) k+ \mathcal{H }k^{2}+ o\left(
k^{2}\right)  , \qquad\qquad k \to0,
\end{equation}
for some constant matrix $\mathcal{H}.$ Furthermore in equation (3.9.237) in
page 189 of \cite{WederBook} it is proved that if $V\in L^{1}_{2}%
(\mathbb{R}^{+}),$
\begin{equation}
\label{ap.39}J(k)^{-1}= \frac{1}{k} \, \mathcal{M}+ \mathcal{E}_{1}+o(1),
\qquad k\to0 \, \text{\textrm{in}}\, \overline{\mathbb{C}^{+}} ,
\end{equation}
where $\mathcal{M}$ and $\mathcal{E}_{1}$ are constant matrices. It is also
easily checked that if $V\in L^{1}_{4}(\mathbb{R}^{+})$ the error term $o(1)$
can be improved using \eqref{ap.37} to obtain,
\begin{equation}
\label{ap.40}J(k)^{-1}= \frac{1}{k} \, \mathcal{M}+ \mathcal{E}_{1}+
\mathcal{E}_{2} k+ o(k), \qquad k\to0 \, \text{\textrm{in}} \,\overline
{\mathbb{C}^{+}},
\end{equation}
for a constant matrix $\mathcal{E}_{2}.$ Using \eqref{Scatteringmatrix} we
write \eqref{153} as follows,
\begin{equation}
\label{ap.41}%
\begin{array}
[c]{l}%
\ddot{S}\left(  k\right)  =-\ddot{J}\left(  -k\right)  J\left(  k\right)
^{-1}- 2\dot{J}\left(  - k\right)  J\left(  k\right)  ^{-1}\dot{J}\left(
k\right)  J\left(  k\right)  ^{-1}+\\
\\
2 S(k)\dot{J}\left(  k\right)  J\left(  k\right)  ^{-1} \dot{J}\left(  k
\right)  J\left(  k\right)  ^{-1} - S(k) \ddot{J}(k) J\left(  k\right)  ^{-1}.
\end{array}
.
\end{equation}

Then, by \eqref{ap.34.1}-\eqref{ap.36}, \eqref{ap.40} and \eqref{ap.41}
\begin{equation}
\label{ap.42}\ddot{S}\left(  k\right)  =\frac{C_{1}}{k}+\frac{C_{2}} {k^{2}}+
g\left(  k\right)  , \qquad k \to0,
\end{equation}
for some constants matrices $C_{1},C_{2}$ and a function $g\in C(\mathbb{R}).
$ Moreover, integrating \eqref{ap.42} we see that for $0 < k < \varepsilon,$
\begin{equation}
\label{ap.43}\dot{S}\left(  k\right)  = \dot{S}\left(  \varepsilon\right)
+C_{1} \ln\left(  \frac{k}{\varepsilon}\right)  + C_{2} \left(  \frac{1}{k}-
\frac{1}{\varepsilon}\right)  - \int_{k}^{\varepsilon}g(\lambda) d\lambda.
\end{equation}
Since $\dot{S}\left(  k\right)  $ is continuous for all $k\in\mathbb{R}$,
\eqref{ap.43} implies that $C_{1}=C_{2}=0.$ Therefore, $S\in C^{2}\left(
\mathbb{R}\right)  $ and by \eqref{ap.33.1}, (\ref{scatmatrixsecondderiv}) is true.
\end{proof}

We prepare the following results that we need.

\begin{lemma2}\label{Lemma2}
Assume that the potential $V$ is self-adjoint,  and that the constant boundary
matrices $A,B$ fulfill \eqref{ap.4}, and \eqref{ap.5}. Then.
\begin{enumerate}
\item
Suppose that  $V\in L^{1}(\mathbb{R}^{+}),$  and that $H$ has no negative
eigenvalues. Then,
\begin{equation}
\label{61}\left\Vert k\left(  \mathbf{F}^{\pm}\psi\right)  \left(  k\right)
\right\Vert _{{L}^{2}(\mathbb{R}^{+})}\leq C\left\Vert \psi\right\Vert
_{{H}^1_{A,B}(\mathbb{R}^{+})}, \qquad\psi\in
{H}^1_{A,B}(\mathbb{R}^{+}).
\end{equation}
\item Assume that $V\in L^{1}(\mathbb{R}^{+}) \cap L^{\infty}(\mathbb{R}%
^{+}),$ and that $H$ has no negative eigenvalues. Then, for all $\psi
\in {H}^{2}(\mathbb R^+)$ that satisfy the boundary condition \eqref{ap.3}
\begin{equation}
\left\Vert k^{2}\left(  \mathbf{F}^{\pm}\psi\right)  \left(  k\right)
\right\Vert _{{L}^{2}(\mathbb{R}^{+})}\leq C\left\Vert \psi\right\Vert
_{{H}^{2}(\mathbb R^+)}.\label{62}%
\end{equation}
\item Suppose that $V\in L^{1}_{ 3}(\mathbb{R}^{+}),$ and that it admits a regular decomposition (see Definition ~\ref{Def1}). Then,
\begin{equation}
\left\Vert k\partial_{k}\mathbf{F}^{\pm}\psi\left(  k\right)  \right\Vert
_{{L}^{2}(\mathbb{R}^{+})}\leq C\left\Vert \psi\right\Vert _{{H}^{1,1}(\er)%
}, \qquad\psi\in {H}^{1,1}(\er). \label{61BIS}%
\end{equation}
\item Assume that $ V  \in L^{1}_{ 3}(\mathbb{R}^{+}).$ Then,
\beq\label{61BIS.2}
\left\| \partial_k \mathbf F^{\pm} \psi\left(k\right) \right\|_{L^2(\er)}\leq C \|\psi\|_{L^2_1(\er)}, \qquad \psi \in L^2_1(\er).
\ene
\item Assume that $ V  \in L^{1}_{ 4}(\mathbb{R}^{+}).$ Then,
\beq\label{61BIS.3}
\left\| \partial_k^2 \mathbf F^{\pm} \psi\left(k\right) \right\|_{L^2(\er)}\leq C \|\psi\|_{L^2_2(\er)}, \qquad \psi \in L^2_2(\er).
\ene
\end{enumerate}
\end{lemma2}

\begin{proof}
If $H$ has no eigenvalues, its absolutely continuous subspace coincides with
$L^{2}(\mathbb{R}^{+}),$ and hence, the generalized Fourier maps
$\mathbf{F}^{\pm}$ are unitary on $L^{2}(\mathbb{R}^{+})$ (see
Theorem~\ref{theospec}). Further, by \eqref{spectralrepr} and functional
calculus,
\begin{equation}
k=\mathbf{F}^{\pm}\sqrt{H}(\mathbf{F}^{\pm})^{\dagger}. \label{ap.44}%
\end{equation}
By \eqref{ap.44}
\begin{equation}
\Vert k\mathbf{F}^{\pm}\psi\Vert_{L^{2}(\mathbb{R}^{+})}^{2}=\Vert\sqrt{H}%
\psi\Vert_{L^{2}(\mathbb{R}^{+})}^{2}=h(\psi,\psi), \label{ap.45}%
\end{equation}
where $h$ is the quadratic form of $H$ defined in \eqref{quadratic}. Further,
by \eqref{quadratic} and Sobolev's inequality,
\begin{equation}
|q(\psi,\psi)|\leq C\|\psi\|_{ {H}^{1}_{A,B}(\mathbb{R}^{+})},\qquad\psi\in
{H}^{1}_{A,B}(\mathbb{R}^{+}). \label{ap.46}%
\end{equation}
Equation \eqref{61} follows from \eqref{ap.45} and \eqref{ap.46}. This proves
1. Let us prove 2. By \eqref{domainh2} we have, $\psi\in D[H].$ Further, by
\eqref{spectralrepr}
\begin{equation}
k^{2}=\mathbf{F}^{\pm}H(\mathbf{F}^{\pm})^{\dagger}. \label{ap.47}%
\end{equation}
Then,
\begin{equation}
\Vert k^{2}\mathbf{F}^{\pm}\psi\Vert_{L^{2}(\mathbb{R}^{+})}=\Vert H\psi
\Vert_{L^{2}(\mathbb{R}^{+})}\leq C\Vert\psi\Vert_{H^{2}(\mathbb{R}^{+})}.
\label{ap.48}%
\end{equation}
This concludes the proof of 2. Let us prove 3. We give the proof for
$\mathbf{F}^{+}.$ The case of $\mathbf{F}^{-}$ follows similarly. We
calculate
\begin{align}
\partial_{k}\Psi\left(  -k,y\right)  ^{\dagger}  &  =e^{-iky}\partial
_{k}m\left(  k,y\right)  ^{\dagger}+e^{iky}S\left(  -k\right)  ^{\dagger
}\partial_{k}m\left(  -k,y\right)  ^{\dagger}+e^{iky}\partial_{k}S\left(
-k\right)  ^{\dagger}m\left(  -k,y\right)  ^{\dagger}\nonumber\\
&  -iy\left(  f\left(  k,y\right)  ^{\dagger}-S\left(  -k\right)  ^{\dagger
}f\left(  -k,y\right)  ^{\dagger}\right)  . \label{c1}%
\end{align}
Integrating by parts, we have%
\[
\int_{0}^{\infty}ke^{-iky}\partial_{k}m\left(  k,y\right)  ^{\dagger}%
\psi\left(  y\right)  dy=-i\partial_{k}m\left(  k,0\right)  ^{\dagger}%
\psi\left(  0\right)  -i\int_{0}^{\infty}e^{-iky}\partial_{y}\left(
\partial_{k}m\left(  k,y\right)  ^{\dagger}\psi\left(  y\right)  \right)  dy.
\]
Then, using (\ref{3.5}) with $\delta>1/2$ and (\ref{3.7}), via Sobolev's
inequality we get%
\begin{equation}
\left\Vert \int_{0}^{\infty}ke^{-iky}\partial_{k}m\left(  k,y\right)
^{\dagger}\psi\left(  y\right)  dy\right\Vert _{{L}^{2}(\mathbb{R}^{+})}\leq
C\left\Vert \psi\right\Vert _{H^{1}(\mathbb{R}^{+})}. \label{c2}%
\end{equation}
Similarly, by using also (\ref{UnitaritySM}) we have%
\begin{equation}
\left\Vert \int_{0}^{\infty}ke^{iky}S\left(  -k\right)  ^{\dagger}\partial
_{k}m\left(  -k,y\right)  ^{\dagger}\psi\left(  y\right)  dy\right\Vert
_{{L}^{2}(\mathbb{R}^{+})}\leq C\left\Vert \psi\right\Vert _{H^{1}%
(\mathbb{R}^{+})}. \label{c3}%
\end{equation}
By (\ref{scatmatrixderiv}) and Parseval's identify
\begin{equation}
\left\Vert \int_{0}^{\infty} k e^{iky}\dot S\left(  -k\right)  ^{\dagger}%
\psi\left(  y\right)  dy\right\Vert _{{L}^{2}(\mathbb{R}^{+})}\leq C\Vert
\psi\Vert_{L^{2}(\mathbb{R}^{+})}. \label{ap.49}%
\end{equation}
Further, by \eqref{3.2} and (\ref{scatmatrixderiv}),
\begin{equation}
\left\Vert \int_{0}^{\infty} k e^{iky}\dot S\left(  -k\right)  ^{\dagger
}(m(-k,y)^{\dagger}-I)\psi\left(  y\right)  dy\right\Vert _{{L}^{2}%
(\mathbb{R}^{+})}\leq C\Vert\psi\Vert_{L^{2}(\mathbb{R}^{+})}. \label{ap.50}%
\end{equation}
By \eqref{ap.49} and \eqref{ap.50},
\begin{equation}
\left\Vert \int_{0}^{\infty}\, k e^{iky}\dot S\left(  -k\right)  ^{\dagger
}\left(  m\left(  -k,y\right)  ^{\dagger}\right)  \psi\left(  y\right)
dy\right\Vert _{{L}^{2}(\mathbb{R}^{+})}\leq C\left\Vert \psi\right\Vert
_{{L}^{2}(\mathbb{R}^{+})}. \label{c4}%
\end{equation}
Integrating by parts, and recalling that $f(k,y)= e^{iky}\, m(k,y), $ we have
\begin{equation}
\int_{0}^{\infty}kf(k,y)^{\dagger}iy\psi(y)dy=-i\int_{0}^{\infty}%
e^{-iky}\partial_{y}\left(  m\left(  k,y\right)  ^{\dagger}iy\psi\left(
y\right)  \right)  dy. \label{ap.51}%
\end{equation}
By \eqref{3.6}
\begin{equation}
\left\Vert \int_{0}^{\infty}e^{-iky}\left(  \partial_{y}m\left(  k,y\right)
^{\dagger}\right)  y\psi\left(  y\right)  dy\right\Vert _{L^{2}(\mathbb{R}%
^{+})}\leq\Vert\psi\Vert_{L^{2}(\mathbb{R}^{+})}. \label{ap.52}%
\end{equation}
By Parseval's identity,
\begin{equation}
\left\Vert \int_{0}^{\infty}e^{-iky}\psi\left(  y\right)  \right\Vert
_{L^{2}(\mathbb{R}^{+})}\leq C\Vert\psi\Vert_{L^{2}(\mathbb{R}^{+})}.
\label{ap.53}%
\end{equation}
By \eqref{3.2}
\begin{equation}
\left\Vert \int_{0}^{\infty}e^{-iky}\left(  m\left(  k,y\right)  ^{\dagger
}-I\right)  \psi\left(  y\right)  dy\right\Vert _{L^{2}(\mathbb{R}^{+})}\leq
C\Vert\psi\Vert_{L^{2}(\mathbb{R}^{+})}. \label{ap.54}%
\end{equation}
It follows from \eqref{ap.53} and \eqref{ap.54},
\begin{equation}
\left\Vert \int_{0}^{\infty}e^{-iky}m\left(  k,y\right)  ^{\dagger}\psi\left(
y\right)  dy\right\Vert _{L^{2}(\mathbb{R}^{+})}\leq C\Vert\psi\Vert
_{L^{2}(\mathbb{R}^{+})}. \label{ap.55}%
\end{equation}
In the same way we prove,
\begin{equation}
\left\Vert \int_{0}^{\infty}e^{-iky}m\left(  k,y\right)  ^{\dagger}y
\partial_{y} \psi(y) \left(  y\right)  dy\right\Vert _{L^{2}(\mathbb{R}^{+}%
)}\leq C\Vert y \partial_{y} \psi\Vert_{L^{2}(\mathbb{R}^{+})}. \label{ap.56}%
\end{equation}
By \eqref{ap.51}, \eqref{ap.52}, \eqref{ap.55} and \eqref{ap.56},
\begin{equation}
\left\Vert \int_{0}^{\infty}kf\left(  k,y\right)  ^{\dagger}iy\psi\left(
y\right)  dy\right\Vert _{{L}^{2}(\mathbb{R}^{+})}\leq C\left\Vert
\psi\right\Vert _{H^{1,1}(\mathbb{R}^{+})}. \label{c5}%
\end{equation}
Similarly, using (\ref{UnitaritySM}) we show%
\begin{equation}
\left\Vert \int_{0}^{\infty}kS\left(  -k\right)  ^{\dagger}iyf\left(
-k,x\right)  ^{\dagger}\psi\left(  y\right)  dy\right\Vert _{{L}%
^{2}(\mathbb{R}^{+})}\leq C\left\Vert \psi\right\Vert _{H^{1,1}(\mathbb{R}%
^{+})}. \label{c6}%
\end{equation}
Hence, (\ref{61BIS}) follows from \eqref{gefoma}, (\ref{c1}), (\ref{c2}),
(\ref{c3}),~(\ref{c4}), (\ref{c5}) and (\ref{c6}).

Let us now prove 4 in the case of $\mathbf{F}^{+}.$ The case of $\mathbf{F}%
^{-}$ follows similarly. By \eqref{Physicalsolution}
\begin{equation}
\label{c6.1}\partial_{k} \psi(-k,x)^{\dagger}= T_{1}(k,y)+T_{2}(k,y),
\end{equation}
where,
\begin{equation}
\label{c6.2}T_{1}(k,y):= -iy e^{-iky} +i y S(-k)^{\dagger}e^{iky} + e^{iky}
\partial_{k} S(-k)^{\dagger},
\end{equation}
and
\begin{equation}
\label{c6.3}%
\begin{array}
[c]{l}%
T_{2}(k,y):= e^{-iky} \partial_{k} m(k,y)^{\dagger}+ e^{iky} S(-k)^{\dagger
}\partial_{k} m(-k,y)\\
\\
+ e^{iky} \partial_{k}S(-k)^{\dagger}(m(-k,y)^{\dagger}-I)- iy e^{-iky}
(m(k,y)^{\dagger}-I)+ iy e^{iky} S(-k)^{\dagger}(m(-k,y)^{\dagger}-I).
\end{array}
\end{equation}
By \eqref{UnitaritySM}, \eqref{scatmatrixderiv} and Parseval's identity,
\begin{equation}
\label{c6.4}\left\|  \int_{0}^{\infty}\, T_{1}(k,y) \psi(y)\, dy \right\|
_{L^{2}(\mathbb{R}^{+})} \leq C \|\psi\|_{L^{2}_{1}(\mathbb{R}^{+})}.
\end{equation}
Further, by \eqref{UnitaritySM}, \eqref{scatmatrixderiv}, \eqref{3.2} and
\eqref{3.5},
\begin{equation}
\label{c6.5}\left\|  \int_{0}^{\infty}\, T_{2}(k,y) \psi(y)\, dy \right\|
_{L^{2}(\mathbb{R}^{+})} \leq C \|\psi\|_{L^{2}(\mathbb{R}^{+})}.
\end{equation}
Equation \eqref{61BIS.2} follows from \eqref{gefoma}, \eqref{c6.1},
\eqref{c6.4} and \eqref{c6.5}.

Finally, we prove 5 in the case of $\mathbf{F}^{+}.$ The case of
$\mathbf{F}^{-}$ follows in a similar way. By \eqref{c6.1}, \eqref{c6.2} and
\eqref{c6.3}
\begin{equation}
\label{c6.7}\partial_{k}^{2} \psi(-k,^{2}x)^{\dagger}= T_{3}(k,y)+T_{4}(k,y),
\end{equation}
where,
\begin{equation}
\label{c6.8}T_{3}(k,y):= -y^{2} e^{-iky} - y^{2} S(-k)^{\dagger}e^{iky}+ iy
e^{iky} \partial_{k} S(-k)^{\dagger}+ e^{iky} \partial^{2}_{k} S(-k)^{\dagger
}.
\end{equation}
and
\begin{equation}
\label{c6.9}%
\begin{array}
[c]{l}%
T_{4}(k,y):= -iy e^{-iky} \partial_{k} m(k,y)^{\dagger}+ e^{-iky} \partial
^{2}_{k} m(k,y)^{\dagger}+ iy e^{iky} S(-k)^{\dagger}\partial_{k} m(-k,y) +
e^{iky} S(-k)^{\dagger}\partial^{2}_{k} m(-k,y )\\
\\
+ e^{iky} \partial_{k} S(-k)^{\dagger}\partial_{k} m(-k,y) +iy e^{iky}
\partial_{k}S(-k)^{\dagger}(m(-k,y)^{\dagger}-I)+ e^{iky} \partial^{2}_{k}
S(-k)^{\dagger}(m(-k,y)^{\dagger}-I)\\
\\
+ e^{iky} (\partial_{k}S(-k)^{\dagger}) \partial_{k} m(-k,y)^{\dagger}- y^{2}
e^{-iky} (m(k,y)^{\dagger}-I)\\
\\
- iy e^{-iky} \partial_{k} m(k,y)^{\dagger}-y^{2} e^{iky} S(-k)^{\dagger
}(m(-k,y)^{\dagger}-I) -y^{2} e^{iky} S(-k)^{\dagger}(m(-k,y)^{\dagger}-I)\\
\\
+ iy e^{iky} \partial_{k} S(-k)^{\dagger}(m(-k,y)^{\dagger}-I) +iy e^{iky}
S(-k)^{\dagger}\partial_{k} m(-k,y)^{\dagger}.
\end{array}
\end{equation}

By \eqref{UnitaritySM}, \eqref{scatmatrixderiv}, \eqref{scatmatrixsecondderiv}
and Parseval's identity,
\begin{equation}
\label{c6.10}\left\|  \int_{0}^{\infty}\, T_{3}(k,y) \psi(y)\, dy \right\|
_{L^{2}(\mathbb{R}^{+})} \leq C \|\psi\|_{L^{2}_{2}(\mathbb{R}^{+})}.
\end{equation}
Further, by \eqref{UnitaritySM}, \eqref{scatmatrixderiv},
\eqref{scatmatrixsecondderiv}, \eqref{3.2}, \eqref{3.5}, and \eqref{3.5bis}
\begin{equation}
\label{c6.11}\left\|  \int_{0}^{\infty}\, T_{4}(k,y) \psi(y)\, dy \right\|
_{L^{2}(\mathbb{R}^{+})} \leq C \|\psi\|_{L^{2}(\mathbb{R}^{+})}.
\end{equation}
Equation \eqref{61BIS.3} follows from \eqref{gefoma}, \eqref{c6.7},
\eqref{c6.10} and \eqref{c6.11}.
\end{proof}

\subsection{The evolution group}

\label{intev} In this subsection we obtain a convenient representation of the
restriction of the evolution group to the continuous subspace of the Hamiltonian.

The evolution group is given by,
\[
\mathcal{U}(t):=e^{-itH}.
\]
The restriction of $\mathcal{U}(t)$ to the continuous subspace of $H$ is given
by,
\[
\mathcal{U}(t)P_{c}(H).
\]
We derive our representation of $\mathcal{U}(t)P_{c}$ using the generalized
Fourier map $\mathbf{F}^{+}.$ Note, however, that we could also obtain a
representation using $\mathbf{F}^{-}.$ To simplify the notation we denote by
$\mathbf{F}$ the generalized Fourier map $\mathbf{F}^{+}.$

By \eqref{spectralgroup} and as $\mathbf{F}^{\dagger}\,\mathbf{F }= P_{c}(H),
\mathbf{F }\mathbf{F}^{\dagger}=I $ (see Theorem~\ref{theospec}) ,
\begin{equation}
\mathcal{U}(t) P_{\operatorname*{c}} =\mathbf{F}^{\ast}e^{-it\mathcal{M}%
}\mathbf{F}. \label{2.13}%
\end{equation}
By \eqref{gefoma}, \eqref{fad}, and \eqref{2.13}
\begin{equation}
\label{Spectralrepresentation1}%
\begin{array}
[c]{l}%
\mathcal{U}(t)\,P_{\operatorname*{c}}(H)\psi(x)=\frac{1}{\sqrt{2\pi}}\int%
_{0}^{\infty} f(k,x)e^{-itk^{2}}\mathbf{F}\psi\left(  k\right)  dk\\
\\
+ \int_{0}^{\infty}\, f(-k,x) e^{-itk^{2}}\, S(-k)\, \mathbf{F}\psi\left(
k\right)  \, dk.
\end{array}
\end{equation}
Further, for any $\phi\in L^{2}(\mathbb{R}^{+}),$ it follows from
\eqref{Spectralrepresentation1},
\begin{equation}
\label{sp2}%
\begin{array}
[c]{l}%
\mathcal{U }(t)\mathbf{F}^{\dagger}\phi=\mathcal{U}(t)\,P_{\operatorname*{c}%
}(H)\mathbf{F}^{\dagger}\phi(x)=\frac{1}{\sqrt{2\pi}}\int_{0}^{\infty}
f(k,x)e^{-itk^{2}}\phi\left(  k\right)  dk\\
\\
+\frac{1}{\sqrt{2\pi}} \int_{0}^{\infty}\, f(-k,x) e^{-itk^{2}}\,
S(-k)\,\phi\left(  k\right)  \, dk\\
\\
= \frac{1}{\sqrt{2\pi}}\int_{-\infty}^{\infty}\, e^{ikx}\, m(k,x)e^{-itk^{2}%
}\, (\mathcal{E }\phi)(k) \, dk,
\end{array}
\end{equation}
with, as above, $m\left(  k,x\right)  =f\left(  k,x\right)  e^{-ikx}$, and
where we define
\begin{equation}
\label{sp3}\mathcal{E}\phi(k)= \left\{
\begin{array}
[c]{l}%
\phi(k), \qquad k \geq0,\\
S(k) \phi(-k), \qquad k < 0.
\end{array}
\right.
\end{equation}
Note that since by \eqref{UnitaritySM} $\vert S(k)\vert=1,$
\begin{equation}
\label{sp3.uu}\|\mathcal{E}\phi\|_{L^{2}( {\mathbb{R}})}= \sqrt{2}
\|\phi\|_{L^{2}(\mathbb{R}^{+})}, \qquad\phi\in L^{2}(\mathbb{R}^{+}).
\end{equation}
Moreover, since by \eqref{UnitaritySM} $S(k) S(-k)= S(k) S(k)^{\dagger}=I,$
\begin{equation}
\label{sp4}S(k) (\mathcal{E }\phi)(-k) = (\mathcal{E}\phi)(k), \qquad k
\in{\mathbb{R}}, \phi\in L^{2}(\mathbb{R}^{+}).
\end{equation}

By \eqref{sp2}, we get%
\begin{equation}
\left\Vert \frac{1}{\sqrt{2\pi}}\int_{-\infty}^{\infty}\,e^{ikx}\,m\left(
k,x\right)  e^{-itk^{2}} (\mathcal{E}\phi)\left(  k\right)  dk\right\Vert
_{{L}^{2}(\mathbb{R}^{+})}=\left\Vert \phi\right\Vert _{{L}^{2}(\mathbb{R}%
^{+})},\qquad\phi\in L^{2}(\mathbb{R}^{+}). \label{10}%
\end{equation}
We find it convenient to extend $m(k,x)$ to $x \leq0$ as an even function,
\begin{equation}
\label{sp2.x.1}m(k,x) = m (k,-x), \qquad x \in{\mathbb{R}}.
\end{equation}

It follows from a simple calculation that for $t\neq0$%
\begin{equation}
\frac{1}{\sqrt{2\pi}}\int_{-\infty}^{\infty}e^{ikx}e^{-itk^{2}}m\left(
k,x\right)  \phi\left(  k\right)  dk=M\mathcal{D}_{t}\mathcal{W}\left(
t\right)  \phi, \label{ap.64}%
\end{equation}
where we denote%
\begin{equation}
\mathcal{W}\left(  t\right)  \phi=\sqrt{\frac{it}{2\pi}}\int_{-\infty}%
^{\infty}e^{-it\left(  k-\frac{x}{2}\right)  ^{2}}m\left(  k,tx\right)
\phi\left(  k\right)  dk,\qquad t,x\in{\mathbb{R}}, \label{sp5}%
\end{equation}
and we recall that,
\begin{equation}
M:=e^{\frac{ix^{2}}{4t}},\qquad(\mathcal{D}_{t}\phi)(x)=\frac{1}{\sqrt{it}%
}\,\phi(xt^{-1}). \label{sp6}%
\end{equation}
We define,
\begin{equation}
\mathcal{Q}(t):=\mathcal{W}(t)\mathcal{E}. \label{sp7}%
\end{equation}
Hence, by \eqref{sp2} and \eqref{ap.64} we obtain,
\begin{equation}
\mathcal{U}\left(  t\right)  \mathbf{F}^{\dagger}\phi=M\mathcal{D}%
_{t}\mathcal{Q}\left(  t\right)  \phi,\qquad\phi\in L^{2}(\mathbb{R}^{+}).
\label{2.5}%
\end{equation}
Note that \eqref{2.5} implies,
\begin{equation}
\mathcal{U}\left(  t\right)  P_{c}(H)\psi=M\mathcal{D}_{t}\mathcal{Q}\left(
t\right)  \mathbf{F}\psi,\qquad\psi\in L^{2}(\mathbb{R}^{+}), \label{ap.65}%
\end{equation}
where we used that $P_{c}(H)=\mathbf{F}^{\dagger}\mathbf{F}.$
Moreover, if $H$ has no negative eigenvalues, $P_{c}(H)=I,$ and $\mathbf{F}$
is unitary in $L^{2}(\mathbb{R}^{+}).$ Then, using \eqref{2.5} and as $M$ and
$\mathcal{D}_{t}$ are unitary in $L^{2}(\mathbb{R}^{+})$ we obtain,
\begin{equation}
\left\Vert \mathcal{Q}\left(  t\right)  \phi\right\Vert _{{L}^{2}%
(\mathbb{R}^{+})}=\left\Vert \phi\right\Vert _{{L}^{2}(\mathbb{R}^{+})}.
\label{isom1}%
\end{equation}
Moreover, using (\ref{ap.65}) we see that if $H$ has no negative eigenvalues,%
\begin{equation}
\mathcal{Q}^{-1}\left(  t\right)  =\mathbf{F}\,e^{itH}\,M\mathcal{D}%
_{t}=e^{itk^{2}}\mathbf{F}M\mathcal{D}_{t}, \label{2.6}%
\end{equation}
where we used $\mathbf{F}\mathcal{U}^{-1}(t)=\mathbf{F}e^{itH}=e^{itk^{2}%
}\,\mathbf{F}.$ Further, from the first equality in \eqref{2.6} we obtain,
\begin{equation}
\mathbf{F}e^{itH}=\mathcal{Q}^{-1}\,\mathcal{D}_{t}^{-1}\overline{M}.
\label{sp7.ll}%
\end{equation}
Moreover, by \eqref{2.6}
\begin{equation}
\left\Vert \mathcal{Q}^{-1}\left(  t\right)  \phi\right\Vert _{{L}%
^{2}(\mathbb{R}^{+})}=\left\Vert \phi\right\Vert _{{L}^{2}(\mathbb{R}^{+})},
\label{isom2}%
\end{equation}
and using \eqref{UnitaritySM}
\begin{equation}
\mathcal{Q}^{-1}\left(  t\right)  \phi(k)=S\left(  k\right)  \mathcal{W}%
_{+}\left(  t\right)  \phi+\mathcal{W}_{-}\left(  t\right)  \phi,\qquad
k\geq0, \label{winverse}%
\end{equation}
where%
\begin{equation}
\mathcal{W}_{\pm}\left(  t\right)  \phi(k)=\sqrt{\frac{t}{2\pi i}}\int%
_{0}^{\infty}e^{it\left(  k\pm\frac{x}{2}\right)  ^{2}}m^{\dagger}\left(  \mp
k,tx\right)  \phi\left(  x\right)  dx,\qquad k\in{\mathbb{R}}. \label{sp8}%
\end{equation}
Note that in \eqref{sp8}, for convenience, we have defined the quantities
$\mathcal{W}_{\pm}(t)$ for $k\in{\mathbb{R}},$ but that in \eqref{winverse} we
only use then for $k\geq0.$ We also find it useful to introduce the following
quantity,
\begin{equation}
\widehat{\mathcal{W}}(t)\phi(k):=S\left(  k\right)  \mathcal{W}_{+}\left(
t\right)  \phi+\mathcal{W}_{-}\left(  t\right)  \phi,\qquad k\in{\mathbb{R}}.
\label{sp9}%
\end{equation}
Observe that,
\begin{equation}
\widehat{\mathcal{W}}(t)\phi(k)=\mathcal{Q}^{-1}\left(  t\right)
\phi(k),\qquad k\geq0. \label{sp10}%
\end{equation}
Furthermore, by \eqref{UnitaritySM}
\begin{equation}
S(k)\widehat{\mathcal{W}}(t)\phi(-k)=\widehat{\mathcal{W}}(t)\phi(k),\qquad
k\in{\mathbb{R}}. \label{sp11.1}%
\end{equation}
In the following proposition we prove some useful results.
\begin{prop2}\label{propl2}
Assume that  $V \in L^1_{2+\delta}(\er),$ for some $ \delta \geq 0$.
Then.
\begin{enumerate}
\item
\beq\label{l2.1}
\|\mathcal W \phi\|_{L^2(\ere)} \leq C \|\phi\|_{L^2(\ere)}.
\ene
\item
\beq\label{EST8}
\|\mathcal W_\pm \phi\|_{L^2(\ere)} \leq C \|\phi\|_{L^2(\er)}.
\ene
\item
\beq\label{l2.3}
\|\widehat{\mathcal W} \phi\|_{L^2(\ere)} \leq C \|\phi\|_{L^2(\er)}.
\ene
\item
Further, if $ V\in L^{1}_{2+\delta},$ for some $ \delta > 1/2,$ and $V$ admits a regular decomposition (see Definition~\ref{Def1}),
\beq\label{EST5}
\|\partial_x \mathcal W \phi\|_{L^2(\ere)} \leq C \sqrt{\langle |t|\rangle}  \|\phi\|_{H^1(\ere)}.
\ene
\end{enumerate}
\end{prop2}

\begin{proof}
By \eqref{sp5}
\begin{equation}
\label{l2.3.1}%
\begin{array}
[c]{l}%
\mathcal{W}\left(  t\right)  \phi= e^{-it x^{2}/4} \sqrt{\frac{it}{2\pi}}
\int_{-\infty}^{\infty} e^{i ktx } e^{-it k^{2}} \phi\left(  k\right)  dk\\
\\
+ e^{-it x^{2}/4} \sqrt{\frac{i|t|}{2\pi}} \int_{-\infty}^{\infty} e^{i ktx }
e^{-it k^{2}} [ m\left(  k,t x\right)  -I] \phi\left(  k\right)  dk.
\end{array}
\end{equation}
Hence, \eqref{l2.1} follows from \eqref{3.2}, \eqref{l2.3.1} and Parseval's
identity. Further, by \eqref{sp8},
\begin{equation}
\label{l2.4}%
\begin{array}
[c]{l}%
\mathcal{W}_{\pm}\left(  t\right)  \phi= e^{it k^{2}} \sqrt{\frac{it}{2\pi}}
\int_{0}^{\infty} e^{\pm i ktx } e^{-it x^{2}/4} \phi\left(  x\right)  dx\\
\\
+ e^{it k^{2}} \sqrt{\frac{it}{2\pi}} \int_{0}^{\infty} e^{\pm i ktx } e^{-it
x^{2}/4} [ m^{\dagger}\left(  \mp k,tx\right)  -I] \phi\left(  x\right)  dx.
\end{array}
\end{equation}
Then, \eqref{EST8} follows from \eqref{3.2}, \eqref{l2.4} and Parseval's
identity. Equation \eqref{l2.3} follows from \eqref{UnitaritySM} and
\eqref{EST8}. Let us prove \eqref{EST5}. Derivating both sides of \eqref{sp5},
using that
\begin{equation}
\partial_{x}e^{-it\left(  k-\frac{x}{2}\right)  ^{2}}=-\frac{1}{2}\partial
_{k}e^{-it\left(  k-\frac{x}{2}\right)  ^{2}}, \label{9}%
\end{equation}
and integrating by parts, we have
\begin{align}
\label{2}\partial_{x}\mathcal{W}\left(  t\right)  \phi &  = \frac{1}{2}
\mathcal{W}\left(  t\right)  \left(  \partial_{k}\phi\right)  +\frac{1}%
{2}\sqrt{\frac{it}{2\pi}}\int_{-\infty}^{\infty}e^{-it\left(  k-\frac{x}%
{2}\right)  ^{2}}\left(  \partial_{k}m\right)  \left(  k,tx\right)
\phi\left(  k\right)  dk\\
\\
&  +\sqrt{\frac{it}{2\pi}}\int_{-\infty}^{\infty}e^{-it\left(  k-\frac{x}%
{2}\right)  ^{2}}t\partial_{x}m\left(  k,tx\right)  \phi\left(  k\right)
dk.\nonumber\\
\end{align}
Then, using (\ref{3.5}) with $\delta>1/2,$ (\ref{3.6}) and \eqref{l2.1} we show%

\begin{align*}
\left\Vert \partial_{x}\mathcal{W}\left(  t\right)  \phi\right\Vert _{L^{2}(
{\mathbb{R}})}  &  \leq C \left\Vert \phi\right\Vert _{H^{1}( {\mathbb{R}})}
+C\sqrt{|t|}\left\Vert \left\langle tx\right\rangle ^{-\delta}\right\Vert
_{L^{2}}\left\Vert \left\langle k\right\rangle ^{-1}\right\Vert _{L^{2}%
}\left\Vert \phi\right\Vert _{L^{2}( {\mathbb{R}})}\\
&  +C|t|\sqrt{|t|}\left\Vert \left\langle tx\right\rangle ^{-2-\tilde{\delta}%
}\right\Vert _{L^{1}}\left\Vert \phi\right\Vert _{L^{\infty}( {\mathbb{R}})}.
\end{align*}
Hence, (\ref{EST5}) follows from Sobolev's inequality.
\end{proof}

In the following lemma we obtain further properties of $\mathcal{W}_{\pm}$ and
of $\mathcal{W}.$

\begin{lemma2}
Assume  that $V \in L^1_{2+\delta}(\er),$ for some $ \delta >0 $.
\begin{enumerate}
\item
Suppose that $\phi \in H^1(\er),$ and that besides being in $L^2(\er), \partial_k \phi$ admits the following decomposition,  $\partial_{k}\phi=\phi_{1}+\phi_{2},$  where $ \phi_1 \in L^{q_1}(\er),$ and $ \phi_2 \in L^{q_2}(\er),$ for some $q_1, q_2,$ with $ 1\leq q_1, q_2 < \infty.$ Then, we have
\begin{equation}
\left\Vert \mathcal{W}_\pm\left(  t\right)  \phi\right\Vert
_{L^{\infty}(\ere)}\leq C\left(  \left[1+ \frac{1}{\sqrt{|t|}} \right]   \left\Vert \phi\right\Vert _{L^{\infty}(\er)}
+|t|^{-\frac{1}{2p_1}}\left\Vert \phi_1\right\Vert _{L^{q_1}(\er)}+|t|^{-\frac{1}{2p_2}}\left\Vert \phi_2\right\Vert _{L^{q_2}(\er)}\right),
\label{89}
\end{equation}
and
\begin{equation}
\left\Vert \widehat{\mathcal{W}}\left(  t\right)  \phi\right\Vert
_{L^{\infty}(\ere)}\leq C  \left(  \left[ 1+ \frac{1}{\sqrt{|t|}} \right]    \left\Vert \phi\right\Vert _{L^{\infty}(\er)}
+|t|^{-\frac{1}{2p_1}}\left\Vert \phi_1\right\Vert _{L^{q_1}(\er)}+|t|^{-\frac{1}{2p_2}}\left\Vert \phi_2\right\Vert _{L^{q_2}(\er)}\right),
\label{89.1}
\end{equation}
where $ \frac{1}{p_1}+ \frac{1}{q_1}=1,$ and $ \frac{1}{p_2}+ \frac{1}{q_2}=1.$
\item
Suppose that $\phi \in H^1(\ere),$ and that besides being in $L^2(\ere), \partial_k \phi$ admits the following decomposition,  $\partial_{k}\phi=\phi_{1}+\phi_{2},$  where $ \phi_1 \in L^{q_1}(\ere),$ and $ \phi_2 \in L^{q_2}(\ere),$ for some $q_1, q_2,$ with $ 1\leq  q_1, q_2 < \infty.$ Then,
\begin{equation}
\left\Vert \mathcal{W}\left(  t\right)  \phi\right\Vert _{L^{\infty}(\ere)}\leq
C \left( 1+\frac{1}{\sqrt{|t|}}\right)  \left\Vert \phi\right\Vert _{L^{\infty}(\ere)}
+ C |t|^{-\frac{1}{2p_1}}\left\Vert \phi_1\right\Vert _{L^{q_1}(\ere)}+ C |t|^{-\frac{1}{2p_2}}\left\Vert \phi_2\right\Vert _{L^{q_2}(\ere)},
\label{EST4}%
\end{equation}
where $ \frac{1}{p_1}+ \frac{1}{q_1}=1,$ and $ \frac{1}{p_2}+ \frac{1}{q_2}=1.$
\end{enumerate}
\end{lemma2}

\begin{proof}
By \eqref{sp8}
\begin{equation}
\mathcal{W}_{\pm}\phi=T_{1}+T_{2}, \label{89.1.0}%
\end{equation}
where
\begin{equation}
T_{1}:=\sqrt{\frac{t}{2\pi i}}\int_{0}^{\infty}e^{it\left(  k\pm\frac{x}%
{2}\right)  ^{2}}\phi\left(  x\right)  dx, \label{89.1.1}%
\end{equation}
and
\begin{equation}
T_{2}:=\sqrt{\frac{t}{2\pi i}}\int_{0}^{\infty}e^{it\left(  k\pm\frac{x}%
{2}\right)  ^{2}}[m^{\dagger}\left(  \mp k,tx\right)  -I]\phi\left(  x\right)
dx. \label{89.1.2}%
\end{equation}
Since
\begin{equation}
e^{it\left(  k\pm\frac{x}{2}\right)  ^{2}}=\pm\frac{2\partial_{x}\left(
\left(  k\pm\frac{x}{2}\right)  e^{it\left(  k\pm\frac{x}{2}\right)  ^{2}%
}\right)  }{1+2it\left(  k\pm\frac{x}{2}\right)  ^{2}}, \label{89.2}%
\end{equation}
integrating by parts we have%
\begin{align*}
\int_{0}^{\infty}e^{it\left(  k\pm\frac{x}{2}\right)  ^{2}}\phi\left(
x\right)  dx  &  =\mp\frac{2ke^{itk^{2}}}{1+2itk^{2}}\phi\left(  0\right)
+\int_{0}^{\infty}\frac{4it\left(  k\pm\frac{x}{2}\right)  ^{2}e^{it\left(
k\pm\frac{x}{2}\right)  ^{2}}}{\left(  1+2it\left(  k\pm\frac{x}{2}\right)
^{2}\right)  ^{2}}\phi\left(  x\right)  dx\\
&  \mp\int_{0}^{\infty}\frac{2\left(  \left(  k\pm\frac{x}{2}\right)
e^{it\left(  k\pm\frac{x}{2}\right)  ^{2}}\right)  }{1+2it\left(  k\pm\frac
{x}{2}\right)  ^{2}}\partial_{x}\phi(x))dx.
\end{align*}
Then, by H\"{o}lder's and Sobolev's inequalities we get%
\begin{align}
\left\vert \int_{0}^{\infty}e^{it\left(  k\pm\frac{x}{2}\right)  ^{2}}%
\phi\left(  x\right)  dx\right\vert  &  \leq C\frac{1}{\sqrt{|t|}}\left\Vert
\phi\right\Vert _{L^{\infty}(\mathbb{R}^{+})}+C\left\Vert \frac{1}{1+|t|x^{2}%
}\right\Vert _{L^{1}({\mathbb{R}})}\left\Vert \phi\right\Vert _{L^{\infty
}(\mathbb{R}^{+})}\nonumber\\
&  +C\left\Vert \frac{x}{1+|t|x^{2}}\right\Vert _{L^{p_{1}}({\mathbb{R}}%
)}\left\Vert \phi_{1}\right\Vert _{L^{q_{1}}(\mathbb{R}^{+})}+C\left\Vert
\frac{x}{1+|t|x^{2}}\right\Vert _{L^{p_{2}}({\mathbb{R}})}\left\Vert \phi
_{2}\right\Vert _{L^{q_{2}}(\mathbb{R}^{+})}\nonumber\\
&  \leq C\frac{1}{\sqrt{|t|}}\left\Vert \phi\right\Vert _{L^{\infty
}(\mathbb{R}^{+})}+\frac{C}{|t|^{\frac{1}{2}\left(  1+\frac{1}{p_{1}}\right)
}}\left\Vert \phi_{1}\right\Vert _{L^{q_{1}}(\mathbb{R}^{+})}+\frac
{C}{|t|^{\frac{1}{2}\left(  1+\frac{1}{p_{2}}\right)  }}\left\Vert \phi
_{2}\right\Vert _{L^{q_{2}}}. \label{90.1}%
\end{align}
Hence, by \eqref{89.1.1} and \eqref{90.1}
\begin{equation}
|T_{1}|\leq C\left(  \left\Vert \phi\right\Vert _{L^{\infty}(\mathbb{R}^{+}%
)}+|t|^{-\frac{1}{2p_{1}}}\left\Vert \phi_{1}\right\Vert _{L^{q_{1}%
}(\mathbb{R}^{+})}+|t|^{-\frac{1}{2p_{2}}}\left\Vert \phi_{2}\right\Vert
_{L^{q_{2}}(\mathbb{R}^{+})}\right)  . \label{90.1.1}%
\end{equation}
By \eqref{3.2}
\begin{equation}
|T_{2}|\leq C\frac{1}{\sqrt{|t|}}\left\Vert \phi\right\Vert _{L^{\infty
}(\mathbb{R}^{+})}. \label{90.2.yy}%
\end{equation}

Then, (\ref{89}) follows from \eqref{89.1.0}, \eqref{90.1.1}, and
\eqref{90.2.yy}. Further, \eqref{89.1} follows from \eqref{UnitaritySM},
\eqref{sp9}, and \eqref{89}. Let us prove \eqref{EST4}. By \eqref{sp5} ,
\begin{equation}
\label{89.1.0.a}\mathcal{W }\phi= T_{3}+T_{4},
\end{equation}
where
\begin{equation}
\label{89.1.xx}T_{3}:= \sqrt{\frac{t}{2\pi i}}\int_{-\infty}^{\infty
}e^{-it\left(  k-\frac{x}{2}\right)  ^{2}} \phi\left(  x\right)  dx,
\end{equation}
and
\begin{equation}
\label{89.1.2.xx}T_{4}:= \sqrt{\frac{t}{2\pi i}}\int_{-\infty}^{\infty
}e^{-it\left(  k-\frac{x}{2}\right)  ^{2}} [ m^{\dagger}\left(  k,tx\right)
-I ] \phi\left(  x\right)  dx.
\end{equation}
By \eqref{89.2} with $t $ replaced by $-t,$ and integrating by parts we have,%

\begin{align*}
\int_{-\infty}^{\infty}e^{-it\left(  k- \frac{x}{2}\right)  ^{2}} \phi\left(
x\right)  dx  &  = \int_{-\infty}^{\infty}\frac{-4it\left(  k-\frac{x}%
{2}\right)  ^{2}e^{-it\left(  k-\frac{x}{2}\right)  ^{2}}}{\left(
1-2it\left(  k-\frac{x}{2}\right)  ^{2}\right)  ^{2}} \phi\left(  x\right)
dx\\
&  + \int_{-\infty}^{\infty}\frac{2\left(  \left(  k-\frac{x}{2}\right)
e^{-it\left(  k-\frac{x}{2}\right)  ^{2}}\right)  }{1-2it\left(  k-\frac{x}%
{2}\right)  ^{2}}\partial_{x} \phi(x) ) dx.
\end{align*}
Then, by H\"{o}lder's and Sobolev's inequalities we get%
\begin{align}
\left\vert \int_{-\infty}^{\infty}e^{-it\left(  k-\frac{x}{2}\right)  ^{2}}
\phi\left(  x\right)  dx\right\vert  &  \leq C\left\Vert \frac{1}{1+|t|x^{2}%
}\right\Vert _{L^{1}( {\mathbb{R}})}\left\Vert \phi\right\Vert _{L^{\infty
}(\mathbb{R}^{+})}+\label{90.1.x.x}\\
&  C\left\Vert \frac{x}{1+|t|x^{2}}\right\Vert _{L^{p_{1}}( {\mathbb{R}}%
)}\left\Vert \phi_{1}\right\Vert _{L^{q_{1}}(\mathbb{R}^{+})}+C\left\Vert
\frac{x}{1+|t|x^{2}}\right\Vert _{L^{p_{2}}( {\mathbb{R}})}\left\Vert \phi
_{2}\right\Vert _{L^{q_{2}}(\mathbb{R}^{+})}\nonumber\\
&  \leq\frac{C}{\sqrt{|t|}}\left\Vert \phi\right\Vert _{L^{\infty}%
(\mathbb{R}^{+})}+\frac{C}{|t|^{\frac{1}{2}\left(  1+\frac{1}{p_{1}}\right)
}}\left\Vert \phi_{1}\right\Vert _{L^{q_{1}}(\mathbb{R}^{+})}+\frac
{C}{|t|^{\frac{1}{2}\left(  1+\frac{1}{p_{2}}\right)  }}\left\Vert \phi
_{2}\right\Vert _{L^{q_{2}}}.\nonumber
\end{align}
Hence, by \eqref{89.1.xx} and \eqref{90.1.x.x}
\begin{equation}
\label{90.1.1xx}|T_{3}| \leq C \left(  \left\Vert \phi\right\Vert _{L^{\infty
}(\mathbb{R}^{+})} +|t|^{-\frac{1}{2p_{1}}}\left\Vert \phi_{1}\right\Vert
_{L^{q_{1}}(\mathbb{R}^{+})}+|t|^{-\frac{1}{2p_{2}}}\left\Vert \phi
_{2}\right\Vert _{L^{q_{2}}(\mathbb{R}^{+})}\right) .
\end{equation}
By \eqref{3.2}
\begin{equation}
\label{90.2}|T_{4}|\leq C \frac{1}{\sqrt{|t|}} \left\Vert \phi\right\Vert
_{L^{\infty}(\mathbb{R}^{+})}.
\end{equation}

Then \eqref{EST4} follows from \eqref{89.1.0.a}, \eqref{90.1.1xx} and \eqref{90.2}.
\end{proof}

We prepare the following lemma. \begin{lemma2}
Assume that the potential $V$ is self-adjoint, and that the constant boundary
matrices $A,B$ fulfill \eqref{ap.4} and \eqref{ap.5}. Then.
\begin{enumerate}
\item Suppose that $V\in L^{1}(\mathbb{R}^{+})\cap L^{\infty}(\mathbb{R}%
^{+}).$ Then,  for all $\psi\in
{H}^{2}(\er)$ satisfying the boundary condition \eqref{ap.3}
\begin{equation}
\left\Vert \partial_{x}^{2}\mathcal{U}\left(  t\right)  \psi\right\Vert
_{{L}^{2}(\mathbb{R}^{+})}\leq C\left\Vert \psi\right\Vert _{{H}^{2}(\er)%
}.\label{Uderiv}%
\end{equation}
\item Suppose that $H$ has no negative eigenvalues, that $V \in L^\infty(\er)\cap L_{3}%
^{1}(\mathbb{R}^{+}),$ and that $V$ admits a regular decomposition (see Definition~\ref{Def1}). Then, for all $ \psi \in
H^2(\er)$ that satisfy the boundary condition \eqref{ap.3},
\begin{equation}
\left\Vert x\partial_{x}\mathcal{U}\left(  t\right)  \psi\right\Vert _{{L}%
^{2}(\mathbb{R}^{+})}\leq C(1+|t|)\left(  \left\Vert \psi\right\Vert
_{{H}^{1,1}(\er)}+\left\Vert \psi\right\Vert _{{H}^{2}(\er)}\right).
\label{Uweightderiv}%
\end{equation}
\item
Assume that $V\in L^1_3(\er),$  and that $H$ has no negative eigenvalues. Then,
\beq\label{zzz}
\| x U(t)  \psi\|_{L^2(\er)}  \leq C (1+|t|^2)\left( \| \psi\|_{L^2_1(\er)}+ \|\psi\|_{H^{1}(\er)}\right).
\ene
\item
Suppose that $V\in L^1_4(\er),$ that it admits a regular decomposition, and that $H$ has no negative eigenvalues. Then, for all $\psi \in H^2(\er)$ that satisfy the boundary condition \eqref{ap.3},
\beq\label{zzz.1}
\| x^2 U(t)  \psi\|_{L^2(\er)}  \leq C (1+|t|^3) \left(\| \psi\|_{L^2_2(\er)}+ \|\psi\|_{H^{1,1}(\er)}+ \|\psi\|_{H^2(\er)}\right).
\ene
\item
Assume that $V\in L^1_{2+\tilde{\delta}}(\er), \tilde{\delta} >0,$ that it admits a regular decomposition, and that $H$ has no negative eigenvalues. Then, for $ \psi \in H^1_{A,B}(\er),$
\beq\label{zzz.2}
\| \partial_x U(t)  \psi\|_{L^2(\er)}  \leq C  \|\psi\|_{H^{1}_{A,B}(\er)}.
\ene
\end{enumerate}
\end{lemma2}

\begin{proof}
Since $V\in L^{\infty}(\mathbb{R}^{+}),$ by \eqref{domainh2}, the domain of
$H$ consists of all those functions in $H^{2}(\mathbb{R}^{+})$ that satisfy
the boundary condition \eqref{ap.3}. Hence $\psi\in D[H].$ Further,
\[
\partial_{x}^{2}\mathcal{U}\left(  t\right)  \psi=(-H+V)\mathcal{U}\left(
t\right)  \psi=-\mathcal{U}\left(  t\right)  H\psi+V\mathcal{U}\left(
t\right)  \psi.
\]
Then,
\[
\left\Vert \partial_{x}^{2}\mathcal{U}\left(  t\right)  \psi\right\Vert
_{{L}^{2}(\mathbb{R}^{+})}\leq\Vert H\psi\Vert_{L^{2}(\mathbb{R}^{+})}+\Vert
V\Vert_{L^{\infty}(\mathbb{R}^{+})}\Vert\psi\Vert_{L^{2}(\mathbb{R}^{+})}\leq
C\Vert\psi\Vert_{H^{2}(\mathbb{R}^{+})},
\]
where we used that as $D[H]\subset H^{2}(\mathbb{R}^{+}),$
\[
\Vert H\psi\Vert_{L^{2}(\mathbb{R}^{+})}\leq C\Vert\psi\Vert_{H^{2}%
(\mathbb{R}^{+})}.
\]
This completes the proof of \eqref{Uderiv}. Let us now prove
\eqref{Uweightderiv}. As $H$ has no negative eigenvalues, $P_{c}(H)=I (see
Theorem~\ref{theospec} .$ Then, from (\ref{Spectralrepresentation1}) we have
\begin{equation}
\label{24}x\partial_{x}\mathcal{\mathcal{U}}(t)\,\psi(x)= T_{1}+ T_{2},
\end{equation}
where,
\begin{equation}
\label{24.1}T_{1}=- \frac{1}{\sqrt{2\pi}}\int_{0}^{\infty} e^{ikx}
\partial_{k} \left(  k e^{-itk^{2}}\mathbf{F}\psi\left(  k\right)  \right)  dk
+ \frac{1}{\sqrt{2\pi}} \int_{0}^{\infty}\, e^{-ikx} \partial_{k} \left(  k
e^{-itk^{2}}\, S(-k)\, \mathbf{F}\psi\left(  k\right)  \right)  dk,
\end{equation}
and,
\begin{equation}
\label{24.2}%
\begin{array}
[c]{l}%
T_{2}= \frac{i}{\sqrt{2\pi}}\, \int_{0}^{\infty}\, x e^{-it k^{2}} \left(  \,
e^{ikx}( m(k,x)-I)- e^{-ik x} (m(-k,x)-I)\, S(-k)\right)  k \mathbf{F}\psi(k)
\, dk\\
\\
+\frac{1}{\sqrt{2\pi}}\, \int_{0}^{\infty}\, x e^{-it k^{2}} \left(  \,
e^{ikx} \partial_{x} m(k,x)+ e^{-ik x} \partial_{x} m(-k,x)\, S(-k)\right)
\mathbf{F}\psi(k) \, dk.
\end{array}
\end{equation}
By \eqref{UnitaritySM},\eqref{isom}, item 3 of Lemma~ \ref{scontinuity}, items
2 and 3 of Lemma ~\ref{Lemma2}, and Parseval's identity, we estimate,%
\begin{equation}
\label{24.3}%
\begin{array}
[c]{l}%
\|T_{1}\|_{L^{2}(\mathbb{R}^{+})} \leq\| \partial_{k} \left(  k e^{-itk^{2}%
}\mathbf{F}\psi\left(  k\right)  \right)  \|_{L^{2}(\mathbb{R}^{+})} +\|
\partial_{k} \left(  k e^{-itk^{2}}\, S(-k)\, \mathbf{F}\psi\left(  k\right)
\right)  \|_{L^{2}(\mathbb{R}^{+})}\\
\\
\leq C (1+|t|)\left(  \Vert\psi\Vert_{H^{1,1}(\mathbb{R}^{+})}+\Vert\psi
\Vert_{H^{2}(\mathbb{R}^{+})}\right)  .
\end{array}
\end{equation}
Moreover, by \eqref{UnitaritySM}, (\ref{3.2}) with $\delta>1/2$, (\ref{3.6}),
and \eqref{61},
\begin{equation}
\label{24.4}\|T_{2}\|_{L^{2}(\mathbb{R}^{+})}\leq\|\psi\|_{H^{1}( {\mathbb{R}%
})}.
\end{equation}
Hence, we deduce (\ref{Uweightderiv}) from \eqref{24}, \eqref{24.3} and \eqref{24.4}.

Let us now prove item 3. Let $g\in C_{0}^{\infty}(\mathbb{R})$ satisfy
$g(k)=1,|k|\leq1,g(k)=0,|k|\geq2.$ Then, using (\ref{Spectralrepresentation1})
with $P_{c}(H)=I,$ we write,
\begin{equation}
xU(t)\phi=T_{3}+T_{4}, \label{24.4.1}%
\end{equation}
where,
\begin{equation}
T_{3}:=xU(t)g(H)\psi, \label{24.4.2}%
\end{equation}
and
\begin{equation}
T_{4}:=xU(t)(I-g(H))\psi. \label{24.4.3}%
\end{equation}
Further,
\begin{equation}
T_{3}=\psi-i\int_{0}^{t}\,xU(s)g(H)H\psi ds . \label{24.4.4}%
\end{equation}
By \eqref{Spectralrepresentation1} with $P_{c}(H)=I,$
\begin{equation}
xU(s)g(H)H\phi=T_{5}(s)+T_{6}(s), \label{24.4.5}%
\end{equation}
where,
\begin{equation}
T_{5}(s):=\frac{1}{\sqrt{2\pi}}\,\int_{0}^{\infty}\left[  -i(\partial
_{k}e^{ikx})e^{-isk^{2}}k^{2}\mathbf{F}\psi(k)+i(\partial_{k}e^{-ikx}%
)\,e^{-isk^{2}}\,S(-k)k^{2}\mathbf{F}\psi(k)\right]  \,g(k^{2})\,dk,
\label{24.4.6}%
\end{equation}
and,
\begin{equation}%
\begin{array}
[c]{l}%
T_{6}(s):=\frac{1}{\sqrt{2\pi}}\,\int_{0}^{\infty}\left[  e^{ikx}%
x(m(k,x)-1)e^{-isk^{2}}k^{2}\mathbf{F}\psi(k)\right. \\
\\
\left.  +e^{-ikx}\,e^{-isk^{2}}x(m(-k,x)-1)\,S(-k)k^{2}\mathbf{F}%
\psi(k)\right]  \,g(k^{2})dk.
\end{array}
\label{24.4.7}%
\end{equation}

Integrating by parts in \eqref{24.4.6} and using \eqref{UnitaritySM},
\eqref{isom}, \eqref{scatmatrixderiv},\eqref{61BIS.2} and by Parseval's
identity, we obtain,
\begin{equation}
\Vert T_{5}(s)\Vert_{L^{2}(\mathbb{R}^{+})}\leq C(1+|s|)\,\Vert\psi
\Vert_{L_{1}^{2}(\mathbb{R}^{+})}. \label{24.4.8}%
\end{equation}
Moreover, by \eqref{UnitaritySM}, \eqref{isom}, and \eqref{3.2},
\begin{equation}
\Vert T_{6}(s)\Vert_{L^{2}(\mathbb{R}^{+})}\leq C\Vert\psi\Vert_{L^{2}%
(\mathbb{R}^{+})}. \label{24.4.9}%
\end{equation}
By \eqref{24.4.4}, \eqref{24.4.5}, \eqref{24.4.8} and \eqref{24.4.9}
\begin{equation}
\Vert T_{3}\Vert_{L^{2}(\mathbb{R}^{+})}\leq C(1+|t|^{2})\Vert\psi\Vert
_{L_{1}^{2}(\mathbb{R}^{+})}. \label{24.4.10}%
\end{equation}
Similarly,
\begin{equation}
T_{4}(t)=T_{7}(t)+T_{8}(t), \label{24.4.11}%
\end{equation}
where,
\begin{equation}
T_{7}(t):=\frac{1}{\sqrt{2\pi}}\,\int_{0}^{\infty}\left[  -i(\partial
_{k}e^{ikx})e^{-itk^{2}}\mathbf{F}\psi(k)+i(\partial_{k}e^{-ikx}%
)\,e^{-itk^{2}}\,S(-k)\mathbf{F}\psi(k)\right]  \,(I-g(k^{2}))\,dk,
\label{24.4.12}%
\end{equation}
and
\begin{equation}%
\begin{array}
[c]{l}%
T_{8}(t)=\frac{1}{\sqrt{2\pi}}\,\int_{0}^{\infty}\left[  e^{ikx}%
x(m(k,x)-1)e^{-itk^{2}}\mathbf{F}\psi(k)\right. \\
\\
\left.  +e^{-ikx}\,e^{-itk^{2}}x(m(-k,x)-1)\,S(-k)\mathbf{F}\psi(k)\right]
\,(I-g(k^{2}))dk.
\end{array}
\label{24.4.13}%
\end{equation}
Integrating by parts in \eqref{24.4.12} and using \eqref{UnitaritySM},
\eqref{isom}, \eqref{scatmatrixderiv}, \eqref{61}, \eqref{61BIS.2}, and
Parseval's identity, we obtain,
\begin{equation}
\Vert T_{7}(t)\Vert_{L^{2}(\mathbb{R}^{+})}\leq C(1+|t|)\,\left(  \Vert
\psi\Vert_{L_{1}^{2}(\mathbb{R}^{+})}+\Vert\psi\Vert_{H^{1}(\mathbb{R}^{+}%
)}\right)  . \label{24.4.14}%
\end{equation}
Further, by \eqref{UnitaritySM}, \eqref{isom}, and \eqref{3.2},
\begin{equation}
\Vert T_{8}(t)\Vert_{L^{2}(\mathbb{R}^{+})}\leq C\Vert\psi\Vert_{L^{2}%
(\mathbb{R}^{+})}. \label{24.4.15}%
\end{equation}
Hence, by \eqref{24.4.11}, \eqref{24.4.14} and \eqref{24.4.15},
\begin{equation}
\Vert T_{4}\Vert_{L^{2}(\mathbb{R}^{+})}\leq C(1+|t|)\left(  \Vert\psi
\Vert_{L_{1}^{2}(\mathbb{R}^{+})}+\Vert\psi\Vert_{H^{1}(\mathbb{R}^{+}%
)}\right)  . \label{24.4.16}%
\end{equation}
Equation \eqref{zzz} follows from \eqref{24.4.1}, \eqref{24.4.10} and
\eqref{24.4.16}. Let us now prove item 4. Using (\ref{Spectralrepresentation1}%
) with $P_{c}(H)=I,$ we write
\begin{equation}
x^{2}U(t)\psi=T_{9}+T_{10}, \label{24.4.33}%
\end{equation}
where,
\begin{equation}
T_{9}:=x^{2}U(t)g(H)\psi, \label{24.4.34}%
\end{equation}
and
\begin{equation}
T_{10}:=x^{2}U(t)(I-g(H))\psi. \label{24.4.35}%
\end{equation}
Further,
\begin{equation}
T_{9}=\psi-i\int_{0}^{t}\,x^{2}U(s)g(H)H\psi. \label{24.4.36}%
\end{equation}
By \eqref{Spectralrepresentation1} with $P_{c}(H)=I,$
\begin{equation}
x^{2}U(s)g(H)H\phi=T_{11}(s)+T_{12}(s), \label{24.4.37}%
\end{equation}
where,
\begin{equation}
T_{11}(s):=\frac{1}{\sqrt{2\pi}}\,\int_{0}^{\infty}\left[  -(\partial_{k}%
^{2}e^{ikx})e^{-isk^{2}}k^{2}\mathbf{F}\psi(k)-(\partial_{k}^{2}%
e^{-ikx})\,e^{-isk^{2}}\,S(-k)k^{2}\mathbf{F}\psi(k)\right]  \,g(k^{2})\,dk,
\label{24.4.38}%
\end{equation}
and,
\begin{equation}%
\begin{array}
[c]{l}%
T_{12}(s):=\frac{1}{\sqrt{2\pi}}\,\int_{0}^{\infty}\left[  e^{ikx}%
x^{2}(m(k,x)-1)e^{-isk^{2}}k^{2}\mathbf{F}\psi(k)\right. \\
\\
\left.  +e^{-ikx}\,e^{-isk^{2}}x^{2}(m(-k,x)-1)\,S(-k)k^{2}\mathbf{F}%
\psi(k)\right]  \,g(k^{2})dk.
\end{array}
\label{24.4.39}%
\end{equation}
Integrating by parts in \eqref{24.4.38} and using \eqref{UnitaritySM},
\eqref{isom}, \eqref{scatmatrixderiv}, \eqref{scatmatrixsecondderiv},
\eqref{61BIS.2}, \eqref{61BIS.3} and Parseval's identity, we obtain,
\begin{equation}
\Vert T_{11}(s)\Vert_{L^{2}(\mathbb{R}^{+})}\leq C(1+|s|^{2})\,\Vert\psi
\Vert_{L_{2}^{2}(\mathbb{R}^{+})}. \label{24.4.40}%
\end{equation}
Moreover, by \eqref{UnitaritySM}, \eqref{isom} and \eqref{3.2},
\begin{equation}
\Vert T_{12}(s)\Vert_{L^{2}(\mathbb{R}^{+})}\leq C\Vert\psi\Vert
_{L^{2}(\mathbb{R}^{+})}. \label{24.4.41}%
\end{equation}
By \eqref{24.4.36}, \eqref{24.4.37}, \eqref{24.4.40} and \eqref{24.4.41}
\begin{equation}
\Vert T_{9}\Vert_{L^{2}(\mathbb{R}^{+})}\leq C(1+|t|^{3})\Vert\psi\Vert
_{L_{2}^{2}(\mathbb{R}^{+})}. \label{24.4.42}%
\end{equation}
Similarly,
\begin{equation}
T_{10}(t)=T_{13}(t)+T_{14}(t), \label{24.4.43}%
\end{equation}
where,
\begin{equation}
T_{13}(t):=\frac{1}{\sqrt{2\pi}}\,\int_{0}^{\infty}\left[  -(\partial_{k}%
^{2}e^{ikx})e^{-itk^{2}}\mathbf{F}\psi(k)-(\partial_{k}^{2}e^{-ikx}%
)\,e^{-itk^{2}}\,S(-k)\mathbf{F}\psi(k)\right]  \,(I-g(k^{2}))\,dk,
\label{24.4.43.1}%
\end{equation}
and
\begin{equation}%
\begin{array}
[c]{l}%
T_{14}(t):=\frac{1}{\sqrt{2\pi}}\,\int_{0}^{\infty}\left[  e^{ikx}%
x^{2}(m(k,x)-1)e^{-itk^{2}}\mathbf{F}\psi(k)\right. \\
\\
\left.  +e^{-ikx}\,e^{-itk^{2}}x^{2}(m(-k,x)-1)\,S(-k)\mathbf{F}%
\psi(k)\right]  \,(I-g(k^{2})dk.
\end{array}
\label{24.4.44}%
\end{equation}
Integrating by parts in \eqref{24.4.43.1} and using \eqref{UnitaritySM},
\eqref{isom}, \eqref{scatmatrixderiv}, \eqref{scatmatrixsecondderiv},
\eqref{62}, \eqref{61BIS}, \eqref{61BIS.2}, \eqref{61BIS.3}, and Parseval's
identity, we obtain,
\begin{equation}
\Vert T_{13}(t)\Vert_{L^{2}(\mathbb{R}^{+})}\leq C(1+|t|^{2})\,\left(
\Vert\psi\Vert_{L_{2}^{2}(\mathbb{R}^{+})}+\Vert\psi\Vert_{H^{1,1}%
(\mathbb{R}^{+})}+H^{2}(\mathbb{R}^{+})\right)  . \label{24.4.46}%
\end{equation}
Further, by \eqref{UnitaritySM}, \eqref{isom}, and \eqref{3.2},
\begin{equation}
\Vert T_{14}(t)\Vert_{L^{2}(\mathbb{R}^{+})}\leq C\Vert\psi\Vert
_{L^{2}(\mathbb{R}^{+})}. \label{24.4.47}%
\end{equation}
Hence, by \eqref{24.4.43}, \eqref{24.4.46} and \eqref{24.4.47},
\begin{equation}
\Vert T_{10}\Vert_{L^{2}(\mathbb{R}^{+})}\leq C(1+|t|^{2})\left(  \Vert
\psi\Vert_{L_{2}^{2}(\mathbb{R}^{+})}+\Vert\psi\Vert_{H^{1,1}(\mathbb{R}^{+}%
)}+H^{2}(\mathbb{R}^{+})\right)  . \label{24.4.48}%
\end{equation}
Equation \eqref{zzz.1} follows from \eqref{24.4.33}, \eqref{24.4.42} and
\eqref{24.4.48}. Let us now prove item 5. By \eqref{Spectralrepresentation1}
with $P_{c}(H)=I,$
\begin{equation}%
\begin{array}
[c]{l}%
\partial_{x}\mathcal{U}(t)\,\psi(x)=\frac{1}{\sqrt{2\pi}}\int_{0}^{\infty
}e^{-itk^{2}}\left[  (ikm(k,x)+(\partial_{x}m(k,x))e^{ikx}\right. \\
\\
\left.  +(-ikm(-k,x)+(\partial_{x}m(-k,x))e^{-ikx}S(-k)\right]  \mathbf{F}%
\psi\left(  k\right)  dk.
\end{array}
\label{24.4.48.1}%
\end{equation}
Equation \eqref{zzz.2} follows from \eqref{UnitaritySM}, \eqref{3.2},
\eqref{3.6}, \eqref{61}, \eqref{24.4.48.1} and Parseval's identity.
\end{proof}

We prepare the following lemma \begin{lemma2}\label{ext}
Assume that the potential $V$ is self-adjoint, $ V\in L^1(\er),$ that the constant boundary
matrices $A,B$ fulfill \eqref{ap.4} and \eqref{ap.5}, and that $H$ has no negative eigenvalues. We define
\beq\label{ab.1}
w(k):= \left\{\begin{array}{l} w(k):= \mathbf F\phi(k), \qquad k \geq 0 \\ \\
w(k)= S(k) w(-k), \qquad k \leq 0.
\end{array}\right.
\ene
Then.
\begin{enumerate}
\item
\beq\label{ab.2}
w(k)= \frac{1}{\sqrt{2\pi}}\, \int_0^\infty \left[ f(k,x)^\dagger +  S(-k)^\dagger f(-k,x)^\dagger\right] \phi(x) \, dx, \qquad k \in \ere.
\ene
\item
\beq\label{ab.2.1}
S(k) w(-k)= w(k), \qquad k \in \mathbb R.
\ene
\item
We have that  $\phi \in  H^1_{A,B}(\er)$  if and only if  $ w \in L^2_1(\ere),$ and for some $C_1, C_2 >0,$
\beq\label{ab.2.1.1.0}
C_1 \|w\|_{L^2_1(\ere)} \leq \|\phi\|_{ H^1_{A,B}(\er)}\leq C_2 \|w\|_{L^2_1(\ere)}.
\ene
\item
Further, assume that $ V \in L^\infty(\er).$ Then, we have that  $\phi \in H^2(\er)$ and it  satisfies the boundary condition \eqref{ap.3} if and only if  $ w \in L^2_2(\ere),$ and for some $C_1, C_2 >0,$
\beq\label{ab.2.1.1}
C_1 \|w\|_{L^2_2(\ere)} \leq \|\phi\|_{H^2(\er)}\leq C_2 \|w\|_{L^2_2(\ere)}.
\ene
\item Suppose that $V \in L^1_{2+j}(\er), j=1,2.$ Then,  $\phi \in L^2_j(\er)$ if and only if   $w \in H^j(\ere), j=1,2.$ Moreover,
\beq\label{ab.2.1.2}
C_1 \|\phi\|_{L^2_j(\er)}  \leq   \|w\|_{H^j(\ere)}  \leq C_2 \|\phi\|_{L^2_j(\er)}, \qquad j=1,2.
\ene
\end{enumerate}
\end{lemma2}

\begin{proof}
Let us prove 1. We define $q(k)$ by the right-hand side of \eqref{ab.2},
\begin{equation}
q(k):=\frac{1}{\sqrt{2\pi}}\,\int_{0}^{\infty}\left[  f(k,x)^{\dagger
}+S(-k)^{\dagger}f(-k,x)^{\dagger}\right]  \phi(x)\,dx,\qquad k\in{\mathbb{R}%
}. \label{ab.3}%
\end{equation}
By \eqref{gefoma} and \eqref{ab.1}, $w(k)=q(k),k\geq0.$ Suppose that $k\leq0.$
Then,
\[%
\begin{array}
[c]{l}%
w(k)=S(k)w(-k)=S(k)\frac{1}{\sqrt{2\pi}}\,\int_{0}^{\infty}\left[
f(-k,x)^{\dagger}+S(k)^{\dagger}f(k,x)^{\dagger}\right]  \phi(x)\,dx\\
=\frac{1}{\sqrt{2\pi}}\,\int_{0}^{\infty}\left[  S(-k)^{\dagger}%
f(-k,x)^{\dagger}+f(k,x)^{\dagger}\right]  \phi(x)\,dx=q(k),
\end{array}
\]
where we used \eqref{UnitaritySM}. It follows that also $w(k)=q(k),k\leq0.$
This proves \eqref{ab.2}. Equation \eqref{ab.2.1} follows from
\eqref{UnitaritySM} and \eqref{ab.2}. We now prove 3. Since $H$ has no
negative eigenvalues, $\mathbf{F}$ is unitary. If $\phi\in H_{A,B}%
^{1}({\mathbb{R}})$ it follows from \eqref{UnitaritySM}, \eqref{61} and
\eqref{ab.1} that $w\in L_{1}^{2}(\mathbb{R}^{+})$ and,
\begin{equation}
\Vert w\Vert_{L_{1}^{2}({\mathbb{R}})}\leq\sqrt{2}\Vert w\Vert_{L_{1}%
^{2}(\mathbb{R}^{+})}\leq C\Vert\phi\Vert_{H_{A,B}^{1}}(\mathbb{R}^{+}).
\label{ab.3.0}%
\end{equation}
Suppose that $w\in L_{1}^{2}({\mathbb{R}}).$ Denote by $w_{+}$ the restriction
of $w$ to $\mathbb{R}^{+}.$ Then, $\phi=\mathbf{F}^{\dagger}w_{+}.$ Further,
$kw_{+}\in L^{2}(\mathbb{R}^{+}),$ and by \eqref{spectralrepr} and functional
calculus, $\phi\in D[\sqrt{H}]=Q(h)=H_{A,B}^{1}(\mathbb{R}^{+})$ (see
\eqref{quadratic}). Moreover, since $D[\sqrt{H}]=H_{A,B}^{1}(\mathbb{R}^{+}),$
it follows from \eqref{spectralrepr},
\begin{equation}
\Vert\phi\Vert_{H_{A,B}^{1}(\mathbb{R}^{+})}\leq C\left[  \Vert\sqrt{H}%
\phi\Vert_{L^{2}(\mathbb{R}^{+})}+\Vert\phi\Vert_{L^{2}(\mathbb{R}^{+}%
)}\right]  \leq C\Vert w_{+}\Vert_{L_{1}^{2}(\mathbb{R}^{+})}\leq C\Vert
w\Vert_{L_{1}^{2}({\mathbb{R}})}. \label{ab.3.0.0}%
\end{equation}
Equation \eqref{ab.2.1.1.0} follows from \eqref{ab.3.0} and \eqref{ab.3.0.0}.

Let us prove item 4. Suppose that $\phi\in H^{2}(\mathbb{R}^{+})$ and that it
satisfies the boundary condition \eqref{ap.3}. Then, by \eqref{domainh2}
$\phi\in D[H]$ and by \eqref{spectralrepr},
\begin{equation}
\Vert k^{2}w\Vert_{L^{2}(\mathbb{R}^{+})}=\Vert H\phi\Vert_{L^{2}%
(\mathbb{R}^{+})}\leq\Vert-\phi^{\prime\prime}\Vert_{L^{2}(\mathbb{R}^{+}%
)}+\Vert V\phi\Vert_{L^{2}(\mathbb{R}^{+})}\leq C\Vert\phi\Vert_{H^{2}%
(\mathbb{R}^{+})}. \label{ab.3.1}%
\end{equation}
Further by \eqref{UnitaritySM}, \eqref{ab.2.1} and \eqref{ab.3.1}
\begin{equation}
\Vert w\Vert_{L_{2}^{2}({\mathbb{R}})}=\sqrt{2}\Vert w\Vert_{L_{2}%
^{2}(\mathbb{R}^{+})}\leq C\Vert\phi\Vert_{H^{2}({\mathbb{R}})}.
\label{ab.3.2}%
\end{equation}
Assume that $w\in L_{2}^{2}({\mathbb{R}}).$Further, by \eqref{spectralrepr}
$\phi\in D[H],$ and
\begin{equation}
\Vert H\phi\Vert_{L^{2}(\mathbb{R}^{+})}=\Vert k^{2}w_{+}\Vert_{L^{2}%
(\mathbb{R}^{+})}. \label{ab.3.3}%
\end{equation}
Then, moreover, as $D[H]\subset H^{2}(\mathbb{R}^{+}),$ using \eqref{ab.3.3}
we have,
\begin{equation}
\Vert\phi\Vert_{H^{2}(\mathbb{R}^{+})}\leq C[\Vert H\phi\Vert_{L^{2}%
(\mathbb{R}^{+})}+\Vert\phi\Vert_{L^{2}(\mathbb{R}^{+})}]\leq C\Vert
w_{+}\Vert_{L_{2}^{2}(\mathbb{R}^{+})}\leq C\Vert w\Vert_{L_{2}^{2}%
({\mathbb{R}})}. \label{ab.3.2.1}%
\end{equation}
Equation \eqref{ab.2.1.1} follows from \eqref{ab.3.1} and \eqref{ab.3.2.1}.
This completes the proof of 4. Let us prove 5. By \eqref{61BIS.2} and
\eqref{61BIS.3} if $\phi\in L_{j}^{2}(\mathbb{R}^{+}),j=1,2$ we have that
$w\in H^{j}(\mathbb{R}^{+})\cup H^{j}(\mathbb{R}^{-}),j=1,2.$ However, by
\eqref{ab.2} $w$ and $\partial_{k}w$ are continuous at $k=0,$ and then, $w\in
H^{j}({\mathbb{R}}),j=1,2.$ Moreover, using \eqref{UnitaritySM},
\eqref{scatmatrixderiv}, \eqref{scatmatrixsecondderiv}, \eqref{61BIS.2},
\eqref{61BIS.3}, and \eqref{ab.1},
\begin{equation}
\Vert w\Vert_{H^{j}({\mathbb{R}})}\leq C\Vert\phi\Vert_{L_{j}^{2}%
(\mathbb{R}^{+})},\qquad j=1,2 \label{ab.3.3.3}%
\end{equation}
Suppose that $w\in H^{j}({\mathbb{R}}),$ $j=1,2.$ Since for $k\geq0,$
$w_{+}(k)=w(k)=\mathbf{F}\phi,$ we have,
\begin{equation}
\phi(x)=(\mathbf{F}^{\dagger}w_{+})(x)=\frac{1}{\sqrt{2\pi}}\,\int_{0}%
^{\infty}\left[  f(k,x)+f(-k,x)S(-k)\right]  w(k)\,dk, \label{ab.4}%
\end{equation}
and by \eqref{UnitaritySM}, and \eqref{ab.1}
\begin{equation}
\phi(x)=\frac{1}{\sqrt{2\pi}}\,\int_{-\infty}^{\infty}f(k,x)w(k)\,dk.
\label{ab.5}%
\end{equation}
By \eqref{ab.5} and as $f(k,x)=e^{ikx}m(k,x),$
\begin{equation}
x^{j}\phi(x)=R_{1,j}(x)+R_{2,j}(x),\qquad j=1,2, \label{ab.6}%
\end{equation}
where,
\begin{equation}
R_{1,j}(x):=-\frac{1}{\sqrt{2\pi}}\,\int_{-\infty}^{\infty}(\partial_{k}%
^{j}e^{ikx})w(k)\,dk,\qquad j=1,2, \label{ab.7}%
\end{equation}
and
\begin{equation}
R_{2,j}(x):=\frac{1}{\sqrt{2\pi}}\,\int_{-\infty}^{\infty}e^{ikx}%
x^{j}(m(k,x)-I)w(k)\,dk,\qquad j=1,2. \label{ab.8}%
\end{equation}
Integrating by parts in \eqref{ab.7} and using Parseval's identity we obtain,
\begin{equation}
\Vert R_{1,j}\Vert_{L^{2}(\mathbb{R}^{+})}\leq C\Vert w\Vert_{H^{j}%
({\mathbb{R}})},\qquad j=1,2. \label{ab.9}%
\end{equation}
Further, by \eqref{3.2}
\begin{equation}
\Vert R_{2,j}\Vert_{L^{2}(\mathbb{R}^{+})}\leq C\Vert w\Vert_{L^{2}%
({\mathbb{R}})},\qquad j=1,2. \label{ab.10}%
\end{equation}
By \eqref{ab.6}, \eqref{ab.9} and \eqref{ab.10}
\begin{equation}
\Vert\phi\Vert_{L_{j}^{2}(\mathbb{R}^{+})}\leq C\Vert w\Vert_{H^{j}%
({\mathbb{R}})},\qquad j=1,2. \label{ab.11}%
\end{equation}
and it follows that $\phi\in L_{j}^{2}(\mathbb{R}^{+}),j=1,2.$ Finally,
\eqref{ab.2.1.2} follows from \eqref{ab.3.3.3} and \eqref{ab.11}.
\end{proof}

\subsection{The free evolution group}

\label{freeev} In this subsection we consider the case where the potential $V
$ is identically zero. Let us denote by $H_{0}$ the free Hamiltonian with
potential identically zero,
\[
H_{0}:= H_{A,B,0}.
\]
The free evolution group is given by,
\[
\mathcal{U}_{0}(t):= e^{-it H_{0}}.
\]
Let us denote by $P_{c}(H_{0})$ the projector onto the continuous subspace of
$H_{0},$ that coincides with the projector onto the absolutely continuous
subspace.
As above, we denote by $S_{0}(k)$ the scattering matrix
\eqref{Scatteringmatrix} in the case where the potential is zero, and by
$\Psi_{0}(k)$ the physical solution for zero potential. It is given by,
\begin{equation}
\label{ap.57}\Psi_{0}(k,x)= e^{-ikx}+ e^{ikx} S_{0}(k),
\end{equation}
where we used that for potential zero the Jost solution is given by $e^{ikx}.
$
To simplify the notation we denote by $\mathbf{F}_{0}$ the generalized Fourier
map \eqref{gefoma} in the plus case and with zero potential. It is given by,
\begin{equation}
\label{ap.58}(\mathbf{F}_{0} Y)(k):=\sqrt{\frac{1}{2\pi}} \,\int_{0}^{\infty
}\, \left(  \Psi_{0}(-k,x)\right)  ^{\dagger}\, Y(x) \, dx,
\end{equation}
and the adjoint is given by,
\begin{equation}
\label{ap.59}(\mathbf{F}_{0}^{\dagger}Y)(x):=\sqrt{\frac{1}{2\pi}} \,\int%
_{0}^{\infty}\, \ \Psi_{0}(-k,x) \, Y(k) \, dk.
\end{equation}

The case where $V$ is zero \eqref{sp2} reads
\begin{equation}%
\begin{array}
[c]{l}%
\mathcal{U}_{0}(t)\mathbf{F}_{0}^{\dagger}\phi=\mathcal{U}_{0}%
(t)\,P_{\operatorname*{c}}(H_{0})\mathbf{F}_{0}^{\dagger}\phi(x)=\frac
{1}{\sqrt{2\pi}}\int_{0}^{\infty}e^{ikx}\,e^{-itk^{2}}\phi\left(  k\right)
dk\\
\\
+\frac{1}{\sqrt{2\pi}}\int_{0}^{\infty}\,e^{-ikx}e^{-itk^{2}}\,S_{0}%
(-k)\,\phi\left(  k\right)  \,dk\\
\\
=\frac{1}{\sqrt{2\pi}}\int_{-\infty}^{\infty}\,e^{ikx}\,e^{-itk^{2}%
}\,(\mathcal{E}_{0}\phi)(k)\,dk,
\end{array}
\label{sp2.0}%
\end{equation}
where we define
\begin{equation}
\mathcal{E}_{0}\phi(k)=\left\{
\begin{array}
[c]{l}%
\phi(k),\qquad k\geq0,\\
S_{0}(k)\phi(-k),\qquad k<0.
\end{array}
\right.  \label{sp3.xx}%
\end{equation}
Equation \eqref{sp3.uu} is now given by,%

\begin{equation}
\label{sp3.0}\|\mathcal{E}_{0}\phi\|_{L^{2}( {\mathbb{R}})}= \sqrt{2}
\|\phi\|_{L^{2}(\mathbb{R}^{+})}, \qquad\phi\in L^{2}(\mathbb{R}^{+}).
\end{equation}
Moreover, \eqref{sp4} reads,
\begin{equation}
\label{sp4.0}S_{0}(k) (\mathcal{E}_{0} \phi)(-k) = (\mathcal{E}_{0}\phi)(k),
\qquad k \in{\mathbb{R}}, \phi\in L^{2}(\mathbb{R}^{+}).
\end{equation}

We denote%
\begin{equation}
\label{sp5.xxx}\mathcal{V}\left(  t\right)  \phi=\sqrt{\frac{it}{2\pi}}%
\int_{-\infty}^{\infty}e^{-it\left(  k-\frac{x}{2}\right)  ^{2}} \phi\left(
k\right)  dk,
\end{equation}
and%

\begin{equation}
\label{sp7.ppp}\mathcal{Q}_{0}(t):= \mathcal{V}(t) \mathcal{E}_{0}.
\end{equation}

Equation \eqref{2.5} now reads
\begin{equation}
\mathcal{U}_{0}\left(  t\right)  \mathbf{F_{0}}^{\dagger}\phi=M\mathcal{D}%
_{t}\mathcal{Q}_{0}\left(  t\right)  \phi,\qquad\phi\in L^{2}(\mathbb{R}^{+}).
\label{2.5.1}%
\end{equation}
Further, \eqref{ap.65} is now given by,
\begin{equation}
\mathcal{U}_{0}\left(  t\right)  P_{c}(H_{0})\psi=M\mathcal{D}_{t}%
\mathcal{Q}_{0}\left(  t\right)  \mathbf{F}_{0}\psi,\qquad\psi\in
L^{2}(\mathbb{R}^{+}). \label{ap.65.0}%
\end{equation}
Moreover, if $H_{0}$ has no bound states, $P_{c}(H_{0})=I,$ and \eqref{isom1}
reads now,
\begin{equation}
\left\Vert \mathcal{Q}_{0}\left(  t\right)  \phi\right\Vert _{{L}%
^{2}(\mathbb{R}^{+})}=\left\Vert \phi\right\Vert _{{L}^{2}(\mathbb{R}^{+})},
\label{isom1.xxx}%
\end{equation}
and \eqref{2.6} is given by,
\begin{equation}
\mathcal{Q}_{0}^{-1}\left(  t\right)  \phi=\mathbf{F}_{0}\,e^{itH_{0}%
}\,M\mathcal{D}_{t}\phi=e^{itk^{2}}\mathbf{F}_{0}M\mathcal{D}_{t}\phi.
\label{2.6.1.yyy}%
\end{equation}
Moreover, for $V$ equal zero \eqref{sp7.ll} reads,
\begin{equation}
\mathbf{F}_{0}e^{itH_{0}}=\mathcal{Q}_{0}^{-1}\,\mathcal{D}_{t}^{-1}%
\overline{M}, \label{sp7.0}%
\end{equation}
and \eqref{isom2} corresponds to,
\begin{equation}
\left\Vert \mathcal{Q}_{0}^{-1}\left(  t\right)  \phi\right\Vert _{{L}%
^{2}(\mathbb{R}^{+})}=\left\Vert \phi\right\Vert _{{L}^{2}(\mathbb{R}^{+})}.
\label{isom2.0}%
\end{equation}
Further \eqref{winverse} reads,
\begin{equation}
\mathcal{Q}_{0}^{-1}\left(  t\right)  \phi(k)=S\left(  k\right)
\mathcal{V}_{+}\left(  t\right)  \phi+\mathcal{V}_{-}\left(  t\right)
\phi,\qquad k\geq0, \label{winverse.0}%
\end{equation}
where%
\begin{equation}
\mathcal{V}_{\pm}\left(  t\right)  \phi(k)=\sqrt{\frac{t}{2\pi i}}\int%
_{0}^{\infty}e^{it\left(  k\pm\frac{x}{2}\right)  ^{2}}\phi\left(  x\right)
dx,\qquad k\in{\mathbb{R}}. \label{sp8.1}%
\end{equation}
As in the case of nontrivial potential, for convenience, we have defined the
quantities $\mathcal{V}_{\pm}(t)$ for $k\in{\mathbb{R}},$ but in
\eqref{winverse.0} we only use then for $k\geq0.$ We also find it useful to
introduce the following quantity,
\begin{equation}
\widehat{\mathcal{V}}(t)\phi(k):=S_{0}\left(  k\right)  \mathcal{V}^{+}\left(
t\right)  \phi+\mathcal{V}^{-}\left(  t\right)  \phi,\qquad k\in{\mathbb{R}}.
\label{sp9.1}%
\end{equation}
Observe that,
\begin{equation}
\widehat{\mathcal{V}}(t)\phi(k)=\mathcal{Q}_{0}^{-1}\left(  t\right)
\phi(k),\qquad k\geq0. \label{sp10.1}%
\end{equation}
Furthermore, by \eqref{UnitaritySM}
\begin{equation}
S_{0}(k)\widehat{\mathcal{V}}(t)\phi(-k)=\widehat{\mathcal{V}}(t)\phi
(k),\qquad k\in{\mathbb{R}}. \label{sp11.2}%
\end{equation}
Since for $\phi\in L^{2}(\mathbb{R}),$
\begin{equation}
\sqrt{\frac{it}{2\pi}}\int_{-\infty}^{\infty}e^{-\frac{it}{2}\left(
k-x\right)  ^{2}}\phi\left(  k\right)  dk=\mathcal{F}e^{\frac{ix^{2}}{2t}%
}\mathcal{F}^{-1}\phi, \label{sp11.2.0}%
\end{equation}
it follows from \eqref{sp8.1}
\begin{equation}
\mathcal{V}\left(  t\right)  \phi(x)=\,\frac{1}{\sqrt{2}}\left(
\mathcal{F}^{-1}M\mathcal{F}\phi\right)  \left(  \frac{x}{2}\right)  .
\label{sp11.2.1}%
\end{equation}
Then, by Parseval's identity, we see that
\begin{equation}
\left\Vert \mathcal{V}\left(  t\right)  \phi\right\Vert _{{L}^{2}(\mathbb{R}%
)}=\left\Vert \mathcal{F}^{-1}M\mathcal{F}\phi\right\Vert _{{L}^{2}%
(\mathbb{R})}=\left\Vert \phi\right\Vert _{{L}^{2}(\mathbb{R})}. \label{2.11}%
\end{equation}
Moreover,
\begin{equation}%
\begin{array}
[c]{l}%
\left\Vert \mathcal{V}\left(  t\right)  \phi(x)-\frac{1}{\sqrt{2}}%
\phi(x/2)\right\Vert _{{L}^{2}(\mathbb{R})}=\left\Vert \mathcal{F}^{-1}\left(
M-1\right)  \mathcal{F}\phi\right\Vert _{{L}^{2}(\mathbb{R})}\\
\leq C\left\Vert \left(  \frac{x^{2}}{4t}\right)  ^{\frac{1}{2}}%
\mathcal{F}\phi\right\Vert _{{L}^{2}(\mathbb{R})}\leq C|t|^{-\frac{1}{2}%
}\left\Vert \partial_{k}\phi\right\Vert _{{L}^{2}(\mathbb{R})},
\end{array}
\label{2.11.1}%
\end{equation}
and similarly,
\begin{equation}
\left\Vert \partial_{x}\left(  \mathcal{V}\left(  t\right)  \phi(x)-\frac
{1}{\sqrt{2}}\phi(x/2)\right)  \right\Vert _{{L}^{2}(\mathbb{R})}\leq
C\left\Vert \partial_{k}\phi\right\Vert _{{L}^{2}(\mathbb{R})}. \label{2.11.2}%
\end{equation}
Then, using that
\[
\left\Vert f\right\Vert _{{L}^{\infty}(\mathbb{R})}\leq C\left\Vert
f\right\Vert _{{L}^{2}(\mathbb{R})}^{\frac{1}{2}}\left\Vert \partial
_{x}f\right\Vert _{{L}^{2}(\mathbb{R})}^{\frac{1}{2}},
\]
we obtain the estimate
\begin{equation}
\left\Vert \mathcal{V}\left(  t\right)  \phi(x)-\frac{1}{\sqrt{2}}%
\phi(x/2)\right\Vert _{{L}^{\infty}(\mathbb{R})}\leq C|t|^{-\frac{1}{4}%
}\left\Vert \partial_{k}\phi\right\Vert _{{L}^{2}(\mathbb{R})}. \label{2.3}%
\end{equation}
Moreover, by \eqref{l2.4} with $m^{\dagger}(k,x)=I,$ and Parseval's identity,
we obtain,
\begin{equation}
\Vert\mathcal{V}_{\pm}\phi\Vert_{L^{2}({\mathbb{R}})}\leq C\Vert\phi
\Vert_{L^{2}(\mathbb{R}^{+})}. \label{EST8.0}%
\end{equation}
Further, by \eqref{UnitaritySM}, and \eqref{EST8.0},
\begin{equation}
\Vert\widehat{\mathcal{V}}\phi\Vert_{L^{2}({\mathbb{R}})}\leq C\Vert\phi
\Vert_{L^{2}(\mathbb{R}^{+})}. \label{isomafr}%
\end{equation}
Furthermore, as
\[
\pm\partial_{x}e^{it(k\pm x/2)^{2}}=\frac{1}{2}\partial_{k}e^{it(k\pm
x/2)^{2}},
\]
integrating by parts in \eqref{sp5.xxx} and using \eqref{2.11} we get,
\begin{equation}
\left\Vert \partial_{x}^{n}\mathcal{V}\left(  t\right)  \phi\right\Vert
_{L^{2}({\mathbb{R}})}\leq C\left\Vert \phi\right\Vert _{H^{n}({\mathbb{R}}%
)},\text{ n}\in\mathbb{N\cup}\{0\}\text{.} \label{EST5fr}%
\end{equation}
Suppose that $\phi\in H^{1}(\mathbb{R}^{+}),$ and that besides being in
$L^{2}(\mathbb{R}^{+}),\partial_{k}\phi$ admits the following decomposition,
$\partial_{k}\phi=\phi_{1}+\phi_{2},$ where $\phi_{1}\in L^{q_{1}}%
(\mathbb{R}^{+}),$ and $\phi_{2}\in L^{q_{2}}(\mathbb{R}^{+}),$ for some
$q_{1},q_{2},$ with $1\leq q_{1},q_{2}<\infty.$ Then, as in the proof of
\eqref{89}, but with $m(k,x)=I,$ we obtain,
\begin{equation}
\Vert\mathcal{V}_{\pm}\phi\Vert_{L^{\infty}({\mathbb{R}})}\leq C\left(
\left\Vert \phi\right\Vert _{L^{\infty}(\mathbb{R}^{+})}+|t|^{-\frac{1}%
{2p_{1}}}\left\Vert \phi_{1}\right\Vert _{L^{q_{1}}(\mathbb{R}^{+}%
)}+|t|^{-\frac{1}{2p_{2}}}\left\Vert \phi_{2}\right\Vert _{L^{q_{2}%
}(\mathbb{R}^{+})}\right)  , \label{89.xx}%
\end{equation}
where $\frac{1}{p_{1}}+\frac{1}{q_{1}}=1,$ and $\frac{1}{p_{2}}+\frac{1}%
{q_{2}}=1,$ and by \eqref{UnitaritySM}, \eqref{sp9.1} and \eqref{89.xx} we
get,
\begin{equation}
\Vert\widehat{\mathcal{V}}\phi\Vert_{L^{\infty}({\mathbb{R}})}\leq C\left(
\left\Vert \phi\right\Vert _{L^{\infty}(\mathbb{R}^{+})}+|t|^{-\frac{1}%
{2p_{1}}}\left\Vert \phi_{1}\right\Vert _{L^{q_{1}}(\mathbb{R}^{+}%
)}+|t|^{-\frac{1}{2p_{2}}}\left\Vert \phi_{2}\right\Vert _{L^{q_{2}%
}(\mathbb{R}^{+})}\right)  , \label{89.xxx}%
\end{equation}
with $\frac{1}{p_{1}}+\frac{1}{q_{1}}=1,$ and $\frac{1}{p_{2}}+\frac{1}{q_{2}%
}=1.$ Assume that $\phi\in H^{1}({\mathbb{R}}),$ and that besides being in
$L^{2}({\mathbb{R}}),\partial_{k}\phi$ admits the following decomposition,
$\partial_{k}\phi=\phi_{1}+\phi_{2},$ where $\phi_{1}\in L^{q_{1}}%
({\mathbb{R}}),$ and $\phi_{2}\in L^{q_{2}}({\mathbb{R}}),$ for some
$q_{1},q_{2},$ with $1\leq q_{1},q_{2}<\infty.$ Then, by \eqref{sp5.xxx}, as
in the proof of \eqref{EST4}, we prove,
\begin{equation}
\Vert\mathcal{V}\phi\Vert_{L^{\infty}(\mathbb{R}^{+}}\leq C\left(  \left\Vert
\phi\right\Vert _{L^{\infty}({\mathbb{R}})}+|t|^{-\frac{1}{2p_{1}}}\left\Vert
\phi_{1}\right\Vert _{L^{q_{1}}({\mathbb{R}})}+|t|^{-\frac{1}{2p_{2}}%
}\left\Vert \phi_{2}\right\Vert _{L^{q_{2}}({\mathbb{R}})}\right)  ,
\label{89.xxxx}%
\end{equation}
where $\frac{1}{p_{1}}+\frac{1}{q_{1}}=1,$ and $\frac{1}{p_{2}}+\frac{1}%
{q_{2}}=1.$

\begin{lemma}
The following estimates hold,
\begin{equation}
\left\Vert x\mathcal{V}(t)\phi\right\Vert _{L^{2}({\mathbb{R}})}\leq C\left(
\frac{1}{|t|}\Vert\partial_{k}\phi(k)\Vert_{L^{2}({\mathbb{R}})}+\Vert
k\phi(k)\Vert_{L^{2}({\mathbb{R}})}\right) , \label{fin.1}%
\end{equation}
and
\begin{equation}%
\begin{array}
[c]{l}%
\left\Vert x^{2}\mathcal{V}(t)\phi\right\Vert _{L^{2}({\mathbb{R}})}\leq
C\left(  \frac{1}{|t|}\Vert k\partial_{k}\phi(k)\Vert_{L^{2}({\mathbb{R}}%
)}+\frac{1}{|t|^{2}}\Vert\partial_{k}^{2}\phi(k)\Vert_{L^{2}({\mathbb{R}}%
)}+\frac{1}{|t|}\Vert\phi(k)\Vert_{L^{2}({\mathbb{R}})}+\Vert k^{2}\phi
\Vert_{L^{2}({\mathbb{R}})}\right)  .
\end{array}
\label{fin.2}%
\end{equation}

\end{lemma}

\begin{proof}
The lemma follows using that $x e^{\pm ikx}= \pm1/ it\, \partial_{k} e^{\pm
ikx}$ and integrating by parts.
\end{proof}


\subsection{Asymptotic estimates}

\label{aymp} We prepare the following lemma \begin{lemma2}
\label{Lemma3}Suppose that $V\in L^{1}_{5/2+\delta}(\er),$ for some $\delta>0.$ Then
\begin{enumerate}
\item
For each $ a >0$ there is a constant $C_a>0$ such that
\begin{equation}
\left\Vert \left(  \mathcal{W}\left(  t\right)  -\mathcal{V}\left(  t\right)
\right)  \phi\right\Vert _{L^{2}(\ere)}\leq\frac{C_a}{\sqrt{|t|}}\left\Vert
\phi\right\Vert _{H^{1}(\ere)}, \qquad |t| \geq a. \label{L2}%
\end{equation}
\item
\begin{equation}
\left\Vert \mathcal{W}\left(  t\right)  \phi-     \frac{1}{\sqrt{2}}m\left(
\frac{x}{2},tx\right)  \phi\left(  \frac{x}{2}\right)  \right\Vert
_{L^{\infty}(\ere)}\leq\frac{C}{|t|^{1/4}}\left\Vert \phi\right\Vert _{H^{1}(\ere)}.
\label{Linf}%
\end{equation}
\end{enumerate}
\end{lemma2}.

\begin{proof}
In order to prove (\ref{L2}) we write%
\begin{equation}
\left(  \mathcal{W}\left(  t\right)  -\mathcal{V}\left(  t\right)  \right)
\phi=\sqrt{\frac{it}{2\pi}}\int_{-\infty}^{\infty}e^{-it\left(  k-\frac{x}%
{2}\right)  ^{2}}\left(  m\left(  k,tx\right)  -1\right)  \phi\left(
k\right)  dk. \label{26}%
\end{equation}
Using that%
\begin{equation}
\label{26.1}e^{-it\left(  k-\frac{x}{2}\right)  ^{2}}=\frac{\partial
_{k}\left(  \left(  k-\frac{x}{2}\right)  e^{-it\left(  k-\frac{x}{2}\right)
^{2}}\right)  }{1-2it\left(  k-\frac{x}{2}\right)  ^{2}}%
\end{equation}
and integrating by parts we get%
\begin{align}
\int_{-\infty}^{\infty}e^{-it\left(  k-\frac{x}{2}\right)  ^{2}}g(k) dk  &
=-\int_{-\infty}^{\infty}\frac{\left(  k-\frac{x}{2}\right)  e^{-it\left(
k-\frac{x}{2}\right)  ^{2}}}{1-2it\left(  k-\frac{x}{2}\right)  ^{2}}%
\partial_{k}g(k) dk\nonumber\\
&  -2\int_{-\infty}^{\infty}\frac{2it\left(  k-\frac{x}{2}\right)
^{2}e^{-it\left(  k-\frac{x}{2}\right)  ^{2}}}{\left(  1-2it\left(  k-\frac
{x}{2}\right)  ^{2}\right)  ^{2}}g(k) dk. \label{27}%
\end{align}
Then, taking $g=\left(  m\left(  k,tx\right)  -1\right)  \phi\left(  k\right)
,$ by (\ref{3.2}), (\ref{3.5}) with $\delta>1/2,$ via Sobolev's inequality
theorem we get%

\begin{equation}
\label{www}%
\begin{array}
[c]{l}%
\left\Vert \int_{-\infty}^{\infty}e^{-it\left(  k-\frac{x}{2}\right)  ^{2}%
}\left(  m\left(  k,tx\right)  -1\right)  \phi\left(  k\right)  dk\right\Vert
_{L^{2}( {\mathbb{R}})} \leq C\left\Vert \left\langle tx\right\rangle
^{-\delta}\right\Vert _{L^{2}( {\mathbb{R}})}\left\Vert \frac{k}{1+|t|k^{2}%
}\right\Vert _{L^{2}( {\mathbb{R}})}\left\Vert \phi\right\Vert _{H^{1}(
{\mathbb{R}})}\\
\\
+C\left\Vert \left\langle tx\right\rangle ^{-1-\delta}\right\Vert _{L^{2}(
{\mathbb{R}})}\left\Vert \frac{1}{1+|t|k^{2}}\right\Vert _{L^{1}( {\mathbb{R}%
})}\left\Vert \phi\right\Vert _{L^{\infty}( {\mathbb{R}})} \displaystyle \leq
\frac{C_{a}}{|t|}\left\Vert \phi\right\Vert _{H^{1}( {\mathbb{R}})},
\qquad|t|\geq a.
\end{array}
\end{equation}
Hence, (\ref{L2}) follows from (\ref{26}), \eqref{www}.

Using that
\[
\sqrt{\frac{it}{2\pi}}\int_{-\infty}^{\infty}e^{-itk^{2}}dk=\frac{1}{\sqrt{2}%
},
\]
we write
\begin{equation}
\mathcal{W}\left(  t\right)  \phi(x)= \frac{1}{\sqrt{2}}m\left(  \frac{x}%
{2},tx\right)  \phi\left(  \frac{x}{2}\right)  +\sqrt{\frac{it}{2\pi}}%
\int_{-\infty}^{\infty}e^{-it\left(  k-\frac{x}{2}\right)  ^{2}}g_{1}dk,
\label{28}%
\end{equation}
where we denote $g_{1}(k,x)=\left(  m\left(  k,tx\right)  \phi\left(
k\right)  -m\left(  \frac{x}{2},tx\right)  \phi\left(  \frac{x}{2}\right)
\right)  .$ Note that,
\begin{equation}
\label{28.www}\left|  g_{1}(k, x)\right|  \leq\left|  \int_{x/2}^{k}
\partial_{s} g_{1}(s,x) ds \right|  \leq\sqrt{|k- x/2|} \, \|\partial_{k}
g_{1}(\cdot, x)\|_{L^{2}( {\mathbb{R}})}.
\end{equation}
By (\ref{27}) with $g=g_{1}$ and \eqref{28.www} we show%
\begin{equation}
\label{zzuuv}%
\begin{array}
[c]{l}%
\left\Vert \int_{-\infty}^{\infty}e^{-it\left(  k-\frac{x}{2}\right)  ^{2}%
}g_{1}(\cdot,x) dk\right\Vert _{L^{\infty}(\mathbb{R}^{+})} \leq C\left\Vert
\frac{k}{1+|t|k^{2}}\right\Vert _{L^{2}( {\mathbb{R}})} \, \sup_{x
\in{\mathbb{R}}} \left\Vert \partial_{k}g_{1}(\cdot,x)\right\Vert _{L^{2}(
{\mathbb{R}})}\\
\\
+C\left\Vert \frac{\sqrt{k}}{1+|t|k^{2}}\right\Vert _{L^{1}( {\mathbb{R}})}
\sup_{x \in\mathbb{R}^{+}} \left\Vert \partial_{k} g_{1}(\cdot, x)\right\Vert
_{L^{2}( {\mathbb{R}})} \leq\frac{C}{|t|^{3/4}} \sup_{x \in\mathbb{R}^{+}}
\left\Vert \partial_{k}g_{1}(\cdot,x)\right\Vert _{L^{2}( {\mathbb{R}})}%
\end{array}
\end{equation}
Then, from (\ref{3.2}), (\ref{3.5}) and \eqref{zzuuv} we deduce,
\[
\left\Vert \sqrt{\frac{it}{2\pi}}\int_{-\infty}^{\infty}e^{-it\left(
k-\frac{x}{2}\right)  ^{2}}g_{1}dk\right\Vert _{L^{\infty}(\mathbb{R}^{+}%
)}\leq\frac{C}{|t|^{1/4}}\left\Vert \phi\right\Vert _{H^{1}( {\mathbb{R}})}.
\]
Together with (\ref{28}), this proves (\ref{Linf}).
\end{proof}

We prepare the following lemma.

\begin{lemma2}
Assume  that $ V\in L^1_{2+\delta}(\er),$  for some $ \delta \geq 0.$
\begin{enumerate}
\item
We have,
\begin{equation}
\left\vert \left(  \left(  \mathcal{W}\left(  t\right)  -\mathcal{V}\left(
t\right)  \right)  \phi \right)  \left(  x\right)  \right\vert \leq C\sqrt
{|t|}\left\langle tx\right\rangle ^{-1-\delta}\left\Vert \phi\right\Vert _{L^{2}(\ere)}.
\label{EST11b}%
\end{equation}%
\item
Assume that $\phi \in H^1(\ere),$ and that besides being in $L^2(\ere), \partial_k \phi$ admits the following decomposition,  $\partial_{k}\phi=\phi_{1}+\phi_{2},$  where $ \phi_1 \in L^{q_1}(\ere),$ and $ \phi_2 \in L^{q_2}(\ere),$ for some $q_1, q_2,$ with $ 1\leq q_1, q_2 < \infty.$ Then,  for $ 1 \leq l < \infty,$ we have,
\begin{equation}\begin{array}{l}
\left\vert   \left(  \mathcal{W}\left(  t\right)  -\mathcal{V}\left(
t\right)  \right)  \phi\left(  x\right)  \right\vert \leq
C \left\langle tx\right\rangle ^{-1-\delta}
\left(  \left\Vert \phi\right\Vert _{L^{\infty}(\ere)}+|t|^{-\frac{1}{2p}}\left\Vert \left\langle k\right\rangle ^{-1} \phi_1\right\Vert _{L^{q}(\ere)}\right.
\\\\ \left. +|t|^{-\frac{1}{2p_{1}}}\left\Vert \left\langle
k\right\rangle ^{-1} \phi_2 \right\Vert _{L^{q_{1}}(\ere)}\right)+ C_l |t|^{-\frac{1}{2l}}\left \langle tx\right\rangle^{-\delta}\,
\|\phi\|_{L^\infty(\ere)},
\label{EST11}
\end{array}
\end{equation}
with  $ \frac{1}{p_1}+ \frac{1}{q_1}=1,$ and $ \frac{1}{p_2}+ \frac{1}{q_2}=1$, and where $C_l$ depends on $l.$
\item
Let $n\in\mathbb{N}\cup\{0\}$ and \bigskip\bigskip assume that $V\in
L^{1}_{2+\delta}\left( \er\right) ,$ for some $\delta>n.$ Then
\begin{equation}
\left\Vert \left(  \mathcal{W}_{\pm}\left(  t\right)  -\mathcal{V}%
_{  \pm  }\left(  t\right)  \right)  x^{n} \phi\right\Vert _{L^{2}(\ere)%
}\leq C|t|^{-n}\min\{|t|^{-1/2}\left\Vert \phi\right\Vert _{L^{\infty}(\ere)},\left\Vert
\phi\right\Vert _{L^{2}(\ere)}\}, \label{EST12j}%
\end{equation}
and%
\begin{equation}
\left\Vert \left( \widehat{ \mathcal{W}}\left(  t\right)  -\widehat{\mathcal{V}}\left(  t\right)  \right)  x^{n} \phi\right\Vert
_{L^{2}(\ere)}\leq C|t|^{-n}\min\{|t|^{-1/2}\left\Vert \phi\right\Vert _{L^{\infty}(\ere)%
},\left\Vert \phi\right\Vert _{L^{2}(\ere)}\}. \label{EST12}%
\end{equation}
\end{enumerate}
\end{lemma2}

\begin{proof}
By \eqref{sp5} and \eqref{sp5.xxx},
\begin{equation}
\left(  \mathcal{W}\left(  t\right)  -\mathcal{V}\left(  t\right)  \right)
\phi(x)=\sqrt{\frac{it}{2\pi}}\int_{-\infty}^{\infty}e^{-it\left(  k-\frac
{x}{2}\right)  ^{2}}[m\left(  k,tx\right)  -I]\phi\left(  k\right)  dk.
\label{xxx.1}%
\end{equation}
Then \eqref{EST11b}, follows from \eqref{3.2} and \eqref{xxx.1}. Further,
using (\ref{3.2}), \eqref{3.5}, \eqref{27} with\linebreak$g=(m(k,tx)-I)\phi
(k),$ and \eqref{xxx.1} we show (\ref{EST11}). Moreover, by \eqref{sp8} and
\eqref{sp8.1}
\begin{equation}
\left[  \left(  \mathcal{W}_{\pm}-\mathcal{V}_{\pm}\right)  x^{n}\phi\right]
(k)=\sqrt{\frac{t}{2\pi i}}\int_{0}^{\infty}e^{it\left(  k\pm\frac{x}%
{2}\right)  ^{2}}[m^{\dagger}\left(  \mp k,tx\right)  -I]x^{n}\phi\left(
x\right)  dx,\qquad k\in{\mathbb{R}}. \label{xxx.2}%
\end{equation}
Using (\ref{3.2}) we estimate,
\begin{align}
\left\vert \left(  \mathcal{W}_{\pm}\left(  t\right)  -\mathcal{V}_{\left(
\pm\right)  }\left(  t\right)  \right)  x^{n}\phi\right\vert  &  \leq
C\sqrt{t}\left\langle k\right\rangle ^{-1}\int_{0}^{\infty}\left\langle
tx\right\rangle ^{-1-\delta}x^{n}\left\vert \phi\left(  x\right)  \right\vert
dx\nonumber\\
&  \leq C\sqrt{t}\left\langle k\right\rangle ^{-1}\min\{\left\Vert
\left\langle tx\right\rangle ^{-1-\delta}x^{n}\right\Vert _{L^{1}({\mathbb{R}%
})}\left\Vert \phi\right\Vert _{L^{\infty}({\mathbb{R}})},\left\Vert
\left\langle tx\right\rangle ^{-1-\delta}x^{n}\right\Vert _{L^{2}({\mathbb{R}%
})}\left\Vert \phi\right\Vert _{L^{2}({\mathbb{R}})}\}. \label{33}%
\end{align}
Equation (\ref{EST12j}) follows from \eqref{33}. Estimate (\ref{EST12})
follows from (\ref{UnitaritySM}) and (\ref{EST12j}).
\end{proof}


\begin{thebibliography}{99}                                                                                               %


\bibitem {ablow}Ablowitz M. J., and Clarkson P.A., \textit{Solitons, Nonlinear
Evolution Equations and Inverse Scattering,} Canbridge University Press,
Cambridge, 1991.

\bibitem {segur}Ablowitz M. J., and Segur H., \textit{Solitons and the Inverse
Scattering Transform,} SIAM, Philadelphia, 1981.

\bibitem {adams}{Adams R.A., and Fournier J.J.F.}, {\ \textit{Sobolev spaces.
Second edition.} Pure and Applied Mathematics (Amsterdam) \textbf{140}.
Elsevier/Academic Press, Amsterdam, (2003)}

\bibitem {AgrMarch}Agranovich Z. S., and Marchenko V. A., \textit{The Inverse
Problem of Scattering Theory.} Gordon and Breach, New York, 1963.

\bibitem {akv}Aktosun T., Klaus M., and Van der Mee C., Small energy
asymptotics of the scattering matrix for the matrix Schr\"odinger equation on
the line, J. Math. Phys. \textbf{42}, (2001) 4627-4652 .

\bibitem {WederBook}Aktosun T., and Weder R, \textit{Direct and Inverse
Scattering for the Matrix Schr\"{o}dinger Equation}. Applied Mathematical
Sciences \textbf{203}, Springer, New York, 2021.

\bibitem {Carles3}Antonelli P., Carles R., and Silva J. D., Scattering for
nonlinear Schr\"{o}dinger equation under partial harmonic confinement, Comm.
Math. Phys. \textbf{334} (2015) 367--396.

\bibitem {Yajima}Asano N., Taniuti T., and Yajima N., Perturbation method for
nonlinear wave modulation, II, J. Math. Phys., \textbf{10}, (1969) 2020-2024.

\bibitem {Barab}Barab J.E., Non-existence of asymptotically free solutions for
nonlinear Schr\"{o}dinger equation, J. Math.Phys., \textbf{25} (1984) 3270--3273.



\bibitem {Berkolaio}Berkolaio G., and Kuchment P., \textit{Introduction to
Quantum Graphs.} Mathematical Surveys and Monographs \textbf{186} Amer. Math.
Soc., Providence, 2013.

\bibitem {Bespalov}Bespalov V.I., and Talanov V. I., Filamentary structure of
light beams in nonlinear liquids, JETP Lett., \textbf{3 }(1966) 307-310.

\bibitem {bourgain}Bourgain J., \textit{Global Solutions of Nonlinear
Schr\"odinger Equations.} Colloquium Publications \textbf{46} A. M. S.,
Providence, 1999.

\bibitem {caudrelier}Caudrelier V., On the inverse scattering method for
integrable PDEs on a star graph, Comm. Math. Phys. \textbf{338} (2015) 893-917.

\bibitem {Cazenave}T. Cazenave, \textit{Semilinear Schr\"{o}dinger equations},
Courant Institute of Mathematical Sciences, New York, Am. Math. Soc.,
Providence, RI, (2003).

\bibitem {CazenaveW1}{Cazenave T., and Weissler F. B.}, Rapidly decaying
solutions of the nonlinear Schr\"{o}\-din\-ger equation,{\ Comm. Math. Phys.
\textbf{147} (1992) 75--100.}

\bibitem {CN}{Cazenave T., and Naumkin I.}, Local existence, global existence,
and scattering for the nonlinear Schr\"{o}\-din\-ger equation,{\ Commun.
Contemp. Math. \textbf{19} 2 (2017).}

\bibitem {CN1}{Cazenave T., and Naumkin I.,} Modified scattering for the
critical nonlinear Schr\"{o}dinger equation, J. Funct. Anal. \textbf{274}
(2018) 402-432.


\bibitem {Chen}Chen G., Long-time dynamics of small solutions to 1d cubic
nonlinear Schr\"{o}dinger equations with a trapping potential (2021)
arXiv:2106.10106 [math.AP].

\bibitem {Pusateri}Chen G., and Pusateri F., The 1d nonlinear Schr\"{o}dinger
equation with a weighted $L^{1}$ potential, to appear in Analysis \& PDE.
ArXiv 1912.10949 [math.AP]



\bibitem {Cuccagna}Cuccagna S., Georgiev V., and Visciglia N., Decay and
scattering of small solutions of pure power NLS in $R$ with $p>3$ and with a
potential, Comm. Pure Appl. Math. \textbf{67 } (2014) 957--981.

\bibitem {Cuccagna1}Cuccagna S., and Maeda M., A survey on asymptotic
stability of ground states of nonlinear Schr\"{o}dinger equations II. Discrete
Contin. Dyn. Syst. Ser. S \textbf{14} (2021) 1693--1716.

\bibitem {Cuccagna2}Cuccagna, S., Maeda, M., On Selection of Standing Wave at
Small Energy in the 1D Cubic Schr\"{o}dinger Equation with a Trapping
Potential. Commun. Math. Phys. (2022) https://doi.org/10.1007/s00220-022-04487-7. arXiv:2109.08108 [math.AP].

\bibitem {Delort}Delort J.M., Modified scattering for odd solutions of cubic
nonlinear Schr\"{o}dinger equations with potential in dimension one, (2016) (hal-01396705).

\bibitem {deGennes}de Gennes P. G., \textit{Superconductivity of Metals and
Alloys}, Benjamin, New York, 1966.

\bibitem {deift}Deift P., and Park J., Long-time asymptotics for solutions of
the NLS equation with a delta potential and even initial data, Int. Math. Res.
Notices \textbf{24} (2011), 5505--5624

\bibitem {fokas}Fokas A., \textit{A Unified Approach to Boundary Value
Problems,} SIAM, Philadelphia, 2008.

\bibitem {Germain}Germain P., Pusateri F., and Rousset F., The nonlinear
Schr\"{o}dinger equation with a potential, Ann. Inst. H. Poincar\'{e} Anal.
Non Lin\'eaire \textbf{35} 6 (2018) 1477--1530.

\bibitem {Ginibre}Ginibre, J. Introduction aux \'{e}quations de
Schr\"{o}dinger non Iin\'{e}aires Cours de DEA 1994-1995, Paris Onze \'Edition
L 161, Universit\'{e} de Paris-Sud, Orsay.

\bibitem {GinibreV1}{Ginibre J., and Velo G.}, On a class of nonlinear
Schr\"{o}\-din\-ger equations. II. Scattering theory, general case, {\ J.
Funct. Anal. \textbf{32} (1979) 33--71.}

\bibitem {GinibreV2}{Ginibre J., and Velo G.}, On a class of nonlinear
Schr\"{o}\-din\-ger equations. III. Special theories in dimensions 1, 2 and
3,{\ Ann. Inst. H. Poincar\'{e} Section A \textbf{28} (1978) 287--316.}

\bibitem {GinibreOV}{Ginibre J., Ozawa T., and Velo G.}, On the existence of
the wave operators for a class of nonlinear Schr\"{o}\-din\-ger
equations,{\ Ann. Inst. H.~Poin\-ca\-r\'{e} Phys. Th\'{e}or. \textbf{60} 2
(1994) 211--239.}



\bibitem {haka1998}Hayashi N., Kaikina E.I., and Naumkin P.I., On the
scattering theory for the cubic nonlinear Schr\"{o}dinger and Hartree type
equations in one space dimension, Hokkaido Math. J. \textbf{27 } (1998) 651--667.

\bibitem {HayashiNau1}{Hayashi N., and Naumkin P. I.}, Asymptotics for large
time of solutions to the nonlinear Schr\"{o}\-din\-ger and Hartree equations,
{\ Amer. J. Math. \textbf{120} (1998) 369--389.}

\bibitem {HayashiNau2}{Hayashi N., and Naumkin P. I.}, Large time behavior for
the cubic nonlinear Schr\"{o}\-din\-ger equation, {\ Canad. J. Math.
\textbf{54} (2002) 1065--1085.}

\bibitem {HNST}{Hayashi N., Naumkin P.I. , Shimomura A., and Tonegawa S.},
Modified wave operators for nonlinear Schr\"{o}\-din\-ger equations in one and
two dimensions,{\ Electron. J. Differential Equations \textbf{62} (2004).}

\bibitem {H-O}Hayashi N., and Ozawa T., Scattering theory in the weighted
$L^{2}(${$\mathbb{R}$}$^{n})$ spaces for some Schr\"{o}dinger equations, Ann.
Inst. H. Poincar\'e Phys. Th\'{e}or., \textbf{48\ (}1988) 17-37.

\bibitem {hora}Hora A., and Obata N., \textit{Quantum Probability and Spectral
Analysis of Graphs,} Springer, Berlin, 2007.

\bibitem {Karpman}Karpman V.I., and Kruskal E. M., Modulated waves in a
nonlinear dispersive media, Sov. Phys.-JETP\textbf{\ 28} (1969) 277-281.

\bibitem {Kato}Kato T., \textit{Perturbation Theory for Linear Operators,
Second Edition,} Springer, Berlin, 1976.





\bibitem {Soffer}Lindblad H., L\"{u}hrmann J., and Soffer A., \textit{Decay
and Asymptotics for the One-Dimensional Klein--Gordon Equation with Variable
Coefficient Cubic Nonlinearities, }SIAM J. Math.Anal., \textbf{52} 6 (2020) 6379--6411.

\bibitem {Soffer1}Lindblad H., L\"uhrmann J., Schlag W., and Soffer A., On
modified scattering for 1 D quadratic Klein-Gordon equations with non-generic
potentials, arXiv 2012.15191v2 [math.AP].

\bibitem {Martel1}Martel Y., Asymptotic $N-$ soliton-like solutions of the
subcritical and critical generalized Korteweg-de Vries equations, Amer. J.
Math. \textbf{127} (2005) 1103--1140.

\bibitem {MartelMerle}Martel Y., and Merle F., Multi-solitary waves for
nonlinear Schr\"{o}dinger equations, Ann. Inst. H. Poincar\'{e} Anal. Non
Lin\'{e}aire \textbf{23 } (2006) 849-864.

\bibitem {Merle}Merle F., Construction of solutions with exactly k blow-up
points for the Schr\"{o}dinger equation with critical nonlinearity, Comm.
Math. Phys. \textbf{129} (1990) 223--240.

\bibitem {Mizumachi}Mizumachi T., Asymptotic stability of small solitary waves
to 1D nonlinear Schr\"{o}dinger equations with potential, J. Math. Kyoto Univ.
\textbf{48} (2008) 471--497.

\bibitem {NakanishiO}{Nakanishi K., and Ozawa T.}, Remarks on scattering for
nonlinear Schr\"{o}\-din\-ger equations,{\ NoDEA Nonlinear Differential
Equations Appl. \textbf{9} 1 (2002) 45--68.}

\bibitem {Ivan}Naumkin I., Sharp asymptotic behavior of solutions for cubic
nonlinear Schr\"{o}dinger equations with a potential, J. Math. Phys.
\textbf{57} (2016).

\bibitem {Ivan1}Naumkin I., Nonlinear Schr\"{o}dinger equations with
exceptional potentials, J. Differential Equations \textbf{265} (2018) 4575--4631.

\bibitem {RicardoIvan}Naumkin I., and Weder R., $L^{p}-L^{p^{\prime}}$
estimates for matrix Schr\"{o}dinger equations, J. Evol. Equ. \textbf{21}
(2021) 891--919.

\bibitem {nov}Novikov S., Manakov S. V., Pitaevskii L. P., and Zakharov V. E.,
\textit{Theory of Solitons: The Inverse Scattering Method,} Consultants
Bureau, New York, 1984.

\bibitem {Ozawa1}{Ozawa T.}, Long range scattering for the nonlinear
Schr\"{o}\-din\-ger equation in one space dimension, {\ Comm. Math. Phys.
\textbf{139} (1991) 479--493.}



\bibitem {racke}Racke R., \textit{Lectures in Nonlinear Evolution Equations.
Initial Value Problems.} Aspects of Mathematics \textbf{19}, F. Vieweg \& Son,
Braunschweig/ Wiesbaden, 1992.

\bibitem {rs.2}Reed M., and Simon B.,\textit{\ Methods of Modern Mathematical
Physics I Functional Analysis.} Academic Press, New York, 1972.




\bibitem {Segal1}Segal I. E., Non-linear semi-groups, Ann. of Math. (2),
\textbf{78 }(1963) 339--364.

\bibitem {Segal2}Segal I. E., Quantization and dispersion for non-linear
relativistic equations, \textit{Proceeding Conference Mathematical Theory of
Elementary Particles,} MIT Press, Cambridge, Mass. (1966), 79--108.

\bibitem {Segal3}Segal I. E., Dispersion for non-linear relativistic
equations, II, Ann. Sci. \'Ecole Norm. Sup. (4), \textbf{1} (1968), 459--497.

\bibitem {Shimizu}Shimizu K., and Ichikawa Y. H., Automodulation of ion
oscillation modes in plasma, J. Phys. Soc. Japan \textbf{33} (1972) 189-792.

\bibitem {Strauss}{Strauss W. A.}, Nonlinear scattering theory, in
\textit{Scattering Theory in Mathematical Physics}, NATO Advanced Study
Institutes Series \textbf{9} (1974) 53-78.

\bibitem {Strauss2}{Strauss W. A.}, Nonlinear scattering theory at low energy,
J. Funct. Anal. \textbf{41} (1981) ,110--133.

\bibitem {straussl}Strauss W. A., \textit{Nonlinear Wave Equations}, CBMS-RSMC
\textbf{73} Amer. Math. Soc., Providence, 1989.

\bibitem {Sulem}Sulem C. , and Sulem P.L., \textit{The nonlinear
Schr\"{o}dinger equation. Self-focusing and wave collapse,} App. Math.
Sciences,\textbf{\ 139} Springer, New York, (1999).

\bibitem {Taniuti}Taniuti T., and Washimi~H., Self trapping and instability of
hydromagnetic waves along the magnetic field in a cold plasma, Phys. Rev.
Lett., \textbf{21} (1968) 209-212.



\bibitem {Wederkg1}Weder R., Inverse scattering on the line for the nonlinear
Klein-Gordon equation with a potential, J. Math. Anal. Appl .\textbf{252}
(2000) 102-123.

\bibitem {Weder2000}Weder R., $L^{p}-L^{p^{\prime}}$ estimates for the
Schr\"{o}dinger equation on the line and inverse scattering for the nonlinear
Schr\"{o}dinger equation with a potential, J. Funct. Anal. \textbf{170} (2000) 37--68.

\bibitem {Wedercenter}Weder R., Center manifold for nonintegrable nonlinear
Schr\"odinger equations on the line, Comm. Math. Phys. \textbf{215} (2000), 343-356.

\bibitem {Weder2001}Weder R., Inverse scattering for the nonlinear
Schr\"odinger equation: reconstruction of the potential and the nonlinearity,
Math. Meth. Appl. Sci. \textbf{24} (2001) 245-254.

\bibitem {WederPAMS}Weder R., Inverse scattering for the nonlinear
Schr\"odinger equation II: reconstruction of the potential and the
nonlinearity in the multidimensional case , Proc. Amer. Math. Soc.
\textbf{129} (2001) 3637-3645.

\bibitem {Wederkg}Weder R., Multidimensional inverse scattering for the
nonlinear Klein-Gordon equation with a potential, J. Differerential Equations,
\textbf{184} (2002) 62-77.

\bibitem {Wederhalf}Weder R., The forced nonlinear Schr\"odinger equation with
a potential on the half-line, Math. Methods Appl. Sci.\textbf{28} (2005) 1237-1255.

\bibitem {Wederhalfscat}Weder R., Scattering for the forced nonlinear
Schr\"odinger equation with a potential on the half-line, Math. Methods Appl.
Sci. \textbf{28} (2005) 1219-1236.

\bibitem {Wlp}Weder R., The $L^{p}$ boundedness of the wave operators for
matrix Schr\"{o}dinger equations, arXiv: 1912.12793v3 [math-ph]. To appear in
J. Spectral Theory.

\end{thebibliography}
\end{document}